\documentclass[11pt,a4paper]{amsart} 
\usepackage[english]{babel}
\usepackage[T1]{fontenc} 
\usepackage{mathpazo,amssymb}
\usepackage{upref} 
\usepackage{enumerate} 
\usepackage{graphicx,color}
\usepackage{dsfont} %fuer ein schoenes \mathbb{R}: \mathds{R}
\usepackage[colorlinks=true, pdftitle={Regularisation effects of
  nonlinear semigroups}, pdfauthor={Thierry Coulhon and Daniel
  Hauer},citecolor=red,%pagebackref,
hypertexnames=false,breaklinks]{hyperref}

\title[Regularisation effects of nonlinear semigroups]{
  Regularisation effects of nonlinear semigroups}

\author{Thierry Coulhon} \address[Thierry Coulhon]{PSL Research University,
  F-75005 Paris, France}
\email{\href{mailto:thierry.coulhon@univ-psl.fr}{\nolinkurl{thierry.coulhon@univ-psl.fr}}}
\thanks{T.C.'s research was done while he was employed by the Australian National University and was supported by an Australian Research Council
  (ARC) grant DP 130101302.}

\author{Daniel Hauer} \address[Daniel Hauer]{School of Mathematics and
  Statistics, The University of Sydney, NSW 2006, Australia}
\email{\href{mailto:daniel.hauer@sydney.edu.au}{\nolinkurl{daniel.hauer@sydney.edu.au}}}

\subjclass[2010]{47H06,47H20,35K55,46B70,35B65}

\keywords{Nonlinear semigroups, p-Laplace operator, porous media
  operator, doubly nonlinear diffusion operator, nonlocal operators, regularity.}

\numberwithin{equation}{section}

\newtheorem{theorem}{Theorem}[section]
\newtheorem{proposition}[theorem]{Proposition}
\newtheorem{lemma}[theorem]{Lemma}
\newtheorem{corollary}[theorem]{Corollary}

\theoremstyle{definition}
\newtheorem{definition}[theorem]{Definition}

\theoremstyle{remark}
\newtheorem{remark}[theorem]{Remark}

\newcommand\R{{\mathbb{R}}}
\newcommand\N{\mathbb{N}}

\newcommand\dx{\mathrm{d}x }
\newcommand\dy{\mathrm{d}y }
\newcommand\dr{\mathrm{d}r }
\newcommand\ds{\mathrm{d}s }

\newcommand\dmu{\mathrm{d}\mu}
\newcommand\dt{\mathrm{d}t }

\newcommand\dH{\mathrm{d}\mathcal{H}}

\DeclareMathOperator*{\esssup}{ess\;sup}

\DeclareMathOperator*{\divergence}{div}

\newcommand\abs[1]{\lvert#1\rvert}
\newcommand\labs[1]{\left\lvert#1\right\rvert} 
\newcommand\norm[1]{\lVert#1\rVert}

\definecolor{darkred}{rgb}{0.7,0.1,0.1}

\begin{document}
\date{\today}

%\setpagewiselinenumbers

\begin{abstract} 
  One introduces natural and simple methods to deduce
  $L^{s}$-$L^{\infty}$-re\-gularisation estimates for $1\le s< \infty$
  of nonlinear semigroups holding uniformly for all time with sharp
  exponents from natural Gagliardo-Nirenberg inequalities. From
  $L^{q}$-$L^{r}$ Gagliardo-Nirenberg inequalities,
  $1\le q, r\le \infty$, one deduces $L^{q}$-$L^{r}$ estimates for the
  semigroup. New nonlinear interpolation techniques of independent
  interest are introduced in order to extrapolate such estimates to
  $L^{\tilde{q}}$-$L^{\infty}$ estimates for some $\tilde{q}$,
  $1\le \tilde{q}<\infty$. Finally one is able to extrapolate to
  $L^{s}$-$L^{\infty}$ estimates for $1\le s<q$.  The theory developed
  in this monograph allows to work with minimal regularity assumptions
  on solutions of nonlinear parabolic boundary value problems as
  illustrated in a plethora of examples including nonlocal diffusion
  processes.
 \end{abstract}

\maketitle

\tableofcontents

\section{Introduction}
\subsection{The story}
\label{subsec:story}
It begins in the linear semigroup theory: let $\{T_{t}\}_{t\geq 0}$ be
a symmetric semigroup with infinitesimal generator $-A$ of linear
operators acting on $L^2(\Sigma,\mu)$, where $(\Sigma,\mu)$ is a
$\sigma$-finite measure space. Assume that $\{T_{t}\}_{t\geq 0}$ is
submarkovian, meaning that $0\le u\le 1$ implies $0\le T_{t}u\le 1$ for
all $t>0$. If follows that $\{T_{t}\}_{t\geq 0}$ acts on
$L^q(\Sigma,\mu)$ for all $1\le q\le \infty$.

In this framework, there has been many works in the last four decades
that connect a variety of $L^q$-$L^r$, $1\le q<r\le \infty$,
regularisation properties of $\{T_{t}\}_{t\geq 0}$ to a variety of
abstract %(logarithmic)
Sobolev type inequalities involving $A$. The first regularisation
property of $\{T_{t}\}_{t\geq 0}$ that attracted much attention was
the so-called {\em hypercontractivity}: 
% A linear semigroup $\{T_{t}\}_{t\geq 0}$ is said to be
% \emph{hypercontractive}
for some (all) $1<q<r<\infty$, there exists $t_0=t_0(q,r)>0$ such that
$T_{t_0}$ maps $L^q$ to $L^r$ and
\begin{equation}
  \label{eq:237}
  \norm{T_{t_{0}}}_{q\to r}\le 1
\end{equation}
where $\norm{T}_{q\to r}:=\sup_{\norm{u}_{q}\le 1}\norm{Tu}_{r}$
denotes the operator norm of a linear bounded operator
$T : L^{q}\to L^{r}$, $1\le q, r\le \infty$. The theory of
hypercontractive semigroups was introduced by Nelson
in~\cite{MR0210416}, who also provided the most basic example: for the
harmonic oscillator
$A=-\tfrac{1}{2}\frac{d^2}{dx^2}+\tfrac{1}{2}x^{2}-\tfrac{1}{2}$ on
$L^{2}$ equipped with the Gaussian measure
$d\mu= (2\pi)^{-1/2}\,\exp(-\frac{1}{2}x^2)\,\dx$ on $\R$, $T_{t}$ is
a linear contraction from $L^{2}$ to $L^{4}$ if $e^{-t}\le 1/\sqrt{3}$
(cf.~\cite{MR0210416,MR0343816}). One reason for the popularity of
hypercontractivity was its deep connection to constructive quantum
field theory. The ideas in~\cite{MR0210416} were followed up rapidly
and further developed. For instance, Simon and
H{\o}egh-Krohn~\cite{MR0293451} combined the property that the
considered semigroup $\norm{T_{t}}_{q\to r}$ is contractive on $L^{q}$
for all $1\le q\le\infty$ with Riesz-Thorin's and Stein's
interpolation techniques to \emph{extrapolate} the
$L^{q}$-$L^{r}$-regularisation estimate~\eqref{eq:237} to an
$L^{\tilde{q}}$-$L^{\tilde{r}}$-regularisation estimate 
$\norm{T_{t_{0}}}_{\tilde{q}\to \tilde{r}}\le \tilde{C}$ for some
$\tilde{q},\tilde{r}$ such that $1\le \tilde{q}<q$ and
$r<\tilde{r}<\infty$.
%We emphasise that the extrapolation result
%\cite{MR0293451} neither relies on the construction of a
%\emph{one-parameter family} of Sobolev type inequalities nor employs
%other properties of the generator $-A$ of the semigroup
%$\{T_{t}\}_{t\geq 0}$. 
Hypercontractivity is a natural property of some infinite-dimensional
semigroups such as Ornstein-Uhlenbeck. Note that Nelson~\cite{MR0343816} proved
that the Ornstein-Uhlenbeck semigroup does not admit an
$L^{q}$-$L^{\infty}$-regularisation effect for $1\le q<\infty$. % Recall that the harmonic oscillator and the
% Ornstein-Uhlenbeck operator are intimately connected via the Gaussian
% measure on $\R^{d}$. 

% It essentially depends on the
% underlying measure space $(\Sigma,\mu)$ and on the properties of the
% generator $-A$ of the semigroup if an $L^{q}$-$L^{r}$-regularisation
% effect can be extended to an $L^{\tilde{q}}$-$L^{\infty}$-regularisation effect. This statement was
% underlined by Nelson~\cite{MR0343816} proving that the
% Ornstein-Uhlenbeck semigroup admits no
% $L^{\tilde{q}}$-$L^{\infty}$-regularisation effect.

In the mid 70's, Gross~\cite{MR0420249} characterised 
hypercontractivity  in terms of a \emph{single
  logarithmic Sobolev inequality in $L^{2}$}. % Let $\{T_{t}\}_{t\ge 0}$
% be a symmetric submarkovian semigroup on some $L^2(\Sigma,\mu)$ and
% $-A$ its infinitesimal generator.
% \begin{displaymath}
%   \langle Au,v\rangle=\int_{\Sigma}\overline{\nabla u}\nabla v\,\dmu
% \end{displaymath}
% for every $u$, $v\in C^{\infty}_{c}$. 
Then  $\{T_{t}\}_{t\ge 0}$
 is hypercontractive if and only if there is some
 $C>0$ such that
 \begin{equation}
   \label{eq:238}
   \int_{\Sigma}\abs{u}^{2}\,\log \abs{u}\,\dmu\le
   C\, \langle Au,u \rangle+\norm{u}_{2}^{2}\log \norm{u}_{2}
 \end{equation}
 for every $u\in D(A)$. Here, $D(A)$ denotes the domain of  $A$ in $L^2(\Sigma,\mu)$. 
 %Gross' approach in~\cite{MR0420249} is based on
% the idea that, one can deduce from the logarithmic Sobolev
% inequality~\eqref{eq:238} in $L^{2}$ the following \emph{one-parameter family
% of logarithmic Sobolev inequalities}, for every $1<\tilde{q}<\infty$, one has
% \begin{equation}
%   \label{eq:16}
%   \int_{\Sigma} \abs{u}^{q}\,\log \abs{u}\,\dmu\le
%   c(q)\,\rm{Re}\langle
%   Au,u_{q}\rangle+\norm{u}_{q}^{q}\log \norm{u}_{q}
% \end{equation}
% for every $u\in D(A)$ with $c(q)=\frac{c}{2}\frac{q}{q-1}$, where  
% $\langle\cdot,\cdot\rangle$ denotes the duality brackets on $L^{2}$
% (cf.~\cite[Lemma~6.1]{MR0420249}) and $u_{q}:=\abs{u}^{q-2}u$.

In the 80's, in the context of heat
kernels on Lie groups and manifolds, the focus shifted towards a
stronger property, namely {\em ultracontractivity}: for all $t>0$,
$T_{t}$ maps $L^1$ to $L^\infty$.  The game is
then to estimate $\|T_t\|_{1\to\infty}$ from above by an explicit
function of $t$. Of particular interest is the estimate
\begin{equation}
  \label{ultrad}
  \norm{T_t}_{1\to\infty}\le C\,t^{-d/2}, \quad \text{for every $t>0$,}
\end{equation}
where $d>0$ plays the role of a dimension. Davis and Simon~\cite{MR766493}
(see also~\cite{MR1103113}) adapted Gross' approach~\cite{MR0420249}
to the ultracontractivity framework and established the equivalence of
estimates~\eqref{ultrad} with the following one-para\-meter family of
logarithmic Sobolev inequalities: for every $\varepsilon>0$, 
\begin{equation}
  \label{eq:246}
  \int_{\Sigma}\abs{u}^{2}\,\log \abs{u}\,\dmu\le
  \varepsilon\, \langle Au,u \rangle +\varepsilon^{-d/4}\norm{u}_{2}^{2}+\norm{u}_{2}^{2}\log \norm{u}_{2}
\end{equation}
for every $u\in D(A)$. Estimate~\eqref{ultrad} was also
characterised in terms of \emph{($d$-dimen\-sional) Sobolev
  inequalities}: if there exists $C>0$ such that
\begin{equation}
  \label{eq:13}
  \norm{u}_{\frac{2d}{d-2}}^{2}\le C\,\langle  Au,u\rangle\qquad
  \text{for every $u\in D(A)$}
\end{equation}
by Varopoulos~\cite{MR803094} (see also \cite{MR1077272},
\cite{MR1164643} for simplifications and further developments) and in
terms of ($d$-dimensional) Nash inequalities by
Carlen-Kusu\-oka-Stroock~\cite{MR898496}). % It is worth noting that
% Varo\-poulos' characterisation~\cite{MR803094} of Sobolev inequalities with
% inequality~\eqref{ultrad} and Nelson's
% non-ultracontractivity result~\cite{MR0343816} imply that for the
% measure space $(\R^{d},\mu)$ equipped with the Gaussian measure
% \begin{math}
%  \textrm{d}\mu=(2\pi)^{-d/2}\,\exp(-\frac{1}{2}\abs{x}^{2})\,\dx,
% \end{math}
% a $d$-dimensional Sobolev inequality can not be valid. In particular, this shows
% that in terms of an $L^{q}$-$L^{r}$-regularisation estimates of linear
% semigroups, a Sobolev type inequality is stronger than a logarithmic
% Sobolev type inequality.
Further, an intermediate property
called {\em supercontractivity} was also considered
(\cite{MR0425601}): for all $1<q<r<\infty$ and all $t>0$, $T_{t}$ maps
$L^q$ to $L^r$ with the polynomial estimate 
\begin{equation}
\label{eq:177}
\norm{T_t}_{q\to r}\le C\,
t^{-d\left(\frac{1}{q}-\frac{1}{r}\right)}\quad\text{for all $t>0$.}
\end{equation}
Note that if the semigroup is uniformly bounded on $L^1$ and
$L^\infty$, then \eqref{eq:177} implies \eqref{ultrad} (see
\cite{MR1077272}).

% However, in order to keep this monograph readable, we avoid using the
% terminology \emph{hypercontractivity}, \emph{supercontractivity} and
% \emph{ultracontractivity}, but rather prefer to speak about the
% \emph{$L^{q}$-$L^{r}$-regularisation effect} of the semigroup
% $\{T_{t}\}_{t\geq 0}$.

The above outlined development of characterising
$L^{q}$-$L^{r}$-regularisation estimates of the semigroup
$\{T_{t}\}_{t\geq 0}$ with abstract (logarithmic) Sobolev inequalities
is exclusively concerned with \emph{linear semigroups}. Thus, it is
interestingly enough that prior to Varo\-poulos' theorem
\cite{MR803094}, the fact that an abstract Sobolev type inequality
associated with an operator $A$ implies an
$L^1$-$L^\infty$regularisation effect for the semigroup
$\{T_{t}\}_{t\geq 0}$ generated by $-A$ had been discovered in the late
70's by B\'enilan (\cite{Benilan1978}, see
also~\cite[p. 25]{MR1218884}) in the context of \emph{nonlinear
  semigroups}: %
%
%At this point, the nonlinear semigroup theory interferes into the
%story: 
let $\{T_{t}\}_{t\geq 0}$ be a semigroup of mappings $T_{t}$
acting on $L^q$ for all $1\le q\le \infty$, of a
$\sigma$-finite measure space $(\Sigma,\mu)$, with infinitesimal generator $-A$. 

In the paper \cite{Benilan1978}, B\'enilan established first $L^{q}$-$L^{r}$-regularisation
estimates ($1\le q<r\le \infty$) of nonlinear semigroups $\{T_{t}\}_{t\geq 0}$ generated
either by operators of similar type as the \emph{Dirichlet $p$-Laplace operator}
$\Delta_{p}^{\! D}u=\textrm{div}(\abs{\nabla u}^{p-2}\nabla u)$, ($1<p<\infty$),
% \begin{displaymath}
%   \Delta_{p}^{\! D}u=\textrm{div}(\abs{\nabla u}^{p-2}\nabla u)
% \end{displaymath}
or by operators similar to the \emph{Dirichlet porous media operator}
$\Delta^{\! D}(u^{m})=\textrm{div}(\nabla u^{m})$, ($m>0$).
% \begin{displaymath}
%   \Delta^{\! D}(u^{m})=\textrm{div}(\nabla u^{m}),
% \end{displaymath}
Here, the name \emph{Dirichlet} and the superscript $D$ refer to the fact that the
differential operators $\Delta^{\! D}_{p}$ and $\Delta^{\! D}(\cdot^{m})$ are
equip\-ped with homogeneous Dirichlet boundary conditions on a
bounded domain $\Sigma$ of $\R^{d}$, and $u^{m}$ is the shorthand of
$\abs{u}^{m-1}u$. B\'enilan's method employs a truncation technique
on the sublevel sets of the resolvent combined with the
regularisation effect of the resolvent given by the $d$-dimensional
Sobolev inequality. 

Only one year later, V\'eron~\cite{MR554377} simplified B\'enilan's
method essentially and adapted it to the general nonlinear semigroup
framework acting on $L^{q}$ for $1\le q\le \infty$. % By using the
% notion of $\phi$-accretive operators, 
V\'eron introduces an abstract Sobolev type inequality (similar
to~\eqref{eq:13}) satisfied by the generator $A$, from which one can
conclude an $L^{q}$-$L^{r}$ regularity estimate ($1\le q<r\le\infty$)
of the corresponding semigroup $\{T_{t}\}_{t\ge 0}$. In particular,
$L^{q}$-$L^{\infty}$ estimates ($1\le q<\infty$) of
$\{T_{t}\}_{t\ge 0}$ are obtained by using a \emph{one-parameter
  family of Sobolev type inequalities} satisfied by $A$ combined with
an iteration method in the time-variable of $\{T_{t}\}_{t\ge 0}$. To
be more precise, one easily sees that, for instance, for $1<p<d$, the
Dirichlet $p$-Laplace operator $A=-\Delta_{p}^{D}$ on $L^{2}$
satisfies the following one-parameter family (in $q\ge p$) of Sobolev
type inequalities
 \begin{equation}
   \label{eq:17}
   \norm{u}_{\frac{dq}{d-p}}^{q}\le
   C\,\tfrac{1}{q-p+1}\,\left(\tfrac{q}{p}\right)^{p}\, \langle
   -\Delta_{p}^{D}u, u_{q-p+2}\rangle
 \end{equation}
 for every $u\in D(A)\cap L^{\infty}$ and $q\ge p$. To the best of our knowledge,
 it goes back to V\'eron~\cite{MR554377} who established that the semigroup
$\{T_{t}\}_{t\geq 0}$ generated by $\Delta_{p}^{\! D}$ satisfies the
$L^{q}$-$L^{r}$-regularisation estimate
\begin{equation}
  \label{eq:176}
  \norm{T_{t}u-T_{t}\hat{u}}_{r}\le C\,t^{-\delta}
  \,\norm{u-\hat{u}}_{q}^{\gamma}
\end{equation}
for every $t>0$, $u$, $\hat{u}\in L^{q}$ with exponents $\delta$,
$\gamma>0$ depending on $d$, $p$, $r$ and $q$ for every
$q\le r\le \infty$ ($1\le q\le\infty$, $2\le p<\infty$), and that the
semigroup $\{T_{t}\}_{t\geq 0}$ generated by
$\Delta^{\! D}(\cdot^{m})$ satisfies the $L^{q}$-$L^{r}$-regularisation
estimate
\begin{equation}
  \label{eq:146}
  \norm{T_{t}u}_{r}\le C\,t^{-\delta}\,\norm{u}_{q}^{\gamma}
\end{equation}
for every $t>0$, $u\in L^{q}$, with  $q=1$, $r=\infty$  and exponents
$\delta$, $\gamma>0$ depending on $r$, $q$, $m>1$ and $d$. % , where
% the constant $C>0$ depends only on $m>1$ and dimension $d$.
V\'eron's approach was quickly adapted  to many nonlinear parabolic problems
(cf. for instance~\cite{MR656651,MR1741878}).

The analogue of estimates~\eqref{ultrad} and \eqref{eq:177} concerning
linear semigroups are in the nonlinear semigroup theory the
estimates~\eqref{eq:176} or \eqref{eq:146}. To emphasise the fact that these estimates involving
nonlinear semigroups appear with an exponent
$\gamma$ at the initial datum $u\in L^{q}$, which is, in general,
different of one, we avoid calling the
estimates~\eqref{eq:176} and \eqref{eq:146} supercontractive or
ultracontractive estimates, but rather speak from an
\emph{$L^{q}$-$L^{r}$ regularisation estimate} of the nonlinear semigroup
$\{T_{t}\}_{t\geq 0}$ if $1\le q<r\le \infty$ (see also Remark~\ref{rem:8} in
Section~\ref{gn}).

In 2001, Cipriani and Grillo~\cite{MR1867617} adapted the approach by
Davis and Simon~\cite{MR766493} to establish
$L^{q}$-$L^{\infty}$-regularisation estimates of solutions of parabolic
diffusion equations involving quasilinear operators of $p$-Laplace
type equipped with homogeneous Dirichlet boundary conditions on
bounded domains. The approach in~\cite{MR1867617} is essentially based
on the following two steps (cf. \cite{MR1867617}): firstly, one
employs the \emph{classical} Sobolev inequality
\begin{equation}
 \label{eq:178}
  \norm{u}_{\frac{pd}{d-p}}\le C\,\norm{|\nabla u|}_{p}
\end{equation}
with respect to the Lebesgue measure, in order to derive a
\emph{one-parameter family of logarithmic Sobolev inequalities} in $L^{p}$
(similar to~\eqref{eq:246} with $2$ replaced by $p$) associated with the energy functional
of the Dirichlet-$p$-Laplace operator $\Delta_{p}^{D}$. Then one uses this family of
 inequalities to show that for a
solution $u$ of the parabolic equation under consideration the function
\begin{displaymath}
y(t):=\log \norm{u(t)}_{r(t)}\quad\text{ for $t\ge 0$,}
\end{displaymath}
satisfies a differential inequality from which one can deduce an
$L^{q}$-$L^{r}$-regulari\-sation estimate for $1\le q<r\le \infty$. 

Comparing this method with the one by V\'eron,  the
approach in~\cite{MR554377} seems to be more direct in order to achieve
$L^{q}$-$L^{r}$-regularisation estimates for $1\le q<r\le \infty$ with
optimal exponents. % On the other hand, V\'eron's method can not be
% applied if a Sobolev-type inequality seems not to be available for the
% given measure space $(\Sigma,\mu)$. This is in consent with Nelson's
% non-ultracontractivity result~\cite{MR0343816}.

% Note, the \emph{linear} and \emph{nonlinear semigroup theory} coincide
% in the \emph{symmetric} situation, that is, for every positive linear
% \emph{self-adjoint} operator $A$ on $L^{2}$
% (cf.~\cite[Proposition~2.15]{MR0348562}) the linear and
% nonlinear semigroup theory applies to $A$ for generating a
% semigroup $\{T_{t}\}_{t\ge0}$ of contractions on $L^{2}$. Thus, looking back to the beginning of
% this story, the answer to this question was already given by
% Nelson~\cite{MR0343816}: the Ornstein-Uhlen\-beck semigroup is the
% prototype example demonstrating there are situations where for the
% underlying measure space $(\Sigma,\mu)$ a Sobolev inequality is not
% valid, but rather a logarithmic inequality holds. Under this
% constraints, Gross' logarithmic Sobolev approach \cite{MR0420249}
% still achieves a $L^{q}$-$L^{r}$-regularisation effect of the
% semigroup. In 2009, this statement has been revisited by
% Grillo~\cite{MR2508974} in the framework of parabolic problems
% involving weighted $p$-Laplace operators, where the weights are of
% Gaussian type.

Many authors followed the approach in~\cite{MR1867617}; they derive
from the classical Sobolev inequality~\eqref{eq:178} new families of
\emph{energy entropy inequalities} (generalising the logarithmic
Sobolev inequality) and then apply these inequalities to nonlinear
parabolic problems (see, for
instance, \cite{MR2053885,MR1892935,Takac05,MR2149917,MR2129606,
  MR2268115,MR2379911,MR3158845}).

Another approach worth mentioning in this context is~\cite{MR2529737}
by Porzio. In this paper, Porzio employs directly the classical Sobolev inequality to establish
$L^{q}$-$L^{r}$-regularisation estimates ($1\le q<r\le \infty$) for
solutions of nonlinear parabolic equations involving
non-autonomous quasilinear differential operators of $p$-Laplace type.

In order to conclude this section, we want to emphasise that
$L^{q}$-$L^{r}$-regularity estimates of semigroups
$\{T_{t}\}_{t\ge 0}$  have many applications, such as new existence results 
(see, for instance,~\cite{MR2529737}), global H\"older
continuity (see, for instance,~\cite{MR1230384,MR1218742,MR1156216}) (and
higher regularity (see \cite{MR1230384})) of weak energy solutions of the underlying
parabolic boundary value problem (see
Section~\ref{Asec:proof-of-Linfty-regularity} and
Section~\ref{subsec:Mild-is-strong} of this monograph),  finite time of extinction 
results with respect to the initial data (see, for
instance, \cite[pp 234]{MR2582280} or \cite{MR2286292}) or uniqueness of
solutions (see~\cite{MR1028745}), and others.

\subsection{Main results}  
\label{sec:main-thms}
In the present monograph, rather than making a  detour via
a family of Log-Sobolev inequalities, we shall use the tools that
are more directly relevant to establish general $L^{q}$-$L^{r}$-regularity estimates for
$1\le q$, $r\le \infty$ and $L^{q}$-$L^{\infty}$-regularisation estimates
for $1\le q<\infty$ of nonlinear semigroups $\{T_{t}\}_{t\ge 0}$,
%polynomial ultracontractivity and
%supercontractivity, 
namely, Sobolev type inequalities or, more generally,
\emph{Gagliardo-Nirenberg} type inequalities 

In the following, $(\Sigma,\mu)$ will be  a $\sigma$-finite
measure space. For  $1\le q\le \infty$, we shall say that $A$ is an \emph{operator on
  $L^{q}(\Sigma,\mu)$} if 
$A$ is a subset of $L^{q}(\Sigma,\mu)\times L^{q}(\Sigma,\mu)$. Further notions and notation used
throughout this monograph can be found in Section~\ref{fm}. 

\begin{definition}
  \label{def:Gag-Nire-inequality}
  Let $1\le q<\infty$ and $1\le r\le \infty$. We say an operator
  $A$ on $L^{q}$ satisfies an $L^{q}$-$L^{r}$- \emph{Gagliardo-Nirenberg type
    inequality} for some $\varrho\ge 0$, $\sigma>0$, $\omega\in \R$ and
    $(u_{0},0)\in A$ if there is a constant $C>0$ such that
   \begin{equation}
     \label{eq:242}
    \norm{u-u_{0}}_{r}^{\sigma} \le C\,
    \Big( [u-u_{0},v]_{q}+\omega \norm{u-u_{0}}_{q}^{q}\Big)\;
    \norm{u-u_{0}}_{q}^{\varrho}
  \end{equation}
  for every $(u,v)\in A$. Moreover, we say that an operator $A$ on
  $L^{q}$ satisfies an $L^{q}$-$L^{r}$- \emph{Gagliardo-Nirenberg type
    inequality with differences} for some $\varrho\ge 0$,
  $\sigma>0$ and $\omega\in \R$ if there is a constant $C>0$ such that
   \begin{equation}
    \label{eq:5}
    \norm{u-\hat{u}}_{r}^{\sigma} \le C\,
    \Big( [u-\hat{u},v-\hat{v}]_{q}+\omega \norm{u-\hat{u}}_{q}^{q}\Big)\;
    \norm{u-\hat{u}}_{q}^{\varrho}
  \end{equation}
  for every $(u,v)$, $(\hat{u},\hat{v})\in A$. % Note, we identify an
  % operator $A$ by its graph and write $(u,v)\in A$ to denote that
  % $v\in Au$ since we allow $A$ to be multi-valued.
\end{definition}

For example, the negative Dirichlet $p$-Laplace operator
$-\Delta_{p}^{D}$ satisfies the Ga\-gliar\-do-Nirenberg
inequality~\eqref{eq:242} with $u_{0}=0$ if $1\le p<2$ and \eqref{eq:5} if
$2\le p<\infty$ (see Section~\ref{sec:p-laplace}) and for $m>0$, the
negative doubly nonlinear operator $-\Delta_{p}^{D}(\cdot^{m})$
equipped with Dirichlet boundary conditions
satisfies the Gagliardo-Nirenberg inequality~\eqref{eq:242} for $u_{0}=0$ (see
Section~\ref{sec:doubly-nonl-diff}). Further examples of operators
and other type of boundary conditions are discussed in Section~\ref{sec:examples}.

In the paper, we intend to come back to B\'enilan's and V\'eron's viewpoint, and provide a
systematic semigroup approach in order to establish
$L^{s}$-$L^{\infty}$-regulari\-sation estimates of the
form~\eqref{eq:176} and \eqref{eq:146} for any $1\le s < \infty$ for
(nonlinear) semigroups $\{T_{t}\}_{t\ge 0}$ under the assumption that
the corresponding infinitesimal generator $-A$ satisfies an
$L^{q}$-$L^{r}$-Gagli\-ardo-Nirenberg type inequality either without
differences~\eqref{eq:242} or with differences~\eqref{eq:5} for
some $1\le q$, $r\le \infty$.

%the notion of $\phi$-accre\-tive operators. In particular,

We simplify B\'enilan's and V\'eron's method by avoiding for a large
class of operators the construction of a one-parameter family of
Sobolev inequalities (such as the family of inequalities given by
\eqref{eq:17}) to establish $L^{q}$-$L^{\infty}$-regularisation
estimates of semigroups $\{T_{t}\}_{t\ge 0}$. We rather tried to make
the extrapolation techniques from the \emph{linear} semigroup theory
by Simon and H{\o}egh-Krohn \cite{MR0293451} available for the
\emph{nonlinear} semigroup theory. To achieve this, we have
established a new nonlinear interpolation theorem (see
Theorem~\ref{thm:nonlinear-interpol},
Theorem~\ref{thm:interpol-no-differences} and
Theorem~\ref{thm:interpol-no-differences-bis} in
Section~\ref{sec:nonlinear-interpolation}). Our techniques require the
validity of only one $L^{q}$-$L^{r}$ Gagliardo-Nirenberg type
inequality satisfied by the generator $A$ for some $1\le q$,
$r\le \infty$ in order to establish
$L^{s}$-$L^{\infty}$-regulari\-sation estimates for $1\le s<\infty$ of
the corresponding semigroup $\{T_{t}\}_{t\ge 0}$. This simplifies
essentially the known techniques in the existing literature
(cf. \cite{Benilan1978,MR554377,MR0420249,MR2053885,MR1892935,Takac05,MR2149917,MR2129606,MR2268115,MR2379911,MR3158845}
and many more), but also allows us to establish
$L^{q}$-$L^{r}$-regularity estimates, $1\le q$, $r\le \infty$, for
solutions of nonlinear parabolic problems involving nonlocal diffusion
processes (see Section~\ref{sec:nonlocal}
and~\ref{subsection:fractional-p-laplace}). Estimates of this type for
solutions of nonlinear nonlocal diffusion problems are know to hold
only for the fractional porous media equation on the whole space (cf.
\cite{MR2737788}).  Further, we provide a nonlinear version of the
methods from \cite{MR1077272} and \cite{MR1164643} to conclude that if
a semigroup $\{T_{t}\}_{t\ge 0}$ satisfies a
$L^{q}$-$L^{r}$-regularisation estimate of the form~\eqref{eq:176} or
\eqref{eq:146} for some $1<q<r\le \infty$ then the semigroup admits,
in particular, a $L^{1}$-$L^{r}$-regularisation estimate of the
form~\eqref{eq:176} or \eqref{eq:146} (see
Theorem~\ref{thm:extrapol-L1-differences} and
Theorem~\ref{thm:extrapol-L1-bis} in
Section~\ref{sec:extra-pol-twoards-one}).

Similar to~\cite{MR554377}, we focus our attention on two
important classes of operators generating nonlinear semigroups acting
on $L^{q}$ for all $1\le q\le \infty$:

\begin{itemize}
\item \emph{quasi $m$-completely accretive operators in $L^{q_{0}}$
    for some $1\le q_{0}<\infty$}, 
 
\item \emph{quasi $m$-$T$-accretive operators in $L^{1}$ with complete
    resolvent}. 
 \end{itemize}

 % Semigroups $\{T_{t}\}_{t\geq 0}$ of the first class are generated by
 % an operators $-A$ with the property that $A+\omega I$ is
 % \emph{$m$-completely accretive} in some $L^{q_{0}}$ space with dense
 % domain (see Definition \ref{def:completely-accretive-operators} in
 % Section~\ref{sec:comp} and cf.  \cite{MR1164641}).

 In order to keep this subsection for an overview of the main results
 of this monograph, we refer for the definition of these two classes
 of operators to Section~\ref{sec:comp} and Section~\ref{subsec:L1}
 and note briefly that prototypes of the first class of operators are
 of the form $A+F$, where $F$ is a Lipschitz continuous mapping on
 $L^{q_{0}}$ and $A$ is, for instance, the celebrated negative
 $p$-Laplace operator $-\Delta_{p}$ (see Section~\ref{sec:p-laplace})
 but also the negative nonlocal fractional $p$-Laplace operator
 $-(-\Delta_{p})^{s}$ (see \cite{MaRoTo2015} and
 Section~\ref{subsection:fractional-p-laplace}) respectively equipped
 with some boundary conditions and the Dirichlet-to-Neumann operator
 associated with the $p$-Laplace operator (see
 Section~\ref{subsec:DtN}). Examples of the second class of operators
 are also of the form $A+F$, where $F$ is a Lipschitz continuous
 mapping on $L^{1}$ and $A$ is, for instance, the negative porous
 media operator $-\Delta (\cdot^{m})$ and its nonlocal counterpart
 (\cite{MR2737788}) or, more generally, doubly nonlinear operators
 $\Delta_{p}(\cdot^{m})$ (see Section~\ref{sec:doubly-nonl-diff}),
 where each of them is equipped with some boundary conditions.

Our first main result is concerned with
$L^{q}$-$L^{r}$-regularity estimates of semigroups
$\{T_{t}\}_{t\ge 0}$ generated by $-A$ for an operator $A$ of the
first class satisfying $L^{q}$-$L^{r}$-Gagliardo-Nirenberg type
inequality~\eqref{eq:5} with \emph{differences}.

\begin{theorem}
  \label{thm:main-1}
  For some $q\in[1,+\infty)$ and  $\omega\ge 0$, let $A+\omega I$ be $m$-completely accretive in
  $L^{q}(\Sigma,\mu)$  with dense
  domain. If $A$ satisfies the Gagliardo-Nirenberg type
    inequality~\eqref{eq:5} with parameters $q$, $1\le r\le \infty$,
    $\varrho\ge 0$ and $\sigma>0$, then the semigroup
    $\{T_{t}\}_{t\ge 0}$ generated by $-A$ on $L^{q}(\Sigma,\mu)$ satisfies
    \begin{equation}
      \label{eq:18}
      \norm{T_{t}u-T_{t}\hat{u}}_{r}\le 
      \left(\tfrac{C}{q}\right)^{1/\sigma}\,t^{-\alpha}\,e^{\omega\, \beta\,t}\,
      \norm{u-\hat{u}}_{q}^{\gamma}
    \end{equation}
    for every $t>0$, $u$, $\hat{u}\in L^{q}(\Sigma,\mu)$ with
    exponents $\alpha=\frac{1}{\sigma}$, $\beta=\gamma+1$ and
    $\gamma=\tfrac{q+\varrho}{\sigma}$. Moreover, if $1\le r<
    \infty$, $\gamma\,r>q$ and there is $(u_{0},0)\in A$ for some
    $u_{0}\in L^{1}\cap L^{\infty}(\Sigma,\mu)$, then
    \begin{equation}
      \label{eq:80}
        \norm{T_{t}u-T_{t}\hat{u}}_{\infty}\lesssim \; t^{-\alpha_s}\;
        e^{\omega\,  \beta_{s}\,t}\; \norm{u-\hat{u}}_{s}^{\gamma_s}
    \end{equation}
    for every $t>0$, $u$, $\hat{u}\in L^{s}(\Sigma,\mu)$, $1\le s\le
    \gamma\,r\,q^{-1}\,m_{0}$ satisfying $\gamma (1-\frac{s q}{\gamma
      r m_{0}})<1$ for every $m_{0}\ge q\,\gamma^{-1}$ satisfying
    \begin{equation}
      \label{eq:75}
      (\tfrac{\gamma\,r}{q}-1)\,m_{0}+q(\tfrac{1}{\gamma}-1)>0,
    \end{equation}
    with exponents
    \begin{equation}
      \label{eq:76}
      \begin{split}
        & \alpha^{\ast}= \tfrac{\sigma^{-1}\,q\,\gamma^{-1}}{
          (\frac{\gamma\,r}{q}-1)\,m_{0}+q(\frac{1}{\gamma}-1)},\;
        \beta^{\ast}=\tfrac{\gamma^{2} r
          q^{-1}-1}{(\frac{\gamma\,r}{q}-1)\,m_{0}+q(\frac{1}{\gamma}-1)}+1,\;
        \gamma^{\ast}= \tfrac{(\gamma r q^{-1}-1)\,m_{0}}{
          (\frac{\gamma\,r}{q}-1)\,m_{0}+q(\frac{1}{\gamma}-1)},\\
 	& \alpha_{s}=\tfrac{\alpha^{\ast}}{1-\gamma^{\ast}(1-\frac{s q}{\gamma r m_{0}})},\quad
        \beta_{s}=\tfrac{\beta^{\ast} 2^{-1}+\gamma^{\ast} \frac{s
            q}{\gamma r m_{0}}}{1-\gamma^{\ast}(1-\frac{s q}{\gamma r m_{0}})},\quad 
        \gamma_{s}=\tfrac{\gamma^{\ast}\frac{s q}{\gamma r m_{0}} }{
          1-\gamma^{\ast}(1-\frac{s q}{\gamma r m_{0}})}.
      \end{split}
    \end{equation}
\end{theorem}

The proof of Theorem~\ref{thm:main-1} follows by combining
Theorem~\ref{thm:quasi-super-contractivity} (Section~\ref{gn}),
Theorem~\ref{thm:extrapolation-to-infty}
(Section~\ref{sec:extrapolation-towards-infinity}) and subsequently by
applying Theorem~\ref{thm:extrapol-L1-differences} (Section~\ref{sec:extra-pol-twoards-one}).

\begin{remark}
  \label{rem:11}
  At first glance, the condition $\gamma\,r>q$ and the choice of
  $m_{0}\ge q\,\gamma^{-1}$ satisfying~\eqref{eq:75} in
  Theorem~\ref{thm:main-1} and in the subsequent two theorems seem
  rather mysterious. They are sufficient conditions to conclude an
  $L^{s}$-$L^{\infty}$ regularisation estimate for
  $s=\gamma\,r\,q^{-1}\,m_{0}$ from an $L^{q}$-$L^{r}$ regularity
  estimate for some $1\le q$, $r<\infty$ (cf. Remark~\ref{rem:10} in
  Chapter~\ref{sec:extrapolation-towards-infinity}). But on the other
  hand, the parameters $\gamma$, $r$ and $q$ are intimately related
  with the given operator $A$. In fact, the condition $\gamma\,r>q$
  changes if and only if $A$ changes. This is not the case for the
  parameter $m_{0}$ satisfying $m_{0}\ge q\,\gamma^{-1}$
  and~\eqref{eq:75} since for sufficiently large $m_{0}$ both
  conditions always hold. In certain cases, but not all,
  $m_{0}=q\,\gamma^{-1}$ satisfies ~\eqref{eq:75}, in which case this
  choice of $m_0$ is optimal. This is well demonstrated by the example
  of the $p$-Laplace operator $A=-\Delta_{p}^{\!\R^{d}}$ on $\R^{d}$
  satisfying vanishing conditions at infinity (see
  Theorem~\ref{thm:Dirichlet-p-laplace} in
  Section~\ref{sec:homog-dirichl-bound}).
\end{remark}

Our second main result is concerned with
$L^{q}$-$L^{r}$-regularisation estimates of semigroups
$\{T_{t}\}_{t\ge 0}$ generated by $-A$ for an operator $A$ of the
first class but satisfying the Gagliardo-Nirenberg type
inequality~\eqref{eq:242} without \emph{differences}.

\begin{theorem}
  \label{thm:GN-implies-reg-bis}
  For some $q\in[1,+\infty)$ and  $\omega\ge 0$, let $A+\omega I$ be $m$-completely accretive in
  $L^{q}(\Sigma,\mu)$  with dense
  domain. If $A$ satisfies the Gagliardo-Nirenberg type
    inequality~\eqref{eq:242} with  parameters $q$, $1\le r\le \infty$,
    $\varrho\ge 0$ and $\sigma>0$ and some $(u_{0},0)\in A$ satisfying
    $u_{0}\in L^{1}\cap L^{\infty}(\Sigma,\mu)$, then the semigroup
    $\{T_{t}\}_{t\ge 0}$ generated by $-A$ on $L^{q}(\Sigma,\mu)$ satisfies
    \begin{equation}
      \label{eq:20}
      \norm{T_{t}u-u_{0}}_{r}\le 
      \left(\tfrac{C}{q}\right)^{1/\sigma}\,t^{-\alpha}\,e^{\omega\, \beta\,t}\,
      \norm{u-u_{0}}_{q}^{\gamma}
    \end{equation}
    for every $t>0$, $u\in L^{q}(\Sigma,\mu)$ with
    exponents $\alpha=\frac{1}{\sigma}$, $\beta=\gamma+1$ and
    $\gamma=\tfrac{q+\varrho}{\sigma}$. Moreover, if $1\le r<
    \infty$ and $\gamma\,r>q$, then
    \begin{equation}
      \label{eq:168}
        \norm{T_{t}u-u_{0}}_{\infty}\lesssim \; t^{-\alpha_s}\;
        e^{\omega \, \beta_{s}\,t}\; \norm{u-u_{0}}_{s}^{\gamma_s}
    \end{equation}
    for every $t>0$, $u \in L^{s}(\Sigma,\mu)$,
    $1\le s\le \gamma\,r\,q^{-1}\,m_{0}$ satisfying
    $\gamma (1-\frac{s q}{\gamma r m_{0}})<1$ for every
    $m_{0}\ge q\,\gamma^{-1}$ satisfying~\eqref{eq:75} with
    exponents~\eqref{eq:76}.
\end{theorem}

The statements of Theorem~\ref{thm:GN-implies-reg-bis} follows from
Theorem~\ref{thm:quasi-super-contractivity-bis} (Section~\ref{gn}) and
by Theorem~\ref{thm:extrapolation-to-infty-bis}
(Section~\ref{sec:extrapolation-towards-infinity}) combined with
Theorem~\ref{thm:extrapol-L1-bis}
(Section~\ref{sec:extra-pol-twoards-one}).

The last main result of this monograph focuses on the
$L^{q}$-$L^{r}$-regularisation estimates of semigroups
$\{T_{t}\}_{t\ge 0}$ generated by $-A$ for an operator $A$ in
$L^{1}(\Sigma,\mu)$ of the second class satisfying the Gagliardo-Nirenberg
type inequality~\eqref{eq:242} without \emph{differences}. However,
applications show that operators $A$ of the second class, generally, do not
satisfy the Gagliardo-Nirenberg type inequality~\eqref{eq:242} for $u\in
D(A)\cap L^{q}(\Sigma,\mu)$ and some $q>1$ but for $u\in
D(A)\cap L^{\infty}(\Sigma,\mu)$. Thus, it is very useful to introduce
the \emph{trace}
\begin{displaymath}
    A_{1\cap \infty}:= A\cap ((L^{1}\cap
  L^{\infty}(\Sigma,\mu))\times (L^{1}\cap
  L^{\infty}(\Sigma,\mu)))
  \end{displaymath}
of $A$ on $L^{1}\cap L^{\infty}(\Sigma,\mu)$. %  satisfies Gagliardo-Nirenberg type
% inequality~\eqref{eq:242}. 
Furthermore, note that the
notion of \emph{$c$-complete resolvent} is defined in Section~\ref{subsec:L1}. 

The nonlinear interpolation theorems used in the proofs of Theorems
\ref{thm:main-1} and \ref{thm:GN-implies-reg-bis} (see
Theorem~\ref{thm:nonlinear-interpol},
Theorem~\ref{thm:interpol-no-differences} and
Theorem~\ref{thm:interpol-no-differences-bis} in
Section~\ref{sec:nonlinear-interpolation}) cannot be applied to
semigroups $\{T_{t}\}_{t\ge 0}$ generated by $-A$ in $L^{1}$ for
operators of the second class. One essential reason for this is
that in the latter case each mapping
$T_{t} : L^{1}\cap L^\infty \to L^{1}\cap L^\infty$ is, in general, not Lipschitz
continuous with respect to the $L^{\infty}$-norm.  % We aim to improve
% this in a forthcoming paper.
% Further, operators of the first class satisfy range
% condition~\eqref{eq:148} in
% Theorem~\ref{thm:GN-implies-reg-for-oper-in-L1}. This is, in general,
% not clear for operators of the second class.
Another important observation
is that operators satisfying an $L^{q}$-$L^{r}$ Gagliardo-Nirenberg
type inequality~\eqref{eq:5} with differences are necessarily
quasi-accretive in $L^{q}$. But there are operators of the second
class, as for instance, the negative porous media operator $-\Delta(\cdot^{m})$
(cf.~\cite{MR2286292}) that are not accretive in $L^{q}$ for
$q>1$. Hence we have to provide an alternative approach which applies
to the second class of operators.

\begin{theorem}
  \label{thm:GN-implies-reg-for-oper-in-L1}
  Let $A+\omega I$ be an $m$-$T$-accretive operator in $L^{1}(\Sigma,\mu)$
  for some $\omega\ge 0$ with complete resolvent (respectively,
  $c$-complete resolvent and $\omega=0$). Suppose the
  trace $A_{1\cap \infty}$ of $A$ on
  $L^{1}\cap L^{\infty}(\Sigma,\mu)$ satisfies the range condition
  \begin{equation}
    \label{eq:148}
    L^{1}\cap L^{\infty}(\Sigma,\mu)\subseteq Rg(I+(A_{1\cap \infty}+\omega I)).
  \end{equation}
  and the Gagliardo-Nirenberg type
  inequality~\eqref{eq:242} for parameters $1\le q$,
  $r\le \infty$, $(q<\infty)$, $\sigma>0$, $\varrho\ge 0$ and some
  $(u_{0},0)\in A_{1\cap \infty}$ satisfying $u_{0}\in L^{1}\cap L^{\infty}(\Sigma,\mu)$.
  Then the semigroup $\{T_{t}\}_{t\geq 0}$ generated by $-A$ on
    $\overline{D(A)}^{\mbox{}_{L^{1}}}$ satisfies
    \begin{equation}
      \tag{\ref{eq:20}}
      \norm{T_{t}u-u_{0}}_{r}\le 
      \left(\tfrac{C}{q}\right)^{1/\sigma}\; t^{-\alpha}\; e^{\omega\,\beta\,t}\;
      \norm{u-u_{0}}_{q}^{\gamma}
    \end{equation}
    for every $t>0$, $u\in \overline{D(A)}^{\mbox{}_{L^{1}}}\cap L^{\infty}(\Sigma,\mu)$ with
    exponents $\alpha=\frac{1}{\sigma}$, $\beta=\gamma+1$ and
    $\gamma=\tfrac{q+\varrho}{\sigma}$.
    % \begin{equation}
    %   \label{eq:170}
    %   \begin{split}
    %     \norm{T_{t}u-u_{0}}_{r}&\le
    %     \left(\tfrac{C}{q}\right)^{1/\sigma}\,e^{\omega\,t
    %     (\beta+1)}\,t^{-\frac{1}{\sigma}}\,\norm{u-u_{0}}_{q}^{\beta}\\
    %     &\qquad \text{for every $t>0$ and
    %     $u\in L^{1}\cap L^{\infty}(\Sigma,\mu)$}
    %   \end{split}
    % \end{equation}
    % with exponent $\beta=\frac{q+\varrho}{\sigma}$.
   Moreover, if for parameters $\kappa>1$, $m>0$ and $q_{0}\ge p\ge 1$
   satisfying $\kappa m q_{0}\ge 1$ and
    \begin{equation}
      \label{eq:164}
      (\kappa-1)q_{0}+p-1-\tfrac{1}{m}>0,
    \end{equation}
   the trace $A_{1\cap \infty}$ satisfies the one-parameter family of Sobolev
    type inequalities
    \begin{equation}
      \label{eq:10single}
      \begin{split}
        \norm{u-u_{0}}_{\kappa mq}^{mq}&\le
        \tfrac{C\,(q/p)^{p}}{q-p+1}\,
        \left[[u-u_{0},v]_{(q-p+1)m+1}+\omega\norm{u-u_{0}}_{(q-p+1)m+1}^{(q-p+1)m+1}\right]
      \end{split}
    \end{equation}
    for every $(u,v)\in A_{1\cap \infty}$ and every $q\ge q_{0}$, then
    there is a $\beta^{\ast}\ge 0$ such that the semigroup
    $\{T_{t}\}_{t\geq 0}$ satisfies
    \begin{equation}
      \tag{\ref{eq:168}}
      \norm{T_{t}u-u_{0}}_{\infty} \lesssim \;
      t^{-\alpha_{s}}\;e^{\omega\,
        \beta_{s}\,t}\,\norm{u-u_{0}}_{s}^{\gamma_{s}} %\kappa m q_{0}
    \end{equation}
    for every $u\in \overline{D(A)}^{\mbox{}_{L^{1}}}\cap L^{\infty}(\Sigma,\mu)$, $1\le s\le
    \kappa m q_{0}$ satisfying $\gamma^{\ast}\left(1-\tfrac{q}{\kappa m q_{0}}\right)<1$
    with exponents
    \begin{equation}
      \label{eq:261}
      \begin{split}
        &\qquad
        \alpha^{\ast}=\tfrac{1}{m((\kappa-1)q_{0}+p-1-\frac{1}{m})},\quad
        \gamma^{\ast}=\tfrac{(\kappa-1)
          q_{0}}{(\kappa-1)q_{0}+p-1-\frac{1}{m}},\\
        &\alpha_{s}=\tfrac{\alpha^{\ast}}{1-\gamma^{\ast}\left(1-\tfrac{s}{\kappa
              m q_{0}}\right)},\; \beta_{s}=\tfrac{\beta^{\ast}
          2^{-1}+\gamma^{\ast}s\kappa^{-1} m^{-1} q_{0}^{-1}}{
          1-\gamma^{\ast}\left(1-\tfrac{s}{\kappa m q_{0}}\right)},\;
        \gamma_{s}=\tfrac{\gamma^{\ast}\,s}{\kappa m
          q_{0}(1-\gamma^{\ast}(1-\tfrac{s}{\kappa m q_{0}}))}.
      \end{split}
    \end{equation}
   % \begin{equation}
   %    \label{eq:173}
   %    \begin{split}
   %      \norm{T_{t}u-u_{0}}_{\infty} &\lesssim \;e^{\omega t
   %        \beta_{q}}\;
   %      t^{-\delta_{q}}\,\norm{u-u_{0}}_{q}^{\gamma_{q}}\\
   %      &\qquad\text{for every $t>0$ and
   %        $u\in L^{1}(\Sigma,\mu)\cap L^{\infty}(\Sigma,\mu)$,}
   %    \end{split}
   %  \end{equation}
   %  where
    % \begin{displaymath}
    %   \beta_{q}=
    %   \tfrac{\frac{\beta^{\ast}}{2}+\gamma\frac{q}{\kappa m q_{0}}}{
    %     1-\gamma\left(1-\tfrac{q}{\kappa m q_{0}}\right)},\quad
    %   \delta_{q}=\tfrac{\delta}{1-\gamma\left(1-\tfrac{q}{\kappa m
    %         q_{0}}\right)},\quad
    %   \gamma_{q}=\tfrac{\gamma\,q}{\kappa m q_{0}(1-\gamma(1-\tfrac{q}{\kappa m q_{0}}))}.
    % \end{displaymath}
\end{theorem}

Suppose the domain $D(A)$ of the operator $A$ considered in
Theorem~\ref{thm:GN-implies-reg-for-oper-in-L1} is dense in
$L^{1}(\Sigma,\mu)$. If the measure space $(\Sigma,\mu)$
is finite then a standard density result yields the
inequalities \eqref{eq:20} and \eqref{eq:168} in
Theorem~\ref{thm:GN-implies-reg-for-oper-in-L1} hold for all
$u\in L^{q}(\Sigma,\mu)$, respectively, $u\in L^{s}(\Sigma,\mu)$. If
$\Sigma$ has infinite measure $\mu$ and if the inequalities
\eqref{eq:20} and~\eqref{eq:168} hold for $q=1$ and $s=1$ then
\eqref{eq:20} and~\eqref{eq:168} hold also for all
$u\in L^{1}(\Sigma,\mu)$.

By comparing Theorem~\ref{thm:GN-implies-reg-for-oper-in-L1} with
Theorem~\ref{thm:GN-implies-reg-bis}, we note that
quasi-$m$-completely accretive operators $A$ in $L^{q_{0}}$ for some
$q_{0}\ge 1$ (that is, operators of the first class considered in
Theorem~\ref{thm:GN-implies-reg-bis}), it is not difficult to see that
the trace $A_{1\cap \infty}$ of $A$ in $L^{1}\cap L^{\infty}$
satisfies the \emph{range condition}~\eqref{eq:148}. This is not
immediately clear for quasi-$m$-accretive operators $A$ in $L^{1}$
with complete resolvent (that is, the second class of operators
considered in Theorem~\ref{thm:GN-implies-reg-for-oper-in-L1}). Thus,
we provide in Proposition~\ref{propo:Range-cond-in-Rd} sufficient
conditions yielding that operators $A\phi$ composed of an operator $A$
of the first class and a monotone graph $\phi$ satisfies range
condition~\eqref{eq:148}. On the other hand, there are operators of
the first class which do not satisfy the one-parameter family of
Sobolev type inequalities~\ref{eq:10single} in
Theorem~\ref{thm:GN-implies-reg-for-oper-in-L1} but they do satisfy
all assumptions in Theorem~\ref{thm:GN-implies-reg-bis}. Examples of
such operators are the nonlinear Dirichlet-to-Neumann operator
associated with $p$-Laplace type operators (see
Section~\ref{subsec:DtN}) or the fractional $p$-Laplace operator
equipped with some boundary conditions (see
Section~\ref{subsection:fractional-p-laplace}). Here, it is interesting to
note that both operators are of \emph{nonlocal} character (cf.~\cite{MR3369257,MaRoTo2015}).

The proof of Theorem~\ref{thm:GN-implies-reg-for-oper-in-L1} follows from
Corollary~\ref{cor:Gagliardo-for-accretive-op-in-L1-complete-res}
(respectively,
Corollary~\ref{cor:Gagliardo-for-accretive-op-c-complete}) and
Theorem~\ref{thm:Moser}. Comparing
Theorem~\ref{thm:GN-implies-reg-for-oper-in-L1} with the existing
literature (cf., for instance,
\cite{MR2268115,MR2409206,MR2379911,Takac05,MR2286292}), we note that
so far, in order to establish an $L^{1}$-$L^{\infty}$-regularisation
estimate~\eqref{eq:146} for semigroups $\{T_{t}\}_{t\ge 0}$ generated
by doubly nonlinear operators $\Delta_{p}(\cdot^{m})$ on
$L^{1}(\R^{d})$, one had either to assume more regularity on the
solution of the associated parabolic problem (cf., for instance,
\cite{MR2268115,MR2379911,Takac05}) or to approximate the semigroup
$\{T_{t}\}_{t\ge 0}$ on $L^{1}(\R^{d})$ by the semigroup
$\{T_{t}^{n}\}_{t\ge 0}$ generated by the Dirichlet-doubly nonlinear
operator $\Delta_{p}^{D}(\cdot^{m})$ on $L^{1}(\Sigma_{n})$ for a
sequence $(\Sigma_{n})_{n}$ of open sets $\Sigma_{n}\subseteq \R^{d}$
with finite measure, smooth boundary, and satisfying $\Sigma_{n}\subseteq \Sigma_{n+1}$
and $\bigcup_{n\ge 1}\Sigma_{n}=\R^{d}$ (cf., for instance,
\cite{MR2286292}). Our approach presented in
Theorem~\ref{thm:GN-implies-reg-for-oper-in-L1} is different to this
one and to the one given by V\'eron~\cite{MR554377}. Thus, the statement of
Theorem~\ref{thm:GN-implies-reg-for-oper-in-L1} unfolds its complete
strength in the case $(\Sigma,\mu)$ is an infinite measure space. It is based on
the result stated in Proposition~\ref{propo:Range-cond-in-Rd}. We
emphasise that we generalise in
Proposition~\ref{propo:Range-cond-in-Rd} an idea by Crandall and
Pierre \cite{MR647071} for the composition $A\phi$ of a
\emph{nonlinear} completely accretive operator $A$ and a non-decreasing
function $\phi : \R\to \R$. 

As an application of the theory developed in this monograph, we
provide in Section~\ref{Asec:proof-of-Linfty-regularity} an abstract
approach showing that a global $L^{1}$-$L^{\infty}$-regularisation
estimate satisfied by a semigroup $\{T_{t}\}_{t\ge 0}$ implies that
the trajectories $u(t):=T_{t}u_{0}$ in $L^{1}$, $t\ge 0$, for initial
values $u_{0}\in L^{1}$, are, in fact, \emph{weak energy solutions} of
the corresponding abstract initial value problem (see
Definition~\ref{Def:weak-sols-in-V} in
Section~\ref{Asec:proof-of-Linfty-regularity}). Here, the semigroup
$\{T_{t}\}_{t\ge 0}$ is generated by a quasi-$m$-$T$-accretive
operator in $L^{1}$ of the form $A_{1\cap\infty}\phi+F$, where $A$ is
the realisation in $L^{2}$ of the G\^ateaux-derivative
$\Psi' : V\to V'$ of a convex real-valued functional $\Psi$ defined on
a Banach space $V\hookrightarrow L^{2}$, and $A$ has the property to
be an $m$-completely accretive operator in $L^{2}$. Further, $\phi$ is
assumed to be a continuous, strictly increasing function on $\R$, and
$F$ a Lipschitz mapping on $L^{1}$. In the case $F\equiv 0$ and
$A\phi$ is the celebrated negative porous media operator
$-\Delta\phi$, this result is well-known
(see~\cite{MR1164641,MR0454360,MR2286292}). Our results
(Theorem~\ref{thm:weak-solutions-Li-Linfty},
Theorem~\ref{thm:Linfty-implies-mild-are-strong-in-L1} and
Theorem~\ref{thm:reg-mild-sol-are-weak}) in
Section~\ref{Asec:proof-of-Linfty-regularity} extend the known
literature by providing a uniform approach which can be applied to
general quasi-$m$-$T$-accretive operators $A_{1\cap\infty}\phi+F$ in
$L^{1}$ where the operator $A_{1\cap\infty}$ needs \emph{not to be
  linear}, but has the important property to be completely accretive
in $L^{2}$. % We apply this theory to the famous \emph{doubly nonlinear
%   operator} $\Delta_{p}\phi$ in
% Section~\ref{subsec:Mild-is-strong}.
Concerning the regularity of trajectories $u(t):=T_{t}u_{0}$ in
$L^{1}$ generated by operators $-A\phi$ with similar characteristics as
$\Delta_{p}\phi$, the notion of \emph{entropy solutions} was developed
(see~\cite{MR1354907,MR2437810,MR2033382}).
Theorem~\ref{thm:weak-solutions-Li-Linfty} in
Section~\ref{Asec:proof-of-Linfty-regularity} improves the regularity
of these trajectories $u(t):=T_{t}u_{0}$ in $L^{1}$ essentially.

We demonstrate the efficiency of the methods and techniques developed
in the Sections~\ref{gn}, \ref{sec:extra-pol-twoards-one},
\ref{sec:extrapolation-towards-infinity}
and~\ref{sec:Alternative-Moser-iteration} with a plethora of examples
gathered in Section~\ref{sec:examples}. Section~\ref{sec:p-laplace} is
concerned with establishing $L^{q}$-$L^{r}$-regulari\-sation estimates
of the semigroup generated by a negative Leray-Lions type operator
equipped with either homogeneous Dirichlet, Neumann or Robin boundary
conditions. Our results in this section improve some known results in
the literature. For instance, we prove that the exponents in the
$L^{q}$-$L^{r}$-regularisation estimates remain unchanged by adding a
monotone (multi-valued) or a Lipschitz perturbation. Note that our methods
yield sharp exponents as one can see  in
Section~\ref{sec:p-laplace}. Section~\ref{sec:nonlocal} is dedicated
to establishing the $L^{q}$-$L^{r}$-re\-gu\-lari\-sation estimates of two
nonlocal parabolic problems where known methods fail. In this section, we establish the
$L^{q}$-$L^{r}$-re\-gu\-lari\-sation estimates of the semigroup
generated by the negative Dirichlet-to-Neumann operator associated
with a Leray-Lions type operator and of the semigroup generated by the
fractional $p$-Laplace operator equipped with either Dirichlet or
Neumann boundary conditions. In Section~\ref{sec:doubly-nonl-diff}, we
establish the $L^{q}$-$L^{r}$-regulari\-sation estimates of
\emph{mild} solutions of parabolic problems involving the negative doubly
nonlinear operator $-\Delta_{p}(\cdot^{m})$ for $m>0$ equipped with
either homogeneous Dirichlet, Neumann or Robin boundary conditions. In
this section, we use the properties of the $p$-Laplace operator
established in Section~\ref{sec:p-laplace}, to construct the operator
$-\Delta_{p}\phi$.

In Section~\ref{subsec:Mild-is-strong}, we
employ  the $L^{1}$-$L^{\infty}$-regularisation estimate~\eqref{eq:168} satisfied by the
semigroup $\{T_{t}\}_{t\ge 0}$ of the doubly nonlinear
operator $\Delta_{p}(\cdot^{m})$ equipped with either Dirichlet,
Neumann or Robin boundary conditions to show that for every given
initial value $u_{0}\in L^{1}$, the mild solution $u(t):=T_{t}u_{0}$,
$t\ge 0$, in $L^{1}$ is a \emph{strong energy
  solution} (see Theorem~\ref{thm:weak-is-strong}). We give the exact definition of
strong solutions in the next section, where we set the general
framework in which we are working and briefly
review some classical definitions and
important results in nonlinear semigroup theory.

\subsection{Acknowledgements}  
\label{sec:acknowledgements}

Both authors are very grateful to Professor
Michel Pierre (ENS Cachan) for his tireless willingness to discuss
important problems concerning nonlinear semigroups generated by the
porous media operator and for his valuable suggestions.

%--------------------------------------------------------------
%--------------------------------------------------------------
%--------------------------------------------------------------
%--------------------------------------------------------------
%--------------------------------------------------------------
%--------------------------------------------------------------
%--------------------------------------------------------------
%--------------------------------------------------------------
%--------------------------------------------------------------
%--------------------------------------------------------------
%--------------------------------------------------------------
%--------------------------------------------------------------
%--------------------------------------------------------------
%--------------------------------------------------------------
%--------------------------------------------------------------
%--------------------------------------------------------------
%--------------------------------------------------------------
%--------------------------------------------------------------
%--------------------------------------------------------------
%--------------------------------------------------------------
%--------------------------------------------------------------
%--------------------------------------------------------------
%--------------------------------------------------------------
%--------------------------------------------------------------

\section{Framework}\label{fm}

Throughout this monograph, $(\Sigma,\mu)$ denotes a $\sigma$-finite
measure space and $M(\Sigma,\mu)$ the space (of all classes) of
measurable real-valued functions on $\Sigma$. We denote by
$L^{q}(\Sigma,\mu)$, $1\le q\le \infty$, the corresponding standard
Lebesgue space with norm $\norm{\cdot}_{q}$. For $1\le q<\infty$, we
identify the dual space $(L^{q}(\Sigma,\mu))'$ with
$L^{q^{\mbox{}_{\prime}}}(\Sigma,\mu)$ and use the notation
$\langle u^{\mbox{}_{\prime}},u\rangle$ to denote the natural pairing
of $u^{\mbox{}_{\prime}}\in L^{q^{\mbox{}_{\prime}}}(\Sigma,\mu)$ and
$u\in L^{q}(\Sigma,\mu)$, where $q^{\mbox{}_{\prime}}$ is the
conjugate exponent of $q$ given by
$1=\tfrac{1}{q}+\tfrac{1}{q^{\mbox{}_{\prime}}}$. More generally, for
every topological vector space $X\subseteq M(\Sigma,\mu)$, we denote
by $\langle \psi,u\rangle$ the value of $\psi\in X'$ at $u\in X$.  In
the case $1<q<\infty$, we shall write $u_{q}$ to denote
$\abs{u}^{q-2}u$ for every $u\in L^{q}(\Sigma,\mu)$.

\subsection{Nonlinear semigroup theory: old and new}

Most of this section can be skipped by readers familiar with classical
nonlinear semigroup theory. Most of the results of this theory can be
found in the books~\cite{MR2582280} by Barbu, \cite{MR1192132} by
Miyadera or in the famous draft~\cite{Benilanbook} of the
\emph{unpublished} book by B\'enilan, Crandall and Pazy.

\medskip

We call a mapping $A$ from $M(\Sigma,\mu)$ into
the set of all subsets of $M(\Sigma,\mu)$, denoted by
$2^{M(\Sigma,\mu)}$, an \emph{operator} on $M(\Sigma,\mu)$. As usual, we
identify an operator $A$ on $M(\Sigma,\mu)$ with its \emph{graph},
that is, with the set 
\begin{displaymath}
  \Big\{(u,v)\in M(\Sigma,\mu)\times M(\Sigma,\mu)\,\big\vert\; v\in Au\Big\},
\end{displaymath}
and thus, we shall say that $(u,v)\in A$ if $v\in Au$. The
\emph{effective domain} $D(A)$ of $A$ denotes the set of all
$u\in M(\Sigma,\mu)$ satisfying $Au\neq \emptyset$ and the
\emph{range} $Rg(A)$ of $A$ the set
$\bigcup_{u\in D(A)}Au\subseteq M(\Sigma,\mu)$. The \emph{inverse
  operator} $A^{-1}$ of $A$ is given by the set of all pairs
$(u,v)\in M(\Sigma,\mu)\times M(\Sigma,\mu)$ satisfying $u \in Av$,
hence $D(A^{-1})=Rg(A)$ and $Rg(A^{-1})=D(A)$. Given two operators $A$
and $B$ on $M(\Sigma,\mu)$ and a scalar $\alpha\in \R$, the operator
$A+\alpha B$ is given by $(A+\alpha B)u=Au+\alpha(Bu)$ for every
$u\in D(A)\cap D(B)$. Further, the \emph{composition} $AB:=A\circ B$
of two operators $A$ and $B$ on $M(\Sigma,\mu)$ is defined by
\begin{displaymath}
  AB = \Big\{(u,v)\in M(\Sigma,\mu)\times M(\Sigma,\mu)\,\Big\vert\;
  \text{ there is }z\in Bu \text{ such that }v\in Az\Big\}.
\end{displaymath}

Let $X\subseteq M(\Sigma,\mu)$ be a Banach space with norm
$\norm{.}_X$. Then, an operator $A$ on $X$, meaning that
$A\subseteq X\times X$, is said to be \emph{densely defined} or with
\emph{dense domain} if its effective domain $D(A)$ is dense in $X$. We
denote by $\overline{A}$ the closure of the graph of $A$ in $X$ and
call $\overline{A}$ the \emph{closure of $A$} in $X$. We call $A$
closed if $A=\overline{A}$. Obviously, the domain $D(\overline{A})$ of
the closure $\overline{A}$ of $A$ in $X$ is closed in $X$. For any
sequence $(A_{n})_{n\ge 0}$ of operators $A_{n}$ on $X$, the
\emph{limit inferior of $(A_{n})$} denoted by
$\liminf_{n\to\infty}A_{n}$ is defined by
\begin{displaymath}
  \Big\{(u,v)\in X\times X\,\Big\vert\,\text{there are
  }(u_{n},v_{n})\in A_{n}\text{
    s. t. }\lim_{n\to\infty}(u_{n},v_{n})=(u,v)\text{ in }X\times X\Big\}.
\end{displaymath}

Now, an operator $A$ on $X$ is called \emph{accretive (in $X$)} if for
every $(u,v)$, $(\hat{u},\hat{v})\in A$ and every $\lambda\ge 0$, one
has
\begin{equation}
  \label{eq:63}
  \norm{u-\hat{u}}_{X}\le \norm{u-\hat{u}+\lambda (v-\hat{v})}_{X}.
\end{equation}
In other words, $A$ is accretive in $X$ if and only if for every
$\lambda>0$, the \emph{resolvent operator}
$J_{\lambda}=(I+\lambda A)^{-1}$ of $A$ is a single-valued mapping
from $Rg(I+\lambda A)$ to $D(A)$, which is \emph{contractive} (also called \emph{nonexpansive})
with respect to the norm of $X$. Recall a mapping $S : D(S)\to X$ with domain $D(S)\subseteq X$
is call \emph{contractive} with respect to the norm of $X$ or a \emph{contraction} in $X$ if
\begin{displaymath}
  \norm{Su-S\hat{u}}_{X}\le \norm{u-\hat{u}}_{X},
\end{displaymath}
for all $u,\hat{u}\in D(S)$. In order to  better
grasp  the definition of accretive operators, one might
first  consider
accretive operators $\beta$ on $\R$. If $X=\R$ is equipped with the
absolute value $\abs{\cdot}$ then for every $(u,v)$ and
$(\hat{u},\hat{v})\in \beta$, inequality~\eqref{eq:63} is equivalent
to the inequality
\begin{equation}
  \label{eq:65}
  (v-\hat{v})(u-\hat{u})\ge 0.
\end{equation}
This shows that a single-valued operator $\beta$ on $\R$ is accretive
if and only if $\beta$ is non-increasing. From now on we refer to
accretive operators on $\R$ as \emph{monotone graphs}
(cf.~\cite[Example~(2.3)]{Benilanbook}). For a given monotone graph
$\beta$ on $\R$ and for every $1\le q\le\infty$, we denote by
$\beta_{q}$ the \emph{associated accretive operator with $\beta$} in
$L^{q}(\Sigma,dx)$ given by
\begin{equation}
  \label{eq:66}
  \beta_{q}=\Big\{(u,v)\in L^{q}\times L^{q}(\Sigma,\mu)\,\Big\vert\;v(x)\in
  \beta(u(x))\text{ for a.e. $x\in \Sigma$}\Big\}.
\end{equation}

There are important characterisations of accretivity, which we use from time to time
throughout this paper. Here is the first one: an operator $A$ is accretive in $X$ if and only if
\begin{equation}
  \label{eq:40}
  \begin{cases}
    &\textrm{for every $(u,v)$,
      $(\hat{u},\hat{v})\in A$, there exists $\psi\in J(u-\hat{u})$}\\
    &\textrm{satisfying }\langle \psi,v-\hat{v} \rangle\ge0,
  \end{cases}
\end{equation}
where $J : X \to 2^{X'}$ denotes the \emph{duality mapping} of $X$,
which is given by
\begin{displaymath}
  J(u)=\Big\{\psi\in X'\,\Big\vert\;\langle \psi,u\rangle=\norm{u}_{X}\text{
    and }\norm{\psi}_{X'}\le 1\Big\}
\end{displaymath}
for every $u\in X$ (cf.~\cite[Theorem~(2.15)]{Benilanbook}
or~\cite[Proposition~3.1]{MR2582280}).

Now, it is not difficult to verify (cf.~\cite[Example
(2.11)]{Benilanbook}) that  for
$q=1$, the duality mapping
$J$ on $L^{1}(\Sigma,\mu)$ is given by
\begin{displaymath}
  J(u)=\Big\{\psi\in L^{\infty}(\Sigma,\mu)\,\big\vert\,\psi(x)\in
  \textrm{sign}(u(x))\text{ for a.e. $x\in \Sigma$ }\Big\}
\end{displaymath}
for every $u\in L^{1}(\Sigma,\mu)$, where the multi-valued
\emph{signum} function is defined by
\begin{displaymath}
  \textrm{sign}(s):=
  \begin{cases}
    1 & \text{if $s>0$,}\\
    [-1,1] & \text{if $s=0$,}\\
    -1 & \text{if $s<0$}
  \end{cases}
\end{displaymath}
for every $s\in \R$, and for $1<q<\infty$, the duality mapping $J$
on $L^{q}(\Sigma,\mu)$ is a well-defined mapping $J :
L^{q}(\Sigma,\mu)\to L^{q^{\mbox{}_{\prime}}}(\Sigma,\mu)$ given by
\begin{equation}
  \label{eq:37}
  J(u)=u_{q}\,\norm{u}_{q}^{1-q}
\end{equation}
for every $u\in L^{q}(\Sigma,\mu)$. In the case $q=1$, $J(u)$ is multi-valued
exactly when the set $\{u=0\}$ has strictly positive
$\mu$-measure. However, the multi-valued signum function
$\textrm{sign}(\cdot)$ can be approximated by the sequence
$(\gamma_{\varepsilon})_{\varepsilon>0}$ of piecewise smooth functions
$\gamma_{\varepsilon} : \R\to \R$ defined by
\begin{equation}
  \label{eq:64}
  \gamma_{\varepsilon}(r)=
  \begin{cases}
    1 & \text{if $r>\varepsilon,$}\\
   \tfrac{r}{\varepsilon} & \text{if $-\varepsilon\le r\le
      \varepsilon$,}\\
    -1 & \text{if $r<-\varepsilon$.}
  \end{cases}
\end{equation}

For our purposes, it is convenient to introduce the notion of
\emph{$q$-brackets}. For $1\le q<\infty$, the \emph{$q$-bracket}
$[\cdot,\cdot]_{q} : L^{q}(\Sigma,\mu)\times L^{q}(\Sigma,\mu)\to \R$ is defined by 
\begin{displaymath}
  [u,v]_{q}=\lim_{\lambda\to 0+}
  \frac{\frac{1}{q}\norm{u+\lambda v}_{q}^{q}-\frac{1}{q}\norm{u}_{q}^{q}}{\lambda}
\end{displaymath}
for every $u$, $v\in L^{q}(\Sigma,\mu)$. For given $u$, $v\in
L^{q}(\Sigma,\mu)$, the number $[u,v]_{q}$ is the
right-hand directional derivative of the function $u\mapsto
\frac{1}{q}\norm{u}_{q}^{q}$. Since the function $\lambda\mapsto 
\frac{1}{q}\norm{u+\lambda v}_{q}^{q}$ is convex on $\R$, we can
define $[\cdot,\cdot]_{q}$, equivalently, by
\begin{equation}
  \label{eq:53}
  [u,v]_{q}=\inf_{\lambda>0}
  \frac{\frac{1}{q}\norm{u+\lambda v}_{q}^{q}-\frac{1}{q}\norm{u}_{q}^{q}}{\lambda}
\end{equation}
for every $u$, $v\in L^{q}(\Sigma,\mu)$. The $q$-bracket
$[\cdot,\cdot]_{q} : L^{q}(\Sigma,\mu)\times L^{q}(\Sigma,\mu)\to \R$
is upper semicontinuous (respectively, continuous if $1<q<\infty$) and
\begin{equation}
  \label{eq:35}
  [u,v]_{q}=\langle u_{q},v\rangle\qquad\text{for every 
    $u$, $v\in L^{q}(\Sigma,\mu)$ if $1<q<\infty$,}
\end{equation}
while for $q=1$, $[\cdot,\cdot]_{1}$ reduces to the classical
\emph{brackets} $[\cdot,\cdot]$ on $L^{1}(\Sigma,\mu)$ given by
\begin{equation}
  \label{eq:67}
  [u,v]_{1}=\int_{\{u\neq 0\}}\textrm{sign}_{0}(u)\,v\,\dmu + \int_{\{u=0\}}\abs{v}\,\dmu
\end{equation}
for every $u$, $v\in L^{1}(\Sigma,\mu)$, where the \emph{restricted
  signum} $\textrm{sign}_{0}$ is defined by
\begin{displaymath}
  \textrm{sign}_{0}(s)=
  \begin{cases}
    1 & \text{if $s>0$,}\\
    0 & \text{if $s=0$,}\\
    -1 & \text{if $s<0$}
  \end{cases}
\end{displaymath}
for every $s\in \R$ (cf.~\cite[Section 2.2 \& Example~(2.8)]{Benilanbook} or \cite[pp
102]{MR2582280}). By characterisations~\eqref{eq:40}, \eqref{eq:37} and \eqref{eq:35} if
$1<q<\infty$, respectively, by \cite[Theorem~2.14]{Benilanbook} (or,
alternatively, \cite[p 103 formula (3.15)]{MR2582280}) if $q=1$, we
see that an operator $A$ on $L^{q}(\Sigma,\mu)$ is accretive if and
only if
\begin{displaymath}
  [u-\hat{u},v-\hat{v}]_{q}\ge0 \qquad\text{for every $(u,v)$, 
    $(\hat{u},\hat{v})\in A$.}
\end{displaymath}
In Section~\ref{gn}, we shall need the following two properties of $q$-brackets:
\begin{equation}
  \label{eq:56}
   [u,v]_{q}\le
   \frac{1}{q}\norm{u+v}_{q}^{q}-\frac{1}{q}\norm{u}_{q}^{q}\qquad
   \text{for every $u$, $v\in L^{q}(\Sigma,\mu)$,}
\end{equation}
and 
\begin{equation}
  \label{eq:57}
  [u,\alpha\, v+\omega u]_{q}=\alpha\, [u,v]_{q}+\omega\,\norm{u}_{q}^{q}
\end{equation}
for every $u$, $v\in L^{q}(\Sigma,\mu)$, $\omega$, $\alpha\in
\R$. Here, note that inequality~\eqref{eq:56} is an immediate consequence
of~\eqref{eq:53}.  Property~\eqref{eq:57} is shown for $q=1$ in
\cite[Proposition~(2.5)]{Benilanbook} (or, alternatively,
\cite[Proposition~3.7]{MR2582280}) and if $1<q<\infty$
then~\eqref{eq:57} can be easily deduced from \eqref{eq:35}.

If $X\subseteq M(\Sigma,\mu)$ is a Banach lattice, then we shall
denote the usual lattice operations $u\vee\hat{u}$ and
$u\wedge \hat{u}$ to be the almost everywhere pointwise supremum and
infimum of $u$ and $\hat{u}\in X$. In addition, $u^{+}=u\vee 0$ is the
positive part, $u^- = (-u)\vee 0$ the negative part, and
$|u| = u^+ + u^-$ the absolute value of an element $u\in X$. For every
$u$, $\hat{u}\in X$, one denotes by $u\le \hat{u}$ the usual order
relation on $X$. In this framework, a mapping $S : D(S) \to X$ with domain
$D(S)\subseteq X$ is called \emph{order preserving} if $Su \le S\hat{u}$
for every $u\le \hat{u}$, \emph{positive} if $Su\ge0$ for every $u\ge
0$, and a \emph{$T$-contraction} if
\begin{displaymath}
  \norm{[Su-S\hat{u}]^{+}}_{X}\le \norm{[u-\hat{u}]^{+}}_{X}
\end{displaymath}
for every $u$, $\hat{u}\in D(S)$. Note that if $S$ is $T$-contractive
then it is order-preserving and that the converse holds if $S$ is
contractive and satisfies $u\vee\hat{u}$ and $u\wedge\hat{u}\in D(S)$
for every $u$, $\hat{u}\in D(S)$ (see \cite[Lemma
(19.11)]{Benilanbook}).  We shall say that an operator $A$ on $X$ is
\emph{$T$-accretive} if for every $\lambda>0$, the resolvent
$J_{\lambda}$ of $A$ defines a $T$-contraction with domain
$D(J_{\lambda})=Rg(I+\lambda A)$.

Note, without further assumptions (cf.~\cite[Proposition~(19.13)]{Benilanbook}),
$T$-con\-tractive does not imply contractive, nor does $T$-accretive imply
accretive, and vice-versa. However, if the norm $\norm{\cdot}_{X}$ on
$X$ satisfies the implication
\begin{equation}
  \label{eq:14}
  \norm{u^{+}}_{X}\le \norm{\hat{u}^{+}}_{X},\quad
    \norm{u^{-}}_{X}\le \norm{\hat{u}^{-}}_{X}\quad\text{implies}\quad
      \norm{u}_{X}\le \norm{\hat{u}}_{X}
\end{equation}
for every $u$, $\hat{u}$, then every $T$-contraction is also a
contraction (cf. \cite[p 267]{Benilanbook}). One easily verifies that
this implication holds, for instance, for the space
$X=L^{q}(\Sigma,\mu)$ for every $1\le q\le \infty$. Thus, for the rest
of this monograph, if we speak about $T$-contractive or $T$-accretive
operators on $X$, then we automatically assume that the underlying
space $X$ is a Banach lattice satisfying~\eqref{eq:14}.

In the space $X=L^{1}(\Sigma,\mu)$, the property that an
operator $A$ in $L^{1}(\Sigma,\mu)$ is $T$-accretive can be
characterised as follows: for every $(u,v)$, $(\hat{u},\hat{v})\in A$, there is a
$w\in L^{\infty}(\Sigma,\mu)$ satisfying $w(x)\in
\textrm{sign}^{+}(u(x)-\hat{u}(x))$ for a.e. $x\in \Sigma$ and 
\begin{displaymath}
  \int_{\Sigma}w\,(v-\hat{v})\,\dmu\ge 0,
\end{displaymath}
where for every $s\in \R$,
\begin{displaymath}
  \textrm{sign}^{+}(s):=
  \begin{cases}
    1 & \text{if $s>0$,}\\
    [0,1] & \text{if $s=0$,}\\
    0 & \text{if $s<0$,}
  \end{cases}
\end{displaymath}
or equivalently (cf.~\cite{Benthesis}) for every $(u,v)$,
$(\hat{u},\hat{v})\in A$, one has
\begin{displaymath}
  [u-\hat{u},v-\hat{v}]_{+}:=
  \int_{\{u=\hat{u}\}}[v-\hat{v}]^{+}\,\dmu+\int_{\{u>\hat{u}\}}(v-\hat{v})\,\dmu\ge 0.
\end{displaymath}

% In this setting, $A$ is called
% \emph{$m$-$T$-accretive} in $X$ if $A$ is $m$-accretive and
% $T$-accretive in $X$.

In order to conclude that the sum $A+B$ of two operators $A$ and
$B$ in $X$ is accretive, the assumption that $A$ and $B$ are both
accretive is not sufficient
(cf.~\cite[Exercise E2.3]{Benilanbook}). For this to be true, we need that
at least one of the two operators $A$ and $B$ admits the following
stronger property. We call an operator $A$ \emph{$s$-accretive} in $X$
if for every $(u,v)$, $(\hat{u},\hat{v})\in A$ and for every $\psi\in J(u-\hat{u})$,
\begin{displaymath}
  \langle \psi,v-\hat{v} \rangle\ge  0.
\end{displaymath}
Now, if $A$ is an operator on $X$ then the sum $A+B$ is accretive in
$X$ for every accretive operator $B$ on $X$ if and only if $A$ is
$s$-accretive (cf.~\cite[Proposition~(2.20)]{Benilanbook}). Obviously,
for $1<q<\infty$, every accretive operator $A$ in $L^{q}(\Sigma,\mu)$
is $s$-accretive in $L^{q}(\Sigma,\mu)$. Unfortunately, this is not
true for accretive operators $A$ in $L^{1}(\Sigma,\mu)$. A counter
example is, for instance, given by the accretive operator $\beta_{1}$
in $L^{1}((-2,1),\textrm{d}x)$ for $\beta(s):=\textrm{sign}(s)$
(cf. ~\cite[Exercise~E2.25]{Benilanbook}). On the other hand, a
prototype of $s$-accretive operators in $L^{1}(\Sigma,\mu)$ is
provided by the accretive operator $\beta_{1}$ in $L^{1}(\Sigma,\mu)$
associated with a \emph{non-decreasing function} $\beta : \R\to \R$.
To see that $\beta_{1}$ is $s$-accretive in $L^{1}(\Sigma,\mu)$, we
need to check that for every $(u,v)$, $(\hat{u},\hat{v})\in \beta_{1}$
every $\psi\in L^{\infty}(\Sigma,\mu)$ satisfying
$\psi(x)\in \textrm{sign}(u(x)-\hat{u}(x))$ for a.e. $x\in \Sigma$,
one has
\begin{equation}
  \label{eq:15}
  \int_{\Sigma}\psi\,(v-\hat{v})\,\dmu \ge 0.
\end{equation}
Since $\beta$ is assumed to be real-valued function, we have that
$v=\beta(u)$ and $\hat{v}=\beta(\hat{u})$. Thus, and by the
monotonicity of $\beta$, for a.e. $x\in
\Sigma$, the condition $v(x)>\hat{v}(x)$ implies $u(x)>\hat{u}(x)$ and
so, $\psi(x)=1$ hence $\psi(x)\,(v(x)-\hat{v}(x))\ge 0$. Analogously, for a.e. $x\in
\Sigma$, the condition $v(x)<\hat{v}(x)$ implies $u(x)<\hat{u}(x)$ and
so $\psi(x)=-1$ hence $\psi(x)\,(v(x)-\hat{v}(x))\ge
0$. Therefore, \eqref{eq:15} holds. 

 Next, an operator $A$ on $X$ is called \emph{$m$-($T$)-accretive in $X$} if
 $A$ is ($T$)-accretive in $X$ and satisfies the \emph{range condition}
\begin{equation}
  \label{eq:range-condition}
  Rg(I+\lambda A)=X\qquad\text{for some (or equivalently all) $\lambda>0$.}
\end{equation}

Coming back to the example of a monotone graph $\beta$ in $\R$, we set
$\beta(r+)=\inf\beta(]r,\infty[)$ and
$\beta(r-)=\sup\beta(]-\infty,r[)$ for every $r\in\R$, where, as
usual, $\inf\emptyset:=+\infty$ and $\sup\emptyset:=-\infty$. Then, a
monotone graph $\beta$ in $\R$ is $m$-accretive if and only if for
every $r\in \R$, one has
\begin{displaymath}
  \beta(r)=[\beta(r-),\beta(r+)]\cap \R.
\end{displaymath}
Therefore, a monotone graph $\beta$ in $\R$ is $m$-accretive
if and only if the graph of $\beta$ is the maximal monotone set in
$\R\times\R$ containing $\beta$ itself. If $\beta$ is $m$-accretive in
$\R$ and either $(0,0)\in\beta$ or $(\Sigma,\mu)$ is finite, then for
every $1\le q\le \infty$, the associated operator $\beta_{q}$ on
$L^{q}(\Sigma,\mu)$ is $m$-$T$-accretive (cf.~\cite[Examples~(8.4)
\&~(8.5)]{Benilanbook} or \cite[Section 3.2]{MR2582280}). Moreover,
the resolvent operator $J_{\lambda}$ of $\beta_{q}$ is given by
$(J_{\lambda}u)(x)=(1+\lambda\beta)^{-1}u(x)$ and the so-called
\emph{Yosida operator}
$\beta_{\lambda}(\cdot):=\lambda^{-1}(I-J_{\lambda})$ is given by
$\beta_{\lambda}(u)(x)=\lambda^{-1}(1-(1+\lambda\beta)^{-1})u(x)$ for
a.e. $x\in \Sigma$. If $1<q<\infty$ then for a given accretive
operator $A$ on $L^{q}(\Sigma,\mu)$, the sum $A+\beta_{q}$ is
accretive in $L^{q}(\Sigma,\mu)$. This is an immediate consequence of
the fact that the duality mapping $J$ of $L^{q}(\Sigma,\mu)$ is
single-valued on $L^{q}(\Sigma,\mu)$. However, in order to conclude
that for an $m$-accretive operator $A$ in $L^{q}(\Sigma,\mu)$,
$(1<q<\infty)$, satisfying $D(A)\cap D(\beta_{q})\neq\emptyset$, the
sum $A+\beta_{q}$ is $m$-accretive in $L^{q}(\Sigma,\mu)$, one needs
an additional condition. A possible one is the following
(cf.~\cite[Proposition~3.8]{MR2582280}):
\begin{equation}
  \label{eq:179}
  [\beta_{\lambda}(u),v]_{q}\ge 0\qquad\text{for every $\lambda>0$,
    $(u,v)\in A$.}
\end{equation}
% If for some $1\le q\le \infty$,
% $A : L^{q}(\Sigma,\mu)\to L^{q}(\Sigma,\mu)$ is a continuous
% $m$-accretive mapping then $A+\beta_{q}$ is $m$-accretive in
% $L^{q}(\Sigma,\mu)$ (cf.~\cite[Theorem~3.1]{MR2582280}). The latter
% result is very usefull in problems where the $m$-accretive operator
% $A$ on $L^{q}(\Sigma,\mu)$ is replaced by its Yosida operator
% $A_{\lambda}:=\tfrac{1}{\lambda}(I-J_{\lambda})$, ($\lambda>0$), which is an
% $m$-accretive Lipschitz continuous mapping from $L^{q}(\Sigma,\mu)$
% to $L^{q}(\Sigma,\mu)$.

Another important example of $m$-accretive operators in
$L^{2}(\Sigma,\mu)$ is given by the \emph{subgradient}
\begin{displaymath}
  \partial_{\! L^2}\Psi:=\Big\{(u,v)\in L^{2}\times
  L^{2}(\Sigma,\mu)\,\Big\vert\; \langle v,\xi-u\rangle\le
  \Psi(\xi)-\Psi(u)\text{ for all }\xi\in L^{2}(\Sigma,\mu)\Big\}
\end{displaymath}
in $L^{2}(\Sigma,\mu)$ of a functional
$\Psi : L^{2}(\Sigma,\mu)\to \R\cup\{+\infty\}$ which is convex,
lower semicontinuous and proper (see~\cite{MR0348562} and 
also~\cite{arXiv:1412.4151}).

More generally, an operator $A$ on $X$ is called \emph{quasi ($T$)-accretive}
if there is an $\omega\in \R$ such that $A+\omega I$ is
($T$)-accretive in $X$. Obviously, if $A+\omega I$ is ($T$)-accretive for some
$\omega\in \R$ then $A+\tilde{\omega} I$ is ($T$)-accretive for every
$\tilde{\omega}\ge \omega$. Thus, there is no loss of generality in assuming
that $A+\omega I$ is ($T$)-accretive for some $\omega\ge 0$.
Finally, we call $A$ \emph{quasi $m$-($T$)-accretive} if $A+\omega I$ is
$m$-($T$)-accretive for some $\omega\in \R$. It is easy to check
that $A+\omega I$ is ($T$)-accretive for some $\omega\in \R$ if and only if
for every $\lambda>0$ satisfying $\lambda\omega<1$,
the resolvent $J_{\lambda}$ of $A$ satisfies
\begin{displaymath}
  \norm{J_{\lambda}u-J_{\lambda}\hat{u}}_{X}\le
  (1-\lambda\omega)^{-1}\,\norm{u-\hat{u}}_{X}
\end{displaymath}
for every $u$, $\hat{u}\in Rg(I+\lambda A)$  (respectively, one has
\begin{displaymath}
  \norm{[J_{\lambda}u-J_{\lambda}\hat{u}]^{+}}_{X}\le
  (1-\lambda\omega)^{-1}\,\norm{[u-\hat{u}]^{+}}_{X}
\end{displaymath}
for every $u$, $\hat{u}\in Rg(I+\lambda A)$). It is important to know
that if $A+\omega I$ is $m$-accretive for some $\omega\in \R$, then
for every $\lambda>0$ such that $\omega \lambda<1$ and $u\in
\overline{D(A)}^{\mbox{}_{X}}$, the closure of $D(A)$ in $X$, one has
$J_{\lambda}u\in D(A)$ and
\begin{equation}
  \label{eq:212}
  \lim_{\lambda\to 0}J_{\lambda}u=u\qquad\text{in $X$}
\end{equation}
(cf.  \cite[Proposition~(4.4)]{Benilanbook} or
\cite[Proposition~3.2]{MR2582280}). Prototype examples of quasi
($T$)-accretive operators $A$ on $X=L^{q}(\Sigma,\mu)$,
$1\le q\le \infty$, are of the form $A=B+F$, where $B$ denotes a
($T$)-accretive operator on $L^{q}(\Sigma,\mu)$ and
$F : L^{q}(\Sigma,\mu)\to L^{q}(\Sigma,\mu)$ is defined by
$F(u)(x):=f(x,u(x))$ for every $u\in L^{q}(\Sigma,\mu)$ of a given
$f : \Sigma\times\R\to \R$ with the properties that
$f(\cdot, u) : \Sigma\to \R$ is measurable on $\Sigma$ for every
$u\in \R$, $f(x,0)=0$ for a.e. $x\in \Sigma$, and there is a constant
$\omega\ge 0$ such that
\begin{equation}
  \label{eq:2}
  \abs{f(x,u)-f(x,\hat{u})}\le L\,\abs{u-\hat{u}}\qquad
  \text{for all $u$, $\hat{u}\in \R$ and a.e. $x\in \Sigma$.}
\end{equation}
A real-valued function $f$ satisfying such properties (or slightly
weaker ones) is also called a \emph{Carath\'eodory function}, and the
mapping $F $ given by $F(u)(x):=f(x,u(x))$ for every
$u\in L^{q}(\Sigma,\mu)$ the \emph{Nemytski operator} on
$L^{q}(\Sigma,\mu)$ associated with $f$.

If $q=1$ then an approximation argument
with the sequence $(\gamma_{\varepsilon})_{\varepsilon>0}$ given
by~\eqref{eq:64} and if $1<q<\infty$, using that the duality map $J$
on $L^{q}(\Sigma,\mu)$ is single-valued, one sees that the operator
$B+F+\omega I$ is ($T$)-accretive in $L^{q}(\Sigma,\mu)$. On
the other hand, for the same $\omega$, the operator $F+\omega I$ on
$L^{q}(\Sigma,\mu)$ is accretive and Lipschitz continuous. Hence a
standard fixed point argument shows that $F+\omega I$ is $m$-($T$)-accretive
in $L^{q}(\Sigma,\mu)$. Therefore, if $B$ is
$m$-($T$)-accretive in $L^{q}(\Sigma,\mu)$ for some $1\le q<\infty$, then $B+F$ is
quasi $m$-($T$)-accretive in $L^{q}(\Sigma,\mu)$
(cf.~\cite[Theorem~3.1]{MR2582280}).

For an accretive operator $A+\omega I$ on $X$, $\omega\in \R$, one
easily verifies that the following properties  hold
(cf.~\cite[Proposition~2.18]{Benilanbook}):
\allowdisplaybreaks
\begin{align}
\label{eq:111} & \textit{ The closure $\overline{A+\omega
                 I}$ of $A+\omega I$ in $X$ coincides with
                 $\overline{A}+\omega I$ and is accretive.}\\
\label{eq:112} & \textit{ If $A$ is closed, then $Rg (I+\lambda
                 (A+\omega I))$ is closed for every
  $\lambda>0$.}\\
\label{eq:113} &\textit{ If there is a $\lambda>0$ such that $Rg
                 (I+\lambda (A+\omega I))$ is closed, then $A$ is closed.}\\
\label{eq:114} &
                 \begin{array}[c]{l}
                   \textit{\!\!\! If $A\subseteq B$, $Rg (I+(B+\omega I))\subseteq
                   Rg(I+(A+\omega I))$ and $B+\omega I$ is}\\
                   \textit{\! accretive, then $A=B$.}
                 \end{array}
\end{align}

By the celebrated Crandall-Liggett
theorem~\cite[Theorem~I]{MR0287357}, the condition \emph{$A$ is
quasi $m$-accretive in $X$} ensures that for all
$u_0\in \overline{D(A)}^{\mbox{}_{X}}$, the abstract initial value problem
\begin{equation}
  \label{eq:9}
  \tfrac{du}{dt}+Au\ni 0,\quad u(0)=u_{0}
\end{equation}
is well-posed in the sense of mild solutions. In particular, if $A$ is quasi
$m$-$T$-accretive in $X$, then for every $u_{0}$, $\hat{u}_{0}\in
\overline{D(A)}^{\mbox{}_{X}}$ satisfying $u_{0}\le \hat{u}_{0}$, the
corresponding mild solutions $u$ and $\hat{u}$ of~\eqref{eq:9}
satisfy $u(t)\le \hat{u}(t)$ for all $t\ge 0$ (cf. \cite[Proposition
(19.12)]{Benilanbook}). To be more precise,
we first recall that for given $u_{0}\in X$, a function $u \in
W^{1,1}_{loc}((0,\infty);X)\cap C([0,\infty);X)$ satisfying
$u(0)=u_{0}$, $u(t)\in D(A)$ and $-\tfrac{du}{dt}(t)\in Au(t)$ a
\emph{strong solution of~\eqref{eq:9}}. Now, a \emph{mild solution $u$
  of Cauchy problem~\eqref{eq:9}} is a function $u\in C([0,\infty);X)$
with the following property: for every $T$, $\varepsilon>0$, for every
partition $0=t_{0}<\cdots<t_{N}=T$ of the interval $[0,T]$ such that
$t_{i}-t_{i-1}<\varepsilon$ for every $i=1,\dots,N$, there exists a
piecewise constant function $u_{\varepsilon,N} : [0,T]\to X$ given by
\begin{displaymath}
  u_{\varepsilon,N}(t)=u_{0}\,\mathds{1}_{\{t_{0}=0\}}(t)+\sum_{i=1}^{N}u_{\varepsilon,i}\;
\mathds{1}_{(t_{i-1},t_{i}]}(t)
\end{displaymath}
  where the values $u_{i}$ on $(t_{i-1},t_{i}]$
solve recursively the finite difference equation
\begin{displaymath}
  u_{i}+(t_{i}-t_{i-1})Au_{i}\ni u_{i-1}\qquad\text{for every $i=1,\dots,N$}
\end{displaymath}
and
\begin{displaymath}
  \sup_{t\in [0,T]}\norm{u(t)-u_{\varepsilon,N}(t)}_{X}\le\varepsilon.
\end{displaymath}
If $A+\omega I$ is $m$-accretive in $X$ for some $\omega\in \R$, then
for every element $u_{0}$ of $\overline{D(A)}^{\mbox{}_{X}}$, there is
a unique mild solution $u$ of~\eqref{eq:9} which can be given by
\emph{exponential formula}
\begin{equation}
  \label{eq:36}
  u(t)=\lim_{n\to\infty}\left(I+\tfrac{t}{n}A\right)^{-n}u_{0}
\end{equation}
uniformly in $t$ on compact intervals. For every
$u_0\in \overline{D(A)}^{\mbox{}_{X}}$, setting $T_{t}u_{0}=u(t)$,
$t\ge0$, defines a (non-linear) \emph{strongly continuous
  semigroup} $\{T_{t}\}_{t\geq 0}$ of \emph{Lipschitz continuous
  mappings}
$T_{t} : \overline{D(A)}^{\mbox{}_{X}}\to
\overline{D(A)}^{\mbox{}_{X}}$ with constant $e^{\omega\,t}$. More
precisely, there is a family $\{T_{t}\}_{t\geq 0}$ of mappings
$T_{t}$ on $\overline{D(A)}^{\mbox{}_{X}}$ obeying the following three properties:
\begin{itemize}
\item (\emph{semigroup property})
  \begin{equation}
    \label{eq:272}
    T_{t+s}=T_{t}\circ T_{s}\qquad\text{for every $t$, $s\ge0$,}
  \end{equation}
\item (\emph{strong continuity}) 
  \begin{displaymath}
    \lim_{t\to0+}\norm{T_{t}u-u}_{X}=0\qquad
    \text{for every $u\in \overline{D(A)}^{\mbox{}_{X}}$,}
  \end{displaymath}
\item (\emph{exponential growth property in $X$}) 
  \begin{displaymath}
    \norm{T_{t}u-T_{t}v}_{X}\le e^{\omega\,t}\norm{u-v}_{X}\qquad\text{for all $u$,
      $v\in \overline{D(A)}^{\mbox{}_{X}}$, $t\ge0$.}
  \end{displaymath}
\end{itemize}
In addition, if $A$ is quasi
$m$-$T$-accretive in $X$, then for every $t\ge 0$, $T_{t}$ satisfies
\begin{itemize}
\item (\emph{exponential $T$-growth property in $X$})
  \begin{displaymath}
     \norm{[T_{t}u-T_{t}v]^{+}}_{X}\le e^{\omega\,t}\norm{[u-v]^{+}}_{X}\qquad\text{for all $u$,
      $v\in \overline{D(A)}^{\mbox{}_{X}}$, $t\ge0$,}
  \end{displaymath}
\end{itemize}
in particular, the semigroup $\{T_{t}\}_{t\geq 0}$ is
order-preserving, that is, every $T_{t}$ is order-preserving.

 To express that
 the semigroup $\{T_{t}\}_{t\geq 0}$ has been obtained by the above 
construction  we say that $\{T_{t}\}_{t\geq 0}$ has been \emph{generated by $-A$ on
  $\overline{D(A)}^{\mbox{}_{X}}$} and we denote $\{T_{t}\}_{t\geq 0}\sim -A$. If $A$ is $m$-accretive in $X$,
then each mapping $T_{t}$ of the semigroup
$\{T_{t}\}_{t\geq 0}\sim -A$ becomes \emph{contractive} in $X$.

With these preliminaries in mind, we turn now to one of the main
topics of this article. It is not difficult to see (\cite[p. 130]{MR2582280}) that every strong solution of
Cauchy problem~\eqref{eq:9} is a mild solution. But, it is still not
well understood under which conditions on the operator $A$ and the
Banach space $X$, for each $u_{0}\in \overline{D(A)}^{\mbox{}_{X}}$,
the mild solution $u(t)=T_{t}u_{0}$, $t\ge 0$, of~\eqref{eq:9} is a
strong one. Obviously, this problem involves a \emph{regularisation
  effect} since the solution $u$ gains a posteriori in regularity, namely, the
property to be differentiable at a.e. $t>0$ with values in $X$. The
current state of knowledge in the literature concerning this problem is the following one
(cf.~\cite[Theorem~4.6]{MR2582280}
or~\cite[Corollary~(7.11)]{Benilanbook}): if $A$ is
quasi-$m$-accretive in $X$ and if the Banach space $X$ and its dual
$X'$ are uniformly convex, then for every initial value
$u_{0}\in D(A)$ the mild solution $u(t):=T_{t}u_{0}$, $t\ge0$,
of~\eqref{eq:9} belongs to the space
$W^{1,\infty}_{loc}([0,\infty);X)$, is almost everywhere
differentiable on $(0,\infty)$, differentiable from the right at every
$t\ge 0$, the right-hand side derivative $\tfrac{d u}{dt}_{+}(t)$ is
right continuous on $[0,\infty)$ and for every $t\ge 0$,
\begin{equation}
  \label{eq:38}
  u(t)\in D(A)\qquad\text{and}\qquad
  \tfrac{d}{dt}_{+}u(t) + A^{\circ}u(t)=0.
\end{equation}
Here, $A^{\circ}$ denotes the \emph{principal section} of $A$ which
assigns to every $u\in D(A)$ the element $A^{\circ}u$ of $Au$ with
minimal norm among all elements of $Au$. Thus, under these
assumptions, the mild solution $u(t)=T_{t}u_{0}$, $t\ge 0$,
of~\eqref{eq:9} for every $u_{0}\in D(A)$ is a strong solution
of~\eqref{eq:9}. Recall that the space $X=L^{1}(\Sigma,\mu)$ is not
uniformly convex. Further, it is natural to ask whether this statement
holds true if $u_{0}\in \overline{D(A)}^{\mbox{}_{X}}$ for general
quasi $m$-accretive operators $A$. Thanks to the pioneering
result~\cite{MR0283635} by Brezis, the answer of this question is
affirmative provided $A$ is the subgradient $\partial_{\! L^{2}}\Psi$
in
$L^{2}(\Sigma,\mu)$ %(or, more generally, in any Hilbert space) of a
of a convex, proper, lower semicontinuous functional
$\Psi : L^{2}(\Sigma,\mu)\to\R\cup\{+\infty\}$. Semigroups
$\{T_{t}\}_{t\geq 0}$ generated by \emph{positive homogeneous
  operators $A$ of order $\alpha>0$ with $\alpha\neq 1$} on
$L^{q}(\Sigma,\mu)$ for $1<q<\infty$ admit the same regularisation
effect (cf.~\cite{MR648452}). As a by-product of our other results
(Theorem~\ref{thm:GN-implies-reg-for-oper-in-L1}) we can show that the
semigroup $\{T_{t}\}_{t\geq 0}$ in $L^{1}(\Sigma,\mu)$ has also this
regularisation effect provided its infinitesimal generator $A$ is the
closure $\overline{(\partial_{\!  L^{2}}\Psi)_{1\cap\infty}\phi}$ of
$(\partial_{\!  L^{2}}\Psi)_{1\cap\infty}\phi$ in $L^{1}(\Sigma,\mu)$,
where $\Psi : L^{2}(\Sigma,\mu)\to\R\cup\{+\infty\}$ is a convex,
proper, lower semicontinuous functional and $\phi : \R\to \R$ a
strictly increasing function such that $\phi$ and $\phi^{-1}$ are
locally Lipschitz continuous (see
Theorem~\ref{thm:Linfty-implies-mild-are-strong-in-L1} for the exact
statement).
% of~\eqref{eq:9} is strong.

As mentioned in the introduction, throughout this monograph, we deal with the
following two classes of accretive operators $A$ on $M(\Sigma,\mu)$
generating nonlinear semigroups acting on all $L^{q}$ spaces for $1\le
q\le \infty$:
\begin{itemize}
\item \emph{quasi $m$-completely accretive operators in $L^{q_{0}}$
    for some $1\le q_{0}<\infty$}
 
\item \emph{quasi $m$-$T$-accretive operators in $L^{1}$ with complete
    resolvent}.
 \end{itemize}

 Besides the already mentioned prototypes, typical examples of the
 first class of operators are, for instance, quasi-linear operators of
 second order of $p$-Laplace type (so-called \emph{Leray-Lions
   operators}~\cite{MR0194733}), nonlocal diffusion operators of
 $p$-La\-place type (see, e.g., \cite{MR2722295}), but also the total
 variational flow operator of local and nonlocal type (cf., for
 instance, \cite{MR2033382} and
 \cite{Ha2015TotalVariationalFlow}). Typical examples of the second
 class of operators are easily obtained by considering the
 \emph{composition operator} $A\phi$ in $L^{1}$ of an $m$-completely
 accretive operator $A$ in $L^{1}$ and a strictly increasing function
 $\phi : \R\to \R$ satisfying $\phi(0)=0$. Operators of both classes
 are equipped with some boundary conditions when required and may be
 perturbed by a monotone (multi-valued) or Lipschitz continuous lower
 order term.

 In the following two subsections, we introduce these two classes of
 operators in more details and list some of their properties relevant
 for this monograph. We establish $L^{q}$-$L^{r}$-regularisation
 estimates for several examples of both classes in
 Section~\ref{sec:examples}.

%-----------------------------------------------------------------------------
%-----------------------------------------------------------------------------
%-----------------------------------------------------------------------------
%-----------------------------------------------------------------------------
%-----------------------------------------------------------------------------
%-----------------------------------------------------------------------------
%-----------------------------------------------------------------------------
%-----------------------------------------------------------------------------
%-----------------------------------------------------------------------------
%-----------------------------------------------------------------------------
%-----------------------------------------------------------------------------
%-----------------------------------------------------------------------------
%-----------------------------------------------------------------------------
%-----------------------------------------------------------------------------
%-----------------------------------------------------------------------------
%-----------------------------------------------------------------------------
%-----------------------------------------------------------------------------
%-----------------------------------------------------------------------------
%-----------------------------------------------------------------------------
\subsection{Completely accretive operators}\label{sec:comp}
The class of \emph{completely accretive} operators is the nonlinear
analogue of the class of linear operators generating a submarkovian
semigroup in the sense that the semigroup they generate extrapolates
to $L^p$, $1<p<\infty$ (see Proposition \ref{propo:extrapol-prop}
below) and is order preserving. This class of nonlinear operators was
introduced by Benilan and Crandall~\cite{MR1164641}.

The notion of complete accretivity we use is the same as
in~\cite{MR1164641} and will be introduced now. We denote by
$\mathcal{J}_{0}$ the set of all convex, lower semicontinuous functions
$j : \R\to[0,\infty]$ satisfying $j(0)=0$. 

\begin{definition}
  A mapping $S : D(S)\to M(\Sigma,\mu)$ with domain $D(S)\subseteq
  M(\Sigma,\mu)$ is called a \emph{complete contraction} if 
  \begin{displaymath}
    \int_{\Sigma}j(Su-S\hat{u})\,\textrm{d}\mu\le 
    \int_{\Sigma}j(u-\hat{u})\,\textrm{d}\mu
  \end{displaymath}
  for all $j\in \mathcal{J}_{0}$ and every $u$, $\hat{u}\in D(S)$.
\end{definition}

\begin{remark}
  \label{rem:1}
  Choosing $j(\cdot)=\abs{[\cdot]^{+}}^{q}\in \mathcal{J}_{0}$
  if $1\le q<\infty$ and $j(\cdot)=[[\cdot]^{+}-k]^{+}\in \mathcal{J}_{0}$ for $k\ge 0$ large
  enough if $q=\infty$ shows that each complete contraction $S$ is
  $T$-contractive in $L^{q}(\Sigma,\mu)$ for every $1\le q\le \infty$.
  % If for a given complete
  % contraction $S$, there is a $u_{0}\in D(S)\cap L^{1}\cap
  % L^{q}(\Sigma,\mu)$ such that $Su_{0}=0$, then for every $1\le q\le
  % \infty$, $S$ admits a unique contractive
  % extension $S_{q} : \overline{ D(S)\cap
  %   L^{q}(\Sigma,\mu)}^{\mbox{}_{L^{q}}}\to L^{q}(\Sigma,\mu)$. For
  % convenience and if there is no risc of confusion, we will denote
  % $S_{q}$ again by $S$. All complete contractions appearing in this
  % paper admit an element $u_{0}\in D(S)\cap L^{1}\cap
  % L^{q}(\Sigma,\mu)$ such that $Su_{0}=0$ and so after passing to the
  % unique extension on $L^{q}(\Sigma,\mu)$, we see that $S$ \emph{acts
  %   on all $L^{q}$ spaces}.
\end{remark}

Now, we can state the
definition of completely accretive operators.

\begin{definition}
  \label{def:completely-accretive-operators}
  An operator $A$ on $M(\Sigma,\mu)$ is called \emph{completely
    accretive} if for every $\lambda>0$, the resolvent operator
  $J_{\lambda}$ of $A$ is a complete contraction. If $X$ is a linear
  subspace of $M(\Sigma,\mu)$ and $A$ an operator on $X$, then $A$ is
  \emph{$m$-completely accretive on $X$} if $A$ is completely
  accretive and satisfies the range condition~\eqref{eq:range-condition}.
  Further, we call an operator $A$ on $M(\Sigma,\mu)$ \emph{quasi
    completely accretive} if there is an $\omega\in \R$ such that
  $A+\omega I$ is completely accretive. Finally, an operator $A$ on a
  linear subspace $X$ is called \emph{quasi $m$-completely accretive}
  if $A+\omega I$ is $m$-completely accretive on $X$ for some
  $\omega\in \R$.
\end{definition}

As a matter of fact, in most applications the following
characterisation is used to verify whether a given operator $A$ on
$X=L^{q}(\Sigma,\mu)$ is completely accretive (see
 also~\cite[Corollary A.43]{MR2722295}). Here, we state
 \cite[Proposition~2.2]{MR1164641} only in a special case since it is
 more convenient for us.

\begin{proposition}[{\cite[Proposition~2.2]{MR1164641}}]
  \label{prop:completely-accretive}
   Let $P_{0}$ denote the set of all functions $T\in C^{\infty}(\R)$
   satisfying $0\le T'\le 1$, 
    $T'$  is compactly supported, and $x=0$ is not contained in the
   support $\textrm{supp}(T)$ of $T$. Then for $u$, $v\in
   L^{1}(\Sigma,\mu)+L^{\infty}(\Sigma,\mu)$ with
   $\mu(\{\abs{u}>k\})<\infty$ for all $k>0$, one has
   \begin{displaymath}
    \int_{\Sigma}j(u)\,\textrm{d}\mu\le 
    \int_{\Sigma}j(u+\lambda v)\,\textrm{d}\mu
  \end{displaymath}
  for every $j\in \mathcal{J}_{0}$ and $\lambda>0$ if and only if
   \begin{equation}
     \label{eq:complete}
     \int_{\Sigma} T(u)\,v\,\dmu\ge 0
   \end{equation}
   for every $T\in P_{0}$. As a consequence, an operator $A$ on
   $L^{q}(\Sigma,\mu)$ for $1\le q<\infty$ is completely accretive
   if and only if
   \begin{equation}
     \label{eq:charact-completelyaccretive}
     \int_{\Sigma}T(u-\hat{u})(v-\hat{v})\,\dmu\ge 0
   \end{equation}
   for every $T\in P_{0}$ and every $(u,v)$,
   $(\hat{u},\hat{v})\in A$.
 \end{proposition}

 For any given monotone graph $\beta$ on $\R$ and $1\le q<\infty$, the
 associated accretive operator $\beta_{q}$ on $L^{q}(\Sigma,\mu)$ is,
 in fact, completely accretive. To see this, note first that every
 $T\in P_{0}$ is continuous and non-decreasing. Thus for all $(u,v)$,
 $(\hat{u},\hat{v})\in \beta_{q}$ and every $T\in P_{0}$, one has
 \begin{displaymath}
   T(u-\hat{u})(v-\hat{v})\ge 0\qquad\text{a.e. on $\Sigma$.}
 \end{displaymath}
 Integrating this inequality over $\Sigma$ yields
 inequality~\eqref{eq:charact-completelyaccretive} in
 Proposition~\ref{prop:completely-accretive}.  

 Further, the property \emph{completely accretive} is preserved
 under perturbation of a Lipschitz continuous mapping. This result
 seems to be known, but we could not find a reference in the 
 literature. It provides an important example of completely accretive
 operators (cf.~\cite[Corollary~2.4]{MR1164641}).

 \begin{proposition}
   \label{propo:Lipschitz-complete-accretive}
   Let $1\le q<\infty$, $B$ be a completely accretive operator on
   $L^{q}(\Sigma,\mu)$ and
   $F : L^{q}(\Sigma,\mu)\to L^{q}(\Sigma,\mu)$ the Nemytski operator
   of a Carath\'eodory function $f : \Sigma\times\R\to \R$
   satisfying~\eqref{eq:2} for some constant $\omega\ge 0$ and
   $F(0)\in L^{q}(\Sigma,\mu)$. Then the following
   statements hold:
   \begin{enumerate}
   \item The operator $A:=B+F+\omega I$ is completely
     accretive.
   \item Let $1<q<\infty$ and $\beta$ be an $m$-accretive graph on
     $\R$ such that either $(0,0)\in \beta$ or $(\Sigma,\mu)$ is finite.
     If $B$ satisfies the range condition~\eqref{eq:range-condition}
     in $L^{q}(\Sigma,\mu)$ with $D(B)\cap D(\beta_{q})\neq\emptyset$
     and if the Yosida operator $\beta_{\lambda}(\cdot)$ of $\beta_{q}$ satisfies
     \begin{equation}
       \label{eq:78}
       [\beta_{\lambda}(u),v]_{q}\ge 0\qquad\text{for all
         $(u,v)\in A$ and $\lambda>0$,}
     \end{equation}
     then $A:=B+\beta_{q}+F$ is
     quasi $m$-completely accretive in $L^{q}(\Sigma,\mu)$.
   \end{enumerate}
 \end{proposition}

 \begin{proof}
   Let $T\in P_{0}$ and $(u,v)$, $(\hat{u},\hat{v})\in A$. Then, in
   order to apply Proposition~\ref{prop:completely-accretive}, we need
   to show that inequality~\eqref{eq:charact-completelyaccretive}
   holds. By assumption, there are $w\in Bu$ and $\hat{w}\in B\hat{u}$
   such that $v=w+f(x,u)+L\,u$ and
   $\hat{v}=\hat{w}+f(x,\hat{u})$. Also, $T$ is non-decreasing and
   $T(0)=0$. Hence $T(u-\hat{u})\ge 0$ if $u\ge \hat{u}$ and
   $T(u-\hat{u})\le 0$ if $u<\hat{u}$. Using this together with
   inequality~\eqref{eq:2} and since, by assumption, $B$ is completely
   accretive, we see that \allowdisplaybreaks
   \begin{align*}
    & \int_{\Sigma}T(u-\hat{u})(v+\omega u-(\hat{v}+\omega \hat{u}))\, \dmu\\
     &\qquad = \int_{\Sigma}T(u-\hat{u})\,(w-\hat{w})\,\dmu +
       \int_{\Sigma}T(u-\hat{u})\,(f(x,u)-f(x,\hat{u})\, \dmu\\
     &\hspace{5cm} + L
       \int_{\Sigma}T(u-\hat{u})\,(u-\hat{u})\,\dmu\\ 
     & \qquad \ge - L \int_{\{u\ge
       \hat{u}\}}T(u-\hat{u})\,\abs{u-\hat{u}}\, \dmu
       + L \int_{\{u< \hat{u}\}}T(u-\hat{u})\,\abs{u-\hat{u}}\, \dmu\\
          &\hspace{5cm} + L \int_{\Sigma}T(u-\hat{u})\,(u-\hat{u})\,\dmu\\   
     &\qquad = -L \int_{\Sigma}\,T(u-\hat{u})\,(u-\hat{u})\, \dmu +
        L \int_{\Sigma}\,T(u-\hat{u})\, (u-\hat{u})\,\dmu = 0.
   \end{align*}
   Thus claim $(1)$ holds. For any given
   monotone graph $\beta$ on $\R$, one sees by using~\eqref{eq:65}
   together with the fact that $T\in P_{0}$ is monotonically increasing and by 
   proceeding as before that $A$ is completely accretive. If $B$ is
   $m$-accretive in $L^{q}(\Sigma,\mu)$ and $1<q<\infty$, then we know
   that $B+F+\omega I$ is $m$-accretive in
   $L^{q}(\Sigma,\mu)$. Now, under the assumptions of
   claim $(2)$ and by using the definition of
   $q$-brackets $[\cdot,\cdot]_{q}$, it follows by
   \cite[Proposition~3.8]{MR2582280},
   $B+F+\omega I+\beta_{q}$ is $m$-accretive in
   $L^{q}(\Sigma,\mu)$. This completes the proof.
 \end{proof}

 The last proposition of this subsection states that a semigroup
 $\{T_{t}\}_{t\ge0}$ generated by a quasi $m$-completely accretive
 operator $-A$ on $L^{q_{0}}(\Sigma,\mu)$ for some $1\le q_{0}<\infty$
 has, in fact, \emph{exponential growth} in all $L^{q}$ spaces.

\begin{proposition}
  \label{propo:complete-expo-grow}
  Let $1\le q_{0}<\infty$ and $\omega\in \R$ such that
  $A+\omega I$ is completely accretive in
  $L^{q_{0}}(\Sigma,\mu)$. Then for every
  $\lambda>0$ satisfying $\lambda\,\omega<1$,
  the resolvent operator $J_{\lambda}$ of $A$ satisfies
  \begin{equation}
    \label{eq:61}
    \norm{J_{\lambda}u-J_{\lambda}\hat{u}}_{\tilde{q}}\le  
    (1-\lambda\omega)^{-1}\norm{u-\hat{u}}_{\tilde{q}}
  \end{equation}
  for every $u$, $\hat{u}\in Rg(I+\lambda A)$ and every
  $1\le \tilde{q}\le \infty$. If $A+\omega I$ is $m$-completely
  accretive in $L^{q_{0}}(\Sigma,\mu)$, then for all
  $1\le \tilde{q}\le \infty$, the semigroup $\{T_{t}\}_{t\ge0}\sim -A$
  on $\overline{D(A)}^{\mbox{}_{L^{q_{0}}}}$ satisfies the ''exponential
  growth property''
\begin{equation}
\label{eq:62}
  \norm{T_{t}u-T_{t}\hat{u}}_{\tilde{q}}\le  e^{\omega\,t}\norm{u-\hat{u}}_{\tilde{q}}
\end{equation}
for every $u$, $\hat{u}\in \overline{D(A)}^{\mbox{}_{L^{q_{0}}}}\cap
L^{\tilde{q}}(\Sigma,\mu)$ and $t>0$. Moreover, if there are
$u_{0}\in L^{q_{0}}\cap L^{^{\tilde{q}}}(\Sigma,\mu)$ and $t\ge 0$ such that
$T_{t}u_{0}\in L^{\tilde{q}}(\Sigma,\mu)$ then $T_{t}$ can
be uniquely extended to a Lipschitz continuous mapping on
$\overline{\overline{D(A)}^{\mbox{}_{L^{q_{0}}}}\cap
  L^{\tilde{q}}(\Sigma,\mu)}^{\mbox{}_{L^{\tilde{q}}}}$ with constant
$e^{\omega t}$.
\end{proposition}

\begin{remark}
  Throughout this monograph, we often assume that there is
  \begin{equation}
    \label{eq:180}
    (u_{0},0)\in A\qquad\text{for some $u_{0}\in L^{q_{0}}\cap
      L^{\tilde{q}}(\Sigma,\mu)$}
  \end{equation}
  and some $1\le q_{0}$, $\tilde{q}\le\infty$ (see, for instance,
  Definition~\ref{def:Gag-Nire-inequality},
  Theorem~\ref{thm:GN-implies-reg-bis} or
  Theorem~\ref{thm:GN-implies-reg-for-oper-in-L1}). Condition~\eqref{eq:180}
  is equivalent to $u_{0}\in L^{q_{0}}\cap L^{\tilde{q}}(\Sigma,\mu)$
  is a \emph{fixed point} for the resolvent $J_{\lambda}$ of $A$, that
  is, $J_{\lambda}u_{0}=u_{0}$ for all $\lambda>0$. Thus and by using
  the exponential formula~\eqref{eq:36}, one sees that
  condition~\eqref{eq:180} is equivalently to the fact that
  $u_{0}\in L^{q_{0}}\cap L^{\tilde{q}}(\Sigma,\mu)$ is a \emph{fixed
    point} for the semigroup $\{T_{t}\}\sim -A$, that is,
  $T_{t}u_{0}=u_{0}$ for all $t\ge 0$. Moreover, it is worth noting
  that under condition~\eqref{eq:180}, the exponential growth
  property~\eqref{eq:62} of the semigroup $\{T_{t}\}$ reduces to
  \begin{equation}
    \label{eq:216}
    \norm{T_{t}u-u_{0}}_{\tilde{q}}\le  e^{\omega\,t}\norm{u-u_{0}}_{\tilde{q}}
  \end{equation}
  for every
  $u\in \overline{D(A)}^{\mbox{}_{L^{q_{0}}}}\cap
  L^{\tilde{q}}(\Sigma,\mu)$ and $t>0$.
\end{remark}

\begin{proof}[Proof of Proposition~\ref{propo:complete-expo-grow}]
  By assumption, the resolvent operator of $A+\omega I$ is a complete
  contraction. Let $1\le \tilde{q}\le \infty$ and $\lambda>0$ such
  that $\lambda\omega<1$. Then by
  Remark~\ref{rem:1},
   \begin{displaymath}
     \norm{u-\hat{u}}_{\tilde{q}}\le 
     \norm{u-\hat{u}+\lambda (\omega (u-\hat{u}) + (v-\hat{v}))}_{\tilde{q}}
   \end{displaymath}
   for every $(u,v)$, $(\hat{u},\hat{v})\in A$. Thus, 
   \begin{align*}
     \norm{u-\hat{u}+\lambda (v-\hat{v})}_{\tilde{q}}& =
     (1-\lambda\omega)\norm{u-\hat{u}+\tfrac{\lambda}{1-\lambda\omega}(\omega
     (u-\hat{u})+(v-\hat{v}))}_{\tilde{q}}\\
     &\ge (1-\lambda\omega)\,\norm{u-\hat{u}}_{\tilde{q}}
   \end{align*}
   for every $(u,v)$, $(\hat{u},\hat{v})\in A$ and every
  $\lambda>0$ such that $\lambda\,\omega<1$, proving
   that the resolvent operator $J_{\lambda}$ of $A$
   satisfies~\eqref{eq:61}. In order to see that the semigroup $\{T_{t}\}_{t\ge0}\sim -A$ on
$\overline{D(A)}^{\mbox{}_{L^{q_{0}}}}$ satisfies~\eqref{eq:62},  for
given $t>0$ and $n\in \N$ large enough, one takes $\lambda=t/n$ such that
$\tfrac{t}{n}\omega<1$. Then, replacing $u$ and $\hat{u}$ by
$J_{t/n}^{n-1}u$ and $J_{t/n}^{n-1}\hat{u}$ in~\eqref{eq:61} yields
\begin{align*}
  \norm{J_{t/n}^{n}u-J_{t/n}^{n}\hat{u}}_{\tilde{q}}&\le   
  (1-\tfrac{t\,\omega}{n})^{-1}\norm{J_{t/n}^{n-1}u-J_{t/n}^{n-1}\hat{u}}_{\tilde{q}}\\
   &\;\;\vdots\\
    &\le  (1-\tfrac{t\,\omega}{n})^{-n}\norm{u-\hat{u}}_{\tilde{q}}.
  \end{align*}
By the exponential formula~\eqref{eq:36}, one has
\begin{displaymath}
\lim_{n\to\infty}J_{t/n}^{n}u-J_{t/n}^{n}\hat{u}
=T_{t}u-T_{t}\hat{u}\qquad\text{in $L^{q_{0}}(\Sigma,\mu)$}.
\end{displaymath}
Since the $L^{\tilde{q}}$ norm on $L^{q_{0}}(\Sigma,\mu)$ is lower semicontinuous, sending
$t\to\infty$ in the previous estimates shows that
inequality~\eqref{eq:62} holds. In order to see that the last
statement of this proposition holds, note that the existence of
$u_{0}\in L^{q_{0}}\cap L^{\tilde{q}}(\Sigma,\mu)$ and $t\ge 0$ such
that $T_{t}u_{0}\in L^{\tilde{q}}(\Sigma,\mu)$ together
with~\eqref{eq:62} imply that $T_{t}$ maps
$\overline{D(A)}^{\mbox{}_{L^{q_{0}}}}\cap L^{\tilde{q}}(\Sigma,\mu)$
into $L^{q_{0}}\cap L^{\tilde{q}}(\Sigma,\mu)$. Thus ~\eqref{eq:62}
and a standard density argument yield that $T_{t}$ has a unique Lipschitz
continuous extension from
$\overline{\overline{D(A)}^{\mbox{}_{L^{q_{0}}}}\cap
  L^{\tilde{q}}(\Sigma,\mu)}^{\mbox{}_{L^{\tilde{q}}}}$
to
$\overline{\overline{D(A)}^{\mbox{}_{L^{q_{0}}}}\cap
  L^{\tilde{q}}(\Sigma,\mu)}^{\mbox{}_{L^{\tilde{q}}}}$
with constant $e^{\omega t}$. This completes the proof.
\end{proof}

The fundamental property of completely accretive operators $A$ is
given by the \emph{extrapolation property} stated in the next
proposition, which is especially meaningful  when
$(\Sigma,\mu)$ is not a finite measure space. The first main steps
towards the statement of this proposition have been established by
B\'enilan and Crandall~\cite{MR1164641}. The main idea of the proof
in~\cite{MR1164641} relies on the following fundamental property of 
sequences of functions in $L^{q}(\Sigma,\mu)$ for $1\le q<\infty$:
\begin{equation}
  \label{eq:225}
  \begin{split}
    &\text{ for any sequence
      $(u_{n})_{n\ge 1}\subseteq L^{q}(\Sigma,\mu)$ satisfying}\\
    &\quad (i)\quad \int_{\Sigma}\abs{u_{n}}^{q}\,\dmu\le
    \int_{\Sigma}\abs{u}^{q}\,\dmu \quad\text{ for all $n\ge 1$,}\\
    &\quad (ii)\quad \lim_{n\to\infty}u_{n}(x)=u(x)\quad \text{ for a.e.
      $x\in \Sigma$,}\\
    &\text{one has }\quad\lim_{n\to\infty}u_{n}=u\text{ in
      $L^{q}(\Sigma,\mu)$.}
  \end{split}
\end{equation}
Statement~\eqref{eq:225} follows from Fatou's lemma in combination
with either the uniform convexity of $L^{q}(\Sigma,\mu)$ if $q>1$ or
with Young's theorem (cf.,  for instance,
\cite[Theorem~2.8.8]{MR2267655}) if $q=1$. Furthermore ~\eqref{eq:225} 
yields the statements of our next proposition. We
leave its easy proof to the interested reader as an exercise.

\begin{proposition}
  \label{propo:extrapol-prop}
  Let $1\le q_{0}<\infty$ and $A$ be a $m$-completely accretive in
  $L^{q_{0}}(\Sigma,\mu)$ with dense domain and satisfying $(0,0)\in
  A$. Let $A_{1\cap \infty}$ be the trace of $A$ on
  $L^{1}\cap L^{\infty}(\Sigma,\mu)$ and for $1\le q<\infty$, let
  $\overline{A_{1\cap\infty}}^{\mbox{}_{L^{q}}}$ be the closure of $A_{1\cap\infty}$ in
  $L^{q}(\Sigma,\mu)$. Then the following statements hold true.
  \begin{enumerate}
  \item For every $1\le q<\infty$,
    $\overline{A_{1\cap \infty}}^{\mbox{}_{L^{q}}}$ is $m$-completely
    accretive in $L^{q}(\Sigma,\mu)$ with dense domain and
    $(0,0)\in \overline{A_{1\cap \infty}}^{\mbox{}_{L^{q}}}$.

  \item For every $1\le q<\infty$ and for all $\lambda>0$, the
    resolvent $J_{\lambda}$ of $A$ admits a unique extension on
    $L^{q}(\Sigma,\mu)$ and this extension coincides with the
    resolvent operator $J_{\lambda}^{q}$ of
    $\overline{A_{1\cap \infty}}^{\mbox{}_{L^{q}}}$.

  \item For every $1\le q<\infty$ and for every $t>0$, the mapping
    $T_{t}$ of the semigroup $\{T_{t}\}_{t\ge0}\sim -A$ on
    $L^{q_{0}}(\Sigma,\mu)$ admits a unique extension on
    $L^{q}(\Sigma,\mu)$ and this extension coincides with the mapping
    $T_{t}^{q}$ on $L^{q}(\Sigma,\mu)$ of the semigroup
    $\{T_{t}^{q}\}_{t\ge0}\sim -\overline{A_{1\cap
        \infty}}^{\mbox{}_{L^{q}}}$ on $L^{q}(\Sigma,\mu)$.
  \end{enumerate}
  If $F : L^{q_{0}}(\Sigma,\mu)\to L^{q_{0}}(\Sigma,\mu)$ is a
  Lipschitz mapping satisfying $F(0)=0$, then the same statements hold
  for the quasi $m$-completely
  accretive operator $A+F$ on $L^{q_{0}}(\Sigma,\mu)$.  
\end{proposition}

%-----------------------------------------------------------------------------
%-----------------------------------------------------------------------------
%-----------------------------------------------------------------------------
%-----------------------------------------------------------------------------
%-----------------------------------------------------------------------------
%-----------------------------------------------------------------------------
%-----------------------------------------------------------------------------
%-----------------------------------------------------------------------------
%-----------------------------------------------------------------------------
%-----------------------------------------------------------------------------
%-----------------------------------------------------------------------------
%-----------------------------------------------------------------------------
%-----------------------------------------------------------------------------
%-----------------------------------------------------------------------------
%-----------------------------------------------------------------------------
%-----------------------------------------------------------------------------
%-----------------------------------------------------------------------------
%-----------------------------------------------------------------------------
%-----------------------------------------------------------------------------
\subsection{$T$-accretive operators in
  $L^{1}$ with complete resolvent}\label{subsec:L1}

The class of \emph{$T$-ac\-cretive operators in $L^{1}$ with complete
  resolvent} was developed in connection with the porous media
equation by B\'enilan~\cite{Benthesis}. To the best of our knowledge,
the first important result in direction to this class of operators has
been the interpolation theorem in~\cite{MR0336050} due to Brezis and
Strauss in order to treat semilinear elliptic equations in
$L^{1}$. This interpolation theorem has been extended by B\'enilan and
Crandall in~\cite{MR1164641} to introduce the class of completely
accretive operators and further investigated by many others (for
instance, see also~\cite[Partie II]{Benthesis},
\cite{MR670925,MR647071} and \cite{MR2286292}).

For the sake of brevity, we merely state
the results in this section and refer for their proofs to
Appendix~\ref{Asec:accretive-L1}.

We begin by introducing the notion of \emph{complete maps} on
$M(\Sigma,\mu)$ in accordance with~\cite[Definition~1.7]{MR1164641}
(see also~\cite[D\'efinition~2.1]{Benthesis}).

\begin{definition}
  \label{def:complete-mapping}
  Let $D(S)$ be a subset of $M(\Sigma,\mu)$. A mapping
  $S : D(S)\to M(\Sigma,\mu)$ is called \emph{complete} if
  \begin{equation}
    \label{eq:109}
    \int_{\Sigma}j(Su)\,\textrm{d}\mu\le 
    \int_{\Sigma}j(u)\,\textrm{d}\mu
  \end{equation}
  for every $j\in \mathcal{J}_{0}$ and $u\in D(S)$. Further, let
  $\mathcal{J}$ be the set of all convex, lower semicontinuous functions
  $j : \R\to[0,\infty]$. We call a
  mapping $S : D(S)\to M(\Sigma,\mu)$ \emph{$c$-complete} if $S$
  satisfies inequality~\eqref{eq:109} for all $u\in D(S)$ and $j\in \mathcal{J}$. 
\end{definition}

Since the set $\mathcal{J}_{0}$ is contained in $\mathcal{J}$, every
$c$-complete mapping is a complete mapping. In
particular, we have the following characterisation.

\begin{proposition}
  \label{propo:characterisation-c-complete}
  Let $P$ denote the set of all functions $T\in C^{\infty}(\R)$
  satisfying $0\le T'\le 1$ and $T'$ is compactly supported.  Suppose
  $(\Sigma,\mu)$ is a finite measure space. Then for $u$,
  $v\in L^{1}(\Sigma,\mu)$, one has
   \begin{equation}
     \label{eq:118}
    \int_{\Sigma}j(u)\,\textrm{d}\mu\le 
    \int_{\Sigma}j(u+\lambda v)\,\textrm{d}\mu
  \end{equation}
  for every $j\in \mathcal{J}$ and $\lambda>0$ if and only if
   \begin{equation}
     \label{eq:char-c-complete}
     \int_{\Sigma} T(u)\,v\,\dmu\ge 0
   \end{equation}
   for every $T\in P$. As a consequence, an operator $A$ on
   $L^{1}(\Sigma,\mu)$ has a $c$-complete resolvent
   if and only if inequality~\eqref{eq:char-c-complete} holds for
   every $(u,v)\in A$.
\end{proposition}

For a proof of this characterisation we refer to
Appendix~\ref{Asec:accretive-L1}.

\begin{remark}
  \label{rem:4}
  By taking $j(\cdot)=\abs{\cdot}^{q}\in \mathcal{J}_{0}$ if
  $1\le q<\infty$ and
  $j(\cdot)=[\abs{\cdot}-k]^{+}\in \mathcal{J}_{0}$ for $k\ge 0$ large
  enough if $q=\infty$, we see that every complete mapping $S$
  has \emph{non-increasing} $L^{q}$
  norm for all $1\le q\le \infty$. More precisely, 
  \begin{displaymath}
    \norm{Su}_{q}\le \norm{u}_{q}
  \end{displaymath}
  for all $u\in D(S)$ and every $1\le q\le \infty$. More generally,
  since for every constant $c\in \R$ and every $1\le
  q<\infty$, the function $j(\cdot):=\abs{\cdot-c}^{q}\in
  \mathcal{J}$ and since $j(\cdot):=[\abs{\cdot-c}-k]^{+}\in
  \mathcal{J}$ for every $k\ge 0$, one has that if $S$ is
  $c$-complete, then
  \begin{displaymath}
    \norm{Su-c}_{q}\le \norm{u-c}_{q}
  \end{displaymath}
  for all $u\in D(S)$, $c\in \R$ and $1\le q\le \infty$.
\end{remark}

Now, we can define the class of accretive operators in $L^{1}$ with
($c$-)complete resolvent.

\begin{definition}
  An operator $A$ on $L^{1}(\Sigma,\mu)$ is called \emph{($m$-) accretive in
    $L^{1}$ with complete resolvent} if $A$
  is ($m$-) accretive in $L^{1}(\Sigma,\mu)$ and for every $\lambda>0$, the
  resolvent operator $J_{\lambda} : Rg(I+\lambda A)\to D(A)$ of $A$ is
  a complete mapping. Further, we call an
  operator $A$ on $M(\Sigma,\mu)$ \emph{quasi ($m$-) accretive in $L^{1}$
    with complete resolvent} if there exists $\omega\in \R$ such that
  $A+\omega I$ is ($m$-) accretive in $L^{1}(\Sigma,\mu)$ 
  and for every $\lambda>0$, the resolvent
  $J_{\lambda}$ of $A+\omega I$ is a complete mapping. Similarly, we
  call an operator $A$ on $L^{1}(\Sigma,\mu)$ \emph{($m$-) accretive in
    $L^{1}$ with $c$-complete resolvent} if $A$
  is ($m$-) accretive in $L^{1}(\Sigma,\mu)$ and for every $\lambda>0$, the
  resolvent $J_{\lambda} : Rg(I+\lambda A)\to D(A)$ of $A$ is
  a $c$-complete mapping.
\end{definition}

\begin{remark}
  \label{rem:9}
  Note that, in contrast to completely accretive operators in $L^{1}$, an
  accretive operator $A$ in $L^{1}$ with ($c$)-complete resolvent does
  not admit, in general, an order-preserving resolvent $J_{\lambda}$ on
  $L^{1}$. For this, one needs the additional assumption $A$ is \emph{$T$-accretive} in
  $L^{1}$.
\end{remark}

Note that it does not make much sense to introduce the notion of
\emph{quasi ($m$-) accretive operators in $L^{1}$ with $c$-complete resolvent}.
This becomes more clear by the following result due to
B\'enilan~\cite[Corollaire~2.3]{Benthesis}. To be more precise,
consider the following situation. Let $B=-\Delta^{N}$ be the Neumann
Laplace operator on $L^{1}$ on a bounded Lipschitz domain
$\Sigma\subseteq \R^{d}$ and $f$ a real-valued Carath\'eodory
function on $\Sigma\times \R$ satisfying $f(x,0)=0$ for a.e. $x\in \Sigma$ and Lipschitz
condition~\eqref{eq:2} for some $\omega\ge 0$. Then $B$ is accretive in
$L^{1}$ has a $c$-complete resolvent and satisfies~\eqref{eq:81}. Let
$F$ denote the Nemytski operator on $L^{1}$ associated with $f$. Then
for $\omega= L$, $A+\omega I$ is accretive in $L^{1}$ and has a
complete resolvent. If one assumes that the resolvent of $A$ is
$c$-complete, then our next proposition implies that $f\equiv -\omega I_{\R}$.

\begin{proposition}[{\cite{Benthesis}}]
  \label{propo:charact-of-c-complete-operators}
  Suppose that $(\Sigma,\mu)$ is a finite measure space and $A$ is an
  accretive operator in $L^{1}(\Sigma,\mu)$ satisfying
  $L^{\infty}(\Sigma,\mu)\subseteq Rg (I+A)$. Then $A$ has a
  $c$-complete resolvent if and only if
  \begin{equation}
    \label{eq:81}
    (c,0)\in A\text{ for all $c\in \R$.}
  \end{equation}
\end{proposition}

Due to Proposition~\ref{propo:charact-of-c-complete-operators}, a
typical example of accretive operator in $L^{1}$ with $c$-complete
resolvent on a finite measure space $(\Sigma,\mu)$ is given by any
second order (nonlinear) diffusion operator equipped with homogeneous
Neumann boundary conditions on a bounded Lipschitz domain.

% \begin{remark}
%   \label{rem:3}
%   If $A$ is completely accretive in $L^{1}(\Sigma,\mu)$ and
%   $(0,0)\in A$, then by Definition~\ref{def:complete-mapping} and
%   since $J_{\lambda}0=0$ for all $\lambda>0$, one has that $A$
%   is accretive in $L^{1}$ with complete resolvent. But, such an operator
%   $A$ needs not to have a $c$-complete resolvent. A counter example
%   is, for instance, the $p$-Laplace operator equipped with homogeneous
%   Dirichlet boundary condition on a bounded set.
% \end{remark}

In the next proposition, we state some important properties of
accretive operators in $L^{1}$ with ($c$-)complete resolvent for later reference.

\begin{proposition}
  \label{proposition:closure-in-L1}
  If $A+\omega I$ is accretive in $L^{1}$ with complete resolvent for
  some $\omega \in \R$, then the closure $\overline{A}+\omega I$ is
  accretive in $L^{1}$ with complete resolvent. Further, if
  $(\Sigma,\mu)$ is a finite measure space and $A$ is accretive in
  $L^{1}$ with $c$-complete resolvent, then the closure
  $\overline{A}$ is accretive in $L^{1}$ with $c$-complete
  resolvent.
\end{proposition}

\begin{proof}
  The operator $\overline{A}+\omega I$ is accretive in
  $L^{1}(\Sigma,\mu)$ by~\eqref{eq:111} and since
  $\overline{A+\omega I}=\overline{A}+\omega I$. Now, suppose that the
  resolvent $J_{\lambda}$ of $A+\omega I$ is a complete mapping
  $J_{\lambda} : Rg(I+\lambda(\omega I+A))\to D(A)$ for every
  $\lambda>0$. Let $(u,v)\in \overline{A}$. Then there are sequences
  $(u_{n})$ and $(v_{n})$ such that $(u_{n},v_{n})\in A$ and $u_{n}$
  converges to $u$ and $v_{n}$ converges to $v$ in
  $L^{1}(\Sigma,\mu)$. By assumption, for every $\lambda>0$, the
  resolvent operator $J_{\lambda}$ of $A+\omega I$ is a complete
  mapping, that is, by Proposition~\ref{prop:completely-accretive},
  for every $(u_{n},v_{n})\in A$, one has
  \begin{equation}
    \label{eq:115}
    \int_{\Sigma}T(u_{n})\,(\omega u_{n}+v_{n})\,\dmu\ge 0
  \end{equation}
  for every $T\in P_{0}$. Since every $T\in P_{0}$ is Lipschitz
  continuous and bounded, 
  \begin{displaymath}
    T(u_{n}) (u_{n}+\lambda(\omega u_{n}+v_{n}))\to T(u_{n}) (u_{n}+\lambda(\omega u_{n}+v_{n}))
  \end{displaymath}
   in $L^{1}(\Sigma,\mu)$ as $n\to\infty$. Thus, sending $n\to\infty$
   in~\eqref{eq:115} yields
   \begin{displaymath}
     \int_{\Sigma}T(u)\,(\omega u+v)\,\dmu\ge 0,
   \end{displaymath}
   showing that the first statement of this proposition holds. In the
   case that the measure space $(\Sigma,\mu)$ is finite and $A$ with
   $c$-complete resolvent, the same arguments show that $\overline{A}$
   has a $c$-complete resolvent.
\end{proof}

A semigroup
 $\{T_{t}\}_{t\ge0}\sim -A$ on $\overline{D(A)}^{\mbox{}_{L^{1}}}$ of
 a quasi $m$-accretive operator in $L^{1}$ with complete resolvent has
 \emph{exponential growth} in all $L^{\tilde{q}}$-norms 
 \begin{equation}
  \label{eq:60}
  \norm{T_{t}u}_{\tilde{q}}\le e^{\omega\,t}\norm{u}_{\tilde{q}}
  \quad\text{ for all $t>0$, $u\in \overline{D(A)}^{\mbox{}_{L^{1}}}\cap L^{\tilde{q}}(\Sigma,\mu)$,}
\end{equation}
and $1\le \tilde{q}\le \infty$. Similarly, a semigroup
$\{T_{t}\}_{t\ge0}\sim -A$ on $\overline{D(A)}^{\mbox{}_{L^{1}}}$ of a
$m$-accretive operator in $L^{1}$ with $c$-complete resolvent has
\emph{modulo a constant ''c'' non-increasing} $L^{\tilde{q}}$-norm
 \begin{equation}
  \label{eq:120}
  \norm{T_{t}u-c}_{\tilde{q}}\le \norm{u-c}_{\tilde{q}}
  \quad\text{ for all $t>0$, $u\in \overline{D(A)}^{\mbox{}_{L^{1}}}\cap L^{\tilde{q}}(\Sigma,\mu)$,}
\end{equation}
$c\in \R$ and $1\le \tilde{q}\le \infty$. We omit the proof of these
statements since they are shown
similarly as the ones of Proposition~\ref{propo:complete-expo-grow}.

\begin{proposition}
  \label{propo:quasi-accretive-operators-in-L1-complete-resolvent}
  Let $A+\omega I$ be an ($m$-) accretive operator in
  $L^{1}$ with complete resolvent for some $\omega\in \R$. Then for every
  $\lambda>0$ such that $\lambda\,\omega<1$ and every $1\le \tilde{q}\le \infty$,
  the resolvent operator $J_{\lambda}$ of $A$ satisfies
  \begin{equation}
    \label{eq:34}
    \norm{J_{\lambda}u}_{\tilde{q}}\le  (1-\lambda\omega)^{-1}\norm{u}_{\tilde{q}}
  \end{equation}
  for every $u\in Rg(I+\lambda A)\cap L^{\tilde{q}}(\Sigma,\mu)$ and
  the semigroup $\{T_{t}\}_{t\ge0}\sim -A$ on $\overline{D(A)}^{\mbox{}_{L^{1}}}$
  satisfies~\eqref{eq:60}. If $A$ is ($m$-) accretive operator in
  $L^{1}$ with $c$-complete resolvent, then for every
  $\lambda>0$ and $1\le \tilde{q}\le \infty$,
  the resolvent operator $J_{\lambda}$ of $A$ satisfies
  \begin{equation}
    \label{eq:121}
    \norm{J_{\lambda}u-c}_{\tilde{q}}\le  \norm{u-c}_{\tilde{q}}
  \end{equation}
  for every $c\in \R$,
  $u\in Rg(I+\lambda A)\cap L^{\tilde{q}}(\Sigma,\mu)$ and the
  semigroup $\{T_{t}\}_{t\ge0}\sim -A$ on
  $\overline{D(A)}^{\mbox{}_{L^{1}}}$ satisfies~\eqref{eq:120}.
\end{proposition}

The next proposition outlines the construction of an operator of the
second class by taking the composition $A\phi$ of an operator $A$ of
the first class and a continuous non-decreasing function $\phi$ on $\R$.
% An important example of accretive operators in $L^{1}$ with
% ($c$-)complete resolvent is given by the composition $AB$ of accretive
% operators $A$ and $B$ in $L^{1}$.
In the case $\varepsilon=0$, the statements of this proposition are well-known
(cf.~\cite[Proposition~11]{PierreDEA} or \cite[Proposition
2.5]{Benthesis}). 

\begin{proposition}
  \label{propo:composition-operators-in-L1}
  Let $A$ be an accretive operator in $L^{1}(\Sigma,\mu)$ and
  $\phi : \R\to \R$ be a non-decreasing function. Suppose that
  one of the following hypotheses hold:
  \begin{enumerate}[($i$)]
   \item\label{propo:composition-operators-in-L1-H1} $A$ is
     $s$-accretive in $L^{1}(\Sigma,\mu)$ and single-valued.
   \item\label{propo:composition-operators-in-L1-H2} $\phi$ is injective. 
  \end{enumerate}
 Then, the following statements hold.
 \begin{enumerate}
 \item For every $\varepsilon\ge 0$, the operator
   $\varepsilon\phi_{1} +A\phi$ is accretive in
   $L^{1}(\Sigma,\mu)$. 
  \item If, in addition, $A$ has a complete resolvent
   and $\phi$ is continuous satisfying $\phi(0)=0$ (respectively,
    $A$ has a $c$-complete
   resolvent and $(\Sigma,\mu)$ is a finite measure space), then for every $\varepsilon\ge 0$,
   $\varepsilon\phi_{1} + A\phi$ is accretive in $L^{1}$ with complete
   resolvent (respectively, with $c$-complete resolvent).
  \item If, in addition, $A$ is $T$-accretive in $L^{1}(\Sigma,\mu)$
   and $\phi$ is injective, then for every $\varepsilon\ge 0$,
   $\varepsilon\phi_{1} + A\phi$ is $T$-accretive in $L^{1}(\Sigma,\mu)$.
 \end{enumerate}
\end{proposition}

% For the sake of completeness, we give the proof of this minor result
% in Appendix~\ref{Asec:accretive-L1}.

Our next result provides sufficient conditions to ensure that the
composition operator $A\phi$ of an operator $A$ of the first class
and a non-decreasing function $\phi$ on $\R$ satisfies the range
condition~\eqref{eq:range-condition} and so, $-A\phi$ generates a
strongly continuous semigroup on $L^{1}(\Sigma,\mu)$. This result generalises
\cite[Proposition~2]{MR647071} to operators $A\phi$ for (possibly
nonlinear) \emph{$m$-completely accretive operators} $A$ in
$L^{1}$. For the proof of this  result, we refer the interested
reader to the Appendix~\ref{Asec:accretive-L1} of this monograph.
% In particular, it outlines the construction
% of an operator of the second class by taking the composing $A\varphi$
% of an operators $A$ of the first class and a continuous non-decreasing
% function $\phi$ on $\R$.

\begin{proposition}
  \label{propo:Range-cond-in-Rd}
  Suppose $A$ is an $m$-completely accretive operator in
  $L^{q}(\Sigma,\mu)$ for some $1<q<\infty$ with $(0,0)\in A$ and 
  $A_{1\cap \infty}$ be the trace of $A$ on $L^{1}\cap L^{\infty}(\Sigma,\mu)$.
  Let $\phi : \R\to \R$ be a continuous, non decreasing function and
  for every $\lambda>0$, $\beta_{\lambda}$ be the Yosida operator of $\beta=\phi^{-1}$.
  Suppose that 
  \begin{equation}
    \label{eq:252}
    \phi(0)=0,\;
    \text{$A$ and $\beta_{\lambda}$ satisfy \eqref{eq:179} in $L^{q}$,}\;
    \text{$A_{1\cap\infty}$ and $\beta_{\lambda}$ satisfy \eqref{eq:179} in
        $L^{1}$,}
 \end{equation}
 and that one of the following hypotheses holds:
  \begin{enumerate}[($i$)]
  \item \label{propo:Range-cond-in-Rd-Hyp-3} $\phi$ is injective.

  \item\label{propo:Range-cond-in-Rd-Hyp-1} $A$ is $s$-accretive in
    $L^{1}(\Sigma,\mu)$ and single-valued, and there are $r_{0}>0$,
    $K>0$ such that
    \begin{displaymath}
      \abs{\phi(s)}\le K\,\abs{s}\qquad\text{ for every $\abs{s}\le r_{0}$.}
  \end{displaymath}

\item\label{propo:Range-cond-in-Rd-Hyp-2} $A$ is $s$-accretive in
  $L^{1}(\Sigma,\mu)$ and single-valued, and the measure space
  $(\Sigma,\mu)$ is finite.
  \end{enumerate}
  Then, the closure $\overline{A_{1\cap\infty}\phi}$ of $A_{1\cap\infty}\phi$ in
  $L^{1}(\Sigma,\mu)$ is $m$-accretive in $L^{1}(\Sigma,\mu)$ with
  complete resolvent. Moreover, under the hypotheses~\eqref{propo:Range-cond-in-Rd-Hyp-1} and
  \eqref{propo:Range-cond-in-Rd-Hyp-2}, one has
  \begin{equation}
    \label{eq:182}
    \begin{split}
     & \text{for every $\lambda>0$,
        $f\in L^{1}\cap L^{\infty}(\Sigma,\mu)$, there is
        $u\in L^{1}\cap L^{\infty}(\Sigma,\mu)$}\\
      & \text{such that $\phi(u)\in D(A_{1\cap \infty})$ with }u+\lambda A_{1\cap\infty}\phi(u)\ni f.
      \end{split}
  \end{equation}  
  % If, in addition, $A$ is densely defined, then
  % $\overline{D(A_{1\cap\infty}\phi)}^{\mbox{}_{L^{1}}}
  %=D(\overline{A_{1\cap\infty}\phi})=L^{1}(\Sigma,\mu)$.
\end{proposition}

% \begin{remark}
%   \label{rem:7}
%   We note that under the hypotheses of
%   Proposition~\ref{propo:Range-cond-in-Rd}, one has that the two traces
%   \begin{displaymath}
%     \overline{A}\phi\cap   ((L^{1}\cap L^{\infty}(\Sigma,\mu))\times (L^{1}\cap
%     L^{\infty}(\Sigma,\mu))
%   \end{displaymath}
%   and
%   \begin{displaymath}
%     A\phi\cap ((L^{1}\cap L^{\infty}(\Sigma,\mu))\times (L^{1}\cap
%     L^{\infty}(\Sigma,\mu))),
%   \end{displaymath}
%   coincide, where $\overline{A}$ denotes the closure of the
%   trace $A\cap ((L^{1}\cap L^{\infty}(\Sigma,\mu))\times (L^{1}\cap
%     L^{\infty}(\Sigma,\mu))$ of $A$ on $L^{1}\cap L^{\infty}(\Sigma,\mu)$.
% \end{remark}

The class of accretive operators in $L^{1}$ with complete resolvent is
invariant under perturbation by a Lipschitz continuous mapping. This
is shown similarly as in the proof of
Proposition~\ref{propo:Lipschitz-complete-accretive}. Thus, we omit the
proof of the first statement of the following proposition.

\begin{proposition}
  \label{propo:Lipschitz-complete} 
  Let $A$ be an accretive operator in $L^{1}$ with complete resolvent and
  $(0,0)\in A$. Further, suppose $F : L^{1}(\Sigma,\mu)\to
  L^{1}(\Sigma,\mu)$ is the Nemytski operator of a Carath\'eodory function
  $f : \Sigma\times\R\to \R$ satisfying $f(x,0)=0$ for a.e. $x\in \Sigma$
  and Lipschitz condition~\eqref{eq:2} for some constant
  $\omega\ge 0$. Then, the following statements hold:
  \begin{enumerate}
    \item The operator $A+F+\omega I$ is accretive in $L^{1}$ with complete resolvent.
    \item Suppose $A$ and $\phi : \R\to \R$ satisfy the hypotheses of
      Proposition~\ref{propo:Range-cond-in-Rd} and $\overline{A_{1\cap
          \infty}\phi}$ be the closure of $A_{1\cap\infty}\phi$. Then,
      $\overline{A_{1\cap \infty}\phi}+F+\omega I$ is $m$-accretive in
      $L^{1}(\Sigma,\mu)$ and for every $\lambda>0$ satisfying
      $\lambda \omega<1$, one has that
  \begin{equation}
    \label{eq:201}
    L^{1}\cap L^{\infty}(\Sigma,\mu)\subseteq Rg(I+\lambda (A_{1\cap\infty}\phi+F)).
  \end{equation}
  \end{enumerate}
 \end{proposition}

% ... the operator $f(x,\cdot)+\omega I$ is
% accretive in $L^{1}(\Sigma)$ and Lipschitz continuous and hence
% $f(x,\cdot)+\omega I$ is $m$-accretive in
% $L^{1}(\Sigma)$. Thus by~\cite[Theorem~3.1]{MR2582280} and
% Proposition~\ref{propo:Lipschitz-complete},

 \begin{proof}
   By employing the same
   notation as in Proposition~\ref{propo:Range-cond-in-Rd},
   $\overline{A_{1\cap\infty}\phi}$ is $m$-accre\-tive in $L^{1}(\Sigma,\mu)$ with
   complete resolvent. Since $F+\omega I$ is accretive and Lipschitz
   continuous in $L^{1}(\Sigma,\mu)$, a standard fixed point argument shows
   that $F+\omega I$ is $m$-accretive in $L^{1}(\Sigma,\mu)$. By the
   continuity of $F+\omega I$ and since $\overline{A_{1\cap\infty}\phi}$ is
   $m$-accretive in $L^{1}(\Sigma,\mu)$, \cite[Theorem~3.1]{MR2582280}
   implies that $\overline{A_{1\cap\infty}\phi}+F+\omega I$ is $m$-accretive in
   $L^{1}(\Sigma,\mu)$. 

   Now, let $\lambda>0$ such that $\lambda \omega<1$. Then,
   Proposition~\ref{propo:quasi-accretive-operators-in-L1-complete-resolvent}
   yields that the resolvent operator $J_{\lambda}$ of $\overline{A_{1\cap\infty}\phi}+F$
   satisfies~\eqref{eq:34} with respect to the
   $L^{\infty}$-norm. Thus, for every
   $v\in L^{1}\cap L^{\infty}(\Sigma,\mu)$, there is a
   $u\in L^{\infty}(\Sigma,\mu)\cap D(\overline{A_{1\cap\infty}\phi})$ such that
   $u+\lambda (\overline{A_{1\cap\infty}\phi}(u)+F(u))=v$ and so, if
   $J_{\lambda}^{\overline{A_{1}\phi}}$ denotes the resolvent of
   $\overline{A_{1\cap\infty}\phi}$,
   $J_{\lambda}^{\overline{A_{1\cap\infty}\phi}}[v-\lambda F(u)]=u$. On the other
   hand, since $v-\lambda F(u)\in L^1\cap L^{\infty}(\Sigma,\mu)$ and
   since $A_{1\cap\infty}\phi$ satisfies the range condition~\eqref{eq:182}, there is a
   $\tilde{u}\in L^{1}\cap L^{\infty}(\Sigma,\mu)$ such that
   $\phi(\tilde{u})\in D(A_{1\cap\infty})$ and
   $J_{\lambda}^{A_{1}\phi}[v-\lambda F(u)]=\tilde{u}$. Since
   $A_{1\cap\infty}\phi\subseteq \overline{A_{1\cap\infty}\phi}$, the resolvents
   $J_{\lambda}^{\overline{A_{1\cap\infty}\phi}}$ and $J_{\lambda}^{A_{1\cap\infty}\phi}$ coincide
   on $Rg(I+\lambda A_{1\cap\infty}\phi)$ and since $J_{\lambda}^{\overline{A_{1}\phi}}$
   is contractive on $L^{1}(\Sigma,\mu)$, we obtain that
   $\tilde{u}=u$, implying that $u$ satisfies
   $u+\lambda ( A_{1\cap\infty}\phi(u)+F(u))=v$. This shows that also the second
   statement of this proposition holds.
 \end{proof}

%-----------------------------------------------------------------------------
%-----------------------------------------------------------------------------
%-----------------------------------------------------------------------------
%-----------------------------------------------------------------------------
%-----------------------------------------------------------------------------
%-----------------------------------------------------------------------------
%-----------------------------------------------------------------------------
%-----------------------------------------------------------------------------
%-----------------------------------------------------------------------------
%-----------------------------------------------------------------------------
%-----------------------------------------------------------------------------
%-----------------------------------------------------------------------------
%-----------------------------------------------------------------------------
%-----------------------------------------------------------------------------
%-----------------------------------------------------------------------------
%-----------------------------------------------------------------------------
%-----------------------------------------------------------------------------
%-----------------------------------------------------------------------------
%-----------------------------------------------------------------------------
\section{Gagliardo-Nirenberg type inequalities \&
  $L^{q}$-$L^{r}$-regularity}\label{gn}

This section is concerned with establishing
$L^{q}$-$L^{r}$-regularisation estimates for $1\le q, r\le \infty$ of
semigroups $\{T_{t}\}_{t\ge 0}$ provided their infinitesimal generator $-A$ satisfies
a \emph{Gagliardo-Nirenberg type inequality} of the
form~\eqref{eq:242} or~\eqref{eq:5}.

\begin{remark}
  We note that for $\omega=0$, the Gagliardo-Nirenberg type
  inequality~\eqref{eq:242} reduces to
\begin{displaymath}
    \norm{u-u_{0}}_{r}^{\sigma} \le C\,[u-u_{0},v]_{q}\;
    \norm{u-u_{0}}_{q}^{\varrho}
\end{displaymath}
for all $(u,v)\in A$, and the Gagliardo-Nirenberg type inequality with
differences~\eqref{eq:5} becomes
\begin{equation}
  \label{eq:243}
  \norm{u-\hat{u}}_{r}^{\sigma} \le C\,
   [u-\hat{u},v-\hat{v}]_{q}\;\norm{u-\hat{u}}_{q}^{\varrho}
\end{equation}
for all $(u,v)$, $(\hat{u},\hat{v})\in A$, which are similar to the classical one
(cf.~\cite{MR0109940}).
\end{remark}

Further, similar to the \emph{classical case}, for $\varrho=0$, Ga\-gliardo-Nirenberg type
inequalities~\eqref{eq:242} and~\eqref{eq:5} reduce to the
following so-called Sobolev type inequalities.

\begin{definition}
  \label{def:Sobolev-inequality}
  We say an operator $A$ on $L^{q}(\Sigma,\mu)$ for some $1\le q<\infty$ satisfies a
  \emph{Sobolev type inequality for some $(u_{0},0)\in A$}
  (respectively, with \emph{differences}) if there exist
  $1\le r\le \infty$,  $\sigma>0$, and $C>0$ such that
  $(u_{0},0)\in A$ and
   \begin{displaymath}
    \norm{u-u_{0}}_{r}^{\sigma} \le C\,
    \Big( [u-u_{0},v]_{q}+\omega \norm{u-u_{0}}_{q}^{q}\Big)
  \end{displaymath}
   for every $(u,v)\in A$  (respectively,
   \begin{displaymath}
    \norm{u-\hat{u}}_{r}^{\sigma} \le C\,
    \Big( [u-\hat{u},v-\hat{v}]_{q}+\omega \norm{u-\hat{u}}_{q}^{q}\Big)
  \end{displaymath}
  for every $(u,v)$, $(\hat{u},\hat{v})\in A$).
\end{definition}

Our first main theorem of this section applies to the class
of operators considered in Section~\ref{sec:comp}. 

\begin{theorem}\label{thm:quasi-super-contractivity}
  Let $A+\omega I$ be an $m$-accretive operator on $L^{q}(\Sigma,\mu)$
  for some $1\le q< \infty$ and $\omega\ge 0$. Suppose $A$
  satisfies the Gagliardo-Nirenberg type inequality~\eqref{eq:5} for some $1\le
  r\le \infty$, $\varrho\ge 0$ and $\sigma>0$, and the semigroup
  $\{T_{t}\}_{t\ge 0}\sim -A$ on $\overline{D(A)}^{\mbox{}_{L^{q}}}$
  has exponential growth~\eqref{eq:62} for $\tilde{q}=r$. Then  $\{T_{t}\}_{t\geq 0}$ satisfies
   \begin{equation}\tag{\ref{eq:18}}
    \norm{T_{t}u-T_{t}\hat{u}}_{r}\le 
    \left(\tfrac{C}{q}\right)^{1/\sigma}\,t^{-\alpha}\,e^{\omega\,\beta\,t}\,
    \norm{u-\hat{u}}_{q}^{\gamma}
  \end{equation}
  for every $t>0$ and $u$, $\hat{u}\in \overline{D(A)}^{\mbox{}_{L^{q}}}$ with
  exponents $\alpha=\frac{1}{\sigma}$, $\beta=\gamma+1$ and $\gamma=\frac{q+\varrho}{\sigma}$.
\end{theorem}

\allowdisplaybreaks
\begin{remark}
 \label{rem:8}
  If $1\le q< r\le \infty$ and if there is an element
  $u_{0}\in \overline{D(A)}^{\mbox{}_{L^{q}}}$ such that
  $T_{t}u_{0}\in L^{r}(\Sigma,\mu)$ for some (all) $t>0$, then
  inequality~\eqref{eq:18} implies that $\{T_{t}\}_{t\ge 0}$ enjoys an
  \emph{$L^{q}$-$L^{r}$-regularisation effect} in the sense that for
  some (all) $t>0$, $T_{t}$ maps $\overline{D(A)}^{\mbox{}_{L^{q}}}$
  into $L^{r}(\Sigma,\mu)$. Thus we call
  inequality~\eqref{eq:18} an \emph{$L^{q}$-$L^{r}$-regularisation
    estimate} if $r>q$. If $q\ge r$ then we call~\eqref{eq:18} an \emph{$L^{q}$-$L^{r}$-regularity
    estimate}. For example, the
  semigroup $\{T_{t}\}_{t\ge 0}$ associated with the
  total variational flow (see~\cite{Ha2015TotalVariationalFlow})
  satisfies inequality~\eqref{eq:18} for some $r<q$ and some $u_{0}\in
  \overline{D(A)}^{\mbox{}_{L^{q}}}\cap L^{\infty}(\Sigma,\mu)$
  satisfying $T_{t}u_{0}=u_{0}$ for all $t\ge 0$. 
\end{remark}

\begin{remark}
  \label{rem:2}
  We want to emphasise that Theorem~\ref{thm:quasi-super-contractivity}
  implies that the parameters $1\le r\le \infty$, $1\le q<\infty$ and
  exponents $\sigma>0$ and $\varrho\ge 0$ in
  $L^{q}$-$L^{r}$-regulari\-sation estimate~\eqref{eq:18} are
  \emph{stable} under a monotone or Lipschitz continuous
  perturbation. To be more specific, suppose $B$ is an accretive
  operator on $L^{q}(\Sigma,\mu)$ satisfying the Gagliardo-Nirenberg type
  inequality~\eqref{eq:243}, $F$ be the Nemytski operator
  on $L^{q}(\Sigma,\mu)$ of a Carath\'eodory function
  $f : \Sigma\times \R\to\R$ satisfying $f(x,0)=0$ for a.e. $x\in \Sigma$
  and Lipschitz condition~\eqref{eq:2} for some constant
  $\omega\ge 0$ and $\beta_{q}$ the accretive operator on
  $L^{q}(\Sigma,\mu)$ associated with a monotone graph $\beta$ on $\R$
  (if $q=1$ suppose, in addition, that $B+\beta_{1}$ is accretive). We
  set $A:=B+\beta_{q}+F$. Then, by property~\eqref{eq:57} of the
  $q$-bracket $[\cdot,\cdot]_{q}$ and since $\beta_{q}$ and
  $F+\omega I$ are accretive in $L^{q}(\Sigma,\mu)$, we see that
  \begin{align*}
    &
      [u-\hat{u},(v_{1}+v_{2}+F(u))-(\hat{v}_{1}+\hat{v}_{2}+F(\hat{u}))]_{q} 
      + \omega \norm{u-\hat{u}}_{q}^{q}\\
    & \qquad= [u-\hat{u},v_{1}-\hat{v}_{1}]_{q}+ [u-\hat{u},v_{2}-\hat{v}_{2}]_{q}\\
    &\qquad\hspace{2,2cm}   + 
      [u-\hat{u},(F(u)+\omega u)-(F(\hat{u})+\omega\hat{u})]_{q}\\
    & \qquad \ge [u-\hat{u},v_{1}-\hat{v}_{1}]_{q}
  \end{align*}
  for every $u$, $\hat{u}\in D(A)\cap D(\beta_{q})$, $v_{1}\in Bu$,
  $\hat{v}_{1}\in B\hat{u}$, $v_{2}\in \beta_{q}(u)$,
  $\hat{v}_{2}\in \beta_{q}(\hat{u})$. Thus,
  the $L^{q}$-$L^{r}$-regularisation effect~\eqref{eq:18} for $1\le q<r\le \infty$ of a
  semigroup $\{T_{t}\}_{t\geq 0}\sim-A$ for $A=B+F$ is only
  determined by $B$.
\end{remark}

\begin{remark}
   The statement of Theorem~\ref{thm:quasi-super-contractivity}
  remains unchanged if one replaces the constant $e^{\omega t}$ in
  condition~\eqref{eq:62} for $\tilde{q}=s$ by $M\,e^{\omega t}$
  for some constant $M>0$.  Then the constant $C$ in~\eqref{eq:18} has to be changed accordingly.
\end{remark}

A common situation in applications is the one where \emph{$A$ is quasi $m$-completely accretive
   on $L^{2}(\Sigma,\mu)$}. Also, we shall see in
Section~\ref{sec:p-laplace} how to derive a  Gagliardo-Nirenberg type
inequality~\eqref{eq:5} for $q=2$. Therefore in practice we shall often use the following special case of
Theorem~\ref{thm:quasi-super-contractivity}.

\begin{corollary}
  \label{coro:super-contractivity}
  Let $A+\omega I$ be an $m$-completely accretive operator on
  $L^{2}(\Sigma,\mu)$ for some $\omega\ge 0$. Suppose there are
  $1\le r\le \infty$, $\varrho\ge 0$, $\sigma>0$ and $C>0$ such that
  \begin{equation}
    \label{eq:245}
    \norm{u-\hat{u}}_{r}^{\sigma} \le C\,
    \Big[ [u-\hat{u},v-\hat{v}]_{2}+\omega \norm{u-\hat{u}}_{2}^{2}\Big]\;
    \norm{u-\hat{u}}_{2}^{\varrho}
  \end{equation}
  for every $(u,v)$,
  $(\hat{u},\hat{v})\in A$. Then the semigroup $\{T_{t}\}_{t\geq 0}\sim-A$ on
  $\overline{D(A)}^{\mbox{}_{L^{2}}}$ satisfies
  \begin{displaymath}
    %\label{eq:6bis}
    \norm{T_{t}u-T_{t}\hat{u}}_{r}\le 
    \left(\tfrac{C}{2}\right)^{1/\sigma}\,t^{-\alpha}\,e^{\omega\,\beta\,t}\,
    \norm{u-\hat{u}}_{2}^{\gamma}
  \end{displaymath}
  for every $t>0$ and $u$, $\hat{u}\in \overline{D(A)}^{\mbox{}_{L^{2}}}$ with
  $\alpha=\tfrac{1}{\sigma}$, $\beta=\gamma+1$ and $\gamma=\frac{2+\varrho}{\sigma}$.
\end{corollary}

Now, we turn to the proof of
Theorem~\ref{thm:quasi-super-contractivity}. For this, we first
consider the case $q>1$. Then by \eqref{eq:35},
the $q$-brackets $[u-\hat{u},v-\hat{v}]_{q}$ can be replaced by 
$\langle (u-\hat{u})_{q},v-\hat{v}\rangle$
in inequality~\eqref{eq:5}. Moreover, the Lebesgue space
$L^{q}(\Sigma,\mu)$ and its dual space are uniformly convex Banach
spaces and so for every $u\in D(A)$, the mild solution
$t\mapsto T_{t}u$ is almost everywhere differentiable, everywhere
differentiable from the right on $[0,\infty)$ with values in
$L^{q}(\Sigma,\mu)$, and satisfies~\eqref{eq:38}. Using this
leads to the following short proof of
Theorem~\ref{thm:quasi-super-contractivity} in this situation
(cf. \cite{MR1164643} in the case of linear semigroups for $\omega=0$
and $\varrho=1$). 

\begin{proof}[First proof of
  Theorem~\ref{thm:quasi-super-contractivity} for $q>1$]
  First, let $u$, $\hat{u}\in D(A)$. By
  hypothesis, one has
  \begin{equation}
    \label{eq:41}
    \norm{T_{t}u-T_{t}\hat{u}}_{\tilde{q}}\le e^{\omega (t-s)}\norm{T_{s}u-T_{s}\hat{u}}_{\tilde{q}}
  \end{equation}
  for every $t\ge s>0$ and for every $\tilde{q}\in \{q,r\}$. Combining
  this with  inequality~\eqref{eq:5}
  and the fact that $\tfrac{d}{dt}_{+}T_{t}u=-A^{0}T_{t}u$ for every
  $t\ge 0$ (cf.~\eqref{eq:38}), we see that \allowdisplaybreaks
  \begin{align*}
      \norm{u-\hat{u}}_{q}^{q+\varrho} & \ge
      \Big[\norm{u-\hat{u}}_{q}^{q}-e^{-\omega\,q\,t}\norm{T_{t}u-T_{t}\hat{u}}_{q}^{q}\Big]
      \,\norm{u-\hat{u}}_{q}^{\varrho}\\
      & = \Big[-
      \int_{0}^{t}\tfrac{d}{ds}\left(e^{-\omega\,q\,s}\norm{T_{s}u-T_{s}\hat{u}}_{q}^{q}\right)\,\ds\Big]\,
      \norm{u-\hat{u}}_{q}^{\varrho}\\
      & =\Big[ q \int_{0}^{t} e^{-\omega\,q\,s}\big( \langle
      (T_{s}u-T_{s}\hat{u})_{q},A^{\circ}T_{s}u-A^{\circ}T_{s}\hat{u}\rangle\,\Big.\\
      &\hspace{5cm}+\omega\,\Big. \norm{T_{s}u-T_{s}\hat{u}}_{q}^{q}\big) 
      \ds\Big]\,\norm{u-\hat{u}}_{q}^{\varrho}\\
      & \ge q \int_{0}^{t} e^{-\omega\,(q+\varrho)\,s}\Big[ \langle
      (T_{s}u-T_{s}\hat{u})_{q},A^{\circ}T_{s}u-A^{\circ}T_{s}\hat{u}\rangle\,\Big.\\
      &\hspace{4cm}+\omega\,\Big. \norm{T_{s}u-T_{s}\hat{u}}_{q}^{q}\Big]\, 
      \norm{T_{s}u-T_{s}\hat{u}}_{q}^{\varrho}\,\ds\\
      & \ge \frac{q}{ C}
      \int_{0}^{t}e^{-\omega\,(q+\varrho)\,s}\norm{T_{s}u-T_{s}\hat{u}}_{r}^{\sigma}\,\ds\\
    &  \ge \frac{q}{ C}
     \left( \int_{0}^{t}e^{-\omega\,(q+\varrho)\,s}\,e^{-\omega (t-s)\sigma}\,\ds \right)
       \norm{T_{t}u-T_{t}\hat{u}}_{r}^{\sigma}.\\
       &  \ge \frac{q}{ C}
     \left( \int_{0}^{t}e^{-\omega\,(q+\varrho)\,s}\,e^{-\omega t\sigma}\,\ds \right)
       \norm{T_{t}u-T_{t}\hat{u}}_{r}^{\sigma}.\\
    & \ge \frac{q}{ C'}\,t\,e^{-\omega(q+\varrho+\sigma)\,t}\,\norm{T_{t}u-T_{t}\hat{u}}_{r}^{\sigma},
  \end{align*}
  showing that inequality~\eqref{eq:18} holds for $u$, $\hat{u}\in D(A)$.

  Now, let $u$, $\hat{u}\in \overline{D(A)}^{\mbox{}_{L^{q}}}$. Then,
  there are sequences $(u_{n})$ and $(\hat{u}_{n})$ in $D(A)$ such
  that $u_{n}$ converges to $u$ and $\hat{u}_{n}$ converges to
  $\hat{u}$ in $L^{q}(\Sigma,\mu)$. Since the semigroup
  $\{T_{t}\}_{t\ge0}$ has exponential growth~\eqref{eq:62} for
  $\tilde{q}=q$, for every $t\ge 0$, the sequence
  $S_{n}(t):=T_{t}u_{n}-T_{t}\hat{u}_{n}$ converges to
  $S(t):=T_{t}u-T_{t}\hat{u}$ in $L^{q}(\Sigma,\mu)$. % After eventually
  % passing to a subsequence of $(u_{n})$, we may assume that $S_{n}(t)$
  % converges to $S(t)$ a.e. on $\Sigma$.
  Moreover, by the first step of
  this proof, inequality~\eqref{eq:18} implies that
  \begin{displaymath}
   % \label{eq:33}
    \norm{S_{n}(t)}_{r}\le 
    \left(\tfrac{C}{q}\right)^{1/\sigma}\,t^{-\alpha}\,e^{\omega \frac{q+\varrho+\sigma}{\sigma}\,t}\;
    \norm{u_{n}-\hat{u}_{n}}_{q}^{\frac{q+\varrho}{\sigma}}
  \end{displaymath}
  for every $n$. Since the $L^{r}$-norm is lower semicontinuous on
  $L^{q}(\Sigma,\mu)$, sending $n\to\infty$ in the previous inequality
  yields % Thus, if $r=\infty$ then the limit laws imply that inequality~\eqref{eq:18} holds for $u$,
  % $\hat{u}\in\overline{D(A)}^{\mbox{}_{L^{q}}}$ and if $r<\infty$ then
  % by Fatou's lemma we can conclude that
  $S(t)\in L^{r}(\Sigma,\mu)$ and
  \begin{displaymath}
    \norm{S(t)}_{r}\le 
    \left(\tfrac{C}{q}\right)^{1/\sigma}\;t^{-\alpha}\; e^{\omega \frac{q+\varrho+\sigma}{\sigma}\,t}\;
    \norm{u-\hat{u}}_{q}^{\frac{q+\varrho}{\sigma}}.
  \end{displaymath}
  Therefore inequality~\eqref{eq:18} holds for every $u$,
  $\hat{u}\in\overline{D(A)}^{\mbox{}_{L^{q}}}$, completing
  the proof of Theorem~\ref{thm:quasi-super-contractivity} for $q>1$.
\end{proof}

Our second proof of Theorem~\ref{thm:quasi-super-contractivity} is
rather technical and uses the definition of mild solutions
(cf. \cite{MR554377} in the case $\omega=\varrho=1$).

\begin{proof}[Second proof of Theorem~\ref{thm:quasi-super-contractivity}]
  Let $u$, $\hat{u}\in \overline{D(A)}^{\mbox{}_{L^{q}}}$.  For given
  $t>0$, we choose $N\ge 1$ large enough such that
  $\tfrac{\omega\,q\,t}{N}<\tfrac{1}{2}$ and set
  $t_{n}=n\,\frac{t}{N}$ for every $n=0,\dots,N$, $u_{0}=u$ and
  $\hat{u}_{0}=\hat{u}$. By hypothesis,
  $Rg(I+\tfrac{\lambda}{1-\lambda \omega}A)=L^{q}(\Sigma,\mu)$ for every
  $0<\lambda<\tfrac{1}{\omega}$. Thus, there are $u_{1}$,
  $\hat{u}_{1}\in D(A)$ solving $u_{1}+\tfrac{t}{N}A u_{1}\ni u_{0}$
  and $\hat{u}_{1}+\tfrac{t}{N}A\hat{u}_{1}\ni
  \hat{u}_{0}$. Iteratively, for every $n=1,\dots,N$, there are solutions $u_{n}$
  and $\hat{u}_{n}\in D(A)$ of
  \begin{equation}
    \label{eq:42}
    u_{n}+\tfrac{t}{N} A u_{n}\ni u_{n-1}\quad\text{and}\quad
    \hat{u}_{n}+\tfrac{t}{N} A \hat{u}_{n}\ni \hat{u}_{n-1},
  \end{equation}
  respectively. We set
  \begin{displaymath}
    U_{N}(s)=u_{0}\mathds{1}_{\{t_{0}=0\}}(s)
    +\sum_{n=1}^{N}u_{n}\;\mathds{1}_{(t_{n-1},t_{n}]}(s)
  \end{displaymath}
  and
  \begin{displaymath}
    \hat{U}_{N}(s)=\hat{u}_{0}\mathds{1}_{\{t_{0}=0\}}(s)+\sum_{n=1}^{N}\hat{u}_{n}\;
    \mathds{1}_{(t_{n-1},t_{n}]}(s)
  \end{displaymath}
  for every $s\in [0,t]$. Further, for $v_{n}=(u_{n-1}-u_{n})\tfrac{N}{t}$
  and $\hat{v}_{n}=(\hat{u}_{n-1}-\hat{u}_{n})\tfrac{N}{t}$, both
  inclusions in~\eqref{eq:42} can be rewritten as $ v_{n}\in A u_{n}$
  and $\hat{v}_{n}\in A\hat{u}_{n}$, or as $J_{t/N}u_{n-1}=u_{n}$ and
  $J_{t/N}\hat{u}_{n-1}=\hat{u}_{n}$ for every $n=1,\dots,N$. Hence by
  Gagli\-ardo-Nirenberg type inequalities~\eqref{eq:5},
  \eqref{eq:57} and~\eqref{eq:56}, we see that
  \begin{align*}
    &\norm{u_{n}-\hat{u}_{n}}_{r}^{\sigma}\\
    &\;\le C\,\Big(\,[u_{n}-\hat{u}_{n},v_{n}-\hat{v}_{n}]_{q}
      +\omega \norm{u_{n}-\hat{u}_{n}}_{q}^{q}\Big)\;
    \norm{u_{n}-\hat{u}_{n}}_{q}^{\varrho}\\
   &\;= C\,\tfrac{N}{t}\Big(\,[u_{n}-\hat{u}_{n},(u_{n-1}-u_{n})-(\hat{u}_{n-1}-\hat{u}_{n})]_{q}
      +\tfrac{\omega\,t}{N} \norm{u_{n}-\hat{u}_{n}}_{q}^{q}\Big)\;
    \norm{u_{n}-\hat{u}_{n}}_{q}^{\varrho}\\
    &\;= C\,\tfrac{N}{t}\Big(\,[u_{n}-\hat{u}_{n},(u_{n-1}-\hat{u}_{n-1})-(u_{n}-\hat{u}_{n})]_{q}
      +\tfrac{\omega\,t}{N} \norm{u_{n}-\hat{u}_{n}}_{q}^{q}\Big)\;
    \norm{u_{n}-\hat{u}_{n}}_{q}^{\varrho}\\
    &\;\le
      C\,\tfrac{N}{t}\Big(\,\tfrac{1}{q}\norm{u_{n-1}-\hat{u}_{n-1}}_{q}^{q}-
      (1-\tfrac{\omega\,q\,t}{N})\tfrac{1}{q}\norm{u_{n}-\hat{u}_{n}}_{q}^{q}\Big)\;
    \norm{u_{n}-\hat{u}_{n}}_{q}^{\varrho}
  \end{align*}
  for every $n=1,\dots,N$. By assumption, $J_{t/N}$ satisfies
  inequality~\eqref{eq:61} for $\tilde{q}=q$. Hence
  \begin{align*}
    \norm{u_{n}-\hat{u}_{n}}_{q} 
      & =\norm{J_{t/N}u_{n-1}-J_{t/N}\hat{u}_{n-1}}_{q}\\
      & \le
        (1-\tfrac{t\omega}{N})^{-1}\,\norm{u_{n-1}-\hat{u}_{n-1}}_{q}\\
      & \;\;\vdots\\
      & \le (1-\tfrac{t\omega}{N})^{-n}\norm{u_{0}-\hat{u}_{0}}_{q}\\
     & \le (1-\tfrac{t\omega}{N})^{-N}\norm{u_{0}-\hat{u}_{0}}_{q}
  \end{align*}
  Using this in order to estimate the term
  $\norm{u_{n}-\hat{u}_{n}}_{q}^{\varrho}$ in the previous inequality
  and multiplying the resulting inequality by
  $\tfrac{t}{N}\,(1-\tfrac{\omega\,q\,t}{N})^{-1}$ yields
  \begin{align*}
   &\tfrac{t}{N} (1-\tfrac{\omega\,q\,t}{N})^{-1}\norm{u_{n}-\hat{u}_{n}}_{r}^{\sigma}\\ 
   &\;\le
      C\,\Big(\,(1-\tfrac{\omega\,q\,t}{N})^{-1}\,\tfrac{1}{q}\norm{u_{n-1}-\hat{u}_{n-1}}_{q}^{q}-
      \tfrac{1}{q}\norm{u_{n}-\hat{u}_{n}}_{q}^{q}\Big)\;\\
    &\hspace{8cm} \times (1-\tfrac{t\omega}{N})^{-N\varrho}
    	\norm{u_{0}-\hat{u}_{0}}_{q}^{\varrho}.
  \end{align*}
  Rearranging the last inequality gives
  \begin{displaymath}
    \tfrac{1}{q}\norm{u_{n}-\hat{u}_{n}}_{q}^{q} \le
    (1-\tfrac{\omega\,q\,t}{N})^{-1}\,\tfrac{1}{q}\norm{u_{n-1}-\hat{u}_{n-1}}_{q}^{q}
    + b_{n}
  \end{displaymath}    
  for every $n=1,\dots,N$, where we set
  \begin{equation}
    \label{eq:183}
    b_{n}:= -\tfrac{t}{N}
    (1-\tfrac{\omega\,q\,t}{N})^{-1}\,\norm{u_{n}-\hat{u}_{n}}_{r}^{\sigma}\,C^{-1}
    \, (1-\tfrac{t\omega}{N})^{N\varrho}\,\norm{u_{0}-\hat{u}_{0}}_{q}^{-\varrho}.
  \end{equation}
  It is easy to see that
  \begin{equation}
    \label{eq:110}
    \begin{cases}
      &\text{for sequences $(\lambda_{n})\subseteq [0,\infty)$ and $(a_{n})$, $(b_{n})\subseteq \R$
   satisfying}\\
 &\text{$a_{n}\le \lambda_{n}a_{n-1}+b_{n}$ for all
  $n=1,\dots,N$, one has that }\\[7pt]
   &\qquad    a_{N}\le a_{0}\,\left(\displaystyle\prod_{n=1}^{N}\lambda_{n}\right) +
      \displaystyle\sum_{n=1}^{N}b_{n} \left(\displaystyle\prod_{k=n+1}^{N}\lambda_{n}\right)
    \end{cases}
  \end{equation}
  (cf.~\cite[Exercise E3.8]{Benilanbook}). Applying this to $\lambda_{n}=(1-\tfrac{\omega\,q\,t}{N})^{-1}$,
  $a_{n}=\tfrac{1}{q}\norm{u_{n}-\hat{u}_{n}}_{q}^{q}$ and $b_{n}$
  given by~\eqref{eq:183}, we obtain 
  \begin{displaymath}
    \tfrac{1}{q}\norm{u_{n}-\hat{u}_{n}}_{q}^{q} \le
    (1-\tfrac{\omega\,q\,t}{N})^{-N}\,
    \tfrac{1}{q}\norm{u_{0}-\hat{u}_{0}}_{q}^{q}+\sum_{n=1}^{N}
    (1-\tfrac{\omega\,q\,t}{N})^{-(N-(n+1))} b_{n}.
  \end{displaymath}
  Using that $(1-\tfrac{\omega\,q\,t}{N})^{n}\le
  (1-\tfrac{\omega\,q\,t}{N})^{N}$ and rearranging this
  inequality yields
  \begin{align*}
    &
      (1-\tfrac{\omega\,q\,t}{N})^{N}\tfrac{1}{q}\norm{u_{n}-\hat{u}_{n}}_{q}^{q}\\
    &\hspace{2cm} + C^{-1}\,
      (1-\tfrac{t\omega}{N})^{N\varrho}\,\norm{u_{0}-\hat{u}_{0}}_{q}^{-\varrho}
      \,(1-\tfrac{\omega\,q\,t}{N})^{N}\, 
      \sum_{n=1}^{N}\tfrac{t}{N} \norm{u_{n}-\hat{u}_{n}}_{r}^{\sigma}\\
   & \qquad \le \tfrac{1}{q}\norm{u_{0}-\hat{u}_{0}}_{q}^{q}
  \end{align*}
  so that
  \begin{align*}
    &
      (1-\tfrac{\omega\,q\,t}{N})^{N}\tfrac{1}{q}\norm{U_{N}(t)-\hat{U}_{N}(t)}_{q}^{q}\\
    &\hspace{2cm}  + C^{-1}\,
      (1-\tfrac{t\omega}{N})^{N\varrho}\,\norm{u-\hat{u}}_{q}^{-\varrho}
      \,(1-\tfrac{\omega\,q\,t}{N})^{N}\, 
      \int_{0}^{t}\norm{U_{N}(s)-\hat{U}_{N}(s)}_{r}^{\sigma}\,\ds\\
   & \qquad \le \tfrac{1}{q}\norm{u-\hat{u}}_{q}^{q}
  \end{align*}
  By the Crandall-Liggett theorem,
  \begin{displaymath}
    \lim_{N\to\infty}U_{N}=T_{t}u\quad\text{ in $L^{q}(\Sigma,\mu)$}\quad\text{ and }\quad
    \lim_{N\to\infty}\hat{U}_{N}=T_{t}\hat{u}\quad\text{ in $L^{q}(\Sigma,\mu)$}
\end{displaymath}
respectively uniformly on $[0,t]$. Thus, sending $N\to \infty$ in the
previous estimate and using the lower
semicontinuity of the $L^{r}$-norm   on $L^{q}(\Sigma,\mu)$ yields
  \begin{align*}
    &e^{-\omega\,q\,t}\,\tfrac{1}{q}\norm{T_{t}u-T_{t}\hat{u}}_{q}^{q}
      + C^{-1}\,e^{-\omega\,\varrho\,t}\,
      \norm{u-\hat{u}}_{q}^{-\varrho}\,e^{-\omega\,q\,t}\, 
      \int_{0}^{t}\norm{T_{s}u-T_{s}\hat{u}}_{r}^{\sigma}\,\ds\\
   & \hspace{2cm} \le \tfrac{1}{q}\norm{u-\hat{u}}_{q}^{q}
  \end{align*}
 and so
 \begin{displaymath}
   C^{-1}\,e^{-\omega\,\varrho\,t}\,
      \norm{u-\hat{u}}_{q}^{-\varrho}\,e^{-\omega\,q\,t}\, 
      \int_{0}^{t}\norm{T_{s}u-T_{s}\hat{u}}_{r}^{\sigma}\,\ds\\
   \le \tfrac{1}{q}\norm{u-\hat{u}}_{q}^{q}
 \end{displaymath}
  By assumption, $\{T_{t}\}_{t\geq 0}$ satisfies~\eqref{eq:41} for $\tilde{q}=r$
  from where we can deduce that \eqref{eq:18} holds.  
\end{proof}

 Even for the class of quasi-$m$-accretive
operators $A$ on $L^{q}$, there are situations in which the
operator $A$ merely satisfies 
the Gagliardo-Nirenberg type inequality~\eqref{eq:242} for some
$(u_{0},0)\in A$. In this situation, we can state the following result.

\begin{theorem}\label{thm:quasi-super-contractivity-bis}
  Let $A+\omega I$ be an $m$-accretive operator on $L^{q}(\Sigma,\mu)$
  for some $1\le q<\infty$ and $\omega\ge 0$. Suppose $A$
  satisfies the Gagliardo-Nirenberg type
  inequality~\eqref{eq:242} for parameters $1\le r\le \infty$,
  $\varrho\ge 0$, $\sigma>0$ and some $(u_{0},0)\in A$ satisfying
  $u_{0}\in L^{q}\cap L^{r}(\Sigma,\mu)$, and the semigroup
  $\{T_{t}\}_{t\ge 0}\sim -A$ on $\overline{D(A)}^{\mbox{}_{L^{q}}}$
  has exponential growth~\eqref{eq:62} for $\tilde{q}=r$. Then
  the semigroup $\{T_{t}\}_{t\geq 0}$ satisfies
   \begin{equation}
      \tag{~\ref{eq:20}}
     %\label{eq:6-L1-setting}
      \norm{T_{t}u-u_{0}}_{r}\le 
      \left(\tfrac{C}{q}\right)^{1/\sigma}\; t^{-\alpha}\; e^{\omega \beta t}\;
      \norm{u-u_{0}}_{q}^{\gamma}
    \end{equation}
    for every $t>0$, $u\in L^{q}(\Sigma,\mu)$ with
    exponents $\alpha=\frac{1}{\sigma}$, $\beta=\gamma+1$ and
    $\gamma=\tfrac{q+\varrho}{\sigma}$. 
\end{theorem}

We omit the proof of Theorem~\ref{thm:quasi-super-contractivity-bis}
since it proceeds along the lines of the second proof of
Theorem~\ref{thm:quasi-super-contractivity}. 

Analogously, as above, the important case $q=2$ and $A$ is quasi $m$-completely accretive
operator on $L^{2}(\Sigma,\mu)$ follows immediately from
Theorem~\eqref{thm:quasi-super-contractivity-bis}.

\begin{corollary}
  \label{coro:super-contractivity-bis}
  Let $A+\omega I$ be $m$-completely accretive operator on
  $L^{2}(\Sigma,\mu)$ for some $\omega\ge 0$. Suppose there are
  $(u_{0},0)\in A$, $2< r\le \infty$, $\varrho\ge 0$, $\sigma>0$ and $C>0$ such that
  \begin{equation}
    \label{eq:190}
    \norm{u-u_{0}}_{r}^{\sigma} \le C\,
    \Big[ [u-u_{0},v]_{2}+\omega \norm{u-u_{0}}_{2}^{2}\Big]\;
    \norm{u-u_{0}}_{2}^{\varrho}
  \end{equation}
  for every $(u,v)\in A$. Then the semigroup $\{T_{t}\}_{t\geq 0}\sim-A$
  on $\overline{D(A)}^{\mbox{}_{L^{2}}}$ satisfies
  \begin{displaymath}
    \norm{T_{t}u-u_{0}}_{r}\le 
    \left(\tfrac{C}{2}\right)^{1/\sigma}\; t^{-\alpha}\; e^{\omega \beta t}\;
    \norm{u-u_{0}}_{2}^{\gamma}
  \end{displaymath}
  for every $t>0$ and $u\in \overline{D(A)}^{\mbox{}_{L^{2}}}$ with exponents
  $\alpha=\tfrac{1}{\sigma}$, $\beta=\gamma+1$ and $\gamma=\frac{2+\varrho}{\sigma}$.
\end{corollary}

%%%%%%%%%%%%%%%%%%%%%%%%%%%%%%%%%%%%%%%%%%%%%%%%%%%%%%%%%%%
%
%
%
%               Gagliardo-Nirenberg inequality and operators in $L^{1}$
%
%
%
%%%%%%%%%%%%%%%%%%%%%%%%%%%%%%%%%%%%%%%%%%%%%%%%%%%%%%%%%%%

Our third main theorem of this section considers the second class of
operators introduced in Section~\ref{subsec:L1}. As a matter of fact,
many examples show that the Gagliardo-Nirenberg type
inequality~\eqref{eq:242} is not satisfied by a quasi
$m$-accretive operator $A$ in $L^{1}(\Sigma,\mu)$ with ($c$-)complete
resolvent. But in order to obtain $L^{q}$-$L^{r}$-regularisation
estimates with $1\le q$, $r\le \infty$ for the semigroup
$\{T_{t}\}_{t\geq 0}\sim -A$ on $\overline{D(A)}^{\mbox{}_{L^{1}}}$,
it turns out that it is sufficient that for some
$1\le q\le q_{0}\le \infty$, the \emph{trace}
\begin{displaymath}
    A_{1\cap q_{0}}:= A\cap ((L^{1}\cap L^{q_{0}}(\Sigma,\mu))\times 
    (L^{1}\cap L^{q_{0}}(\Sigma,\mu))
  \end{displaymath}
  of $A$ on $L^{1}\cap L^{q_{0}}(\Sigma,\mu)$
  satisfies~\eqref{eq:242}. Note that, for $1\le q\le q_{0}\le \infty$,
  $L^{1}\cap L^{q_{0}}(\Sigma,\mu)$ injects continuously into
  $L^{q}(\Sigma,\mu)$. Hence, then trace $A_{1\cap q_{0}}$ is contained in
  the \emph{part} $A_{q}:=A\cap (L^{q}\times L^{q}(\Sigma,\mu))$ of $A$ in $L^{q}(\Sigma,\mu)$.

\begin{theorem}
  \label{thm:Sobolev-for-porousmedia}
  Let $A+\omega I$ be $m$-accretive in $L^{1}(\Sigma,\mu)$ for some
  $\omega\ge 0$.  Suppose, there are $1\le q\le q_{0}\le \infty$,
  $(q<\infty)$, such that the trace $A_{1\cap q_{0}}$ of $A$ on
  $L^{1}\cap L^{q_{0}}(\Sigma,\mu)$ satisfies the range condition
  \begin{equation}
    \label{eq:122}
    L^{1}\cap L^{q_{0}}(\Sigma,\mu)\subseteq Rg(I+(A_{1\cap q_{0}}+\omega I)),
  \end{equation}
  and the Gagliardo-Nirenberg type
  inequality~\eqref{eq:242} for some $1\le r\le \infty$,
  $\varrho\ge 0$, $\sigma>0$ and $(u_{0},0)\in A_{1\cap q_{0}}$, and for
  every $\lambda>0$ satisfying $\lambda \omega<1$, the resolvent
  $J_{\lambda}$ of $A$ satisfies
  \begin{equation}
    \label{eq:123}
    \norm{J_{\lambda}u-u_{0}}_{\tilde{q}}\le (1-\lambda \omega)^{-1}\,\norm{u-u_{0}}_{\tilde{q}}
  \end{equation}
  for $\tilde{q}=r$, every $u \in Rg(I+\lambda A_{1\cap q_{0}})$,
  and for $\tilde{q}=q$ provided $\varrho>0$. Then the semigroup
  $\{T_{t}\}_{t\geq 0}\sim -A$ on $\overline{D(A)}^{\mbox{}_{L^{1}}}$
  satisfies inequality~\eqref{eq:20} for every $t>0$ and
  $u\in \overline{D(A)}^{\mbox{}_{L^{1}}}\cap L^{q_{0}}(\Sigma,\mu)$ with
  exponents $\alpha=\tfrac{1}{\sigma}$, $\beta=\gamma+1$ and
  $\gamma=\frac{q+\varrho}{\sigma}$.
\end{theorem}

\begin{remark}
  \label{rem:6}
  One easily verifies that a similar statement as given in
  Remark~\ref{rem:2} holds for accretive operators in
  $L^{1}(\Sigma,\mu)$. More precisely, for an $m$-accretive operator
  $A$ on $L^{1}(\Sigma,\mu)$ satisfying the hypotheses of
  Theorem~\ref{thm:Sobolev-for-porousmedia} with $\omega=0$ and a
  Lipschitz continuous mapping
  $F: L^{1}(\Sigma,\mu)\to L^{1}(\Sigma,\mu)$ with $F(0)=0$ and
  Lipschitz constant $L>0$, if the trace $A_{1\cap q_{0}}$ of $A$ on
  $L^{1}\cap L^{q_{0}}(\Sigma,\mu)$ satisfies~\eqref{eq:122} for
  $\omega=0$ and satisfies the Gagliardo-Nirenberg type
  inequality~\eqref{eq:242} for $(u_{0},0)$ and $\omega=0$,
  then $A_{1\cap q_{0}}+F$ satisfies the Gagliardo-Nirenberg type
  inequality~\eqref{eq:242} for $(u_{0},0)$ and $\omega=L$.
\end{remark}

From Theorem~\ref{thm:Sobolev-for-porousmedia},
we can immediately conclude the following result concerning
quasi $m$-accretive operators in $L^{1}$ with complete
resolvent. 

\begin{corollary}
  \label{cor:Gagliardo-for-accretive-op-in-L1-complete-res}
  Let $A+\omega I$ be $m$-accretive operator in $L^{1}(\Sigma,\mu)$
  with complete resolvent for some $\omega\ge 0$. Suppose, there are
  $1\le q\le q_{0}\le \infty$, $(q<\infty)$, such that the trace
  $A_{1\cap q_{0}}$ of $A$
  in $L^{1}\cap L^{q_{0}}(\Sigma,\mu)$ satisfies range
  condition~\eqref{eq:122} and Gagliardo-Nirenberg type
  inequality~\eqref{eq:242} for some $1\le r\le \infty$,
  $\varrho\ge 0$, $\sigma>0$ and $(0,0)\in A_{1\cap q_{0}}$. Then
  the semigroup $\{T_{t}\}_{t\geq 0}\sim -A$ on
  $\overline{D(A)}^{\mbox{}_{L^{1}}}$ satisfies
  \begin{equation}
    \label{eq:172}
    \norm{T_{t}u}_{r}\le 
    \left(\tfrac{C}{q}\right)^{1/\sigma}\,e^{\omega (\gamma+1) t}\,t^{-\alpha}
    \norm{u}_{q}^{\gamma}
  \end{equation}
  for every $t>0$ and
  $u\in \overline{D(A)}^{\mbox{}_{L^{1}}}\cap L^{q_{0}}(\Sigma,\mu)$
  with exponents $\alpha=\tfrac{1}{\sigma}$, $\beta=\gamma+1$ and
  $\gamma=\frac{q+\varrho}{\sigma}$.
\end{corollary}

Furthermore, by Theorem~\ref{thm:Sobolev-for-porousmedia}, we can
deduce the following result concerning $m$-accretive operators in $L^{1}$ with
$c$-complete resolvent. 

\begin{corollary}
  \label{cor:Gagliardo-for-accretive-op-c-complete}
  Let $A$ be an $m$-accretive operator in $L^{1}(\Sigma,\mu)$ with
  $c$-complete resolvent. Suppose, there are
  $1\le q\le q_{0}\le \infty$, $(q<\infty)$, such that the trace
  $A_{1\cap q_{0}}$ of $A$ on $L^{1}\cap L^{q_{0}}(\Sigma,\mu)$ satisfies
  the range condition~\eqref{eq:122} and the Gagliardo-Nirenberg type
  inequality~\eqref{eq:242} for some $1\le r\le \infty$,
  $\varrho\ge 0$, $\sigma>0$ and $c\in \R$ with $(c,0)\in A_{1\cap q_{0}}$. Then the
  semigroup $\{T_{t}\}_{t\geq 0}\sim -A$ on
  $\overline{D(A)}^{\mbox{}_{L^{1}}}$ satisfies
  \begin{displaymath}
    \norm{T_{t}u-c}_{r}\le 
    \left(\tfrac{C}{q}\right)^{1/\sigma}\,t^{-\alpha}
    \norm{u-c}_{q}^{\gamma}
  \end{displaymath}
   for every $t>0$ and
  $u\in \overline{D(A)}^{\mbox{}_{L^{1}}}\cap L^{q_{0}}(\Sigma,\mu)$ with exponents
  $\alpha=\tfrac{1}{\sigma}$ and $\gamma=\frac{q+\varrho}{\sigma}$.
\end{corollary}

\allowdisplaybreaks
\begin{proof}[Proof of Theorem~\ref{thm:Sobolev-for-porousmedia}]
  Let $u \in \overline{D(A)}^{\mbox{}_{L^{1}}}\cap L^{q_{0}}(\Sigma,\mu)$
  and for given $t>0$, let $N\ge 1$ be large enough such that
  $\tfrac{t\,\omega\,q}{N}<\tfrac{1}{2}$. Then, we set
  $t_{n}=n\,\frac{t}{N}$ for every $n=0,\dots,N$ and
  $\hat{u}_{0}=u$. By range condition~\eqref{eq:122}, for every $n=1,\dots,N$, there is
  iteratively a $\hat{u}_{n}\in D(A_{1\cap q_{0}})$ satisfying
  \begin{equation}
  \label{eq:108}
    \hat{u}_{n}+\tfrac{t}{N} A_{1\cap q_{0}} \hat{u}_{n}\ni \hat{u}_{n-1}.
  \end{equation}
  We set
  \begin{displaymath}
    \hat{U}_{N}(s)=\hat{u}_{0}\mathds{1}_{\{t_{0}=0\}}(s)
    +\sum_{n=1}^{N}\hat{u}_{n}\;\mathds{1}_{(t_{n-1},t_{n}]}(s)
  \end{displaymath}
  for every $s\in [0,t]$ and
  $\hat{v}_{n}=(\hat{u}_{n-1}-\hat{u}_{n})\tfrac{N}{t}$.  Then,
  inclusions~\eqref{eq:108} can be rewritten as
  $ \hat{v}_{n}\in A_{1\cap q_{0}} \hat{u}_{n}$ or as
  $J_{t/N}\hat{u}_{n-1}=\hat{u}_{n}$ for every $n=1,\dots,N$. Hence, since $A_{1\cap q_{0}}$ satisfies Gagli\-ardo-Nirenberg type
  inequalities~\eqref{eq:242} with $(u_{0},0)\in A_{1\cap q_{0}}$, we see
  that by using~\eqref{eq:57} and~\eqref{eq:56} that
  \begin{align*}
    &\norm{\hat{u}_{n}-u_{0}}_{r}^{\sigma}\\
    &\;\le C\,\Big(\,[\hat{u}_{n}-u_{0},\hat{v}_{n}]_{q}
      +\omega \norm{\hat{u}_{n}-u_{0}}_{q}^{q}\Big)\;
    \norm{\hat{u}_{n}-u_{0}}_{q}^{\varrho}\\
   &\;= C\,\tfrac{N}{t}\Big(\,[\hat{u}_{n}-u_{0},\hat{u}_{n-1}-\hat{u}_{n}]_{q}
      +\tfrac{\omega\,t}{N} \norm{\hat{u}_{n}-u_{0}}_{q}^{q}\Big)\;
    \norm{\hat{u}_{n}-u_{0}}_{q}^{\varrho}\\
    &\;= C\,\tfrac{N}{t}\Big(\,[\hat{u}_{n}-u_{0},(\hat{u}_{n-1}-u_{0})-(\hat{u}_{n}-u_{0})]_{q}
      +\tfrac{\omega\,t}{N} \norm{\hat{u}_{n}-u_{0}}_{q}^{q}\Big)\;
    \norm{\hat{u}_{n}-u_{0}}_{q}^{\varrho}\\
    &\;\le
      C\,\tfrac{N}{t}\Big(\,\tfrac{1}{q}\norm{\hat{u}_{n-1}-u_{0}}_{q}^{q}-
      (1-\tfrac{\omega\,q\,t}{N})\tfrac{1}{q}\norm{\hat{u}_{n}-u_{0}}_{q}^{q}\Big)\;
    \norm{\hat{u}_{n}-u_{0}}_{q}^{\varrho}
  \end{align*}
  for every $n=1,\dots,N$. By assumption, %  and by
  % Proposition~\ref{propo:quasi-accretive-operators-in-L1-complete-resolvent},
  the resolvent operator $J_{t/N}$ of $A$ satisfies
  inequality~\eqref{eq:123} for $\tilde{q}=q$ provided
  $\varrho>0$. Then,
  \begin{align*}
    \norm{\hat{u}_{n}-u_{0}}_{q} 
      & =\norm{J_{t/N}\hat{u}_{n-1}-u_{0}}_{q}\\
      & \le
        (1-\tfrac{t\omega}{N})^{-1}\,\norm{\hat{u}_{n-1}-u_{0}}_{q}\\
      & \quad\vdots\\
      & \le (1-\tfrac{t\omega}{N})^{-n}\norm{\hat{u}_{0}-u_{0}}_{q}\\
     & \le (1-\tfrac{t\omega}{N})^{-N}\norm{\hat{u}_{0}-u_{0}}_{q}.
  \end{align*}
  Applying this to the previous inequality, in order to estimate
  $\norm{\hat{u}_{n}-u_{0}}_{q}^{\varrho}$ and multiplying the
  resulting inequality by
  $\tfrac{t}{N}\,(1-\tfrac{\omega\,q\,t}{N})^{-1}$ yields
  \begin{align*}
   &\tfrac{t}{N} (1-\tfrac{\omega\,q\,t}{N})^{-1}\norm{\hat{u}_{n}-u_{0}}_{r}^{\sigma}\\ 
   &\qquad\le
      C\,\Big(\,(1-\tfrac{\omega\,q\,t}{N})^{-1}\,\tfrac{1}{q}\norm{\hat{u}_{n-1}-u_{0}}_{q}^{q}-
      \tfrac{1}{q}\norm{\hat{u}_{n}-u_{0}}_{q}^{q}\Big)\;\\
    &\hspace{6cm} \times (1-\tfrac{t\omega}{N})^{-N\varrho}
    	\norm{\hat{u}_{0}-u_{0}}_{q}^{\varrho}.
  \end{align*}
  Rearranging this inequality yields
  \begin{displaymath}
    \tfrac{1}{q}\norm{\hat{u}_{n}-u_{0}}_{q}^{q} \le
    (1-\tfrac{\omega\,q\,t}{N})^{-1}\,\tfrac{1}{q}\norm{\hat{u}_{n-1}-u_{0}}_{q}^{q}
    + b_{n}
  \end{displaymath}    
  with
  \begin{displaymath}
    b_{n}:= -\tfrac{t}{N}
    (1-\tfrac{\omega\,q\,t}{N})^{-1}\,\norm{\hat{u}_{n}-u_{0}}_{r}^{\sigma}\,C^{-1}
    \, (1-\tfrac{t\omega}{N})^{N\varrho}\,\norm{\hat{u}_{0}-u_{0}}_{q}^{-\varrho}.
  \end{displaymath}
  for every $n=1,\dots,N$. By auxiliary inequality~\eqref{eq:110},
  \begin{displaymath}
    \tfrac{1}{q}\norm{\hat{u}_{n}-u_{0}}_{q}^{q} \le
    (1-\tfrac{\omega\,q\,t}{N})^{-N}\,
    \tfrac{1}{q}\norm{\hat{u}_{0}-u_{0}}_{q}^{q}+\sum_{n=1}^{N}
    (1-\tfrac{\omega\,q\,t}{N})^{-(N-(n+1))} b_{n}.
  \end{displaymath}
  Rearranging this inequality and using that $(1-\tfrac{\omega\,q\,t}{N})^{n}\le
  (1-\tfrac{\omega\,q\,t}{N})^{N}$ gives
  \begin{align*}
    &
      (1-\tfrac{\omega\,q\,t}{N})^{N}\tfrac{1}{q}\norm{\hat{u}_{n}-u_{0}}_{q}^{q}\\
    &\hspace{2cm} + C^{-1}\,
      (1-\tfrac{t\omega}{N})^{N\varrho}\,\norm{\hat{u}_{0}-u_{0}}_{q}^{-\varrho}
      \,(1-\tfrac{\omega\,q\,t}{N})^{N}\, 
      \sum_{n=1}^{N}\tfrac{t}{N} \norm{\hat{u}_{n}-u_{0}}_{r}^{\sigma}\\
   & \qquad \le \tfrac{1}{q}\norm{\hat{u}_{0}-u_{0}}_{q}^{q}
  \end{align*}
  and so,
  \begin{equation}
    \label{eq:117}
    \begin{split}
      &
      (1-\tfrac{\omega\,q\,t}{N})^{N}\tfrac{1}{q}\norm{\hat{U}_{N}(t)-u_{0}}_{q}^{q}\\
      &\hspace{1cm} + C^{-1}\,
      (1-\tfrac{t\omega}{N})^{N\varrho}\,\norm{u-u_{0}}_{q}^{-\varrho}
      \,(1-\tfrac{\omega\,q\,t}{N})^{N}\,
      \int_{0}^{t}\norm{\hat{U}_{N}(s)-u_{0}}_{r}^{\sigma}\,\ds\\
      & \quad \le \tfrac{1}{q}\norm{u-u_{0}}_{q}^{q}.
    \end{split}
  \end{equation}
  Recall that $A_{1\cap q_{0}}\subseteq A$ and, by assumption, $A+\omega I$ is $m$-accretive
  in $L^{1}(\Sigma,\mu)$. Thus, the Crandall-Liggett theorem yields
  \begin{displaymath}
  \lim_{N\to\infty}\hat{U}_{N}=T_{t}u\qquad\text{ in
  $L^{1}(\Sigma,\mu)$ uniformly on $[0,t]$.}
\end{displaymath}
Since the $L^{r}$- and $L^{q}$-norm on $L^{1}(\Sigma,\mu)$ are lower semicontinuous in
  $L^{1}(\Sigma,\mu)$, sending $N\to \infty$ in~\eqref{eq:117} and
  applying Fatou's Lemma yields
  \begin{align*}
    &e^{-\omega\,q\,t}\,\tfrac{1}{q}\norm{T_{t}u-u_{0}}_{q}^{q}
      + C^{-1}\,e^{-\omega\,\varrho\,t}\,
      \norm{u-u_{0}}_{q}^{-\varrho}\,e^{-\omega\,q\,t}\, 
      \int_{0}^{t}\norm{T_{s}u-u_{0}}_{r}^{\sigma}\,\ds\\
   & \hspace{2cm} \le \tfrac{1}{q}\norm{u-u_{0}}_{q}^{q}
  \end{align*}
  and so
  \begin{equation}
    \label{eq:125}
    C^{-1}\,e^{-\omega\,\varrho\,t}\,
    \norm{u-u_{0}}_{q}^{-\varrho}\,e^{-\omega\,q\,t}\, 
      \int_{0}^{t}\norm{T_{s}u-u_{0}}_{r}^{\sigma}\,\ds\\
      \le \tfrac{1}{q}\norm{u-u_{0}}_{q}^{q}.
  \end{equation}
  Now, fix $\tilde{u}\in \overline{D(A)}^{\mbox{}_{L^{1}}}\cap 
  L^{r}(\Sigma,\mu)$. Then, applying~\eqref{eq:123} 
  iteratively for $\tilde{q}=r$, we see that
  \begin{equation}
    \label{eq:126}
    \begin{split}
      \norm{J_{t/n}^{n}\tilde{u}-u_{0}}_{r}& = \norm{J_{t/n}^{n-1}(J_{t/n}\tilde{u})-u_{0}}_{r}\\
          &  \le (1-\tfrac{t\omega}{n})^{-1}\, \norm{J_{t/n}^{n-1}\tilde{u}-u_{0}}_{r}\\
          &  \;\;\vdots\\
          & \le (1-\tfrac{t\omega}{n})^{-n}\norm{\tilde{u}-u_{0}}_{r}
        \end{split}
      \end{equation}
  for every $t>0$ and integer $n\ge 1$ such that $\tfrac{t}{n}\omega
  <1$. By Euler's formula~\eqref{eq:36} and
  since $\tilde{u}\in \overline{D(A)}^{\mbox{}_{L^{1}}}$,
  \begin{displaymath}
     \lim_{n\to\infty}J_{t/n}^{n}\tilde{u}=T_{t}\tilde{u}\qquad\text{in $L^{1}(\Sigma,\mu)$}
  \end{displaymath}
  for every $t>0$. Since the $L^{r}$-norm is lower semicontinuous on
  $L^{1}(\Sigma,\mu)$, sending $n\to\infty$ in~\eqref{eq:126} yields
  \begin{displaymath}
    \norm{T_{t}\tilde{u}-u_{0}}_{r}\le e^{\omega t}\norm{\tilde{u}-u_{0}}_{r}
  \end{displaymath}
  for every $t>0$ hence, by using the semigroup property of
  $\{T_{t}\}_{t\ge 0}$, it follows that
 \begin{displaymath}
    \norm{T_{t}\tilde{u}-u_{0}}_{r}\le e^{\omega (t-s)}\norm{T_{s}\tilde{u}-u_{0}}_{r}
  \end{displaymath}
  for every $t\ge s> 0$ and $\tilde{u}\in
  \overline{D(A)}^{\mbox{}_{L^{1}}}\cap L^{r}(\Sigma,\mu)$. Applying
  this inequality to the integrand
  in~\eqref{eq:125}, we see that~\eqref{eq:20} holds.
\end{proof}

%-----------------------------------------------------------------------------
%-----------------------------------------------------------------------------
%-----------------------------------------------------------------------------
%-----------------------------------------------------------------------------
%-----------------------------------------------------------------------------
%-----------------------------------------------------------------------------
%-----------------------------------------------------------------------------
%-----------------------------------------------------------------------------
%-----------------------------------------------------------------------------
%-----------------------------------------------------------------------------
%-----------------------------------------------------------------------------
%-----------------------------------------------------------------------------
%-----------------------------------------------------------------------------
%-----------------------------------------------------------------------------
%-----------------------------------------------------------------------------
%-----------------------------------------------------------------------------
%-----------------------------------------------------------------------------
%-----------------------------------------------------------------------------
%-----------------------------------------------------------------------------
\section{Nonlinear extrapolation}\label{extra}

The aim of this section is to provide simple and sufficient conditions such that
an $L^{q}$-$L^{r}$-regularisation estimate of the type~\eqref{eq:18}
or~\eqref{eq:20} for some $1<q<r\le \infty$ satisfied by a nonlinear semigroup
$\{T_{t}\}_{t\ge 0}$  can be extrapolated to an
$L^{s}$-$L^{r}$-regularisation estimate for every $1\le
s<q$ (this we call below  \emph{extrapolation
  towards $L^{1}$}) and such that an $L^{q}$-$L^{r}$-regularity
estimate of the type~\eqref{eq:18}
or~\eqref{eq:20} for some $1<q, r<\infty$ can be extrapolated to an $L^{\tilde{q}}$-$L^{\infty}$-regularisation
estimate for some $1\le \tilde{q}<\infty$ (\emph{extrapolation
  towards $L^{\infty}$}). We note that in order to extrapolate towards
$L^{\infty}$, the relation $r>q$ is not important, but one rather needs that
the relation $\gamma r>q$ holds for the exponent $\gamma>0$ in the
estimates~\eqref{eq:18} and~\eqref{eq:20}
(cf. Theorems~\ref{thm:main-1} and~\ref{thm:GN-implies-reg-bis} or
Theorems~\ref{thm:extrapolation-to-infty}
and~\ref{thm:extrapolation-to-infty-bis}). Further, the
iteration method (Lemma~\ref{lem:iteration} and
Lemma~\ref{lem:iteration-bis}) used to establish
$L^{\tilde{q}}$-$L^{\infty}$-regularisation  works if $1\le
\tilde{q}<\infty$ is chosen sufficiently large. Thus, if one starts from an $L^{q}$-$L^{r}$-regularity
estimate for some $1<q, r<\infty$ then, first, one extrapolates
towards $L^{\infty}$, and then one extrapolates towards $L^{1}$.\bigskip

The extrapolation towards $L^\infty$ being more involved, we shall begin  in Section \ref{sec:extra-pol-twoards-one}
 by extrapolating towards $L^{1}$, or, more
precisely, towards $L^{s}$ for any $1\le s<q$. Section~\ref{sec:nonlinear-interpolation} is concerned with a new
nonlinear interpolation result  which provides the fundamental auxiliary
tool to establish our \emph{extrapolation result towards $L^{\infty}$} presented
in Section~\ref{sec:extrapolation-towards-infinity}.

%-----------------------------------------------------------------------------
%-----------------------------------------------------------------------------
%-----------------------------------------------------------------------------
%-----------------------------------------------------------------------------
%-----------------------------------------------------------------------------
%-----------------------------------------------------------------------------
%-----------------------------------------------------------------------------
%-----------------------------------------------------------------------------
%-----------------------------------------------------------------------------
%-----------------------------------------------------------------------------
%-----------------------------------------------------------------------------
%-----------------------------------------------------------------------------
%-----------------------------------------------------------------------------
%-----------------------------------------------------------------------------
%-----------------------------------------------------------------------------
%-----------------------------------------------------------------------------
%-----------------------------------------------------------------------------
%-----------------------------------------------------------------------------
%-----------------------------------------------------------------------------

\subsection{Extrapolation towards $L^1$}
\label{sec:extra-pol-twoards-one}
This subsection is dedicated to giving a nonlinear version of~\cite[Lemme 1]{MR1077272}
(see also \cite[Section I]{MR1164643}). The first extrapolation result of
this subsection is adapted to semigroups generated by completely
accretive operators (see Section~\ref{sec:comp}) satisfying the
$L^{q}$-$L^{r}$-regularising effect~\eqref{eq:18} for
\emph{differences} and $1<q<r\le \infty$.

\begin{theorem}\label{thm:extrapol-L1-differences}
  Let $1\le s<q<r\le \infty$ and $\{T_{t}\}_{t\geq 0}$ be a semigroup
  acting on some subset $D$ of $L^{q}(\Sigma,\mu)$ with exponential
  growth~\eqref{eq:62} for $\tilde{q}=s$ and some $\omega\ge 0$. Suppose
  there exist $\alpha>0$, $\beta$, $\gamma>0$ and $C>0$ such that
  \begin{equation}
    \tag{\ref{eq:18}}
    \norm{T_{t}u-T_{t}\hat{u}}_{r}\le \,C\; t^{-\alpha}\; e^{\omega \beta t}\; 
    \norm{u-\hat{u}}_{q}^{\gamma}
  \end{equation}
  for every $t>0$ and $u$, $\hat{u}\in D$. For
  $\theta_{s}=\tfrac{(r-q)s}{q(r-s)}>0$ if $r<\infty$ and
  $\theta_{s}=\tfrac{s}{q}$ if $r=\infty$, assume that
  \begin{equation}
    \label{eq:229}
    \gamma(1-\theta_{s})<1.
  \end{equation}
  Then %for $\theta=1-\gamma(1-\theta_{s})>0$,
  one has
  \begin{equation}
    \label{eq:8}
    \norm{T_{t}u-T_{t}\hat{u}}_{r}\le\,
    (C \,2^{\frac{\alpha}{1-\gamma(1-\theta_{s})}})^{\frac{1}{1-\gamma(1-\theta_{s})}}
   \; t^{-\alpha_{s}}\; e^{\omega \beta_{s} t} \norm{u-\hat{u}}_{s}^{\gamma_{s}}
  \end{equation}
  for every $t>0$ and $u$,
  $\hat{u}\in D\cap L^{s}(\Sigma,\mu)$ with exponents
  \begin{equation}
    \label{eq:19}
    \alpha_{s}=\frac{\alpha}{1-\gamma(1-\theta_{s})},\quad
    \beta_{s}= \frac{(\beta/2)+\gamma\, \theta_{s}}{1-\gamma(1-\theta_{s})},\quad
    \gamma_{s}=\gamma\frac{\theta_{s}}{1-\gamma(1-\theta_{s})}.
 \end{equation}
\end{theorem}

\begin{remark}
  The statement of Theorem~\ref{thm:extrapol-L1-differences}
  remains unchanged if one replaces the constant $e^{\omega t}$ in
  condition~\eqref{eq:62} for $\tilde{q}=s$ by $M\,e^{\omega t}$
  for some constant $M>0$.  Then the constant $C$ in~\eqref{eq:8} has to be changed accordingly.
\end{remark}

% \begin{remark}
%   Under the additional assumption that there is an element
%   \begin{center}
%     $u_{0}\in D\cap L^{s}(\Sigma,\mu)$ such that $T_{t}u_{0}\in
%     L^{r}(\Sigma\,u)$ for some (all) $t\ge0$,
%   \end{center}
%   a standard density argument and inequality~\eqref{eq:8} yields that
%   for some (all) $t>0$, the mapping $T_{t}$ from $D\cap
%   L^{s}(\Sigma,\mu)$ to $D\cap L^{r}(\Sigma,\mu)$
%   can be uniquely extended as a mapping from
%   $\overline{D\cap L^{s}(\Sigma,\mu)}^{\mbox{}_{L^{s}}}$ to
%   $\overline{D\cap L^{r}(\Sigma,\mu)}^{\mbox{}_{r}}$ and~\eqref{eq:8} holds for all $u$,
%   $\hat{u}\in \overline{D\cap L^{s}(\Sigma,\mu)}^{\mbox{}_{L^{s}}}$.
% \end{remark}

\begin{proof}[Proof of Theorem~\ref{thm:extrapol-L1-differences}]
  We outline the proof only for $r<\infty$ since the case $r=\infty$ is
  treated similarly. Then, set $\theta_{s}=\tfrac{(r-q)s}{q(r-s)}$ and
  assume that~\eqref{eq:229} holds. For
  $\theta:=1-\gamma(1-\theta_{s})$,  $u$, $\hat{u}\in L^{s}(\Sigma,\mu)\cap
  D$ satisfying $u\neq \hat{u}$ and $T>0$, set
  \begin{displaymath}
  C_{u,\hat{u},T}:=\sup_{t\in [0,T]}\frac{t^{\alpha/\theta}
  \norm{T_{t}u-T_{t}\hat{u}}_{r}}{e^{\omega \gamma_{s} t}\,\norm{u-\hat{u}}_{s}^{\gamma_{s}}}.
  \end{displaymath}
  By~\eqref{eq:18} and since $\theta_{s}$ satisfies $\frac{1}{q}=\frac{(1-\theta_{s})}{r}+\frac{\theta_{s}}{s}$,
  H\"older's inequality imply
  \begin{align*}
    \norm{T_{t}u-T_{t}\hat{u}}_{r} & \le C\,e^{\omega \beta \frac{t}{2}}\,\big(\tfrac{t}{2}\big)^{-\alpha}
    \,\norm{T_{t/2}u-T_{t/2}\hat{u}}_{q}^\gamma\\
    & \le C\,e^{\omega \beta \frac{t}{2}}\,\big(\tfrac{t}{2}\big)^{-\alpha}\,
       \norm{T_{t/2}u-T_{t/2}\hat{u}}_{r}^{\gamma(1-\theta_{s})}\,
       \norm{T_{t/2}u-T_{t/2}\hat{u}}_{s}^{\gamma\theta_{s}}
  \end{align*}
  Since $\{T_{t}\}_{t\geq 0}$ satisfies~\eqref{eq:62} for
  $\tilde{q}=s$ and some $\omega\ge 0$,
  \begin{displaymath}
    \norm{T_{t}u-T_{t}\hat{u}}_{r} \le
    C\,e^{\omega (\beta+\gamma\theta_{s}) \frac{t}{2}}\,
    \big(\tfrac{t}{2}\big)^{-\alpha}\,
    \norm{T_{t/2}u-T_{t/2}\hat{u}}_{r}^{\gamma(1-\theta_{s})}\,
    \norm{u-\hat{u}}_{s}^{\gamma\theta_{s}}
  \end{displaymath}
  and so by definition of $C_{u,\hat{u},T}$,
  \begin{displaymath}
    \norm{T_{t}u-T_{t}\hat{u}}_{r} \le C\;
    e^{\omega (\beta+\gamma\theta_{s}+\gamma_{s}\gamma(1-\theta_{s})) \frac{t}{2}}\;
    \big(\tfrac{t}{2}\big)^{-\alpha-\alpha\frac{\gamma(1-\theta_{s})}{\theta}}\;
    C_{u,\hat{u},T}^{\gamma(1-\theta_{s})}\,
    \norm{u-\hat{u}}_{s}^{\gamma(\theta_{s}+\gamma_{0}(1-\theta_{s}))}
    \end{displaymath}
  for every $t\in [0,2T]$. Since
  $\gamma\theta_{s}+\gamma_{s}\gamma(1-\theta_{s})=\gamma_{s}$ and
  $1+\frac{\gamma(1-\theta_{s})}{\theta}=\frac{1}{\theta}$, the
  previous estimate becomes
  \begin{displaymath}
    \norm{T_{t}u-T_{t}\hat{u}}_{r} \le 
    C\; e^{\omega (\beta+\gamma_{s}) \frac{t}{2}}\; 2^{\frac{\alpha}{\theta}}\;
    t^{-\frac{\alpha}{\theta}}\; C_{u,\hat{u},T}^{\gamma(1-\theta_{s})}\,
    \norm{u-\hat{u}}_{s}^{\gamma_{s}}
  \end{displaymath}
  and so
  \begin{displaymath}
    \norm{T_{t}u-T_{t}\hat{u}}_{r} \le 
    C\,e^{\omega \beta \frac{t}{2}}\,2^{\frac{\alpha}{\theta}}
    \, C_{u,\hat{u},T}^{\gamma(1-\theta_{s})}\,e^{\omega \gamma_{s} t}\,t^{-\frac{\alpha}{\theta}}\,
    \norm{u-\hat{u}}_{s}^{\gamma_{s}}
  \end{displaymath}
  for every $t\in [0,T]$. Dividing this inequality by $\,e^{\omega \gamma_{s} t}\,t^{-\frac{\alpha}{\theta}}\,
    \norm{u-\hat{u}}_{s}^{\gamma_{s}}$ and taking the
  supremum over $[0,T]$ on the left hand-side of the resulting inequality yields
  \begin{displaymath}
  C_{u,\hat{u},T}\le C\,e^{\omega \beta \frac{T}{2}}\;
  2^{\frac{\alpha}{\theta}}\; C_{u,\hat{u},T}^{\gamma(1-\theta_{s})}.
  \end{displaymath}
  Since $\gamma(1-\theta_{s})<1$, this implies that $C_{u,\hat{u},T}$ is
  uniformly bounded in $u$, $\hat{u}$ by constant
  $(C \,2^{\frac{\alpha}{\theta}})^{\frac{1}{\theta}}\,e^{\omega\frac{\beta}{\theta} \frac{T}{2}}>0$
  with $\theta=1-\gamma(1-\theta_{s})$. In other words,
  \begin{displaymath}
  \norm{T_{t}u-T_{t}\hat{u}}_{r}\le
  (C
  \; 2^{\frac{\alpha}{\theta}})^{\frac{1}{\theta}}\;
  e^{\omega \frac{\beta}{\theta} \frac{T}{2}} \; e^{\omega \gamma_{s} t}
  \; t^{-\frac{\alpha}{\theta}}\; \norm{u-\hat{u}}_{s}^{\gamma_{s}}
  \end{displaymath}
  for every $t\in [0,T]$ and $u$,
  $\hat{u}\in D\cap L^{s}(\Sigma,\mu)$, where $T>0$ was
  arbitrary. Taking $t=T$ in this inequality, we can conclude that
  inequality~\eqref{eq:8} holds for every $t>0$ and $u$,
  $\hat{u}\in D\cap L^{s}(\Sigma,\mu)$. 
\end{proof}

Our second extrapolation result of this subsection is adapted to
semigroups enjoying the $L^{q}$-$L^{r}$-regularising
effect~\eqref{eq:20} for $1<q<r\le \infty$ and some
$u_{0}\in L^{r}\cap L^{s}(\Sigma,\mu)$ generated by
either quasi $m$-completely accretive operators on $L^{q}(\Sigma,\mu)$
(Section~\ref{sec:comp}) or quasi $m$-accretive
operators in $L^{1}$ with ($c$-)complete resolvent (Section~\ref{subsec:L1}).

\begin{theorem}\label{thm:extrapol-L1-bis}
  Let $1\le s<q<r\le \infty$ and $\{T_{t}\}_{t\geq 0}$ be a
  semigroup acting on a subset $D$ of $L^{q}(\Sigma,\mu)$ and
  satisfies the exponential growth property~\eqref{eq:216} for $\tilde{q}=s$,
  some $\omega\ge 0$ and
  $u_{0}\in L^{s}\cap L^{r}(\Sigma,\mu)$, $\{T_{t}\}_{t\geq 0}$.
  Suppose there exist $C>0$ and exponents
  $\alpha>0$, $\beta$, $\gamma>0$ such that
  \begin{equation}
    \tag{\ref{eq:20}}
    \norm{T_{t}u-u_{0}}_{r}\le \,C\; t^{-\alpha}\; e^{\omega \beta t}\;
    \norm{u-u_{0}}_{q}^{\gamma}
  \end{equation}
  for every $t>0$ and $u\in D$. For
  $\theta_{s}=\tfrac{(r-q)s}{q(r-s)}>0$ if $r<\infty$ and
  $\theta_{s}=\tfrac{s}{q}$ if $r=\infty$, assume that
  $\gamma(1-\theta_{s})<1$. Then
  % for $\theta=1-\gamma(1-\theta_{s})>0$, 
  one has
  \begin{equation}
    \label{eq:239}
    \norm{T_{t}u-u_{0}}_{r}\le\,
    (C \,2^{\frac{\alpha}{1-\gamma(1-\theta_{s})}})^{\frac{1}{1-\gamma(1-\theta_{s})}}\;
    t^{-\alpha_{s}}\;
    \,e^{\omega \beta_{s} t}\; \norm{u-u_{0}}_{s}^{\gamma_{s}}
  \end{equation}
  for every $t>0$ and $u\in D\cap L^{q_{0}}(\Sigma,\mu)$ with exponents~\eqref{eq:19}.
\end{theorem}

\begin{proof}[Proof of Theorem~\ref{thm:extrapol-L1-bis}]
  By using the same arguments as outlined in the proof of
  Theorem~\ref{thm:extrapol-L1-differences}, where one replaces
  $\hat{u}$ and $T_{t}\hat{u}$ by $u_{0}$ and condition~\eqref{eq:62}
  by \eqref{eq:216}, one sees
  that the statement of Theorem~\ref{thm:extrapol-L1-bis} holds.
\end{proof}

%-----------------------------------------------------------------------------
%-----------------------------------------------------------------------------
%-----------------------------------------------------------------------------
%-----------------------------------------------------------------------------
%-----------------------------------------------------------------------------
%-----------------------------------------------------------------------------
%-----------------------------------------------------------------------------
%-----------------------------------------------------------------------------
%-----------------------------------------------------------------------------
%-----------------------------------------------------------------------------
%-----------------------------------------------------------------------------
%-----------------------------------------------------------------------------
%-----------------------------------------------------------------------------
%-----------------------------------------------------------------------------
%-----------------------------------------------------------------------------
%-----------------------------------------------------------------------------
%-----------------------------------------------------------------------------
%-----------------------------------------------------------------------------
%-----------------------------------------------------------------------------

We continue this section by establishing a new nonlinear interpolation theorem of independent interest.

%-----------------------------------------------------------
%
%
%                      A nonlinear Interpolation theorem
%
%
%
%-----------------------------------------------------------

\subsection{A nonlinear interpolation theorem}
\label{sec:nonlinear-interpolation}

In this subsection, we state our nonlinear interpolation theorem,
which generalises both Peetre's (\cite[Theorem~3.1]{MR0482280}) and
Tartar's (cf.~\cite[Th\'eor\`eme~4]{MR0310619}) nonlinear
interpolation results. Our nonlinear interpolation theorem complements
the existing literature in three ways, namely, by introducing
additional parameters $p_{0}$, $r_{0}$, $r_{1}$, by treating the
borderline cases $p_{0}=\infty$, $p_{1}<\infty$ and $p_{0}<\infty$,
$p_{1}=\infty$, and by giving exact constants.

We begin by recalling some basic definitions, notations and results
from the classical interpolation theory (cf., for
instance,~\cite{MR0264390} or \cite[Chapter~3]{MR0230022}). Let
$X_{0}$ and $X_{1}$ be two real or complex Banach spaces such that
both are continuously embedded into a Hausdorff topological
vector space $\mathcal{X}$. A pair $\{X_{0},X_{1}\}$ of Banach spaces
$X_{0}$ and $X_{1}$ satisfying these conditions is called an
\emph{interpolation couple}. We equip the \emph{intersection space}
$X_{0}\cap X_{1}$ and the \emph{sum space}
\begin{displaymath}
  X_{0}+X_{1}:=\Big\{x\,\Big\vert\;\text{there are }x_{0}\in
  X_{0},\;x_{1}\in X_{1}\text{ s.t. }x=x_{0}+x_{1} \Big\}
\end{displaymath}
respectively with the norm
$\norm{x}_{X_{0}\cap X_{1}}:=\max\{\norm{x}_{X_{0}},\norm{x}_{X_{1}}\}$
and
\begin{displaymath}
  \norm{x}_{X_{0}+x_{1}}:=\inf\Big\{\norm{x_{0}}_{X_{0}}+\norm{x_{1}}_{X_{1}}\;\Big\vert\;
  x=x_{0}+x_{1},\;x_{0}\in X_{0},\,x_{1}\in X_{1}\Big\}.
\end{displaymath}
Then $X_{0}\cap X_{1}$ and $X_{0}+X_{1}$ are Banach spaces and
\begin{equation}
  \label{eq:49}
  X_{0}\cap X_{1}\hookrightarrow Z\hookrightarrow X_{0}+X_{1}
\end{equation}
for $Z=X_{0}$ and $Z=X_{1}$ each with linear continuous embeddings
(cf.~\cite[Proposition~3.2.1]{MR0230022}).
A Banach space $Z$ satisfying~\eqref{eq:49} is called an \emph{intermediate
  space} (of $X_{0}$ and $X_{1}$).

For any Banach space $X$ equipped with norm $\norm{\cdot}_{X}$ and for
every $1\le
q\le \infty$, we denote by $L^{q}_{\ast}(X)$ the Banach space of all
(classes of) strongly $\textrm{d}t/t$-measurable functions $f :
(0,\infty)\to X$ having finite norm
\begin{displaymath}
  \norm{f}_{L^{q}_{\ast}(X)}:=
  \begin{cases}
    \Bigg\{\displaystyle\int_{0}^{\infty}\norm{f(t)}_{X}^{q}\tfrac{\textrm{d}t}{t}\Bigg\}^{1/q}
    & \text{if $1\le q<\infty$,}\\[7pt]
    \displaystyle\esssup_{t\in (0,\infty)} \norm{f(t)}_{X} & \text{if $q=\infty$.}
  \end{cases}
\end{displaymath}

We shall make use of the so-called \emph{mean-method}, which was introduced by
J.-L. Lions and Peetre (\cite{MR0165343,MR0133693}) and further
elaborated, for instance, in~\cite{MR0169043,zbMATH03276049}. 

We begin by introducing the \emph{mean spaces} (\emph{espaces de
  moyennes}).  Let $(X_{0},X_{1})$ be an interpolation couple. Then
for every $0<\theta<1$ and $1\le p_{0},p_{1}\le \infty$, the \emph{mean
  space} $(X_{0},X_{1})_{\theta,p_{0},p_{1}}$ is defined by the space of all
elements $u\in X_{0}+X_{1}$ with the property
\begin{equation}
  \label{eq:50}
  \begin{cases}
    \;\text{for $i=0,1$, there is a measurable function
      $v_{i} : (0,\infty)\to X_{i}$}&\\
\text{satisfying $u=v_{0}(t)+v_{1}(t)$ in $X_{0}+X_{1}$ for a.e. $t\in
  (0,\infty)$,}&\\
t^{-\theta}v_{0}\in L^{p_{0}}_{\ast}(X_{0})\text{ and }
\;t^{1-\theta}v_{1}\in L^{p_{1}}_{\ast}(X_{1}).&
  \end{cases}
\end{equation}
We equip the mean space $(X_{0},X_{1})_{\theta,p_{0},p_{1}}$ with the
norm 
\begin{displaymath}
  \norm{u}_{\theta,p_{0},p_{1}}:=\inf_{a=v_{0}(t)+v_{1}(t)}
  \max\Big\{\norm{t^{-\theta}v_{0}}_{L^{p_{0}}_{\ast}(X_{0})},
  \norm{t^{1-\theta}v_{1}}_{L^{p_{1}}_{\ast}(X_{1})}\Big\},
\end{displaymath}
where the infimum is taken of all representation pairs $(v_{0},v_{1})$
satisfying~\eqref{eq:50}. Then, it is not difficult to see that each
mean space $(X_{0},X_{1})_{\theta,p_{0},p_{1}}$ is an intermediate
space (cf.~\cite[p. 9]{MR0165343}). Moreover, the spaces
$(X_{0},X_{1})_{\theta,p_{0},p_{1}}$ admits the so-called
\emph{interpolation property} (cf.~\cite[p. 63]{MR3024598}), that is,
for every linear mapping $T : X_{0}+X_{1}\to X_{0}+X_{1}$ such that
its restriction to $X_{i}$ yields a linear and bounded operator from
$X_{i}$ into itself, where $i=0, 1$, one has that the restriction of
$T$ to $(X_{0},X_{1})_{\theta,p_{0},p_{1}}X_{i}$ yields a linear and
bounded operator from $(X_{0},X_{1})_{\theta,p_{0},p_{1}}$ into itself
(\cite[Th\'eor\`eme~(3.1)]{MR0165343}). In particular, one has
\begin{equation}\label{eq:51}
  \norm{u}_{\theta,p_{0},p_{1}}=\inf_{a=v_{0}(t)+v_{1}(t)}
  \norm{t^{-\theta}v_{0}}_{L^{p_{0}}_{\ast}(X_{0})}^{1-\theta}\,
  \norm{t^{1-\theta}v_{1}}_{L^{p_{1}}_{\ast}(X_{1})}^{\theta},
\end{equation}
for every $u\in (X_{0},X_{1})_{\theta,p_{0},p_{1}}$, where the infimum
is taken of all representation pairs $(v_{0},v_{1})$
satisfying~\eqref{eq:50} (cf.~\cite[Lemme~(3.1)]{MR0165343}). In addition, the
following continuous embedding is valid.

\begin{lemma}[{\cite[Th\'eor\`eme~(5.3)]{MR0165343}}]
 \label{lem:embedding-of-mean-spaces}
 Let $0<\theta<1$ and $1\le p_{0}$, $p_{1}$, $s_{0}$, $s_{1}\le
 \infty$. Then for $s_{0}\le p_{0}$ and $s_{1}\le p_{1}$, one has
 \begin{displaymath}
   \norm{u}_{\theta,p_{0},p_{1}}\le
    C_{\theta,r_{0},r_{1}}\,\norm{u}_{\theta,s_{0},s_{1}}
 \end{displaymath}
for all $u\in (X_{0},X_{1})_{\theta,s_{0},s_{1}}$, where the constant
 \begin{equation}\label{eq:52}
   \begin{split}
     C_{\theta,r_{0},r_{1}}:= & 
     \begin{cases}
     	1 & \text{if $s_{0}=p_{0}$ and $s_{1}=p_{1}$,}\\[3pt]
        \displaystyle\inf_{\varphi\in
       D_{+}}\norm{t^{-\theta}\varphi}_{L^{r_{0}}_{\ast}(\R)}^{1-\theta}\,
     & \text{if $s_{1}=p_{1}$}\\[3pt]
        \displaystyle\inf_{\varphi\in
       D_{+}}\norm{t^{1-\theta}\varphi}_{L^{r_{1}}_{\ast}(\R)}^{\theta} 
     & \text{if $s_{0}=p_{0}$}\\[3pt]
	\displaystyle\inf_{\varphi\in
       D_{+}}\norm{t^{-\theta}\varphi}_{L^{r_{0}}_{\ast}(\R)}^{1-\theta}\,
     \norm{t^{1-\theta}\varphi}_{L^{r_{1}}_{\ast}(\R)}^{\theta} & 
     	\text{if otherwise}
    \end{cases}\\
     \quad\text{with } \tfrac{1}{r_{0}}&=1-\left[\tfrac{1}{s_{0}}-\tfrac{1}{p_{0}}\right]\quad
     \text{and}\quad\tfrac{1}{r_{1}}=1-\left[\tfrac{1}{s_{1}}-\tfrac{1}{p_{1}}\right],
   \end{split}
  \end{equation}
  and $D_{+}$ denotes the set of all test functions
  $\varphi\in C^{\infty}_{c}((0,\infty))$ satisfying $\varphi\ge0$ and
  $\int_{0}^{\infty}\varphi(\tfrac{1}{t})\tfrac{\dt}{t}=1$.
\end{lemma}

Due to the result~\cite[Th\'eor\`eme~3.1]{MR0169043} by
Peetre, for every $0<\theta<1$ and $1\le p_{0},p_{1}, p\le \infty$
satisfying
$\tfrac{1}{p}=\tfrac{1-\theta}{p_{0}}+\tfrac{\theta}{p_{1}}$, the mean
space $(X_{0},X_{1})_{\theta,p_{0},p_{1}}$ coincides with the
(classical) \emph{real interpolation space}
$(X_{0},X_{1})_{\theta,p}$ with equivalent norms. For the definition of
the interpolation space $(X_{0},X_{1})_{\theta,p}$ we refer, for instance,
to~\cite[Definition~3.2.4]{MR0230022}. Combining this together with
the density result \cite[Theorem~1.6.2]{MR1328645}, we can state the
following extended version of the density
result~\cite[Th\'eor\`eme~2.1]{MR0165343}.

\begin{lemma}\label{lem:density}
 Let $(X_{0},X_{1})$ be an interpolation couple and suppose that
 one of the following cases holds:
 \begin{enumerate}[(i)]
   \item $1\le p_{0}, p_{1}<\infty$
   \item $1\le p_{0}<\infty$ and
   $p_{1}=\infty$
   \item $1\le p_{1}<\infty$ and $p_{0}=\infty$.
 \end{enumerate}
Then, for every $0<\theta <1$, the intersection space $X_{1}\cap
X_{2}$ is dense in $(X_{0},X_{1})_{\theta,p_{0},p_{1}}$. 
\end{lemma}

Now, we are in a position to state our first nonlinear interpolation theorem.

\begin{theorem}\label{thm:nonlinear-interpol}
  Let $(X_{0},X_{1})$ and $(Y_{0},Y_{1})$ be two interpolation
  couples and $T$ be a mapping from $X_{0}+X_{1}$ into $Y_{0}+Y_{1}$ with
  domain containing $X_{0}\cap X_{1}$. Suppose there are
  exponents $0<\alpha_{0}$, $\alpha_{1}<\infty$ and constants $M_{0}$,
  $M_{1}\ge 0$ such that
  \begin{equation}
    \label{eq:30}
    \norm{Tu-T\hat{u}}_{Y_{0}}\le
    M_{0}\,\norm{u-\hat{u}}_{X_{0}}^{\alpha_{0}}
  \end{equation}
  for all $u$, $\hat{u}\in X_{0}\cap X_{1}$ and
  \begin{equation}
    \label{eq:47}
    \norm{Tu-T\hat{u}}_{Y_{1}}\le
    M_{1}\,\norm{u-\hat{u}}_{X_{1}}^{\alpha_{1}}
  \end{equation}
  for all $u$, $\hat{u}\in X_{0}\cap X_{1}$. For every
  $0<\theta<1$ and $1\le q_{0}$,
    $q_{1}\le \infty$ (excluding $q_{0}=q_{1}=\infty$) satisfying
    $q_{0}\ge \tfrac{1}{\alpha_{0}}$ and
    $q_{1}\ge \tfrac{1}{\alpha_{1}}$, let
    $1\le q, p_{0}, p_{1}\le\infty$, $0<\eta<1$, $0<\alpha<\infty$
    be given by
    \begin{equation}
      \label{eq:68}
      \begin{array}[c]{ll}
        \tfrac{1}{q}=\tfrac{1-\theta}{q_{0}}+\tfrac{\theta}{q_{1}}, 
          &p_{0}= \alpha_{0}q_{0},\quad p_{1}=\alpha_{1}q_{1},\\
        \eta=\tfrac{\theta\,\alpha_{1}}{(1-\theta)\alpha_{0}+\theta\alpha_{1}}, 
          &\alpha=(1-\theta)\alpha_{0}+\theta \alpha_{1}
      \end{array}
    \end{equation}
    and let $1\le s_{0}\le p_{0}$ and $1\le s_{1}\le p_{1}$. Then the
    following statements hold.
  \begin{enumerate}
  \item One has
    \begin{equation}
      \label{eq:48}
      \norm{Tu-T\hat{u}}_{(Y_{0},Y_{1})_{\theta,q_{0},q_{1}}}\le \,
      \Big(\tfrac{\eta\,\alpha_{0}}{\theta}\Big)^{\frac{1}{q}}\,
      M_{0}^{1-\theta}\,M_{1}^{\theta}\,C^{\alpha}_{\eta,r_{0},r_{1}}
      \norm{u-\hat{u}}_{(X_{0},X_{1})_{\eta,s_{0},s_{1}}}^{\alpha}
    \end{equation}
    for every $u$, $\hat{u}\in X_{0}\cap X_{1}$, where the constant
    $C_{\eta,r_{0},r_{1}}$ is given by~\eqref{eq:52}. 
  
   \item If there is a $u_{0}\in X_{0}\cap X_{1}$ such that
    $Tu_{0}\in (Y_{0},Y_{1})_{\theta,q_{0},q_{1}}$, then $T$ can be
    uniquely extended to a mapping $T : (X_{0},X_{1})_{\eta,s_{0},s_{1}}\to
    (Y_{0},Y_{1})_{\theta,q_{0},q_{1}}$ satisfying inequality~\eqref{eq:48} for all $u$,
    $\hat{u}\in (X_{0},X_{1})_{\eta,s_{0},s_{1}}$.
  \end{enumerate}
\end{theorem}

By using the preliminaries of this subsection, we can now outline the
proof of this nonlinear interpolation theorem.

\begin{proof}[Proof of Theorem~\ref{thm:nonlinear-interpol}]
 First, we fix $\hat{u}\in X_{0}\cap
  X_{1}$ and show that
  \begin{equation}
    \label{eq:58}
    \norm{T(u+\hat{u})-T\hat{u}}_{(Y_{0},Y_{1})_{\theta,q_{0},q_{1}}}\le
    \Big(\tfrac{\eta\,\alpha_{0}}{\theta}\Big)^{\frac{1}{q}}\,
    M_{0}^{1-\theta}\,M_{1}^{\theta}\, \norm{u}_{(X_{0},X_{1})_{\eta,p_{0},p_{1}}}^{\alpha}
  \end{equation}
  for all $u\in X_{0}\cap X_{1}$. To do so, let
  $u\in X_{0}\cap X_{1}$. Since $X_{0}\cap X_{1}$ is continuously
  injected into $(X_{0},X_{1})_{\eta,p_{0},p_{1}}$, there is a pair
  $(v_{0},v_{1})$ of measurable functions
  satisfying~\eqref{eq:50}. Since $u\in X_{0}\cap X_{1}$ and
  $u=v_{0}+v_{1}$, it follows that  $v_{i}(t)\in X_{0}\cap X_{1}$ for
  a.e. $t\in (0,\infty)$ and every $i=0$, $1$. For
  $\lambda:=\tfrac{\theta}{\eta\,\alpha_{0}}>0$, we set
  \begin{displaymath}
    w_{0}(t)=T(v_{0}(t^{\lambda})+\hat{u})-T\hat{u}\quad\text{ and }\quad
    w_{1}(t)=T(u+\hat{u})-T\hat{u}-w_{0}(t)
  \end{displaymath}
  for a.e. $t\in (0,\infty)$. Then,
  $T(u+\hat{u})-T\hat{u}=w_{0}(t)+w_{1}(t)$ for a.e.
  $t\in (0,\infty)$, and by using~\eqref{eq:30} and~\eqref{eq:47}, one sees
  that the functions $w_{i} : (0,\infty)\to Y_{i}$ are measurable and
  satisfy
  \begin{equation}\label{eq:59}
    \norm{w_{i}(t)}_{Y_{i}}\le
    M_{i}\,\norm{v_{i}(t^{\lambda})}_{X_{i}}^{\alpha_{i}}
  \end{equation}
  for a.e. $t\in (0,\infty)$ and each $i=0, 1$. Since we have chosen
  $\lambda=\tfrac{\theta}{\eta\,\alpha_{0}}$ and
  $p_{0}=q_{0}\alpha_{0}$, we obtain by applying inequality~\eqref{eq:59} and 
  substituting $s=t^{\lambda}$ that
  \begin{align*}
    \norm{t^{-\theta}w_{0}}_{L^{q_{0}}_{\ast}(Y_{0})}
     %    &\le M_{0}\,\left(\int_{0}^{\infty}t^{-\theta
      %     q_{0}}\,\norm{v_{0}(t^{\lambda})}_{X_{0}}^{\alpha_{0}q_{0}}
      %     \tfrac{\dt}{t}\right)^{\frac{1}{q_{0}}}\\
       %  & = M_{0}\,\Big(\tfrac{\eta\,\alpha_{0}}{\theta}\Big)^{\frac{1}{q_{0}}} 
       %    \left(\int_{0}^{\infty}s^{-\eta p_{0}}\,\norm{v_{0}(s)}_{X_{0}}^{p_{0}}
        %   \tfrac{\ds}{s}\right)^{\frac{1}{q_{0}}}\\
         &\le M_{0}\,\Big(\tfrac{\eta\,\alpha_{0}}{\theta}\Big)^{\frac{1}{q_{0}}}\,
            \norm{s^{-\eta}v_{0}}_{L^{p_{0}}(X_{0})}^{\alpha_{0}}.
  \end{align*}
  On the other hand, 
  $\eta=\tfrac{\theta\,\alpha_{1}}{(1-\theta)\alpha_{0}+\theta\alpha_{1}}$
  is equivalent to $\tfrac{1-\eta}{\eta}=\tfrac{(1-\theta) \alpha_{0}}{\theta
    \alpha_{1}}$ hence $\lambda = \tfrac{1-\theta}{(1-\eta)
    \alpha_{1}}$. Using this together with inequality~\eqref{eq:59},
  the fact that $p_{1}=q_{1} \alpha_{1}$,
  and applying the substitution $s=t^{\lambda}$, we see that
  \begin{align*}
    \norm{t^{1-\theta}w_{1}}_{L^{q_{1}}_{\ast}(Y_{1})}
         % &\le M_{1}\,\left(\int_{0}^{\infty}t^{(1-\theta)
         %   q_{1}}\,\norm{v_{1}(t^{\lambda})}_{X_{1}}^{\alpha_{1}q_{1}}
         %   \tfrac{\dt}{t}\right)^{\frac{1}{q_{1}}}\\
         % & = M_{1}\,\Big(\tfrac{\eta\,\alpha_{0}}{\theta}\Big)^{\frac{1}{q_{1}}} 
         %   \left(\int_{0}^{\infty}s^{(1-\eta) p_{1}}\,\norm{v_{1}(s)}_{X_{1}}^{p_{1}}
         %   \tfrac{\ds}{s}\right)^{\frac{1}{q_{1}}}\\
         &\le M_{1}\,\Big(\tfrac{\eta\,\alpha_{0}}{\theta}\Big)^{\frac{1}{q_{1}}}\,
            \norm{s^{1-\eta}v_{1}}_{L^{p_{1}}(X_{1})}^{\alpha_{1}}.
  \end{align*}
  Thus $T(u+\hat{u})-T\hat{u}\in (Y_{1},Y_{2})_{\theta,q_{0},q_{1}}$. Combining the last
  two estimates together with~\eqref{eq:51} yields
  \begin{align*}
    &\norm{T(u+\hat{u})-T\hat{u}}_{(Y_{0},Y_{1})_{\theta,q_{0},q_{1}}}\\
    	&\qquad \le
    M_{0}^{1-\theta}\,M_{1}^{\theta}\,\Big(\tfrac{\eta\,\alpha_{0}}{\theta}\Big)^{\frac{1}{q}}\,
            \norm{s^{-\eta}v_{0}}_{L^{p_{0}}(X_{0})}^{(1-\theta)\alpha_{0}}\,
            \norm{s^{1-\eta}v_{1}}_{L^{p_{1}}(X_{1})}^{\theta\alpha_{1}}\\
    &\qquad \le
    M_{0}^{1-\theta}\,M_{1}^{\theta}\,\Big(\tfrac{\eta\,\alpha_{0}}{\theta}\Big)^{\frac{1}{q}}\,
          \max\Big\{\norm{s^{-\eta}v_{0}}_{L^{p_{0}}(X_{0})},\,
            \norm{s^{1-\eta}v_{1}}_{L^{p_{1}}(X_{1})}\Big\}^{\alpha}.
  \end{align*}
  Taking the infimum over all representation pairs $(v_{0},v_{1})$
  satisfying~\eqref{eq:50} shows that inequality~\eqref{eq:58}
  holds. Now, for every $u, \hat{u}\in X_{0}\cap X_{1}$, replacing 
  $u$ by $u-\hat{u}\in X_{0}\cap X_{1}$ in~\eqref{eq:58} gives
  \begin{equation}\label{eq:55}
    \norm{Tu-T\hat{u}}_{(Y_{0},Y_{1})_{\theta,q_{0},q_{1}}}\le
    \Big(\tfrac{\eta\,\alpha_{0}}{\theta}\Big)^{\frac{1}{q}}\,
    M_{0}^{1-\theta}\,M_{1}^{\theta}\, \norm{u-\hat{u}}_{(X_{0},X_{1})_{\eta,p_{0},p_{1}}}^{\alpha}
  \end{equation}
  for all $u$, $\hat{u}\in X_{0}\cap X_{1}$. Applying
  Lemma~\ref{lem:embedding-of-mean-spaces} yields
  inequality~\eqref{eq:48} for every $u$, $\hat{u}\in X_{0}\cap X_{1}$,
  proving that the first statement of this theorem holds.

  Under the assumption, there is a $u_{0}\in X_{0}\cap X_{1}$ such
  that $Tu_{0}\in (Y_{0},Y_{1})_{\theta,q_{0},q_{1}}$,
  inequality~\eqref{eq:48} implies that the mapping $T$ maps
  $X_{0}\cap X_{1}$ equipped with the
  $(X_{0},X_{1})_{\eta,p_{0},p_{1}}$-norm into
  $(Y_{0},Y_{1})_{\theta,q_{0},q_{1}}$. Thus by Lemma~\eqref{lem:density} and since
  the spaces $(X_{0},X_{1})_{\eta,p_{0},p_{1}}$ and
  $(Y_{0},Y_{1})_{\theta,q_{0},q_{1}}$ are complete, we can
  conclude that $T$ admits a unique H\"older-continuous extension from
  $(X_{0},X_{1})_{\eta,p_{0},p_{1}}$ to 
  $(Y_{0},Y_{1})_{\theta,q_{0},q_{1}}$
  satisfying~\eqref{eq:55} for all  $u$, $\hat{u}\in
  (X_{0},X_{1})_{\eta,p_{0},p_{1}}$. This completes the proof of this theorem.
\end{proof}

In our second nonlinear interpolation theorem, we consider the
situation when the mapping $T$ admits an element
$u_{0}\in X_{0}\cap X_{1}$ such that $Tu_{0}\in Y_{0}\cap Y_{1}$.

\begin{theorem}\label{thm:interpol-no-differences}
  Let $(X_{0},X_{1})$ and $(Y_{0},Y_{1})$ be two interpolation couples
  and $T$ a mapping from $X_{0}+X_{1}$ into $Y_{0}+Y_{1}$ with domain
  containing $X_{0}\cap X_{1}$. Suppose
  $T$ is continuous from $X_{0}\cap X_{1}$ equipped with the
  $X_{0}$-norm to $Y_{0}$ and there are $u_{0}\in X_{0}\cap X_{1}$ satisfying
  $Tu_{0}\in Y_{0}\cap Y_{1}$, exponents $0<\alpha_{0}$,
  $\alpha_{1}<\infty$, and constants $M_{0}$, $M_{1}\ge 0$ such that
  \begin{equation}
    \label{eq:30bis}
    \norm{Tu-Tu_{0}}_{Y_{0}}\le
    M_{0}\,\norm{u-u_{0}}_{X_{0}}^{\alpha_{0}}\qquad\text{ for all $u\in X_{0}\cap X_{1}$} 
  \end{equation}
  and
  \begin{equation}
    \label{eq:47bis}
    \norm{Tu-T\hat{u}}_{Y_{1}}\le
    M_{1}\,\norm{u-\hat{u}}_{X_{1}}^{\alpha_{1}}\qquad\text{for all $u$, $\hat{u}\in X_{0}\cap X_{1}$.}
  \end{equation}
  For every $0<\theta<1$
  and $1\le q_{0}$, $q_{1}\le \infty$ (excluding $q_{0}=q_{1}=\infty$)
  satisfying $q_{0}\ge \tfrac{1}{\alpha_{0}}$ and
  $q_{1}\ge \tfrac{1}{\alpha_{1}}$, let
  $1\le q, p_{0}, p_{1}\le\infty$, $0<\eta<1$, $0<\alpha<\infty$ given
  by~\eqref{eq:68},
  % \begin{displaymath}
  %   \tfrac{1}{q}=\tfrac{1-\theta}{q_{0}}+\tfrac{\theta}{q_{1}},\; 
  %   p_{0}= \alpha_{0}q_{0},\; p_{1}=\alpha_{1}q_{1},\;
  %   \eta=\tfrac{\theta\,\alpha_{1}}{(1-\theta)\alpha_{0}+\theta\alpha_{1}},\;
  %   \alpha=(1-\theta)\alpha_{0}+\theta \alpha_{1}
  % \end{displaymath}
  and let $1\le s_{0}\le p_{0}$ and $1\le s_{1}\le p_{1}$. Then one has
  \begin{equation}
    \label{eq:48bis}
    \norm{Tu-Tu_{0}}_{(Y_{0},Y_{1})_{\theta,q_{0},q_{1}}}\le \,
    \Big(\tfrac{\eta\,\alpha_{0}}{\theta}\Big)^{\frac{1}{q}}\,
    M_{0}^{1-\theta}\,M_{1}^{\theta}\,C^{\alpha}_{\eta,r_{0},r_{1}}
    \norm{u-u_{0}}_{(X_{0},X_{1})_{\eta,s_{0},s_{1}}}^{\alpha}
  \end{equation}
  for every $u\in X_{0}\cap X_{1}$, 
  % $\hat{u}\in (X_{0},X_{1})_{\eta,s_{0},s_{1}}$, 
  where the constant $C_{\eta,r_{0},r_{1}}$ is given by~\eqref{eq:52}.
\end{theorem}

\begin{proof}[Proof of Theorem~\ref{thm:interpol-no-differences}]
 Let $u\in X_{0}\cap X_{1}$. Since $X_{0}\cap X_{1}$ is continuously
  injected into $(X_{0},X_{1})_{\eta,p_{0},p_{1}}$, there are measurable functions $v_{i} : (0,\infty)\to
  X_{i}$ for $i=0, 1$ satisfying $u-u_{0}=v_{0}(t)+v_{1}(t)$ in $X_{0}+X_{1}$ for a.e. $t\in
  (0,\infty)$,
  \begin{equation}
      \label{eq:69}
      t^{-\theta}v_{0}\in L^{p_{0}}_{\ast}(X_{0})\qquad\text{and}\qquad
      \;t^{1-\theta}v_{1}\in L^{p_{1}}_{\ast}(X_{1}).
    \end{equation}
 For $\lambda:=\tfrac{\theta}{\eta\,\alpha_{0}}>0$, we set
  \begin{displaymath}
    w_{0}(t)=T(v_{0}(t^{\lambda})+u_{0})-Tu_{0}\quad\text{ and }\quad
    w_{1}(t)=Tu-Tu_{0}-w_{0}(t)
  \end{displaymath}
  for a.e. $t\in (0,\infty)$. By construction,
  $Tu-Tu_{0}=w_{0}(t)+w_{1}(t)$ for a.e.  $t\in (0,\infty)$. Since by
  assumption, $T$ is continuous from $X_{0}\cap X_{1}$ equipped with
  the $X_{0}$-norm to $Y_{0}$, the function
  $w_{0} : (0,\infty)\to Y_{0}$ is strongly
  measurable. By~\eqref{eq:47}, $T$ is H\"older-continuous from
  $X_{0}\cap X_{1}$ equipped with the $X_{1}$-norm to $Y_{1}$. Thus,
  the function $w_{1} : (0,\infty)\to Y_{1}$ is strongly
  measurable. Moreover, by \eqref{eq:30bis} and~\eqref{eq:47bis}, we
  have that the inequalities~\eqref{eq:59} hold for $i=0, 1$. Now, we
  can proceed as in the proof of Theorem~\ref{thm:nonlinear-interpol}
  to conclude that inequality~\eqref{eq:48bis} holds for all
  $u\in X_{0}\cap X_{1}$.
\end{proof}

%-------------------------------------------------------------
%-------------------------------------------------------------
%-------------------------------------------------------------
%-------------------------------------------------------------
%
%
%                  Third  interpolation theorem
%
%
%-------------------------------------------------------------
%-------------------------------------------------------------
%-------------------------------------------------------------
%-------------------------------------------------------------

In some applications, the assumption that \emph{the
  mapping $T$ is continuous from $X_{0}\cap X_{1}$ equipped with the
  $X_{0}$-norm topology to $Y_{0}$} in
Theorem~\ref{thm:interpol-no-differences} is too strong. This can be
circumvented, for instance, by the following result.

\begin{theorem}\label{thm:interpol-no-differences-bis}
  Let $(X_{0},X_{1})$ and $(Y_{0},Y_{1})$ be two interpolation couples,
  $Y_{0}$ being a separable Banach space. Let $T$ be a mapping from
  $X_{0}+X_{1}$ into $Y_{0}+Y_{1}$ with domain containing
  $X_{0}\cap X_{1}$. Suppose there is some $u_{0}\in X_{0}\cap X_{1}$
  such that $Tu_{0}\in Y_{0}\cap Y_{1}$ and $T$ satisfies the following three
  conditions. 
  \begin{itemize}
  \item $T$ is continuous from $X_{0}\cap X_{1}$
  equipped with the $X_{0}$-norm to $Y_{0}$ equipped with the weak
  topology,
  % \item $T$ is continuous from $X_{0}\cap X_{1}$ equipped with the
  % $X_{1}$-norm to $Y_{1}$,
  \item there are exponents $0<\alpha_{0}$, $\alpha_{1}<\infty$ and constants $M_{0}$,
  $M_{1}\ge 0$ such that $T$ satisfies~\eqref{eq:30bis} and \eqref{eq:47bis}.
  \end{itemize}
  For every $0<\theta<1$ and
  $1\le q_{0}$, $q_{1}\le \infty$ (excluding $q_{0}=q_{1}=\infty$)
  satisfying $q_{0}\ge \tfrac{1}{\alpha_{0}}$ and
  $q_{1}\ge \tfrac{1}{\alpha_{1}}$, let
  $1\le q, p_{0}, p_{1}\le\infty$, $0<\eta<1$, $0<\alpha<\infty$ given
  by~\eqref{eq:68}, and let $1\le s_{0}\le p_{0}$ and
  $1\le s_{1}\le p_{1}$. Then $T$ satisfies
  inequality~\eqref{eq:48bis} for every $u\in X_{0}\cap X_{1}$, where
  the constant $C_{\eta,r_{0},r_{1}}$ is given by~\eqref{eq:52}.
\end{theorem}

\begin{remark}
  \label{rem:density-argument}
  Consider the following situation: For $1\le q$, $r<\infty$, let
  $X_{0}=L^{q}(\Sigma,\mu)$, $X_{1}=L^{\infty}(\Sigma,\mu)$,
  $Y_{0}=L^{r}(\Sigma,\mu)$ and $Y_{1}=L^{\infty}(\Sigma,\mu)$, where
  one assumes that $(\Sigma,\mu)$ is a \emph{separable measure space} (cf,
  \cite[Definition on p.98]{MR2759829}). Suppose $T$ satisfy the
  assumptions of Theorem~\ref{thm:interpol-no-differences-bis} and we
  choose
 \begin{displaymath}
   \begin{array}[c]{lll}
     q_{0}=r,& q_{1}=\infty, & p_{0}=\beta\,q_{0}=\beta\,r> q\ge1,\\
     p_{1}=q_{1}=\infty, & s_{0}=q< \beta\,r=p_{0}, & s_{1}=\infty.
   \end{array}
 \end{displaymath}
 Then, by Corollary~\ref{corApp:1},
 \begin{displaymath}
    (X_{0},X_{1})_{\eta,s_{0},s_{1}}=L^{\frac{q}{(1-\eta)}}(\Sigma,\mu)
    \qquad\text{and}\qquad
    (Y_{0},Y_{1})_{\theta,q_{0},q_{1}}=L^{\frac{r}{(1-\theta)}}(\Sigma,\mu)
 \end{displaymath}
 with equal norms for every $0<\theta, \eta<1$ and so
 Theorem~\ref{thm:interpol-no-differences-bis} yields
 \begin{equation}
   \label{eq:185}
   \begin{split}
     & \norm{Tu-u_{0}}_{\frac{r}{1-\theta}}\\
     &\qquad \le
     \Big[\tfrac{\beta}{(1-\theta)\beta+\theta}\Big]^{\frac{1-\theta}{r}}\,
     M_{0}^{1-\theta}\,M_{1}^{\theta}\, C_{\eta,r_{0},1}^{(1-\theta)\beta+\theta}\,
     \norm{u-u_{0}}_{\frac{q}{1-\eta(\theta)}}^{(1-\theta)\beta+\theta}
   \end{split}
  \end{equation}
  for every $u\in L^{q}(\Sigma,\mu)\cap L^{\infty}(\Sigma,\mu)$ and every
  $0<\theta<1$, where
  \begin{math}
   %\eta(\theta)=\tfrac{\theta}{(1-\theta)\beta+\theta}\qquad\text{ and }\qquad
   r_{0}=\frac{q\,\beta\,r}{\beta r(q-1)+q}.
 \end{math}
In addition, to the above assumptions, we suppose
 \begin{displaymath}
  \textit{$T$ is continuous from $L^{\frac{q}{1-\eta(\theta)}}(\Sigma,\mu)$ to $L^{\frac{q}{1-\eta(\theta)}}(\Sigma,\mu)$.}
 \end{displaymath}
Since $L^{q}(\Sigma,\mu)\cap L^{\infty}(\Sigma,\mu)$ is dense in
$L^{\frac{q}{1-\eta(\theta)}}(\Sigma,\mu)$, for every $u\in
L^{\frac{q}{1-\eta(\theta)}}(\Sigma,\mu)$, there is a sequence
$(u_{n})$ in $L^{q}(\Sigma,\mu)\cap L^{\infty}(\Sigma,\mu)$ such that
$u_{n}$ converges to $u$ in
$L^{\frac{q}{1-\eta(\theta)}}(\Sigma,\mu)$ and so $Tu_{n}$ converges to $Tu$ in
$L^{\frac{q}{1-\eta(\theta)}}(\Sigma,\mu)$. By~\eqref{eq:185},
$(Tu_{n})$ is bounded in $L^{\frac{r}{1-\theta}}(\Sigma,\mu)$ and
hence, after eventually passing to a subsequence of $(u_{n})$, we may
assume that $Tu_{n}$ converges weakly to $v$ in
$L^{\frac{r}{1-\theta}}(\Sigma,\mu)$ for some $v\in
L^{\frac{r}{1-\theta}}(\Sigma,\mu)$. Since
$L^{\frac{q}{1-\eta(\theta)}}(\Sigma,\mu)$ and
$L^{\frac{r}{1-\theta}}(\Sigma,\mu)$ are both continuously embedded into
$L^{m}_{loc}(\Sigma,\mu)$, with
$m:=\min\{\frac{q}{1-\eta(\theta)},\frac{r}{1-\theta}\}$, we obtain
$v=Tu$ a.e. on $\Sigma$ and so, sending $n\to\infty$ in~\eqref{eq:185}
for $u=u_{n}$ and using Fatou's lemma shows that~\eqref{eq:185} holds
for all $u\in L^{\frac{q}{1-\eta(\theta)}}(\Sigma,\mu)$.
\end{remark}

% \begin{remark}
%   Note, the claim of Theorem~\ref{thm:interpol-no-differences-bis}
%   remains unchanged if one replaces also the second assumption
%   \emph{$T$ is continuous from $X_{0}\cap X_{1}$ equipped with the
%     $X_{1}$-norm to $Y_{1}$} by the following one: $Y_{1}$ is
%   separable and $T$ is continuous from $X_{0}\cap X_{1}$ equipped with
%   the $X_{1}$-norm to $Y_{1}$ equipped with the weak topology.
% \end{remark}

\begin{proof}[Proof of Theorem~\ref{thm:interpol-no-differences-bis}]
  Let $u\in X_{0}\cap X_{1}$ and for $i=0$, $1$, let
  $v_{i} : (0,\infty)\to X_{i}$ be measurable such that
  $u-u_{0}=v_{0}(t)+v_{1}(t)$ in $X_{0}+X_{1}$ for a.e.
  $t\in (0,\infty)$ and~\eqref{eq:69} holds.
  For $\lambda:=\tfrac{\theta}{\eta\,\alpha_{0}}>0$, we set
  \begin{displaymath}
    w_{0}(t)=T(v_{0}(t^{\lambda})+u_{0})-Tu_{0}\quad\text{ and }\quad
    w_{1}(t)=Tu-Tu_{0}-w_{0}(t)
  \end{displaymath}
  for a.e. $t\in (0,\infty)$. By construction,
  $Tu-Tu_{0}=w_{0}(t)+w_{1}(t)$ for a.e.  $t\in (0,\infty)$. By
  assumption, $T$ is continuous from $X_{0}\cap X_{1}$ equipped with
  the $X_{0}$-norm topology to $Y_{0}$ equipped with the
  weak-topology. Hence $w_{0}$ is weakly measurable. But since by
  assumption, $Y_{0}$ is separable, the function
  $w_{0} : (0,\infty)\to Y_{0}$ is strongly measurable due to Pettis's theorem
  (\cite[Theorem~3.5.3]{MR0423094}). By~\eqref{eq:47}, $T$ is
  H\"older-continuous from $X_{0}\cap X_{1}$ equipped with the
  $X_{1}$-norm to $Y_{1}$. Thus, the function
  $w_{1} : (0,\infty)\to Y_{1}$ is strongly measurable. Moreover, by
  \eqref{eq:30bis} and~\eqref{eq:47bis}, we have that the
  inequalities~\eqref{eq:59} hold for $i=0, 1$. Now, we can proceed as
  in the proof of Theorem~\ref{thm:nonlinear-interpol} and see that
  the statement of this theorem holds.
\end{proof}

\subsection{Extrapolation towards $L^\infty$}
\label{sec:extrapolation-towards-infinity}

To the best of our knowledge, first extrapolation results towards
$L^{\infty}$ in the context of \emph{linear semigroups} and 
employing Riesz-Thorin's or Stein's linear interpolation theorems go
back to the pioneering work \cite{MR0293451} by Simon and
H{\o}egh-Krohn (see also~\cite[Theorem 3.3]{MR766493}). An alternative
approach using a duality argument has been given
in~\cite[Lemme~1]{MR1077272}. However, in this article, we are
confronted with a much more difficult situation, since the family of
operators $\{T_{t}\}_{t\ge0}$ are (in general) nonlinear. Hence
neither a duality argument or a linear Riesz-Thorin interpolation
theorem can be used.

Our extrapolation result towards $L^{\infty}$ is a nonlinear
generalisation of the techniques developed
in~\cite{MR0293451,MR766493,MR1077272}. Our proof relies essentially
on the nonlinear interpolation results
Theorem~\ref{thm:nonlinear-interpol} and
Theorem~\ref{thm:interpol-no-differences-bis}, as well as the fact
that the mean spaces involving $L^{p_{0}}(\Sigma,\mu)$ and
$L^{p_{1}}(\Sigma,\mu)$ spaces are isometrically isomorphic to an
appropriate $L^{p}(\Sigma,\mu)$ space (cf. Corollary~\ref{corApp:1}).

Here, we shall use the notation $u\lesssim v$ to say that there exists a
constant $C$ (independent of the important parameters) such that
$u\leq Cv$.

Our first extrapolation result towards $L^{\infty}$ is adapted to
semigroups generated by completely accretive operators
(Section~\ref{sec:comp}) satisfying the $L^{q}$-$L^{r}$-regulari\-sation
effect~\eqref{eq:18} for \emph{differences} and $1\le q$, $r<\infty$.

\begin{theorem}\label{thm:extrapolation-to-infty}
  Let $1\le q$, $r< \infty$ and $\{T_{t}\}_{t\geq 0}$ be a semigroup
  acting on $L^{q}\cap L^{\infty}(\Sigma,\mu)$. Suppose $\{T_{t}\}_{t\geq 0}$ satisfies
  exponential growth~\eqref{eq:62} for $\tilde{q}=\infty$ and some $\omega\ge 0$,
  and there exist $C>0$ and exponents $\alpha$, $\beta$, $\gamma>0$
  such that the estimate
   \begin{equation}
    \tag{\ref{eq:18}}
    \norm{T_{t}u-T_{t}\hat{u}}_{r}\le 
    \left(\tfrac{C}{q}\right)^{1/\sigma}\; t^{-\alpha}\; e^{\omega \beta t}\,
    \norm{u-\hat{u}}_{q}^{\gamma}
  \end{equation}
  holds for every $t>0$ and $u$,
  $\hat{u}\in L^{q}(\Sigma,\mu)\cap L^{\infty}(\Sigma,\mu)$. If
  \begin{equation}
    \label{eq:70}
    \gamma\,r> q
  \end{equation}
  then 
  \begin{equation}
    \label{eq:1}
    \norm{T_{t}u-T_{t}\hat{u}}_{\infty}\lesssim \; t^{-\alpha^{\ast}}\;
    e^{\omega \beta^{\ast} t}\;
    \norm{u-\hat{u}}_{\gamma\,r\,q^{-1}\,m_{0}}^{\gamma^{\ast}}
  \end{equation}
  for every $t>0$ and $u$,
  $\hat{u}\in L^{\gamma\,r\,q^{-1}\,m_{0} }(\Sigma,\mu)$, with exponents
  \begin{equation}
    \label{eq:72}
    \begin{array}[c]{c}
    \alpha^{\ast}=
      \displaystyle  \frac{\alpha\,q\,\gamma^{-1}}{(\frac{\gamma\,r}{q}-1)
      \,m_{0}+q(\frac{1}{\gamma}-1)},\qquad
      \gamma^{\ast}=
      \frac{(\frac{\gamma\,r}{q}-1)\,m_{0}}{(
      \frac{\gamma\,r}{q}-1)\,m_{0}+q(\frac{1}{\gamma}-1)},\\[9pt]
      \displaystyle \beta^{\ast}=\frac{(\beta-1)\gamma\,r\,q^{-1}
      +\gamma-\beta}{(\frac{\gamma\,r}{q}-1)\,m_{0}+q(\frac{1}{\gamma}-1)}+1,
    \end{array}
  \end{equation}
  and $m_{0}\ge q\,\gamma^{-1}$ such that
  \begin{equation}
    \label{eq:71}
  (\tfrac{\gamma\,r}{q}-1)\,m_{0}+q(\tfrac{1}{\gamma}-1)>0.
\end{equation}
\end{theorem}

\begin{remark}
  \label{rem:10}
  The two conditions~\eqref{eq:70} and \eqref{eq:71} are heavily
  involved in the recursive construction 
  \begin{equation}
    \label{eq:247}
    m_{n+1} = m_{n}\,\kappa-  r\, \kappa^{-1}\,(\gamma-1),\qquad\text{($n\ge 1$),}
  \end{equation}
  of a strictly increasing sequence $(m_{n})_{n\ge 0}\subseteq (1,+\infty)$ satisfying
  $\lim_{n\to+\infty}m_{n}=+\infty$. If one chooses $\kappa$ by
  \begin{equation}
    \label{eq:249}
    \kappa = \frac{\gamma\, r}{q}
  \end{equation}
  then condition \eqref{eq:70} yields $\kappa>1$. If, in
  addition, $m_{0}$ satisfies~\eqref{eq:71} then $(m_{n})_{n\ge 0}$ is
  strictly increasing and $\lim_{n\to+\infty}m_{n}=+\infty$. Inserting
  the sequence $(m_{n})_{n\ge 0}$ into inequality~\eqref{eq:24} and
  using the semigroup property of $\{T_{t}\}_{t\ge 0}$, one obtains
  an $L^{\tilde{q}}$-$L^{\infty}$ regularisation effect of the
  $\{T_{t}\}_{t\ge 0}$ for some $\tilde{q}=\gamma\,r\,q^{-1}\,m_{0}\in
  [1,\infty)$. 
\end{remark}

%   by
%   deducing an $L^{\tilde{q}}$-$L^{\infty}$ regularisation
%   estimate~\eqref{eq:1} for some $1\le \tilde{q}<\infty$ from an
%   $L^{q}$-$L^{r}$ regularity estimate~\eqref{eq:18} for some $1\le q$,
%   $r<\infty$. In order to construct a  (cf. iteration Lemma~\ref{lem:iteration}
%   below) and seem to have a pure technical origin. It is certainly not difficult to find an $m_{0}\ge
%   1$ large enough such that \eqref{eq:71} holds. But it depends on the

\begin{proof}[Proof of Theorem~\ref{thm:extrapolation-to-infty}]
  We intend to apply Theorem~\ref{thm:nonlinear-interpol} to the
  following situation: let $X_{0}=L^{q}(\Sigma,\mu)$,
  $X_{1}=L^{\infty}(\Sigma,\mu)$, $Y_{0}=L^{r}(\Sigma,\mu)$,
  $Y_{1}=L^{\infty}(\Sigma,\mu)$, and for any fixed $t>0$, let
  $T=T_{t}$. By assumption, $T_{t}$ satisfies~\eqref{eq:18} and has exponential
  growth~\eqref{eq:62} for $\tilde{q}=\infty$ and some $\omega\ge 0$. Hence
  the mapping $T$ satisfies 
  inequality~\eqref{eq:30} 
  with $\alpha_{0}=\gamma>0$, $M_{0}=C\,e^{\omega \beta t}\,t^{-\alpha}$
  and inequality~\eqref{eq:47} with $\alpha_{1}=1$,
  $M_{1}= e^{\omega t}$. Further, we choose
\begin{displaymath}
  \begin{array}[c]{lll}
  q_{0}=r,& q_{1}=\infty, & p_{0}=\gamma\,q_{0}=\gamma\,r> q\ge1,\\
  p_{1}=q_{1}=\infty, & s_{0}=q< \gamma\,r=p_{0}, & s_{1}=\infty.
  \end{array}
\end{displaymath}
Then, by Corollary~\ref{corApp:1},
  \begin{displaymath}
    (X_{0},X_{1})_{\eta,s_{0},s_{1}}=L^{\frac{q}{(1-\eta)}}(\Sigma,\mu)
    \qquad\text{and}\qquad
    (Y_{0},Y_{1})_{\theta,q_{0},q_{1}}=L^{\frac{r}{(1-\theta)}}(\Sigma,\mu)
  \end{displaymath}
 with equal norms for every $0<\theta, \eta<1$. Thus,
 Theorem~\ref{thm:nonlinear-interpol} yields
  \begin{align*}
\norm{T_{t}u-T_{t}\hat{u}}_{\frac{r}{1-\theta}}
    &\le 
    \Big[\tfrac{\gamma}{(1-\theta)\gamma+\theta}\Big]^{\frac{1-\theta}{r}}\,
    \big[C\,e^{\omega \beta t}\,t^{-\alpha}\big]^{1-\theta}\,e^{\omega \theta t}\times\\
    &\hspace{2cm}\times\big[\inf_{\varphi\in
       D_{+}}\norm{s^{-\theta}\varphi}_{L^{r_{0}}_{\ast}(\R)}^{1-\theta}\big]^{(1-\theta)\gamma+\theta}\,
     \norm{u-\hat{u}}_{\frac{q}{1-\eta(\theta)}}^{(1-\theta)\gamma+\theta}
  \end{align*}
  for every $t>0$, $u$,
  $\hat{u}\in L^{q}(\Sigma,\mu)\cap L^{\infty}(\Sigma,\mu)$ and every
  $0<\theta<1$, where
  \begin{displaymath}
   %\eta(\theta)=\tfrac{\theta}{(1-\theta)\gamma+\theta}\qquad\text{ and }\qquad
   r_{0}=\frac{q\,\gamma\,r}{\gamma r(q-1)r+q}.
 \end{displaymath}

  Next, we choose a
  test function $\rho\in C^{\infty}_{c}((0,\infty))$ with $\rho\ge 0$
  and support
  $\textrm{supp}(\varphi)$ in the closed interval $[1,3]$ satisfying 
  \begin{displaymath}
    e^{-\frac{4}{3}}\,\log \tfrac{5}{3}\le
    \int_{0}^{\infty}\rho(\tfrac{1}{s})\,\tfrac{\ds}{s}\le e^{-1}\,\log 3.
  \end{displaymath}
  Then
  $\varphi^{\ast}:=\left(\int_{0}^{\infty}\rho(\tfrac{1}{t})\,\tfrac{\dt}{t}\right)^{-1}\rho\in
  D_{+}$ and there are $C_{\varphi^{\ast}, 1}$, $C_{\varphi^{\ast}, 2}>0$ such that
  \begin{displaymath}
    C_{\varphi^{\ast}, 1} \le \norm{t^{-\theta}\varphi^{\ast}}_{L^{r_{0}}_{\ast}(\R)}
    \le C_{\varphi^{\ast}, 2}
  \end{displaymath}
  hence
  \begin{align*}
    \norm{T_{t}u-T_{t}\hat{u}}_{\frac{r}{1-\theta}}
    &\le 
    \Big[\tfrac{\gamma}{(1-\theta)\gamma+\theta}\Big]^{\frac{1-\theta}{r}}\,
    \big[C\,e^{\omega \beta t}\,t^{-\alpha}\big]^{1-\theta}\,e^{\omega \theta t}\times\\
    &\hspace{2cm}\times\norm{s^{-\theta}\varphi^{\ast}}_{L^{r_{0}}_{\ast}(\R)}^{(1-\theta)((1-\theta)\gamma+\theta)}\,
     \norm{u-\hat{u}}_{\frac{q}{1-\eta(\theta)}}^{(1-\theta)\gamma+\theta}
  \end{align*}
  for every $t>0$, $u$,
  $\hat{u}\in L^{q}(\Sigma,\mu)\cap L^{\infty}(\Sigma,\mu)$ and every
  $0<\theta<1$. 

  Next, we choose $\kappa$ by~\eqref{eq:249} and set 
  \begin{displaymath}
    \theta_{m}=1-\frac{1}{m}\,\frac{r}{\kappa} \qquad
    \text{for every $m> r\,\kappa^{-1}=\frac{q}{\gamma}$.}
  \end{displaymath}
  Then by hypothesis~\eqref{eq:70}, 
  $\kappa>1$ and for all $m> r\,\kappa^{-1}$, one has
  \begin{displaymath}
    \begin{array}[c]{cc}
    0<\theta_{m}<1, & 1-\theta_{m}= \displaystyle\frac{1}{m}\frac{r}{\kappa},\\[7pt]
  \displaystyle 1-\eta(\theta_{m})=\frac{r\,\kappa^{-1}\gamma}{m+r\kappa^{-1}(\gamma-1)}, &
  \displaystyle\frac{\gamma}{(1-\theta_{m})\gamma+\theta_{m}}=\frac{\gamma\,m}{m+r\kappa^{-1}(\gamma-1)}>0.
    \end{array}
  \end{displaymath}
  Further, we set for all $m> r\,\kappa^{-1}$,
  \begin{displaymath}
    C_{\varphi^{\ast},m}:=\norm{s^{\frac{1}{m}\frac{r}{\kappa}-1}\varphi^{\ast}}_{L^{r_{0}}_{\ast}(\R)}.
  \end{displaymath}
  With this setting in mind, the previous inequality reduces to
  inequality~\eqref{eq:24} below for every $t>0$, $u$,
  $\hat{u}\in L^{q}(\Sigma,\mu)\cap L^{\infty}(\Sigma,\mu)$ and for
  all $m> r\,\kappa^{-1}$. 

  Finally, we choose $m_{0}\ge r\,\kappa^{-1}$ such that~\eqref{eq:71}
  holds (where one notes that with the setting of this proof,
  condition~\eqref{eq:71} coincides with~\eqref{cond:qnod} below) and
  let $m>m_{0}$. The condition on $m_{0}$ is sufficient  to
   run an iteration in the
  time-variable. % ''\emph{parabolic}'' Moser iteration (\cite{MR0159139}),
  This is the contents of the next iteration lemma and from there we
  can conclude that the statement of this theorem holds.
\end{proof}

\begin{lemma}\label{lem:iteration}
  Suppose there are $\kappa>1$, $\beta$, $\gamma>0$, $1\le r<\infty$ and $m_{0}\ge r\,K^{-1}$
  such that
  \begin{equation}
    \label{cond:qnod}
    (\kappa-1) m_{0}+r\,\kappa^{-1}(1-\gamma)>0.
  \end{equation}
  Let $\{T_{t}\}_{t\geq 0}$ be a semigroup acting on $
  L^{\kappa m_{0}}(\Sigma,\mu)\cap L^{\infty}(\Sigma,\mu)$ such that
  \begin{equation}
    \label{eq:24}
    \begin{split}
    \norm{T_{t}u-T_{t}\hat{u}}_{m\,\kappa} &\le 
    \Big[\tfrac{\gamma\,m}{m+r\kappa^{-1}(\gamma-1)}\Big]^{\frac{1}{m\,\kappa}}\,
    \big[C\,e^{\omega \beta t}\,t^{-\alpha}\big]^{
      \frac{1}{m}\frac{r}{\kappa}}\, e^{\omega (1-\frac{1}{m}\frac{r}{\kappa}) t}\times\\
    & \hspace{2cm}\times \,C_{\varphi^{\ast},m}^{\frac{1}{m}\frac{r}{\kappa}(\frac{1}{m}\frac{r}{\kappa}(\gamma-1)+1)}\,
     \norm{u-\hat{u}}_{m+r\,\kappa^{-1}(\gamma-1)}^{\frac{1}{m}r\,\kappa^{-1}(\gamma-1)+1}
  \end{split}
  \end{equation}
  for every $u$, $\hat{u}\in L^{\kappa m_{0}}(\Sigma,\mu)\cap
  L^{\infty}(\Sigma,\mu)$, $t>0$ and $m\ge m_{0}$, where $C_{\varphi^{\ast},m}$
  satisfies
  \begin{equation}\label{eq:4}
    C_{\varphi^{\ast}, 1} \le
    C_{\varphi^\ast,m}\le C_{\varphi^{\ast}, 2}
  \end{equation}
  for some constants $ C_{\varphi^{\ast}, 1}$,
  $ C_{\varphi^{\ast}, 2}>0$ independent of $m\ge m_{0}$ . Then
  \begin{equation}
    \label{eq:11}
    \begin{split}
      \norm{T_{t}u-T_{t}\hat{u}}_{\infty}
     & \lesssim
      \;e^{\omega
        \big(\frac{(\beta-1)\kappa +
          \gamma-\beta}{(\kappa-1)m_{0}+r\,\kappa^{-1}(1-\gamma)}+1\big) t}
      \;t^{-\frac{\alpha\,r\,\kappa^{-1}}{(\kappa-1)
          m_{0}+r\,\kappa^{-1}(1-\gamma)}}\times\\
     &\hspace{4cm}\times 
     \norm{u-\hat{u}}_{\kappa m_{0}}^{\frac{(\kappa-1)\,
          m_{0}}{(\kappa-1)m_{0}+r\,\kappa^{-1}(1-\gamma)}}
    \end{split}
   \end{equation}
  for every $u$, $\hat{u}\in L^{\kappa m_{0}}(\Sigma,\mu)$ and every $t>0$.
\end{lemma}

 For the proof of this lemma, we simplify some techniques
 from~\cite{MR554377} and extend them to semigroups satisfying
 exponential growth condition~\eqref{eq:62} (see also~\cite{MR1277971}
in the linear case).

\begin{proof}
  For $m_{0}\ge r\,\kappa^{-1}$ such that~\eqref{cond:qnod} holds, we
  construct a sequence $(m_{n})_{n\ge 0}$ recursively
  by~\eqref{eq:247}.  Then
  \begin{equation}
    \label{eq:23}
    m_{n+1}=\kappa\,m_{n}+r\,\kappa^{-1}(1-\gamma)
  \end{equation}
  for every integer $n\ge 0$ and so, an induction over $n\in \N_{0}$
  yields
  \begin{equation}
    \label{eq:27i}
    m_{n}=\kappa^{n}[m_{0}+r\,\kappa^{-1}\,(\gamma-1)] +
    r\,\kappa^{-1}(1-\gamma)\,\sum_{\nu=0}^{n}\kappa^{\nu},
  \end{equation}
  that is
  \begin{equation}
    \label{eq:27}
    m_{n}=\kappa^{n}\frac{(\kappa-1)m_{0}+r\,\kappa^{-1}(1-\gamma)}{\kappa-1}
    -\frac{r\,\kappa^{-1}(1-\gamma)}{\kappa-1}.
  \end{equation}

  Using \eqref{eq:27i}, we see that
  \begin{displaymath}
    m_{n+1}-m_{n}= \kappa^{n}\Big[(\kappa-1) m_{0}+r\,\kappa^{-1}(1-\gamma)\Big]
  \end{displaymath}
  hence the sequence $(m_{n})_{n\ge 0}$ is strictly increasing if
  and only if $m_{0}$ satisfies condition~\eqref{cond:qnod}. Moreover,
  by~\eqref{eq:27}, since $\kappa >1$, and by~\eqref{cond:qnod}, we see
  that
  \begin{equation}
    \label{eq:25}
    \lim_{n\to\infty}m_{n}=\infty
  \end{equation}
  and
  \begin{equation}
    \label{rap}
    \lim_{n\to\infty}\frac{m_{n}}{\kappa^{n}}=\frac{(\kappa-1)m_{0}+r\,\kappa^{-1}(1-\gamma)}{\kappa-1}.
  \end{equation}

   Since
  \begin{displaymath}
  \frac{1}{m_{n}}r\,\kappa^{-1}(\gamma-1)+1=\frac{\kappa\,m_{n-1}}{m_{n}}\quad\text{
    and }\quad
  \frac{\gamma\,m_{n}}{m_{n}+r\kappa^{-1}(\gamma-1)}=\frac{\gamma\,m_{n}}{m_{n-1}\kappa},
 \end{displaymath}
 inserting the sequence $(m_{n})_{n\ge0}$ into~\eqref{eq:24} yields
  \begin{equation}
    \label{eq:28}
    \norm{T_{t}u-T_{t}\hat{u}}_{\kappa m_{n}}\le 
    C_{m_{n}}^{\frac{1}{m_{n}\,\kappa}}\,t^{-
      \frac{\alpha\,r}{m_{n}\,\kappa}}
     \,e^{\omega (\beta-1) \frac{r}{m_{n}\,\kappa}
     t}\,e^{\omega t}\,C_{\varphi^{\ast},m_{n}}^{\frac{r}{m_{n}\kappa}\frac{\kappa\,m_{n-1}}{m_{n}}}\,
    \norm{u-\hat{u}}_{m_{n-1}\,\kappa}^{\frac{\kappa\,m_{n-1}}{m_{n}}}
  \end{equation}
  for every $t>0$, $u$,
  $\hat{u}\in L^{\kappa m_{0}}(\Sigma,\mu)\cap
  L^{\infty}(\Sigma,\mu)$, and $n\ge 1$, where
  \begin{displaymath}
    C_{m_{n}}:=\tfrac{\gamma\,m_{n}}{m_{n-1}\kappa}\, C^{r}.
 \end{displaymath}
  
 Now, let $(t_{\nu})_{\nu\ge0}$ be a sequence in $[0,1]$ such that
 $\sum_{\nu=0}^{\infty}t_{\nu}=1$ which we will specify below. By
 assumption, $\{T_{t}\}_{t\ge0}$ is a semigroup and $T_{t}u$,
 $T_{t}\hat{u}\in L^{\kappa m_{0}}(\Sigma,\mu)\cap
 L^{\infty}(\Sigma,\mu)$
 for every $t\ge0$. Thus, we can iterate~\eqref{eq:28} and obtain
  \begin{equation}
    \label{eq:29}
    \begin{split}
    &\norm{T_{t\sum_{\nu=0}^{n}t_{\nu}}u-T_{t\sum_{\nu=0}^{n}t_{\nu}}\hat{u}}_{\kappa m_{n+1}}\\
    &\le \prod_{\nu=1}^{n+1}C_{m_{\nu}}^{\frac{\kappa^{n-\nu}}{m_{n+1}}}\;
    \prod_{\nu=0}^{n} t_{\nu}^{-\frac{\alpha\,r}{\kappa}\,\frac{\kappa^{n-\nu}}{m_{n+1}}}\,
     e^{\omega (\beta-1) \frac{r}{\kappa m_{n+1}}\sum_{\nu=0}^{n}t_{\nu} \kappa^{n-\nu} t}\times\\
    &\times e^{\omega
      \frac{1}{m_{n+1}}\sum_{\nu=0}^{n}t_{\nu} \kappa^{n-\nu}
      m_{\nu+1} t}\,\prod_{\nu=1}^{n+1}
    C_{\varphi^{\ast},m_{\nu}}^{\frac{r}{m_{\nu}}\frac{\kappa^{n+1-\nu}m_{\nu-1}}{m_{n+1}}} 
    \,t^{-\frac{\alpha\,r}{\kappa\,m_{n+1}}\sum_{\nu=0}^{n}\kappa^{\nu}}\, 
      \norm{u-\hat{u}}_{\kappa m_{0}}^{m_{0}\frac{\kappa^{n+1}}{m_{n+1}}}.
    \end{split}
  \end{equation}
  Since by assumption, $\kappa>1$, by~\eqref{eq:27}, and
  by~\eqref{cond:qnod}, we see that
  \begin{equation}
    \label{eq:limit-q-sum-kappa}
    \lim_{n\to\infty} \frac{1}{m_{n+1}}\sum_{\nu=0}^{n}\kappa^{\nu}
    = \frac{1}{(\kappa-1) m_{0}+r\,\kappa^{-1}(1-\gamma)}.
  \end{equation}
  Thus
  \begin{equation}
    \label{lim:t-powers}
    \lim_{n\to\infty} t^{-\frac{\alpha\,r}{\kappa}\frac{1}{m_{n+1}}\sum_{\nu=0}^{n}\kappa^{n-\nu}}
    = t^{-\frac{\alpha\,r}{\kappa}\frac{1}{(\kappa-1)
        m_{0}+r\,\kappa^{-1}(1-\gamma)}}
      \qquad\text{for every $t>0$.}
  \end{equation}
  If we choose, for instance, $t_{\nu}=2^{-\nu-1}$, then
  \begin{displaymath}
    \prod_{\nu=0}^{n}t_{\nu}^{-\frac{\alpha r}{\kappa}\frac{\kappa^{n-\nu}}{m_{n+1}}}=
    2^{\frac{\alpha r}{\kappa}\frac{\kappa^{n}}{m_{n+1}}\sum_{\nu=0}^{n}(\nu+1)\kappa^{-\nu}}.
  \end{displaymath}
  Using
  \begin{displaymath}
    \sum_{\nu=0}^{\infty}(\nu+1)\kappa^{-\nu}=\frac{\kappa^2}{(\kappa-1)^2}
  \end{displaymath}
  and \eqref{rap}, one obtains
  \begin{equation}
    \label{eq:limit-sum-nu-plus-one-kappa}
    \lim_{n\to\infty}\frac{\kappa^{n}}{m_{n+1}}\sum_{\nu=0}^{n}(\nu+1)\kappa^{-\nu}
    = \frac{\kappa}{\kappa-1}\,\frac{1}{(\kappa-1) m_{0}+r\,\kappa^{-1}(1-\gamma)},
  \end{equation}
  therefore
  \begin{equation}
    \label{lim:t-nu-powers}
    \lim_{n\to\infty}\prod_{\nu=0}^{n}t_{\nu}^{-\frac{\alpha r}{\kappa}\frac{\kappa^{n-\nu}}{m_{n+1}}}=
    2^{\frac{\alpha r}{\kappa}\frac{\kappa}{\kappa-1}\,\frac{1}{(\kappa-1) m_{0}+r\,\kappa^{-1}(1-\gamma)}}.
  \end{equation}
  Using again that $t_{\nu}=2^{-\nu-1}$ together with~\eqref{eq:27} and
  ~\eqref{rap}, gives
  \begin{displaymath}
    \lim_{n\to\infty}\frac{1}{\kappa m_{n+1}}\sum_{\nu=0}^{n}t_{\nu}
    \kappa^{n-\nu}= 
    \frac{\kappa^{-1}(\kappa-1)(2\kappa-1)^{-1}}{(\kappa-1)m_{0}+r\kappa^{-1}(1-\gamma)}
  \end{displaymath}
  and so
  \begin{equation}
    \label{eq:74}
    \lim_{n\to\infty} e^{\omega (\beta-1) \frac{r}{\kappa
        m_{n+1}}\sum_{\nu=0}^{n}t_{\nu} \kappa^{n-\nu} t}=
    e^{\omega (\beta-1)
      \frac{r \kappa^{-1}(\kappa-1)(2\kappa-1)^{-1}}{(\kappa-1)m_{0}+r\kappa^{-1}(1-\gamma)} t}.
  \end{equation}
  Similarly, we obtain that
  \begin{displaymath}
    \lim_{n\to\infty} \frac{1}{m_{n+1}}\sum_{\nu=0}^{n}t_{\nu} \kappa^{n-\nu}
    m_{\nu+1}= 1- 
    \frac{r (1-\gamma) \kappa^{-1}(2\kappa-1)^{-1}}{(\kappa-1)m_{0}+r\kappa^{-1}(1-\gamma)}
  \end{displaymath}
  and so
  \begin{equation}
    \label{eq:73}
    \lim_{n\to\infty} e^{\omega \frac{1}{m_{n+1}}\sum_{\nu=0}^{n}t_{\nu} \kappa^{n-\nu}
      m_{\nu+1} t}=e^{\omega \big(1- 
      \frac{r (1-\gamma) \kappa^{-1}(2\kappa-1)^{-1}}{(\kappa-1)m_{0}+r\kappa^{-1}(1-\gamma)} \big) t}.
  \end{equation}
  Next, by \eqref{eq:4}, one has
  \begin{equation}
    \label{eq:147}
    C_{\varphi^{\ast},1}^{\frac{r \kappa^{n+1}}{m_{n+1}}\sum_{\nu=1}^{n+1}\frac{m_{\nu-1}}{\kappa^{\nu}m_{\nu}}} 
    \le \prod_{\nu=1}^{n+1}
    C_{\varphi^{\ast},m_{\nu}}^{\frac{r}{m_{\nu}}\frac{\kappa^{n+1-\nu}m_{\nu-1}}{m_{n+1}}} 
    \le
    \prod_{\nu=1}^{n+1} C_{\varphi^{\ast},2}^{\frac{r \kappa^{n+1}}{m_{n+1}}\sum_{\nu=1}^{n+1}\frac{m_{\nu-1}}{\kappa^{\nu}m_{\nu}}}.
  \end{equation}
  Since by~\eqref{eq:27}, one has that
  $a_{\nu}:=\frac{m_{\nu-1}}{\kappa^{\nu}m_{\nu}}$ satisfies
  $\lim_{\nu\to\infty}\abs{\frac{a_{\nu+1}}{a_{\nu}}}=\frac{1}{\kappa}$,
  the ratio test implies that the series
  $\sum_{\nu=1}^{\infty}\frac{m_{\nu-1}}{\kappa^{\nu}m_{\nu}}$
  converges. Furthermore, \eqref{rap} yields
  \begin{equation}
    \label{eq:39}
    \lim_{n\to\infty} \frac{\kappa^{n+1}}{m_{n+1}}=
    \frac{(\kappa-1)}{(\kappa-1)m_{0}+r\,\kappa^{-1}(1-\gamma)}.
  \end{equation}
  Thus
  \begin{displaymath}
    \lim_{n\to\infty} \frac{r
      \kappa^{n+1}}{m_{n+1}}\sum_{\nu=1}^{n+1}\frac{m_{\nu-1}}{\kappa^{\nu}m_{\nu}}
    =\frac{r
      (\kappa-1)}{(\kappa-1)m_{0}+r\,\kappa^{-1}(1-\gamma)}\sum_{\nu=1}^{\infty}
    \frac{m_{\nu-1}}{\kappa^{\nu}m_{\nu}}
  \end{displaymath}
  so that sending $n\to\infty$ in~\eqref{eq:147} yields
  \begin{equation}
    \label{eq:186}
    \begin{split}
      C_{\varphi^{\ast},1}^{\frac{r
          (\kappa-1)}{(\kappa-1)m_{0}+r\,\kappa^{-1}(1-\gamma)}\sum_{\nu=1}^{\infty}
        \frac{m_{\nu-1}}{\kappa^{\nu}m_{\nu}}} & \le
      \liminf_{n\to\infty}\prod_{\nu=1}^{n+1}
      C_{\varphi^{\ast},m_{\nu}}^{\frac{r}{m_{\nu}}\frac{\kappa^{n+1-\nu}m_{\nu-1}}{m_{n+1}}}\\
      & \le \limsup_{n\to\infty}\prod_{\nu=1}^{n+1}
      C_{\varphi^{\ast},m_{\nu}}^{\frac{r}{m_{\nu}}\frac{\kappa^{n+1-\nu}m_{\nu-1}}{m_{n+1}}}\\
      &\le C_{\varphi^{\ast},2}^{\frac{r
          (\kappa-1)}{(\kappa-1)m_{0}+r\,\kappa^{-1}(1-\gamma)}\sum_{\nu=1}^{\infty}
        \frac{m_{\nu-1}}{\kappa^{\nu}m_{\nu}}}.
  \end{split}
\end{equation}
  Using again~\eqref{eq:39}, we see that
  \begin{equation}
    \label{lim:norm-f-powers}
    \lim_{n\to\infty}\norm{u-\hat{u}}_{\kappa m_{0}}^{m_{0}\frac{\kappa^{n+1}}{m_{n+1}}}
    = \norm{u-\hat{u}}_{\kappa m_{0}}^{\frac{(\kappa-1)\, m_{0}}{(\kappa-1)m_{0}+r\,\kappa^{-1}(1-\gamma)}}.
  \end{equation}

  It remains to control the product
  \begin{equation}
    \label{eq:control-product}
    \begin{split}
      \prod_{\nu=1}^{n+1}C_{m_{\nu}}^{\frac{\kappa^{n-\nu}}{m_{n+1}}}=
      \prod_{\nu=1}^{n+1}\left[\tfrac{\beta\,m_{n}}{m_{n-1}\kappa}\right]^{\frac{\kappa^{n-\nu}}{m_{n+1}}}\times
      C^{\frac{r \sum_{\nu=1}^{n+1}\kappa^{n-\nu}}{m_{n+1}}}
      %\times \prod_{\nu=1}^{n+1} C_{\varphi^{\ast},m_{\nu}}^{\frac{r\,\kappa^{n-\nu}}{m_{n+1}}}
    \end{split}
  \end{equation}
  as $n\to\infty$. Since $\kappa>1$ and by~\eqref{rap},
  \begin{equation}\label{eq:3}
    \lim_{n\to\infty}\frac{1
    }{m_{n+1}}\sum_{\nu=1}^{n+1}\kappa^{n-\nu}= 
    \frac{\kappa^{-1}}{(\kappa-1)m_{0}+r\,\kappa^{-1}(1-\gamma)}
  \end{equation}
  and so
  \begin{equation}
    \label{lim:C-in-power-of-kappa-nu}
    \lim_{n\to\infty} C^{\frac{r
        \sum_{\nu=1}^{n+1}\kappa^{n-\nu}}{m_{n+1}}}= 
    C^{\frac{r\,\kappa^{-1}}{(\kappa-1)m_{0}+r\,\kappa^{-1}(1-\gamma)}}.
  \end{equation}
  For every $n\ge 1$, the quotient
  $\tfrac{\gamma\,m_{n}}{m_{n-1}\kappa}=\tfrac{\gamma\,m_{n}}{m_{n}+r\kappa^{-1}(\gamma-1)}$
  can be controlled by
  \begin{displaymath}
     \gamma<\tfrac{\gamma\,m_{n}}{m_{n-1}\kappa}<\frac{\gamma}{1+\frac{r}{m_{0}}\kappa^{-1}(\gamma-1)}
     \qquad\text{if $0<\gamma<1$}
  \end{displaymath}
  and by
  \begin{displaymath}
    \frac{\gamma}{1+\frac{r}{m_{0}}\kappa^{-1}(\gamma-1)}
    <\tfrac{\gamma\,m_{n}}{m_{n-1}\kappa}<\gamma
    \qquad\text{if $\gamma\ge 1$.}
  \end{displaymath}
  Thus for general $\gamma>0$, there are constants $C_{1}$, $C_{2}>0$
  such that
  \begin{equation}
    \label{eq:184}
   C_{1}^{\frac{1}{m_{n+1}}\sum_{\nu=1}^{n+1}\kappa^{n-\nu}} \le
   \prod_{\nu=1}^{n+1}\left[\tfrac{\gamma\,m_{n}}{m_{n-1}\kappa} \right]^{\frac{\kappa^{n-\nu}}{m_{n+1}}} \le 
   C_{2}^{\frac{1}{m_{n+1}}\sum_{\nu=1}^{n+1}\kappa^{n-\nu}}
  \end{equation}
  for every $n\ge 0$ and so by \eqref{eq:3}, sending $n\to\infty$ in
  \eqref{eq:184} yields
  \begin{displaymath}
    C_{1}^{\frac{\kappa^{-1}}{(\kappa-1)m_{0}+r\,\kappa^{-1}(1-\gamma)}} 
    \le \liminf_{n\to\infty}\prod_{\nu=1}^{n+1}
    \left[\tfrac{\gamma\,m_{n}}{m_{n-1}\kappa} \right]^{\frac{\kappa^{n-\nu}}{m_{n+1}}}
    \le
    \limsup_{n\to\infty}\prod_{\nu=1}^{n+1}C_{m_{\nu}}^{\frac{\kappa^{n-\nu}}{m_{n+1}}}
    \le C_{2}^{\frac{\kappa^{-1}}{(\kappa-1)m_{0}+r\,\kappa^{-1}(1-\gamma)}}.
  \end{displaymath}
  
  Thus sending $n\to\infty$ in inequality~\eqref{eq:29} and
  using~\eqref{lim:t-powers}, \eqref{lim:t-nu-powers}, \eqref{eq:74}, \eqref{eq:73},
  \eqref{lim:norm-f-powers}, \eqref{lim:C-in-power-of-kappa-nu}, \eqref{eq:186}
  together with the fact that $m_{n}\nearrow\infty$ as $n\to\infty$
  yields
  \begin{align*}
   \norm{T_{t}u-T_{t}\hat{u}}_{\infty}
   & \le
    \Big[C_{2}\,
    C^{r}\Big]^{\frac{r\,\kappa^{-1}}{(\kappa-1)m_{0}+r\,\kappa^{-1}(1-\gamma)}}\,e^{\omega
     \Big(\frac{(\beta-1)\kappa +
     \gamma-\beta}{(\kappa-1)m_{0}+r\,\kappa^{-1}(1-\gamma)}+1\Big) t}\,
     t^{-\frac{\alpha\,r\,\kappa^{-1}}{(\kappa-1)
        m_{0}+r\,\kappa^{-1}(1-\gamma)}}
     \times\\
   &\qquad\times C_{\varphi^{\ast},2}^{\frac{r
          (\kappa-1)}{(\kappa-1)m_{0}+r\,\kappa^{-1}(1-\gamma)}\sum_{\nu=1}^{\infty}
        \frac{m_{\nu-1}}{\kappa^{\nu}m_{\nu}}}\,
    \norm{u-\hat{u}}_{\kappa m_{0}}^{\frac{(\kappa-1)\, m_{0}}{(\kappa-1)m_{0}+r\,\kappa^{-1}(1-\gamma)}}
  \end{align*}
  showing that inequality~\eqref{eq:11} holds for $u$,
  $\hat{u}\in L^{\kappa m_{0}}\cap L^{\infty}(\Sigma,\mu)$. By
  hypothesis, the semigroup $\{T_{t}\}$ acts on
  $L^{\kappa m_{0}}\cap L^{\infty}(\Sigma,\mu)$, that is, every
  $T_{t}$ maps $ L^{\kappa m_{0}}\cap L^{\infty}(\Sigma,\mu)$ to $
  L^{\kappa m_{0}}\cap L^{\infty}(\Sigma,\mu)$. Since
  $L^{\kappa m_{0}}\cap L^{\infty}(\Sigma,\mu)$ is dense in
  $L^{\kappa m_{0}}(\Sigma,\mu)$, a standard approximation
  argument shows that the first claim of this iteration lemma
  holds. This completes the proof.
\end{proof}

Our second extrapolation result towards $L^{\infty}$ is adapted to
semigroups enjoying the $L^{q}$-$L^{r}$-regularising
effect~\eqref{eq:20} for $1\le q$, $r< \infty$ and some
$u_{0}\in L^{r}\cap L^{q_{0}}(\Sigma,\mu)$ generated by
quasi $m$-completely accretive operators $A$ on $L^{q}(\Sigma,\mu)$
(Section~\ref{sec:comp}).

\begin{theorem}\label{thm:extrapolation-to-infty-bis}
  Let $(\Sigma,\mu)$ be a separable measure space, $1\le q$,
  $r< \infty$, and $\{T_{t}\}_{t\geq 0}$ be a semigroup acting on
  $L^{q}\cap L^{\infty}(\Sigma,\mu)$ with exponential
  growth~\eqref{eq:62} for some $\omega\ge 0$ and every
  $q\le \tilde{q}\le \infty$. Further, suppose there exists
  $u_{0}\in L^{q}\cap L^{\infty}(\Sigma,\mu)$ satisfying $T_{t}u_{0}=u_{0}$ for all $t\ge 0$, 
  and there exist $C>0$ and exponents $\alpha$, $\beta$, $\gamma>0$ such that
  %~\eqref{eq:7bis}
  \begin{equation}
    \tag{\ref{eq:20}}
       \norm{T_{t}u-u_{0}}_{r}\le 
      \left(\tfrac{C}{q}\right)^{1/\sigma}\,t^{-\alpha}\,e^{\omega \beta t}\,
      \norm{u-u_{0}}_{q}^{\gamma}
   \end{equation}
  holds for every $t>0$ and
  $u \in L^{q}(\Sigma,\mu)\cap L^{\infty}(\Sigma,\mu)$. If the
  parameter $\gamma$, $r$, $q$ satisfy~\eqref{eq:70}, then
  \begin{displaymath}
   % \label{eq:1bis}
    \norm{T_{t}u-u_{0}}_{\infty}\lesssim \;e^{\omega
        (\gamma^{\ast}+1) t}\,
    t^{-\delta}\,
    \norm{u-u_{0}}_{\gamma\,r\,q^{-1}\,m_{0}}^{\gamma}
  \end{displaymath}
  for every $t>0$ and $u\in L^{\gamma\,r\,q^{-1}\,m_{0}}(\Sigma,\mu)$, where $\delta$,
  $\gamma$ and $\beta^{\ast}$ are given by~\eqref{eq:72} and
  $m_{0}\ge q\,\gamma^{-1}$ such that~\eqref{eq:71} holds.
\end{theorem}

The proof of this theorem proceeds analogously as the one for
Theorem~\ref{thm:extrapolation-to-infty}, where one replaces the
application of interpolation Theorem~\ref{thm:nonlinear-interpol} by
Theorem~\ref{thm:interpol-no-differences} or
Theorem~\ref{thm:interpol-no-differences-bis}. 
Furthermore, one applies the extrapolation argument from
Remark~\ref{rem:density-argument} and replaces
Lemma~\ref{lem:iteration} by the following one. We leave the details
of the proof to the interested reader.

\begin{lemma}\label{lem:iteration-bis}
  Suppose there exist $\kappa>1$, $\beta$, $\gamma>0$, $1\le r<\infty$ and
  $m_{0}\ge r\,K^{-1}$ such that~\eqref{cond:qnod} holds. Let
  $\{T_{t}\}_{t\geq 0}$ be a semigroup acting on
  $L^{\kappa m_{0}}(\Sigma,\mu)\cap L^{\infty}(\Sigma,\mu)$ satisfying
    \begin{displaymath}
    \begin{split}
    \norm{T_{t}u-u_{0}}_{m\,\kappa} &\le 
    \Big[\tfrac{\gamma\,m}{m+r\kappa^{-1}(\gamma-1)}\Big]^{\frac{1}{m\,\kappa}}\,
    \big[C\,e^{\omega \beta t}\,t^{-\alpha}\big]^{
      \frac{1}{m}\frac{r}{\kappa}}\,e^{\omega
      (1-\frac{1}{m}\frac{r}{\kappa}) t}\,\times\\
    & \hspace{3cm}\times \,C_{\varphi^{\ast},m}^{\frac{1}{m}\frac{r}{\kappa}(\frac{1}{m}\frac{r}{\kappa}(\gamma-1)+1)}\,
     \norm{u-u_{0}}_{m+r\,\kappa^{-1}(\gamma-1)}^{\frac{1}{m}r\,\kappa^{-1}(\gamma-1)+1}
  \end{split}
  \end{displaymath}
  for every $u \in L^{\kappa m_{0}}(\Sigma,\mu)\cap
  L^{\infty}(\Sigma,\mu)$, $t>0$, $m> m_{0}$ and some $u_{0}\in
  L^{\kappa m_{0}}\cap L^{\infty}(\Sigma,\mu)$, where $C_{\varphi^{\ast},m}$
  satisfies~\eqref{eq:4} for some constants $ C_{\varphi^{\ast}, 1}$,
  $ C_{\varphi^{\ast}, 2}>0$ independent of $m$. Then
  \begin{displaymath}
    %\label{eq:11bis}
    \begin{split}
      \norm{T_{t}u-u_{0}}_{\infty} &\lesssim \;e^{\omega 
        \big(\frac{(\beta-1)\kappa +
          \gamma-\beta}{(\kappa-1)m_{0}+r\,\kappa^{-1}(1-\gamma)}+1\big) t}\;
      t^{-\frac{\alpha\,r\,\kappa^{-1}}{(\kappa-1)
          m_{0}+r\,\kappa^{-1}(1-\gamma)}}\times\\
      &\hspace{4cm}\times \norm{u-u_{0}}_{\kappa
        m_{0}}^{\frac{(\kappa-1)\,
          m_{0}}{(\kappa-1)m_{0}+r\,\kappa^{-1}(1-\gamma)}}.
    \end{split}
   \end{displaymath}
  for every $u\in L^{\kappa m_{0}}(\Sigma,\mu)$ and every $t>0$. 
\end{lemma}

The proof of Lemma~\ref{lem:iteration-bis} proceeds as the one of
Lemma~\ref{lem:iteration}.  We omit the details.

\subsection{An alternative approach to arrive at $L^{\infty}$}
\label{sec:Alternative-Moser-iteration}
It is a fundamental fact that semigroups $\{T_{t}\}_{t\ge 0}$
generated by operators $-A$ in $L^{1}$ with a ($c$-)complete resolvent
( Section~\ref{subsec:L1}) are not, in general, contractive with
respect to the $L^{\infty}$-norm
(cf.~\cite[Section~A.11]{MR2286292}). Thus, if one wants to extend the
$L^{q}$-$L^{r}$-regularisation effect of $\{T_{t}\}_{t\ge 0}$ to an
$L^{q}$-$L^{\infty}$-regularisation effect, one needs to proceed by an
alternative approach. One possible way is the following one: firstly,
show that $A$ satisfies a \emph{one-parameter family of
  Gagliardo-Nirenberg type inequalities}, and then by employing
Theorem~\ref{thm:Sobolev-for-porousmedia}, deduce that the semigroup
$\{T_{t}\}$ satisfies a \emph{sequence of
  $L^{q_{n}}$-$L^{q_{n+1}}$-regularisation effects} for some sequence
$(q_{n})_{n\ge 1}\subseteq (1,\infty)$ with
$q_{n}\nearrow\infty$. This method has been employed in the past by
many authors. But to the best of our knowledge, V\'eron has been the
first to use this method in~\cite{MR554377} in the context of
nonlinear semigroups of \emph{contractive} mappings on
$L^{1}(\Sigma,\mu)$ (see also \cite{MR1277971} for another use of this
type of argument in linear semigroup theory). Here, we extend and
simplify this method to nonlinear semigroups $\{T_{t}\}_{t\ge 0}$ of
\emph{Lipschitz continuous} mappings $T_{t}$ on $L^{1}(\Sigma,\mu)$
with constant $e^{\omega t}$, in other words, of exponential
growth~\eqref{eq:62} for $q=1$.

\begin{theorem}
  \label{thm:Moser}
  Let $A+\omega I$ be $m$-accretive in
  $L^{1}(\Sigma,\mu)$ for some $\omega\ge 0$ with trace $A_{1\cap \infty}$
  of $A$ on $L^{1}\cap L^{\infty}(\Sigma,\mu)$ satisfying range
  condition~\eqref{eq:148}. Suppose there exist $\kappa>1$, $m>0$, and
  $q_{0}\ge p\ge 1$ such that $\kappa m q_{0}\ge 1$ and
  \begin{equation}
    \tag{\ref{eq:164}}
    (\kappa-1)q_{0}+p-1-\tfrac{1}{m}>0,
  \end{equation}
  and there exist $C>0$ and $(u_{0},0)\in A_{1\cap \infty}$ such that
  for every $q\ge q_{0}$, the trace $A_{1\cap \infty}$ satisfies
  Sobolev type inequality
  \begin{equation}
    \tag{\ref{eq:10single}}
    \begin{split}
      \norm{u-u_{0}}_{\kappa mq}^{mq}&\le \tfrac{C\,(q/p)^{p}}{q-p+1}\,
      \left[[u-u_{0},v]_{(q-p+1)m+1}+\omega\norm{u-u_{0}}_{(q-p+1)m+1}^{(q-p+1)m+1}\right]
  \end{split}
  \end{equation}
  for every $(u,v)\in A_{1\cap\infty}$, and for every $\lambda>0$
  satisfying $\lambda \omega<1$, the resolvent $J_{\lambda}$ of $A$
  satisfies~\eqref{eq:123} for $\tilde{q}=\kappa m q$. Then, there is
  a $\beta^{\ast}\ge 0$ such that the semigroup
  $\{T_{t}\}_{t\geq 0}\sim -A$ on $\overline{D(A)}^{\mbox{}_{L^{1}}}$
  satisfies
  \begin{equation}
    \label{eq:167}
      \norm{T_{t}u-u_{0}}_{\infty} \lesssim \;e^{\omega \beta^{\ast} t}\;
      t^{-\frac{1}{m((\kappa-1)q_{0}+p-1-\frac{1}{m})}}
      \norm{u-u_{0}}_{\kappa m q_{0}}^{\frac{(\kappa-1) q_{0}}{(\kappa-1)q_{0}+p-1-\frac{1}{m}}}
   \end{equation}
   for every $t>0$ and $u\in \overline{D(A)}^{\mbox{}_{L^{1}}}\cap L^{\infty}(\Sigma,\mu)$.
\end{theorem}

\begin{proof}
  From Theorem~\ref{thm:Sobolev-for-porousmedia}, we can conclude that
  the semigroup $\{T_{t}\}_{t\geq 0}$ satisfies
  inequality~\eqref{eq:163} below for every $t>0$,
  $u\in \overline{D(A)}^{\mbox{}_{L^{1}}}\cap L^{\infty}(\Sigma,\mu)$
  and $q\ge q_{0}$. Thus, we can deduce the claim of this theorem from
  the subsequent iteration Lemma~\ref{lem:iterationbis}.
\end{proof}

\begin{lemma}\label{lem:iterationbis}
  Suppose there exist $\kappa>1$, $m>0$, $q_{0}\ge p\ge 1$ such that
  $\kappa m q_{0}\ge 1$ and~\eqref{eq:164} hold. Furthermore, suppose,
  there exists $C>0$ such that the semigroup $\{T_{t}\}_{t\geq 0}$ on
  $\overline{D(A)}^{\mbox{}_{L^{1}}}$ satisfies
    \begin{equation}
      \label{eq:163}
      \begin{split}
        \norm{T_{t}u-u_{0}}_{\kappa qm}&\le
            \left[\tfrac{C\,(q/p)^{p}}{(q-p+1)
                  ((q-p+1)m+1)}\right]^{\frac{1}{qm}}\,
                   e^{\omega (\frac{(q-p+1)m+1}{qm} +1) t}\times\\
                  &\hspace{4cm} \times \,t^{-\frac{1}{qm}}\,
                   \norm{u-u_{0}}_{(q-p+1)m+1}^{\frac{(q-p+1)m+1}{qm}}
      \end{split}
  \end{equation}
  for every
  $u\in \overline{D(A)}^{\mbox{}_{L^{1}}}\cap
  L^{\infty}(\Sigma,\mu)$ and $q\ge q_{0}$. Then there is a $\beta^\ast\ge 0$
  such that the semigroup $\{T_{t}\}_{t\geq 0}\sim -A$ on
  $\overline{D(A)}^{\mbox{}_{L^{1}}}$ satisfies inequality~\eqref{eq:167} for every
  $u\in \overline{D(A)}^{\mbox{}_{L^{1}}}\cap L^{\infty}(\Sigma,\mu)$ and every
  $t>0$.
\end{lemma}

\begin{proof}
  We fix some $q_{0}\ge p$ and set
  \begin{displaymath}
    %\label{eq:23bis}
    q_{n+1}=\kappa\,q_{n}+p-1-\frac{1}{m}\qquad\text{for every $n\in \N_{0}$.}
  \end{displaymath}
  Then one can show by induction over $n\in \N_{0}$ that
  \begin{displaymath}
    %\label{eq:27ibis}
    q_{n}=\kappa^{n}(q_{0}-((p-1)-\tfrac{1}{m})) +
    ((p-1)-\tfrac{1}{m})\sum_{\nu=0}^{n}\kappa^{\nu},
  \end{displaymath}
  that is
   \begin{equation}
    \label{eq:27bis}
    q_{n}=\frac{\kappa^{n}}{\kappa-1}\left[(\kappa-1)q_{0}+p-1-\tfrac{1}{m}\right] 
    -\frac{p-1-\frac{1}{m}}{\kappa-1}.
  \end{equation}
  Using \eqref{eq:27bis}, we see that
  \begin{displaymath}
    q_{n+1}-q_{n}= \kappa^{n}\Big[(\kappa-1)q_{0}+p-1-\tfrac{1}{m}\Big]
  \end{displaymath}
  hence the sequence $(q_{n})_{n\ge 0}$ is strictly increasing if
  and only if $q_{0}$ satisfies condition~\eqref{cond:qnod}. Moreover,
  by~\eqref{eq:27bis}, since $\kappa >1$, and by~\eqref{cond:qnod}, we see
  that $q_{n}\to \infty$ as $n\to\infty$ and
  % \begin{equation}
  %   \label{eq:25bis}
  %   q_{n}=\frac{\kappa^{n}[(\kappa-1)q_{0}+p-2]}{\kappa-1} +
  %   \frac{2-p}{\kappa-1}\to \infty\qquad\text{as $n\to\infty$}
  % \end{equation}
  % and
  \begin{equation}
    \label{rapbis}
    \lim_{n\to\infty}\frac{q_{n}}{\kappa^{n}}
    = \frac{(\kappa-1)q_{0}+p-1-\frac{1}{m}}{\kappa-1}. 
  \end{equation}
  Now, let
  $u\in \overline{D(A)}^{\mbox{}_{L^{1}}}\cap
  L^{\kappa q_{0}m}(\Sigma,\mu)$. Then, by construction of $q_{1}$, we
  find that $u\in L^{(q_{1}-p+1)m+1}(\Sigma,\mu)$ and so by \eqref{eq:163}, $T_{t}u\in
  \overline{D(A)}^{\mbox{}_{L^{1}}}\cap L^{\kappa
    q_{1}m}(\Sigma,\mu)$ for all $t>0$. By construction of
  $(q_{n})_{n\ge 1}$ and since $\{T_{t}\}_{t\ge 0}$ is a semigroup satisfying~\eqref{eq:163},
  we see that $T_{t}u\in
  \overline{D(A)}^{\mbox{}_{L^{1}}}\cap L^{\kappa
    q_{n}m}(\Sigma,\mu)$ for all $n\ge 1$ and $t>0$. Thus, inserting the sequence
  $(q_{n})_{n\ge0}$ into~\eqref{eq:163} yields
  \begin{equation}
    \label{eq:28bis}
    \norm{T_{t}u-u_{0}}_{\kappa m q_{n+1}}\le
    C_{q_{n+1}}\,t^{-1/mq_{n+1}}\,e^{\omega (\frac{\kappa m q_{n}}{m q_{n+1}}+1) t}\, 
   \norm{u-u_{0}}_{\kappa m q_{n}}^{\kappa q_{n}/q_{n+1}}
  \end{equation}
  for every $t>0$ with
  \begin{displaymath}
    C_{q_{n+1}}= C^{\frac{1}{m q_{n+1}}}\,\left[\tfrac{
      (\frac{q_{n+1}}{p})^{p}}{(q_{n+1}-p+1) (\kappa
      m q_{n})}\right]^{\frac{1}{mq_{n+1}}}
  \end{displaymath}
  for every $n\in \N_{0}$. Let $(t_{\nu})_{\nu\ge0}$ be any
  sequence in $[0,1]$ such that
  $\sum_{\nu=0}^{\infty}t_{\nu}=1$, which will be specified below. Then
  by~\eqref{eq:28bis}, we obtain that
  \begin{equation}
    \label{eq:29bis}
    \begin{split}
      \norm{T_{t\sum_{\nu=0}^{n}t_{\nu}}u-u_{0}}_{\kappa q_{n+1}} &\le
      \prod_{\nu=1}^{n+1}C_{q_{\nu}}^{\kappa^{n+1-\nu}\frac{q_{\nu}}{q_{n+1}}}\,
      \prod_{\nu=0}^{n}t_{\nu}^{-\frac{\kappa^{n-\nu}}{mq_{n+1}}}\,
      t^{-\frac{1}{mq_{n+1}}\sum_{\nu=0}^{n}\kappa^{\nu}}\times\\ 
      &\qquad \times e^{\omega
     \sum_{\nu=0}^{n}t_{\nu}(\frac{\kappa q_{\nu}}{q_{\nu+1}}+1)
     \frac{\kappa^{n-\nu}q_{\nu+1}}{q_{n+1}} t}
      \,\norm{u-u_0}_{\kappa m q_{0}}^{\kappa^{n+1}q_{0}/q_{n+1}}.
    \end{split}
  \end{equation}
  Since $\kappa>1$, by~\eqref{rapbis}, $q_{n}\to\infty$ as $n\to
  \infty$ and since $(\kappa-1)q_{0}+p-2>0$ by
  assumption~\eqref{cond:qnod}, we see that
  \begin{equation}
    \label{eq:limit-q-sum-kappabis}
    \lim_{n\to\infty} \frac{1}{q_{n+1}}\sum_{\nu=0}^{n}\kappa^{\nu}
    = \frac{1}{(\kappa-1)q_{0}+p-1-\frac{1}{m}}.
  \end{equation}
  Thus
  \begin{equation}
    \label{lim:t-powersbis}
    \lim_{n\to\infty} t^{-\frac{1}{mq_{n+1}}\sum_{\nu=0}^{n}\kappa^{\nu}}
    = t^{-\frac{1}{m((\kappa-1)q_{0}+p-1-\frac{1}{m})}}\qquad\text{for every $t>0$.}
  \end{equation}
  If we choose, for instance, $t_{\nu}=2^{-\nu-1}$, then
  \begin{displaymath}
    \prod_{\nu=0}^{n}t_{\nu}^{-\frac{\kappa^{n-\nu}}{mq_{n+1}}}=
    2^{\frac{\kappa^{n}\sum_{\nu=0}^{n}(\nu+1)\kappa^{-\nu}}{mq_{n+1}}}.
  \end{displaymath}
  Using
  \begin{equation}
    \label{eq:165}
    \sum_{\nu=0}^{\infty}(\nu+1)\kappa^{-\nu}=\frac{\kappa^2}{(\kappa-1)^2}
  \end{equation}
  and \eqref{rapbis}, one obtains
  \begin{equation}
    \label{eq:limit-sum-nu-plus-one-kappabis}
    \lim_{n\to\infty}\frac{\kappa^{n}\sum_{\nu=0}^{n}(\nu+1)\kappa^{-\nu}}{mq_{n+1}}
    = \frac{\kappa}{m(\kappa-1)}\,\frac{1}{(\kappa-1)q_{0}+p-1-\frac{1}{m}}
  \end{equation}
  and so
  \begin{equation}
    \label{lim:t-nu-powersbis}
    \lim_{n\to\infty}\prod_{\nu=0}^{n}t_{\nu}^{-\frac{\kappa^{n-\nu}}{mq_{n+1}}}=
    2^{\frac{\kappa}{m(\kappa-1)[(\kappa-1)q_{0}+p-1-\frac{1}{m}]}}.
  \end{equation}
  Next, by~\eqref{rapbis}, we see that
  \begin{displaymath}
    \lim_{n\to\infty} \kappa^{n+1}\frac{q_{0}}{q_{n+1}}=
    \frac{(\kappa-1)q_{0}}{(\kappa-1)q_{0}+p-1-\frac{1}{m}},
  \end{displaymath}
  thus
  \begin{equation}
    \label{lim:norm-f-powersbis}
    \lim_{n\to\infty}\norm{u-u_{0}}_{\kappa m q_{0}}^{\kappa^{n+1}q_{0}/q_{n+1}}
    =
    \norm{u-u_{0}}^{(\kappa-1)q_{0}/((\kappa-1)q_{0}+p-1-\frac{1}{m})}_{\kappa
      m q_{0}}.
  \end{equation}
  
  Further, since for $a_{\nu}:=2^{-\nu}(\frac{\kappa
    q_{\nu}}{q_{\nu+1}}+1)\kappa^{-\nu}q_{\nu+1}$ for every
  $\nu\ge 0$, \eqref{eq:27bis} yields
  \begin{displaymath}
    \lim_{\nu\to\infty}\labs{\frac{a_{\nu+1}}{a_{\nu}}}=\frac{1}{2},
  \end{displaymath}
  the ratio test implies that the series
  \begin{displaymath}
    \frac{1}{2}\sum_{\nu=0}^{\infty}2^{-\nu}\left(\frac{\kappa
    q_{\nu}}{q_{\nu+1}}+1\right)\kappa^{-\nu}q_{\nu+1}
  \end{displaymath}
  converges; we denote the sum of the series by $S\ge 0$. Thus, by~\eqref{rapbis},
  \begin{displaymath}
    \lim_{n\to\infty} \tfrac{\kappa^{n+1}}{q_{n+1}}\frac{1}{2\kappa}
    \sum_{\nu=0}^{\infty}2^{-\nu}\left(\frac{\kappa
    q_{\nu}}{q_{\nu+1}}+1\right)\kappa^{-\nu}q_{\nu+1}= 
   \frac{\kappa-1}{(\kappa-1)q_{0}+p-1-\frac{1}{m}} \frac{1}{2\kappa}S=:\beta^{\ast} 
  \end{displaymath}
and so
\begin{equation}
  \label{eq:166}
  \lim_{n\to\infty}e^{\omega
     \sum_{\nu=0}^{n}t_{\nu}(\frac{\kappa q_{\nu}}{q_{\nu+1}}+1)
     \frac{\kappa^{n-\nu}q_{\nu+1}}{q_{n+1}} t}=e^{\omega \beta^{\ast} t}.
\end{equation}

  It remains to show that we can control the product
  \begin{equation}
    \label{eq:control-productbis}
    \prod_{\nu=1}^{n+1}C_{q_{\nu}}^{\frac{\kappa^{n+1-\nu}q_{\nu}}{q_{n+1}}}=
    \prod_{\nu=1}^{n+1}C^{\frac{\kappa^{n+1-\nu}}{m q_{n+1}}}\,\times
    \prod_{\nu=1}^{n+1} \left[\tfrac{
      (\frac{q_{n+1}}{p})}{(q_{n+1}-p+1) (\kappa
      m q_{n})}\right]^{\frac{\kappa^{n+1-\nu}}{mq_{n+1}}}
  \end{equation}
  as $n\to\infty$. First, note that
  \begin{displaymath}
    \prod_{\nu=1}^{n+1}C^{\frac{\kappa^{n+1-\nu}}{mq_{n+1}}}=e^{\frac{\log
        C}{mq_{n+1}}\sum_{\nu=1}^{n+1}\kappa^{n+1-\nu}}
    = e^{\frac{\log C}{mq_{n+1}}\sum_{\nu=0}^{n}\kappa^{\nu}}.
  \end{displaymath}
  Thus by~\eqref{eq:limit-q-sum-kappabis},
  \begin{equation}
    \label{lim:C-in-power-of-kappa-nubis}
    \lim_{n\to\infty}
       \prod_{\nu=1}^{n+1}C^{\frac{\kappa^{n+1-\nu}}{mq_{n+1}}}
       =C^{\frac{1}{m((\kappa-1)q_{0}+p-1-\frac{1}{m})}}.
  \end{equation}
  By~\eqref{eq:27bis}, we have that
  \begin{displaymath}
    q_{\nu}=\frac{\kappa^{\nu}[(\kappa-1)q_{0}
      +p-1-\frac{1}{m}]-((p-1)-\frac{1}{m})}{\kappa-1}
  \end{displaymath}
  for every $\nu\in \N_{0}$. From this, we conclude that
  \begin{equation}
    \label{eq:estimate-of-q-nubis}
    q_{\nu}\le M\,\kappa^{\nu}
  \end{equation}
  for every $\nu\in \N_{0}$, where
  \begin{displaymath}
   M:= 
    \begin{cases}
      \frac{[(\kappa-1)q_{0}+p-1-\frac{1}{m}]}{\kappa-1}
    & \text{if $(p-1)-\frac{1}{m}\ge 0$, and}\\
    q_{0}
    & \text{if $(p-1)-\frac{1}{m} <0$.}
    \end{cases}
  \end{displaymath}
  Applying~\eqref{eq:estimate-of-q-nubis} and using that $q_{\nu}\ge
  q_{0}\ge p>1$ and $\kappa m>1$, one sees that on the one hand
  \begin{displaymath}
    \prod_{\nu=1}^{n+1}\left[\tfrac{
        (q_{\nu}/p)^{p}}{(q_{\nu}-p+1)\kappa m q_{\nu}}
    \right]^{\frac{\kappa^{n+1-\nu}}{mq_{n+1}}}
    \le \prod_{\nu=1}^{n+1}
    \left(\tfrac{M}{p}\right)^{\frac{p\kappa^{n+1-\nu}}{mq_{n+1}}}\;
    \prod_{\nu=1}^{n+1}\kappa^{\frac{\nu\kappa^{n+1-\nu}}{mq_{n+1}}},
  \end{displaymath}
  and by~\eqref{eq:limit-q-sum-kappabis}, \eqref{eq:165},
  and~\eqref{rapbis}, one has
  \begin{displaymath}
    \lim_{n\to\infty}\prod_{\nu=1}^{n+1}
    \left(\tfrac{M}{p}\right)^{\frac{p\kappa^{n+1-\nu}}{mq_{n+1}}}=
    \left(\tfrac{M}{p}\right)^{\frac{p}{m((\kappa-1)q_{0}+p-1-\frac{1}{m})}}
  \end{displaymath}
  and
  \begin{displaymath}
    \lim_{n\to\infty}\prod_{\nu=1}^{n+1}\kappa^{\frac{\nu\kappa^{n+1-\nu}}{mq_{n+1}}}=
    \kappa^{\frac{\kappa}{(\kappa-1)m((\kappa-1)q_{0}+p-1-\frac{1}{m})}}.
  \end{displaymath}
  On the other hand, by using that $q_{\nu}-p+1\le q_{\nu}$,
  $q_{\nu}\ge p$ and~\eqref{eq:estimate-of-q-nubis}, one finds
  \begin{displaymath}
    \prod_{\nu=1}^{n+1}\left[\tfrac{
        (q_{\nu}/p)^{p}}{(q_{\nu}-p+1)\kappa m q_{\nu}}
    \right]^{\frac{\kappa^{n+1-\nu}}{mq_{n+1}}}\ge
    \prod_{\nu=1}^{n+1}
      (M^3 K )^{-\frac{\kappa^{n+1-\nu}}{mq_{n+1}}}\quad
    \prod_{\nu=1}^{n+1}\kappa^{-2\nu\frac{\kappa^{n+1-\nu}}{mq_{n+1}}}
  \end{displaymath}
  with
  \begin{displaymath}
    \lim_{n\to\infty}\prod_{\nu=1}^{n+1} (M^3 K )^{-\frac{\kappa^{n+1-\nu}}{mq_{n+1}}}
      = (M^3 K )^{-\frac{1}{m((\kappa-1)q_{0}+p-1-\frac{1}{m})}}
  \end{displaymath}
  and
  \begin{displaymath}
    \lim_{n\to\infty}\prod_{\nu=1}^{n+1}\kappa^{-2\nu\frac{\kappa^{n+1-\nu}}{mq_{n+1}}}=
    \kappa^{-\frac{2\kappa}{(\kappa-1)m((\kappa-1)q_{0}+p-1-\frac{1}{m})}}.
  \end{displaymath}
  Thus, by taking
  \begin{displaymath}
    M_{1}=\left(\tfrac{M}{p}\right)^{\frac{p}{m((\kappa-1)q_{0}+p-1-\frac{1}{m})}}\,
    \kappa^{\frac{\kappa}{(\kappa-1)m((\kappa-1)q_{0}+p-1-\frac{1}{m})}}
  \end{displaymath}
  and
  \begin{displaymath}
    M_{2} =  (M^3 K
    )^{-\frac{1}{m((\kappa-1)q_{0}+p-1-\frac{1}{m})}}\,
    \kappa^{-\frac{2\kappa}{(\kappa-1)m((\kappa-1)q_{0}+p-1-\frac{1}{m})}},
  \end{displaymath}
  we have that
  \begin{displaymath}
    \begin{split}
      0<M_{1}&\le \liminf_{n\to\infty}\prod_{\nu=1}^{n+1}\left[\tfrac{
          (q_{\nu}/p)^{p}}{(q_{\nu}-p+1)\kappa m
          q_{\nu}}\right]^{\frac{\kappa^{n+1-\nu}}{mq_{n+1}}} \\
      &\le
      \limsup_{n\to\infty}\prod_{\nu=1}^{n+1}\left[\tfrac{
          (q_{\nu}/p)^{p}}{(q_{\nu}-p+1)\kappa m
          q_{\nu}}\right]^{\frac{\kappa^{n+1-\nu}}{mq_{n+1}}}\le
      M_{2}<\infty
    \end{split}
  \end{displaymath}
   hence by ~\eqref{eq:control-productbis} and~\eqref{lim:C-in-power-of-kappa-nubis},
   \begin{displaymath}
     \begin{split}
       0<C^{\frac{1}{m((\kappa-1)q_{0}+p-1-\frac{1}{m})}}\,M_{1}
       &\le \liminf_{n\to\infty}\prod_{\nu=1}^{n+1}C_{q_{\nu}}^{\frac{\kappa^{n+1-\nu}q_{\nu}}{q_{n+1}}}\\
       &\le \limsup_{n\to\infty}\prod_{\nu=1}^{n+1}C_{q_{\nu}}^{\frac{\kappa^{n+1-\nu}q_{\nu}}{q_{n+1}}}
       \le  C^{\frac{1}{m((\kappa-1)q_{0}+p-1-\frac{1}{m})}}\,M_{2}<\infty.
     \end{split}
  \end{displaymath}
  Thus sending $n\to\infty$ in inequality~\eqref{eq:29bis}
  and using the limits~\eqref{lim:t-powersbis}, \eqref{lim:t-nu-powersbis},
  \eqref{lim:norm-f-powersbis} and~\eqref{eq:166} together with the fact that
  $q_{n}\nearrow\infty$ as $n\to\infty$ yields the desired
  inequality~\eqref{eq:167}. This completes the proof of the lemma.
\end{proof}

%%%%%%%%%%%%%%%%%%%%%%%%%%%%%%%%%%%%%%%%%%%%%%%%%%%
%%%%%%%%%%%%%%%%%%%%%%%%%%%%%%%%%%%%%%%%%%%%%%%%%%%
%%%%%%%%%%%%%%%%%%%%%%%%%%%%%%%%%%%%%%%%%%%%%%%%%%%
%%%%%%%%%%%%%%%%%%%%%%%%%%%%%%%%%%%%%%%%%%%%%%%%%%%
%%%%%%%%%%%%%%%%%%%%%%%%%%%%%%%%%%%%%%%%%%%%%%%%%%%
%%%%%%%%%%%%%%%%%%%%%%%%%%%%%%%%%%%%%%%%%%%%%%%%%%%
%%%%%%%%%%%%%%%%%%%%%%%%%%%%%%%%%%%%%%%%%%%%%%%%%%%
%%%%%%%%%%%%%%%%%%%%%%%%%%%%%%%%%%%%%%%%%%%%%%%%%%%
%
%                         Section - Application: Regularity of Mild solutions
%
%%%%%%%%%%%%%%%%%%%%%%%%%%%%%%%%%%%%%%%%%%%%%%%%%%%
%%%%%%%%%%%%%%%%%%%%%%%%%%%%%%%%%%%%%%%%%%%%%%%%%%%
%%%%%%%%%%%%%%%%%%%%%%%%%%%%%%%%%%%%%%%%%%%%%%%%%%%
%%%%%%%%%%%%%%%%%%%%%%%%%%%%%%%%%%%%%%%%%%%%%%%%%%%
%%%%%%%%%%%%%%%%%%%%%%%%%%%%%%%%%%%%%%%%%%%%%%%%%%%
%%%%%%%%%%%%%%%%%%%%%%%%%%%%%%%%%%%%%%%%%%%%%%%%%%%
%%%%%%%%%%%%%%%%%%%%%%%%%%%%%%%%%%%%%%%%%%%%%%%%%%%
%%%%%%%%%%%%%%%%%%%%%%%%%%%%%%%%%%%%%%%%%%%%%%%%%%%
%%%%%%%%%%%%%%%%%%%%%%%%%%%%%%%%%%%%%%%%%%%%%%%%%%%
%%%%%%%%%%%%%%%%%%%%%%%%%%%%%%%%%%%%%%%%%%%%%%%%%%%

\section{Application I: Mild solutions in $L^{1}$ are weak energy
  solutions}
\label{Asec:proof-of-Linfty-regularity}

This section is concerned with illustrating a first application of the
$L^{1}$-$L^{\infty}$ regularisation estimate
\begin{equation}
  \label{eq:263}
  \norm{T_{t}u}_{\infty}\le
  \tilde{C}\,t^{-\alpha}\,e^{\omega\,\beta\,t}\,\norm{u}_{1}^{\gamma},
  \qquad\text{holding for all $t>0$, $u\in \overline{D(A_{1\cap\infty}\phi)}^{\mbox{}_{L^{1}}}$}
\end{equation}
for some exponents $\alpha$, $\beta$, $\gamma>0$, and a
constant $\tilde{C}>0$, satisfied by the semigroup $\{T_{t}\}\sim
-(A_{1\cap\infty}\phi+F)$ on $\overline{D(A_{1\cap\infty}\phi)}^{\mbox{}_{L^{1}}}$. 

Let $A$ be an $m$-completely accretive operator on $L^{2}(\Sigma,\mu)$
and which is the realisation in $L^{2}(\Sigma,\mu)$ of a monotone
operator $\Psi' : V\to V'$ of a convex, G\^ateaux-differentiable
real-valued functional $\Psi$ defined on a reflexive Banach space $V$
(see the precise hypotheses on $A$ and $V$ below). Further,
\begin{enumerate}[($\mathcal{H}$a)]
\item \label{hyp:phi} let $\phi$ be a strictly increasing continuous
  functions on $\R$ with Yosida operator $\beta_{\lambda}$ of
  $\beta=\phi^{-1}$ satisfying condition~\eqref{eq:252},
  ($\lambda>0$),

\item \label{hyp:F} let $F$ be the Nemytski operator in
  $L^{q}(\Sigma,\mu)$, ($1\le q\le\infty$), of a Carath\'eodory
  function $f : \Sigma\times \R\to \R$ satisfying Lipschitz
  condition~\eqref{eq:2} for some $\omega\ge 0$ and
  $f(x,0)=0$ for a.e. $x\in \Sigma$.
\end{enumerate}

Then, the aim of this section is to show that for every initial value
$u_{0}\in \overline{D(A_{1\cap\infty}\phi)}^{\mbox{}_{L^{1}}}$, the
mild solution $u$ of Cauchy problem
  \begin{equation}
    \label{eq:149}
    \begin{cases}
      \frac{d u}{dt}+\overline{A_{1\cap\infty}\phi}u+F(u)=0&\qquad\text{in
        $L^{1}(\Sigma,\mu)$ on $(0,+\infty)$,}\\
      \mbox{}\hspace{2,75cm}u(0)=u_{0} & 
    \end{cases}
  \end{equation}
  is, in fact, a \emph{weak energy solution} of problem
  \begin{equation}
   \label{eq:161}
   \begin{cases}
    \frac{d u}{dt}+\Psi'(\phi(u))+F(u)=0&\qquad\text{in $V'$ on
      $(0,\infty)$,}\\
     \mbox{}\hspace{2,9cm}u(0)=u_{0}&
   \end{cases}
 \end{equation}
in the sense of Definition~\ref{Def:weak-sols-in-V} below.\bigskip

In this section, we work in the following framework. We
assume that the classical Lebesgue space $L^{q}(\Sigma,\mu)$,
$1\le q\le \infty$, is defined on a finite measure space
$(\Sigma,\mu)$, and $V$ be a reflexive Banach space such that there
are $a\ge 0$ and a semi-norm $\abs{\cdot}_{V}$ on $V$ such that
\begin{displaymath}
  \abs{\cdot}_{V}+a \norm{\cdot}_{2}
\end{displaymath}
defines an equivalent norm on $V$. Further, suppose the continuous
embedding $i :V \to L^{2}(\Sigma,\mu)$ has a dense image. Then,
the adjoint operator $i^{\ast}$ of $i$ from the dual space $(L^{2}(\Sigma,\mu))'$ of
$L^{2}(\Sigma,\mu)$ to the dual space $V'$ of $V$ is also an injective
linear bounded operator. After identifying $L^{2}(\Sigma,\mu)$ with
$(L^{2}(\Sigma,\mu))'$, we see that
\begin{displaymath}
  V\hookrightarrow L^{2}(\Sigma,\mu)\hookrightarrow V',
\end{displaymath}
where each inclusion ''$\hookrightarrow$'' denotes a continuous
embedding with a dense image. Moreover,
the duality brackets $\langle \cdot,\cdot\rangle_{V',V}$
of $V\times V'$ and the inner product $\langle\cdot,\cdot\rangle$ on
$L^{2}(\Sigma,\mu)$ coincide whenever both make sense
(cf.~\cite[Remark~3, Chapter~5.2]{MR2759829}), that is
\begin{equation}
  \label{eq:214}
  \langle u,v\rangle_{V',V}=\langle u,v\rangle\qquad\text{for all
    $u\in L^{2}(\Sigma,\mu)$ and $v\in V$.}
\end{equation}
Thus, in order to keep our notation simple, we only employ
the brackets $\langle\cdot,\cdot\rangle$.\bigskip

Further, we assume that $\Psi : V\to \R$ is a convex, lower
semicontinuous, and G\^ateaux differentiable functional satisfying
  \begin{enumerate}[($\mathcal{H}$i):\;]
    
     \item\label{hyp:a} \emph{there are $1<p<\infty$, $\eta>0$, $C>0$ such that}
       \begin{align}\label{ineq:p-coercive}
      & \langle \Psi'(v),v\rangle\ge \eta
        \abs{v}^{p}_{V}\qquad\textit{and}\\
       \label{ineq:bounded-in-V}
      & \norm{\Psi'(v)}_{V'}\le C\,\abs{v}_{V}^{p-1}\qquad\textit{for every $v\in V$,}
   \end{align}

   \item\label{hyp:b} \emph{$\Psi' : V\to V'$ is hemicontinuous, that is, for every $u$,
     $v$, $w\in V$, the function $\lambda\mapsto \langle
     \Psi'(v+\lambda u),w\rangle$ is continuous on $\R$,}

   \item\label{hyp:d} \emph{there is a $\varepsilon>0$ such that
     $\Psi+\varepsilon \norm{\cdot}_{2}^{2}$ is weakly
       coercive in $V$, that is, for every $c\in \R$, the sub-level set $E_{c}:=\{v\in
V\,\vert\,\Psi(v)+\varepsilon \norm{v}_{2}^{2}\le c\}$ is relatively compact with respect to
the weak topology on $V$.}
 
   \item\label{hyp:c} \emph{the subgradient
       $A:=\partial_{L^{2}}\Psi^{L^{2}}$ in $L^{2}(\Sigma,\mu)$ of the
       extended functional $\Psi^{L^{2}}$ of $\Psi$ on
       $L^{2}(\Sigma,\mu)$ is an $m$-completely accretive operator in
       $L^{2}(\Sigma,\mu)$.}  

   \item\label{hyp:e} \emph{the functional $\Psi$ is related to a ''Poincar\'e type inequality''}
\begin{equation}
  \label{eq:259}
  \norm{u}_{p}\le C\,\Psi(u)\qquad\text{for all $u\in V$,}
\end{equation}
\emph{where the constant $C>0$ is independent of $u\in V$.}  
\end{enumerate}

\begin{remark}
  We note that hypothesis~($\mathcal{H}$\ref{hyp:d}) is needed only to ensure that
  that the extended functional $\Psi^{L^{2}}$ of $\Psi$ on
  $L^{2}(\Sigma,\mu)$ is lower semicontinuous on
  $L^{2}(\Sigma,\mu)$. For a more detailed discussion on this, we
  refer the interested reader to~\cite{arXiv:1412.4151}.
\end{remark}

 \begin{definition}
 \label{Def:weak-sols-in-V}
 Let $1<p<\infty$ with conjugate exponent $p'=\frac{p}{p-1}$, $T>0$,
 the operator $\Psi' : V\to V'$ satisfy the
 hypotheses~($\mathcal{H}$\ref{hyp:a}) and ($\mathcal{H}$\ref{hyp:b}),
 and $\phi$ be a continuous function on $\R$. Then for given
 $u_{0}\in \overline{D(A_{1\cap\infty}\phi)}^{\mbox{}_{L^{1}}}\cap L^{\infty}(\Sigma,\mu)$, we call a function
 $u\in C([0,T];L^{1}(\Sigma,\mu))$ a \emph{weak energy solution}
 of~\eqref{eq:161} if $u(0)=u_{0}$ in $L^{1}(\Sigma,\mu)$, and for
 every $0<\delta<T$,
  \begin{displaymath}
    \frac{d u}{dt}\in L^{p^{\mbox{}_{\prime}}}(\delta,T;V'),\qquad%\text{and}\qquad
    \phi(u) \in L^{p}(\delta,T;V),
  \end{displaymath}
   and 
  \begin{displaymath}
    \int_{\delta}^{T}\Big\{\Big\langle \frac{d u}{dt},v\Big\rangle_{V',V} +
    \langle \Psi'\phi(u(t)),v\rangle +\langle F(u(t)),v\rangle \Big\}\,\dt=0
  \end{displaymath}
  for all $v\in L^{p}(\delta,T;V)$.
\end{definition}

\begin{remark}
  To the best of our knowledge, the notion of \emph{weak energy
    solution} was introduced in~\cite[Section~5.3.2]{MR2286292} in
  connection with the porous media operator $A\phi=\Delta\phi$. The
  word \emph{energy} in this notion indicates that the solution $u$ of~\eqref{eq:161} 
 has the property
  \begin{displaymath}
     \int_{\delta}^{t}\Psi(\phi(u(t)))\,\dmu\,\ds\qquad\text{is finite
       for every $0<\delta<t$.}
  \end{displaymath}
\end{remark}

\begin{remark}
  We note that under the additional assumptions that $\phi$ is
  non-decreas\-ing on $\R$ and the weak energy solution $u$ of
  problem~\eqref{eq:161} is in
  $L^{\infty}(0,T;L^{\infty}(\Sigma,\mu))$,
  Lemma~\ref{lem:3} below yields for the primitive 
  \begin{equation}
    \label{eq:276}
    \Phi(s):=\int_{0}^{s}\phi(r)\dr,\qquad\text{($s\in \R$),}
  \end{equation}
  of $\phi$ that
  \begin{displaymath}
    \int_{\Sigma}\Phi(u)\,\dmu\in W^{1,1}(0,1)
  \end{displaymath}
   and \emph{integration by parts rule}~\eqref{eq:215} from
   Section~\ref{subsec:weak-sol-for-initial-in-Linfty} holds.
\end{remark}

\begin{definition}
 For a given continuous function $\phi$ on $\R$ satisfying $\phi(0)=0$
 and every $\varepsilon>0$, we call the function
 \begin{displaymath}
   \phi_{\varepsilon}(s):=\tfrac{1}{\varepsilon}\int_{\R}\phi(r)\,
   \rho(\tfrac{s-r}{\varepsilon})\,dr+\varepsilon
   s+ c_{\varepsilon}\qquad\text{for every $s\in \R$,}
 \end{displaymath}
 the \emph{regularisation of $\phi$}. Here, the constant
 $c_{\varepsilon}$ is chosen such that $\phi_{\varepsilon}(0)=0$ and
 $\rho\in C^{\infty}(\R)$, $\rho\ge0$, $\int_{\R}\rho\dr=1$,
 $\rho\equiv0$ on $\R\setminus [-1,1]$.
\end{definition}

The following theorem is the main result of this section, which
provides sufficient conditions that mild solutions are weak
  energy solutions.

\begin{theorem}
  \label{thm:weak-solutions-Li-Linfty}
  Let $\Psi : V\to \R$ be a convex, lower semicontinuous, and
  G\^ateaux differentiable functional satisfying the hypotheses
  ($\mathcal{H}$\ref{hyp:a})-($\mathcal{H}$\ref{hyp:e}), $\phi$ be a
  strictly increasing continuous function on $\R$ satisfying~($\mathcal{H}$\ref{hyp:phi})
  and $F$ be an operator on $L^{q}(\Sigma,\mu)$ satisfying~($\mathcal{H}$\ref{hyp:F}).
  Further, suppose that there are exponents $\alpha$, $\beta$,
  $\gamma>0$ and a constant $\tilde{C}>0$ such that the semigroup
  $\{T_{t}\}\sim - (\overline{A_{1\cap\infty}\phi}+F)$ on
  $\overline{D(A_{1\cap\infty}\phi)}^{\mbox{}_{L^{1}}}$ satisfies
  $L^{1}$-$L^{\infty}$ regularisation estimate~\eqref{eq:259}. Then,
  the following statements hold:
  \begin{enumerate}
    \item For every initial value
    $u_{0}\in \overline{D(A_{1\cap\infty}\phi)}^{\mbox{}_{L^{1}}}$, the
    mild solution $u(t)=T_{t}u_{0}$, $(t\ge 0)$, of Cauchy
    problem~\eqref{eq:149} in $L^{1}(\Sigma,\mu)$ is a weak energy
    solution of Cauchy problem~\eqref{eq:161} satisfying
    \begin{equation}
      \label{eq:260}
      \begin{split}
        &\tfrac{1}{2} \int_{0}^{t} s^{\alpha(p^{\mbox{}_{\prime}}-1)+1} \Psi(\phi(u(s)))\,\ds +
        t^{\alpha(p^{\mbox{}_{\prime}}-1)+1} \int_{\Sigma} \Phi(u(t))\,\dmu\\
        &\qquad\le
        \tfrac{(\alpha(p^{\mbox{}_{\prime}}-1)+1)^{p}\,C^{p-1}}{p^{\mbox{}_{\prime}}(4^{-1}\,p)^{p-1}}\,
        \int_{0}^{t}
        e^{\omega\,(\beta\,(p^{\mbox{}_{\prime}}-1)+1)\,s}\,\ds
        \,\norm{u_{0}}_{1}^{\gamma\,(p^{\mbox{}_{\prime}}-1)+1}\\
        &\hspace{3cm}
        +\tfrac{\omega^{p}\,C^{p-1}}{p^{\mbox{}_{\prime}}(4^{-1}\,p)^{p-1}}\,
        \int_{0}^{t}s\,e^{\omega\,(\beta\,(p^{\mbox{}_{\prime}}-1)+1)\,s}\,\ds
        \,\norm{u_{0}}_{1}^{\gamma\,(p^{\mbox{}_{\prime}}-1)+1}.
      \end{split}
    \end{equation}
    for every $t>0$.
    
    \item If, in addition, $\phi'\in L^{\infty}(\R)$, $\phi^{-1}$
      is locally bounded and $0<\alpha\le 1$ in estimate~\eqref{eq:259}, then for every initial value
    $u_{0}\in \overline{D(A_{1\cap\infty}\phi)}^{\mbox{}_{L^{1}}}$, the
    mild solution $u(t)=T_{t}u_{0}$, $(t\ge 0)$, of Cauchy
    problem~\eqref{eq:149} in $L^{1}(\Sigma,\mu)$ is a strong solution
    of~\eqref{eq:149} in $L^{1}(\Sigma,\mu)$ with the following properties:
    \begin{enumerate}
      \item One has
        \begin{displaymath}
          u\in W^{1,2}_{loc}((0,\infty);L^{2}(\Sigma,\mu))
        \end{displaymath}
    
     \item  the function
    $\phi(u)\in W^{1,2}_{loc}((0,T];L^{2}(\Sigma,\mu))$ with weak
    derivative
    \begin{equation}
      \label{eq:150}
      \frac{d}{dt}\phi(u(t))=\phi'(u(t))\frac{d
        u}{dt}(t)\qquad\text{in $L^{2}(\Sigma,\mu)$ for a.e. $t>0$,}
    \end{equation}

   \item for a.e. $t>0$, one has $\phi(u(t))\in D(A)$ and
    \begin{equation}
      \label{eq:152}
     \tfrac{d u}{dt}(t)+A\phi(u(t))+F(u(t))\ni 0\qquad\text{in $L^{2}(\Sigma,\mu)$,}
    \end{equation}

   \item the real-valued function $t\mapsto \Psi(\phi(u(t)))$ is
     locally absolutely continuous on $(0,\infty)$ satisfying for a.e. $t>0$,
     \begin{equation}
       \label{eq:267}
         \tfrac{d}{dt}\Psi(\phi(u(t)))=- \norm{\tfrac{d
             u}{dt}(t)\sqrt{\phi'(u(t))}}_{2}^{2}- \langle
         F(u(t)),\tfrac{d u}{dt}(t)\,\phi'(u(t)) \rangle,
     \end{equation} 

     \item for every $t>0$,
       \begin{equation}
         \label{eq:268}
         \begin{split}
          &\tfrac{1}{2}\int_{0}^{t}s^{\alpha(p^{\mbox{}_{\prime}}-1)+2} \int_{\Sigma}
          \phi'(u(s))\,\labs{\frac{d
              u}{ds}(s)}^{2}\,\dmu\,\ds +
          t^{\alpha(p^{\mbox{}_{\prime}}-1)+2}\,\Psi(\phi(u(t)))\\
          &\qquad\le \tfrac{(\alpha(p^{\mbox{}_{\prime}}-1)+2)\,2}{p^{\mbox{}_{\prime}}}\,\left[
        \tfrac{(k+1)^{p}\,C^{p-1}}{(4^{-1}\,p)^{p-1}}\,
        \int_{0}^{t}e^{\omega\,(\beta\,(p^{\mbox{}_{\prime}}-1)+1)\,s}\,\ds
        \,\norm{u_{0}}_{1}^{\gamma\,(p^{\mbox{}_{\prime}}-1)+1}\right.\\
    &\hspace{3cm}
        +\tfrac{\omega^{p}\,C^{p-1}}{(4^{-1}\,p)^{p-1}}\,
        \int_{0}^{t}s\,e^{\omega\,(\beta\,(p^{\mbox{}_{\prime}}-1)+1)\,s}\,\ds
        \,\norm{u_{0}}_{1}^{\gamma\,(p^{\mbox{}_{\prime}}-1)+1}\\
  & \hspace{3,5cm}\left.
    + \tfrac{(\alpha(p^{\mbox{}_{\prime}}-1)+1)^{p}\,C^{p}}{p^{\mbox{}_{\prime}}(2^{-1}\,p)^{p-1}}\,
      \int_{0}^{t}e^{\omega\,\beta (p^{\mbox{}_{\prime}}-1)\,s}\,\ds
    \norm{u_{0}}_{1}^{\gamma (p^{\mbox{}_{\prime}}-1)+1}\,\right]\\
   &\hspace{4,5cm} +
     \tfrac{\omega^{2}\,\tilde{C} }{2}\norm{\phi'}_{\infty}
     \int_{0}^{t}s^{1-\alpha}\,e^{\omega(\beta+1)\,s }\,\ds\,\norm{u_{0}}_{1}^{\gamma+1}.
          \end{split}
       \end{equation}
    \end{enumerate}    
  \end{enumerate}
 \end{theorem}

 The proof of Theorem~\ref{thm:reg-mild-sol-are-weak} is divided into
 three steps. The first step is to consider the \emph{smooth case},
 that is, under the assumption that $\phi$ and its inverse $\phi^{-1}$
 are locally Lipschitz continuous (see
 Theorem~\ref{thm:Linfty-implies-mild-are-strong-in-L1}). Then, in the
 second step, we consider a general \emph{continuous strictly increasing}
 function $\phi$ but we take initial values
 $u_{0}\in \overline{D(A_{1\cap\infty}\phi)}^{\mbox{}_{L^{1}}}\cap
 L^{\infty}(\Sigma,\mu)$ (see Theorem~\ref{thm:reg-mild-sol-are-weak}).
 In the last and third step, one uses the estimates established in
 step two to conclude by using the continuous dependence of the
 semigroup $\{T_{t}\}_{t\ge 0}$ and the its $L^{1}$-$L^{\infty}$
 regularisation effect to conclude the statement of the main
 theorem (Theorem~\ref{thm:weak-solutions-Li-Linfty}).

%%%%%%%%%%%%%%%%%%%%%%%%%%%%%%%%%%%%%%%%%%%%%%%%%%%%%%%%
%
%                      Theorem : $\phi^{-1}$ and $\phi$ locally Lipschitz
%                      imply that mild solutions are strong
%
%%%%%%%%%%%%%%%%%%%%%%%%%%%%%%%%%%%%%%%%%%%%%%%%%%%%%%%%%
\subsection{The smooth case}
\label{subsec:smooth-case}

We begin by considering the smooth case. Here, the statement of our
following theorem confirms positively a conjecture stated in
\cite[Remarque~2.13]{Benthesis} and generalises the results
in~\cite[Proposition~2.18]{Benthesis} and partially some results in
\cite[Section~3]{MR0454360} to the general subgradient setting. Our
next theorem is the main results in this subsection.

\begin{theorem}
  \label{thm:Linfty-implies-mild-are-strong-in-L1}
  Let $\Psi : V\to \R$ be a convex, lower semicontinuous, and
  G\^ateaux differentiable functional satisfying
  ($\mathcal{H}$\ref{hyp:d}) and ($\mathcal{H}$\ref{hyp:c}). Further,
  let $\phi$ be a strictly increasing function on $\R$ such that
  $\phi$ and $\phi^{-1}$ are locally Lipschitz continuous, and the
  Yosida operator $\beta_{\lambda}$ of $\beta=\phi^{-1}$ satisfies
  condition~\eqref{eq:252}, ($\lambda>0$), and $F$ be an operator on
  $L^{q}(\Sigma,\mu)$ satisfying~($\mathcal{H}$\ref{hyp:F}). We set
  $\Phi(r)=\int_{0}^{r}\phi(s)\,\ds$ for every $r\in \R$.  Then, for
  every $u_{0}\in \overline{D(A_{1\cap\infty}\phi)}^{\mbox{}_{L^{1}}} \cap
  L^{\infty}(\Sigma,\mu)$,
  the mild solution $u(t)=T_{t}u_{0}$, $(t\ge 0)$, of
  problem~\eqref{eq:149} in $L^{1}(\Sigma,\mu)$ is a strong solution
  of
  \begin{equation}
    \label{eq:206}
    \tfrac{d u}{dt}+A\phi(u)+F(u)\ni 0\quad\text{in
      $L^{2}(\Sigma,\mu)$ on
      $(0,T)$,}\qquad u(0)=u_{0}
  \end{equation}
  with the regularity
  \begin{equation}
    \label{eq:124}
    u \in C([0,T];L^{q}(\Sigma,\mu))\cap
    W^{1,2}_{loc}((0,T];L^{2}(\Sigma,\mu))\cap L^{\infty}([0,T];L^{\infty}(\Sigma,\mu))
  \end{equation}
  for every $1\le q<\infty$ and satisfying
  \begin{enumerate}
  \item the function
    $\phi(u)\in W^{1,2}_{loc}((0,T];L^{2}(\Sigma,\mu))$ with weak
    derivative~\eqref{eq:150},

   \item for a.e. $t>0$, one has $\phi(u(t))\in D(A)$
     and~\eqref{eq:152} holds,
 
   \item the real-valued function $t\mapsto \Psi(\phi(u(t)))$ is
     locally absolutely continuous on $(0,\infty)$
     satisfying~\eqref{eq:267} for a.e. $t>0$,
    
    \item %inequality~\eqref{eq:140} holds and 
      for every $k\ge 0$ and $t>0$, one has
      \begin{equation}
        \label{eq:140}
        \begin{split}
          &\int_{0}^{t}s^{k+1}\Psi(\phi(u(s)))\,\ds +
          t^{k+1} \int_{\Sigma}\Phi(u(t))\,\dmu\\
          &\qquad\le
          (k+1)\int_{0}^{t}s^{k}\int_{\Sigma}\phi(u(t))\,u(t)\,\dmu\ds\\
          &\hspace{4cm} -\int_{0}^{t}s^{k+1}\int_{\Sigma} F(u(s))\phi(u(s))\,\dmu\,\ds.
        \end{split}
      \end{equation}
      and
      \begin{equation}
        \label{eq:139}
        \begin{split}
          &\tfrac{1}{2}\int_{0}^{t}s^{k+2} \int_{\Sigma}
          \phi'(u(s))\,\labs{\frac{d
              u}{ds}(s)}^{2}\,\dmu\,\ds +
          t^{k+2}\,\Psi(\phi(u(t)))\\
          &\qquad \le (k+2)(k+1) \int_{0}^{t}s^{k}\int_{\Sigma}\phi(u(t))\,u(t)\,\dmu\,\ds\\
          &\hspace{3cm} -(k+2) \int_{0}^{t}s^{k+1}\int_{\Sigma}
          F(u(s))\phi(u(s))\,\dmu\,\ds\\
          &\hspace{4cm} +
          \tfrac{1}{2}\int_{0}^{t}s\int_{\Sigma}\phi'(u(s))\,
          \labs{F(u(s))}^{2}\,\dmu\,\ds.
        \end{split}
      \end{equation}
    \end{enumerate}

% where $L_{\phi}>0$ denotes the Lipschitz constant of $\phi$ on
% $[-M,M]$ and $M:=e^{\omega T}\norm{u_{0}}_{\infty}$. 
\end{theorem}

Before outlining the proof of
Theorem~\ref{thm:Linfty-implies-mild-are-strong-in-L1}, we recall the
following \emph{convergence result} (cf.~\cite[Proposition 4.4 \&
Theorem~4.14]{MR2582280}), which we state in a version suitable for
the framework of this paper.

\begin{theorem}
  \label{thm:convergence-result}
  For $\omega\in \R$ and $1\le q\le \infty$, let $(A_{n})_{n\ge 1}$ be
  a sequence of operators $A_{n}$ on $L^{q}(\Sigma,\mu)$ such that $A_{n}+\omega I$ is
  accretive in $L^{q}(\Sigma,\mu)$. For given $u_{0,n}\in
  \overline{D(A_{n})}^{\mbox{}_{L^{q}}}$, let $u_{n}$ be the unique
  mild solution of initial value problem
  \begin{displaymath}
    \tfrac{d u_{n}}{dt}+A_{n}u_{n}\ni 0\quad\text{ on $(0,T)$ and
    }\quad u_{n}(0)=u_{0,n}.
  \end{displaymath}
  Further, let $A$ be an operator on $L^{q}(\Sigma,\mu)$ such that $A+\omega I$ is
  accretive in $L^{q}(\Sigma,\mu)$ and for given $u\in
  \overline{D(A)}^{\mbox{}_{L^{q}}}$, let $u$ be the unique mild
  solution of
  \begin{displaymath}
    \tfrac{d u}{dt}+Au\ni 0\quad\text{ on $(0,T)$ and
    }\quad u(0)=u_{0}.
  \end{displaymath}
  Suppose that $\lim_{n\to\infty}u_{0,n}=u_{0}$ in $L^{q}(\Sigma,\mu)$
  and for every $\lambda>0$
  satisfying $\lambda \omega<1$, the resolvent $J_{\lambda}^{A}$ of
  $A$ and the resolvent $J_{\lambda}^{A_{n}}$ of $A_{n}$ satisfy
  \begin{equation}
    \label{eq:116}
   \lim_{n\to\infty} J_{\lambda}^{A_{n}}x\to J_{\lambda}^{A}x\qquad\text{in
      $L^{q}(\Sigma,\mu)$ for every $x\in L^{q}(\Sigma,\mu)$,}
  \end{equation}
  then $u_{n}\to u$ in $C([0,T];L^{q}(\Sigma,\mu))$.
\end{theorem}

Our proof of
Theorem~\ref{thm:Linfty-implies-mild-are-strong-in-L1}
improves an idea from~\cite{Benthesis}.
\allowdisplaybreaks

\begin{proof}[Proof of Theorem~\ref{thm:Linfty-implies-mild-are-strong-in-L1}]
  By Proposition~\ref{propo:Lipschitz-complete},
  $\overline{A_{1\cap\infty}\phi}+F$ is $m$-accretive in $L^{1}$ with
  complete resolvent and for every $\lambda>0$ such that
  $\omega\lambda<1$, $A_{1\cap\infty}\phi+F$ satisfies range
  condition~\eqref{eq:201}.  In particular, $-(\overline{A\phi}+F)$
  generates a strongly continuous semigroup $\{T_{t}\}_{t\ge0}$ on
  $\overline{D(A_{1\cap\infty}\phi)}^{\mbox{}_{L^{1}}}$. Therefore, for
  $u_{0}\in \overline{D(A_{1\cap\infty}\phi)}^{\mbox{}_{L^{1}}} \cap
  L^{\infty}(\Sigma,\mu)$, the function $u(t):=T_{t}u_{0}$ for every $t\ge0$ is the unique
  mild solution of problem~\eqref{eq:149} in $L^{1}(\Sigma,\mu)$ and by
  Proposition~\ref{propo:quasi-accretive-operators-in-L1-complete-resolvent},
  one has
  \begin{equation}
    \label{eq:132}
    \norm{u(t)}_{q}\le e^{\omega
      t}\norm{u_{0}}_{q}\qquad\text{for every $t\ge 0$ and $1\le q\le \infty$,}
  \end{equation}
  where $\omega\ge 0$ is the Lipschitz constant of $F$. Fix $T>0$
  and set $M:=e^{\omega T}\norm{u_{0}}_{\infty}$. Then, the values of
  $\phi(s)\in \R$ for $\abs{s}\ge M$ do not intervene if one considers
  the solutions merely on the time interval $[0,T]$. Thus, there is no
  loss of generality if we assume, $\phi$ and $\phi^{-1}$ are globally
  Lipschitz continuous on $\R$.

  Now, for every $\lambda>0$, let $\Psi_{\lambda} :
  L^{2}(\Sigma,\mu)\to \R_{+}$ denote the
  \emph{Moreau regularisation} of $\Psi$ on $L^{2}(\Sigma,\mu)$
  (cf.~\cite[Proposition~2.11]{MR0348562}). Then, $\Psi_{\lambda}$ is
  continuously Fr\'echet-differentiable on $L^{2}(\Sigma,\mu)$ and the
  Fr\'echet-derivative $\Psi'_{\lambda}$ of $\Psi$ coincides with the
  Yosida operator $A_{\lambda}:=\tfrac{1}{\lambda}(I-J_{\lambda})$ of $A$ in
  $L^{2}(\Sigma,\mu)$. Since the resolvent operator $J_{\lambda}$ of
  $A$ is contractive on $L^{2}(\Sigma,\mu)$ and since $\phi$ is
  globally Lipschitz continuous, the composition operator
  $J_{\lambda}\phi$ is globally Lipschitz continuous on
  $L^{2}(\Sigma,\mu)$. Hence by~\cite[Corollaire~1.1]{MR0348562}, for
  every $\lambda>0$, there is a unique strong solution
  \begin{displaymath}
    u_{\lambda}\in C^{1}([0,T];L^{2}(\Sigma,\mu))
  \end{displaymath}
  of the Cauchy problem
    \begin{equation}
      \label{eq:130}
      \begin{cases}
        \tfrac{du_{\lambda}}{dt}+A_{\lambda}\phi(u_{\lambda})
        +F(u_{\lambda})=0 &\text{in $L^{2}(\Sigma,\mu)$ on $(0,T)$,}\\
        \mbox{}\hspace{3cm}u_{\lambda}(0)=u_{0}.&
      \end{cases}
    \end{equation}
    Since $A$ is accretive in $L^{1}(\Sigma,\mu)$, one easily verifies
    that for every $\lambda>0$, the Yosida operator $A_{\lambda}$ is
    also accretive in $L^{1}(\Sigma,\mu)$. Moreover, for every
    $p\in P_{0}$, there is a $\theta(x)\in (0,1)$ such that
    \begin{align*}
      \int_{\Sigma}p(u)A_{\lambda}u\,\dmu 
      &=  \int_{\Sigma}p(J_{\lambda}u)A_{\lambda}u\,\dmu
        + \lambda \int_{\Sigma}p'(\theta u+(1-\theta)
        J_{\lambda}u)\abs{A_{\lambda}u}^{2}\,\dmu\\
      & \ge \int_{\Sigma}p(J_{\lambda}u)A_{\lambda}u\,\dmu.  
    \end{align*}
    Since $A_{\lambda} u\in A( J_{\lambda}u)$ and since $A$ has a
    complete resolvent,
    \begin{displaymath}
      \int_{\Sigma}p(J_{\lambda}u)A_{\lambda}u\,\dmu\ge 0
    \end{displaymath}
    yielding the Yosida operator $A_{\lambda}$ has a complete
    resolvent. Thus by Proposition~\ref{propo:Lipschitz-complete}, the
    operator $A_{\lambda}\phi +F$ is quasi accretive in
    $L^{1}(\Sigma,\mu)$ with complete resolvent. Thus,
    \begin{equation}
      \label{eq:133}
      \norm{u_{\lambda}(t)}_{q}\le e^{\omega
        t}\,\norm{u_{0}}_{q}\qquad\text{ for every $t\in [0,T]$,
        $\lambda>0$, $1\le q\le \infty$.}
    \end{equation} 
    and in particular, $\norm{u_{\lambda}(t)}_{\infty}\le M$.

    Next, let $\rho>0$ such that $\rho\,\omega <1$ and
    $x\in L^{1}\cap L^{\infty}(\Sigma,\mu)$. Then, by range
    condition~\eqref{eq:201}, there is
    $u_{\rho}\in D(A_{1\cap\infty}\phi)$ such that
    $u_{\rho}=J_{\rho}^{A_{1\cap\infty}\phi+F}x$ or, equivalently,
    \begin{equation}
      \label{eq:254}
       u_{\rho}=x-\rho(A\phi(u_{\rho})+F(u_{\rho})).
   \end{equation}
    Since $A_{\lambda}\phi +F+\omega I$ is Lipschitz
    continuous and accretive in $L^{1}(\Sigma,\mu)$, $A_{\lambda}\phi
    +F+\omega I$ is $m$-accretive in $L^{1}(\Sigma,\mu)$. Thus for every
    $\lambda>0$, there is $u_{\rho,\lambda}\in D(A_{\lambda}\phi)$
    such that $u_{\rho,\lambda}=J_{\rho}^{A_{\lambda}\phi+F}x$ or, equivalently,
    \begin{equation}
      \label{eq:253}
     u_{\rho,\lambda}=x-\rho (A_{\lambda}\phi(u_{\rho,\lambda})+ F(u_{\rho,\lambda})),
   \end{equation}
    and
    \begin{equation}
      \label{eq:131}
      \norm{u_{\rho,\lambda}}_{q}\le (1-\rho\,\omega)^{-1}\norm{x}_{q}
   \end{equation}
   for every $1\le q\le \infty$. Now, by the two
   equations~\eqref{eq:254} and \eqref{eq:253}, since
   \begin{displaymath}
     A_{\lambda}(\phi(u_{\rho,\lambda}))\in
     A(J_{\lambda}\phi(u_{\rho,\lambda})),
   \end{displaymath}
   since the operators $A$ and $F+\omega I$ are accretive in
   $L^{2}(\Sigma,\mu)$, $\phi$ is non-decreasing, and $F$ Lipschitz
   continuous with constant $\omega\ge 0$, we see
    \begin{align*}
      &(1-\rho \omega)\,\int_{\Sigma}
        (u_{\rho,\lambda}-u_{\rho})(\phi(u_{\rho,\lambda})-\phi(u_{\rho}))\,\dmu\\
      &\phantom{0} = - \rho\;\int_{\Sigma} \Big[
        (A_{\lambda}\phi(u_{\rho,\lambda})+F(u_{\rho,\lambda})+\omega
        u_{\rho,\lambda})-(A\phi(u_{\rho})+F(u_{\rho})+\omega u_{\rho})\Big]\times\\
      &\hspace{9cm}  \times\, (\phi(u_{\rho,\lambda})-\phi(u_{\rho}))\,\dmu\\
      &\phantom{0} = - \rho\; \int_{\Sigma} \Big[A\phi(u_{\rho})
        -A_{\lambda}\phi(u_{\rho,\lambda})\Big]\times\\
      &\hspace{4.5cm}  
        \times (\phi(u_{\rho})-J_{\lambda}^{A}\phi(u_{\rho,\lambda})
        +J_{\lambda}^{A}\phi(u_{\rho,\lambda})-\phi(u_{\rho,\lambda}))\,\dmu\\
      &\hspace{1cm}-\rho\;\int_{\Sigma}
        \Big[(F(u_{\rho,\lambda})+\omega
        u_{\rho,\lambda})-(F(u_{\rho})+\omega u_{\rho})\Big]\,
        (\phi(u_{\rho,\lambda})-\phi(u_{\rho}))\,\dmu\\
      & \phantom{0}\le  \rho\,\lambda\, \int_{\Sigma} \left(A\phi(u_{\rho})
        -A_{\lambda}\phi(u_{\rho,\lambda})\right)\,A_{\lambda}\phi(u_{\rho,\lambda})\,\dmu\\
      & \phantom{0}= -  \lambda\; \int_{\Sigma} (u_{\rho}-u_{\rho,\lambda})
        \,A_{\lambda}\phi(u_{\rho,\lambda})\,\dmu\\
      &\hspace{3cm}-\rho\,\lambda\, \int_{\Sigma}
        (F(u_{\rho})-F(u_{\rho,\lambda}))\,A_{\lambda}\phi(u_{\rho,\lambda})\,\dmu\\
      & \phantom{0}\le \lambda \,(1+\rho\,\omega)\;\int_{\Sigma} \abs{u_{\rho}-u_{\rho,\lambda}}\,
        \abs{A_{\lambda}\phi(u_{\rho,\lambda})}\,\dmu\\
      & \phantom{0}\le \lambda
        \frac{(1+\rho\,\omega)\,2}{\rho\,(1-\rho\omega)}\,
        \norm{x}_{\infty}\int_{\Sigma}
        \abs{u_{\rho,\lambda}-x+\rho\,F(u_{\rho,\lambda})}\,\dmu\\
      & \phantom{0}\le \lambda
        \frac{(1+\rho\,\omega)\,2}{\rho\,(1-\rho\omega)}\,
        \norm{x}_{\infty}\left(\int_{\Sigma}
        \abs{u_{\rho,\lambda}}\,\dmu + \int_{\Sigma}\abs{x}\,\dmu + 
        \rho\,\omega\int_{\Sigma} \abs{u_{\rho,\lambda}}\,\dmu\right)% \\
      % & \phantom{0}\le \lambda
      %   \frac{(1+\rho\,\omega)\,2}{\rho\,(1-\rho\omega)^2}\,
      %   \norm{x}_{\infty}\left( \frac{1}{(1-\rho\omega)}+1+ 
      %   \frac{\rho\omega}{(1-\rho\omega)}\right)\int_{\Sigma}
      %   \abs{x}\,\dmu,
     % & \phantom{0}\le \lambda
     %   \frac{(1+\rho\,\omega)\,2}{\rho\,(1-\rho\omega)}\,
     %   \norm{x}_{\infty}^{2}\,\mu(\Sigma)\,
     %   \left(((1-\rho\omega)^{-1}+1+\rho\omega (1-\rho\omega)^{-1}\right).
    \end{align*}
    and so by~\eqref{eq:131},
    \begin{equation}
      \label{eq:210}
      \begin{split}
        &\int_{\Sigma}
        (u_{\rho,\lambda}-u_{\rho})(\phi(u_{\rho,\lambda})-\phi(u_{\rho}))\,\dmu\\
        & \qquad\le \lambda
        \frac{(1+\rho\,\omega)\,2}{\rho\,(1-\rho\omega)^2}\,
        \norm{x}_{\infty}\left( \frac{1}{(1-\rho\omega)}+1+ 
        \frac{\rho\omega}{(1-\rho\omega)}\right)\int_{\Sigma}
        \abs{x}\,\dmu
      \end{split}
    \end{equation}
    From this we can conclude that
    $\lim_{\lambda\to0}u_{\rho,\lambda}=u_{\rho}$ a.e. on $\Sigma$
    since $\phi$ is continuous, strictly increasing and $\phi(s)=0$ if
    and only if $s=0$. Since $(\Sigma,\mu)$ is finite and by~\eqref{eq:131}, Lebesgue's
    dominated convergence theorem yields
   \begin{displaymath}
    \lim_{\lambda\to0+} J_{\rho}^{A_{\lambda}\phi+F}x=\lim_{\lambda\to0+}u_{\rho,\lambda}=
    u_{\rho}=J_{\rho}^{A_{1\cap\phi}\phi+F}x\qquad\text{ in
      $L^{1}(\Sigma,\mu)$}
   \end{displaymath}
   for every $x\in L^{1}\cap L^{\infty}(\Sigma,\mu)$ and $\rho>0$. Since
   $L^{1}\cap L^{\infty}(\Sigma,\mu)$ is dense in $L^{1}(\Sigma,\mu)$
   and $J_{\rho}^{\overline{A_{1\cap\infty}\phi}+F}$ and
   $J_{\rho}^{A_{\lambda}\phi+F}$ are Lipschitz continuous, a
   standard density argument shows that the hypothesis~\eqref{eq:116}
   in Theorem~\ref{thm:convergence-result} for $q=1$ holds. Therefore,
   \begin{displaymath}
     % \label{eq:202}
      \lim_{\lambda\to0+}u_{\lambda}=u\qquad\text{in $C([0,T];L^{1}(\Sigma,\mu))$.}
   \end{displaymath}
   and by~\eqref{eq:132} and~\eqref{eq:133}, for the strong solution
   $u_{\lambda}$ of~\eqref{eq:130} and the mild solution $u$
   of~\eqref{eq:9}, one has
  \begin{equation}
    \label{eq:136}
    \lim_{\lambda\to0+}u_{\lambda}= u\qquad\text{in
      $C([0,T];L^{q}(\Sigma,\mu))$}\qquad\text{for all $1\le q<\infty$.}
  \end{equation}

  Next, we show that
  \begin{equation}
    \label{eq:204}
    u\in W^{1,2}_{loc}((0,T];L^{2}(\Sigma,\mu))
  \end{equation}
  By the Lipschitz continuity of the Nemytski operator $F$ on
  $L^{2}(\Sigma,\mu)$ and since $u_{\lambda}$ belongs to
  $C([0,T];L^2(\Sigma,\mu))$, we have that
  $F(u_{\lambda})\in C([0,T];L^{2}(\Sigma,\mu))$. Further, since
  $\phi$ is Lipschitz continuous, $\phi(0)=0$, and
  $u_{\lambda}\in C^{1}([0,T];L^2(\Sigma,\mu))$, we can conclude that the function
  $\phi(u_{\lambda})\in W^{1,2}(0,T;L^{2}(\Sigma,\mu))$ with weak
  derivative
  \begin{equation}
    \label{eq:134}
    \frac{d}{dt}\phi(u_{\lambda}(t))=\phi'(u_{\lambda}(t))\frac{d
    u_{\lambda}}{dt}(t)\qquad\text{ for a.e. $t\in (0,T)$.}
  \end{equation}
  By equation~\eqref{eq:130} and by
  $A_{\lambda}=\Psi'_{\lambda}$, 
  % \begin{align*}
  %   \frac{d}{dt}\tfrac{1}{2}\norm{\phi(u_{\lambda}(t))}_{2}^{2} 
  %   &= \langle \frac{d u_{\lambda}}{dt},
  %     \phi'(u_{\lambda}(t))\rangle\\
  %   & = - \langle
  %     (A_{0})_{\lambda}\phi(u_{\lambda}(t)),\phi'(u_{\lambda}(t))\rangle
  %     - \langle \tilde{f}(\cdot,u_{\lambda}(t)), \phi'(u_{\lambda}(t)) \rangle
  % \end{align*}
  \cite[Lemme~3.3]{MR0348562} implies that the function
  $\Psi_{\lambda}(\phi(u_{\lambda})) : [0,T]\to \R$ is absolutely
  continuous and
  \begin{displaymath}
    \frac{d}{dt} \Psi_{\lambda}(\phi(u_{\lambda}(t)))=\langle
    A_{\lambda}\phi(u_{\lambda}(t))\,,\frac{d}{dt}\phi(u_{\lambda}(t))\rangle= \langle
    A_{\lambda}\phi(u_{\lambda}(t))\,,\phi'(u_{\lambda}(t))\frac{d
    u_{\lambda}}{dt}(t)\rangle
  \end{displaymath}
  for a.e. $t\in (0,T)$. Let $k\ge 0$ and multiply equation~\eqref{eq:130} by
  $s^{k+2}\,\frac{d}{ds}\phi(u_{\lambda}(s))$ with respect to the
  $L^{2}$-inner product and integrating over $(0,t)$, for some
  $0<t<T$. Then
  \begin{align*}
           &\int_{0}^{t}s^{k+2} \int_{\Sigma} \phi'(u_{\lambda}(s))\,\labs{\frac{d
           u_{\lambda}}{ds}(s)}^{2}\,\dmu\,\ds +
       t^{k+2}\,\Psi_{\lambda}(\phi(u_{\lambda}(t)))\\
       &\quad = (k+2)\int_{0}^{t}s^{k+1}\psi_{\lambda}(\phi(u_{\lambda}(s)))\,\ds -
       \int_{0}^{t}s^{k+2}\int_{\Sigma} F(u_{\lambda}(s)) \phi'(u_{\lambda}(s))\frac{d
         u_{\lambda}}{ds}(s)\,\dmu\,\ds.
  \end{align*}
  Since $\phi'(u_{\lambda})\ge 0$, Young's inequality gives
  \begin{equation}
    \label{eq:135}
     \begin{split}
       &\tfrac{1}{2}\int_{0}^{t}s^{k+2} \int_{\Sigma} \phi'(u_{\lambda}(s))\,\labs{\frac{d
           u_{\lambda}}{ds}(s)}^{2}\,\dmu\,\ds +
       t^{k+2}\,\Psi_{\lambda}(\phi(u_{\lambda}(t)))\\
       &\quad \le
       (k+2)\int_{0}^{t}s^{k+1}\psi_{\lambda}(\phi(u_{\lambda}(s)))\,\ds\\
       &\hspace{2cm}+\tfrac{1}{2}
       \int_{0}^{t}s^{k+2}\int_{\Sigma} \abs{F(u_{\lambda}(s))}^{2} \phi'(u_{\lambda}(s))\,\dmu\,\ds.
    \end{split}
  \end{equation}
  for every $0<t\le T$. On the other hand, the Yosida operator $A_{\lambda}$
  is the subgradient $\partial_{\!
    L^{2}}\Psi_{\lambda}$ of $\Psi_{\lambda}$, and $A_{\lambda}(0)=0$
  and $\Psi_{\lambda}(0)=0$. Thus
  \begin{displaymath}
    \langle
    A_{\lambda}\phi(u_{\lambda}(t))\,,0-\phi(u_{\lambda}(t))\rangle\le 
    \Psi_{\lambda}(0)-\Psi_{\lambda}(\phi(u_{\lambda}(t)))
  \end{displaymath}
  for every $0<t<T$. Multiplying this inequality by $(-1)$ and taking
  advantage of equation~\eqref{eq:130} yields
  \begin{equation}
    \label{eq:203}
    \Psi_{\lambda}(\phi(u_{\lambda}(t)))\le - \langle \tfrac{d
      u_{\lambda}}{d t}(t)\, ,\phi(u_{\lambda}(t))\rangle -
    \langle F(u_{\lambda}(t))\, ,\phi(u_{\lambda}(t))\rangle
  \end{equation}
  for every $0<t<T$. Since $\phi$ is non-decreasing on $\R$ and
  $\Phi(0)=0$, one has $\Phi(r)\le \phi(r)r$ for every $r\in
  \R$. % bounded and continuous on
  % $[-M,M]$, the function $\Phi(r)=\int_{0}^{r}\phi(s)\,ds$,
  % $(r\ge 0)$, is Lipschitz continuous on $[-M,M]$ and of class
  % $C^{1}$. By~\eqref{eq:133} and since
  % $u_{\lambda}\in C^{1}([0,T];L^{2}(\Sigma,\mu))$, one has that
  % $\Phi(u_{\lambda})\in C^{1}([0,T];L^{2}(\Sigma,\mu))$.
  Thus, by integrating inequality~\eqref{eq:203} over $(0,t)$ for some
  $t\in (0,T]$, we obtain
  \begin{align*}
    &\int_{0}^{t}\Psi_{\lambda}(\phi(u_{\lambda}(s)))\,\ds +
      \int_{\Sigma}\Phi(u_{\lambda}(t))\,\dmu\\
    &\qquad\le \int_{\Sigma}\Phi(u_{\lambda}(t))\,\dmu
      -\int_{0}^{t}\int_{\Sigma} F(u_{\lambda}(s))\phi(u_{\lambda}(s))\,\dmu\,\ds.
  \end{align*}
  Similarly, multiplying inequality~\eqref{eq:203} by $s^{k+1}$ and subsequently
  integrating over $(0,t)$ for some $t\in (0,T]$ gives
  \begin{equation}
  \label{eq:255}
  \begin{split}
    &\int_{0}^{t}s^{k+1}\Psi_{\lambda}(\phi(u_{\lambda}(s)))\,\ds +
      t^{k+1} \int_{\Sigma}\Phi(u_{\lambda}(t))\,\dmu\\
    &\qquad\le
      (k+1)\int_{0}^{t}s^{k}\int_{\Sigma}\Phi(u_{\lambda}(t))\,\dmu\ds\\
    &\hspace{4cm}  -\int_{0}^{t}s^{k+1}\int_{\Sigma} F(u_{\lambda}(s))
      \phi(u_{\lambda}(s))\,\dmu\,\ds\\
    &\qquad\le
      (k+1)\int_{0}^{t}s^{k}\int_{\Sigma}\phi(u_{\lambda}(t))\,u_{\lambda}(t)\dmu\ds\\
     &\hspace{4cm} -\int_{0}^{t}s^{k+1}\int_{\Sigma}
     F(u_{\lambda}(s))\phi(u_{\lambda}(s))\,\dmu\,\ds
  \end{split}
\end{equation}
for every $0<t\le T$. Since $\phi(0)=0$ and $\phi$ is non-decreasing
on $\R$, one has that the function $\Phi(r)\ge 0$
for every $r\in \R$ and so,
\begin{displaymath}
  \int_{\Sigma}\Phi(u_{\lambda}(t))\,\dmu\ge 0.
\end{displaymath}
Thus, applying estimate~\eqref{eq:255} to the right hand-side of~\eqref{eq:135}
yields
\begin{equation}
  \label{eq:205}
  \begin{split}
    &\tfrac{1}{2}\int_{0}^{t}s^{k+2} \int_{\Sigma}
    \phi'(u_{\lambda}(s))\,\labs{\frac{d
        u_{\lambda}}{ds}(s)}^{2}\,\dmu\,\ds +
    t^{k+2}\,\Psi_{\lambda}(\phi(u_{\lambda}(t)))\\
    &\qquad \le (k+2)(k+1)
    \int_{0}^{t}s^{k}\int_{\Sigma}\phi(u_{\lambda}(t))
    \,u_{\lambda}(t)\,\dmu\,\ds\\
    &\hspace{3cm}  -(k+2) \int_{0}^{t}s^{k+1}\int_{\Sigma}
      F(u_{\lambda}(s))\phi(u_{\lambda}(s))\,\dmu\,\ds\\
    &\hspace{4cm} +
    \tfrac{1}{2}\int_{0}^{t}s\int_{\Sigma}\phi'(u_{\lambda}(s))\,
    \labs{F(u_{\lambda}(s))}^{2}\,\dmu\,\ds.
    \end{split}
  \end{equation}
  By assumption, there are constants $\alpha_{1}$, $\alpha_{2}>0$ such
  that $\alpha_{1}\le \phi'(s)\le \alpha_{2}$ for all $s\in
  [-M,M]$, by the boundedness of $\phi$ on $[-M,M]$, by the
  continuity of $F : L^{2}(\Sigma,\mu)\to L^{2}(\Sigma,\mu)$, and by
  \eqref{eq:132} and ~\eqref{eq:133}, we can conclude from
  estimate~\eqref{eq:205} that the sequence %  and~,
%   \begin{displaymath}
%     \labs{\int_{\Sigma}\Big[
%         \Phi(u_{\lambda}(t))-\Phi(u(t))\Big]\,\dmu}\le
%     L_{\Phi}\int_{\Sigma}\abs{u_{\lambda}(t)-u(t)}\,\dmu
%   \end{displaymath}
% for every $\lambda>0$, where we denote by $L_{\Phi}>0$ the Lipschitz
% constant of $\Phi$ on $[-M,M]$. Thus, ~\eqref{eq:202} yields
% \begin{displaymath}
%   \lim_{\lambda\to0+}\int_{\Sigma}\Phi(u_{\lambda}(t))\,\dmu 
%=\int_{\Sigma}\Phi(u(t))\,\dmu. 
% \end{displaymath}
% In particular,
% \begin{displaymath}
% \sup_{\lambda>0} \int_{\Sigma}\Phi(u_{\lambda}(t))\,\dmu\qquad\text{is
% finite.}
% \end{displaymath}
% By the global Lipschitz continuity of $F$ and $\phi$, and since
% $\Psi_{\lambda}(\phi(u_{\lambda}(t)))\ge0$, we can conclude
% from~\eqref{eq:205} that 
$(u_{\lambda})_{\lambda>0}$ is bounded in
$W^{1,2}_{loc}((0,T];L^{2}(\Sigma,\mu))$. Thus, for every sequence
$(\lambda_{n})\in (0,1)$ such that $\lambda_{n}\to0$ as $n\to\infty$,
there is $w\in L^{2}_{loc}((0,T];L^{2}(\Sigma,\mu))$ and after
eventually passing to a subsequence,
$\frac{d u_{\lambda_{n}}}{dt}\rightharpoonup w$ weakly in
$L^{2}([\delta,T];L^{2}(\Sigma,\mu)$ for every $\delta\in
(0,T)$. Hence, sending $n\to\infty$ in
\begin{displaymath}
  u_{\lambda_{n}}(t)-u_{\lambda_{n}}(s)=\int_{s}^{t} \frac{d u_{\lambda_{n}}}{dr}(r)\,\dr
\end{displaymath}
and using~\eqref{eq:136} yields $w(t)=\frac{d u}{dt}(t)$ in
$L^{2}(\Sigma,\mu)$ for a.e. $t\in (0,T)$. Therefore, \eqref{eq:204}
holds and
\begin{equation}
  \label{eq:137}
  \lim_{\lambda\to0+}\frac{d u_{\lambda_{n}}}{dt}=\frac{d u}{dt}
  \qquad\text{weakly in $L^{2}_{loc}((0,T];L^{2}(\Sigma,\mu))$.}
\end{equation}

Next, by~\eqref{eq:136} and since $\phi$ is Lipschitz continuous,
\begin{displaymath}
  \lim_{\lambda\to0+}\phi(u_{\lambda})=\phi(u)\qquad\text{in
    $C([0,T];L^{2}(\Sigma,\mu))$.}
\end{displaymath}
In particular, $\phi(u(t))\in L^{2}(\Sigma,\mu)$ for every
$t\in [0,T]$. By assumption, $\Psi$ is densely defined on
$L^{2}(\Sigma,\mu)$ and so by \cite[Proposition~2.11 \&
Th\'eor\`eme~2.2]{MR0348562}, the resolvent operator $J_{\lambda}$ of $A$
satisfies
\begin{displaymath}
  \lim_{\lambda\to0+}J_{\lambda}\phi(u(t))=\phi(u(t))\qquad\text{in
      $L^{2}(\Sigma,\mu)$ for every $t\in [0,T]$.}
\end{displaymath}
Thus and since for every $t\in [0,T]$,
\begin{align*}
  &\norm{J_{\lambda}\phi(u_{\lambda}(t))-\phi(u(t))}_{2}\\
  &\hspace{1,5cm}\le
    \norm{J_{\lambda}\phi(u_{\lambda}(t))-J_{\lambda}\phi(u(t))}_{2}
    + \norm{J_{\lambda}\phi(u(t))-\phi(u(t))}_{2}\\
  &\hspace{1,5cm} \le
    \norm{\phi(u_{\lambda}(t))-\phi(u(t))}_{2}
    + \norm{J_{\lambda}\phi(u(t))-\phi(u(t))}_{2},
  \end{align*}
  we can conclude that
  \begin{equation}
    \label{eq:138}
    \lim_{\lambda\to0+}J_{\lambda}\phi(u_{\lambda}(t))=\phi(u(t))\qquad\text{in
      $L^{2}(\Sigma,\mu)$ for every $t\in [0,T]$.}
  \end{equation}
  Now, let $(\hat{w},\hat{v})\in A$ and $t\in (0,T)$ such that
  $\frac{d u}{dt}(t)$ exists in $L^{2}(\Sigma,\mu)$. Since
  \begin{displaymath}
    A_{\lambda}(\phi(u_{\lambda}(t)))\in A(J_{\lambda}\phi(u_{\lambda}))
\end{displaymath}
and since $A$ is accretive in $L^{2}(\Sigma,\mu)$, one has
  \begin{displaymath}
    \left[J_{\lambda}\phi(u_{\lambda}(t))-\hat{w},
      \left(-F(u_{\lambda}(t))-\frac{d u_{\lambda}}{dt}(t)
    \right) -\hat{v}\right]_{2}\ge0
  \end{displaymath}
  Sending $\lambda\to0+$ in this inequality and using~\eqref{eq:137}
  and~\eqref{eq:138} yields
  \begin{displaymath}
    \left[\phi(u_{\lambda}(t))-\hat{w},\left(-F(u(t))-\frac{d u}{dt}(t)
    \right) -\hat{v}\right]_{2}\ge0.
  \end{displaymath}
  Since $(\hat{w},\hat{v})\in A$ was arbitrary, $A$ is $m$-accretive
  in $L^{2}(\Sigma,\mu)$ and $\frac{d u}{dt}(t)$ exists in
  $L^{2}(\Sigma,\mu)$ for a.e. $t\in (0,T)$, we can conclude that for
  a.e. $t\in (0,T)$, one has $\phi(u(t))\in D(A)$ satisfying
  inclusion~\eqref{eq:206}, showing that $u$ is a strong solution
  of~\eqref{eq:206} in $L^{2}(\Sigma,\mu)$. Now, proceeding as in the
  previous steps of this proof, one sees that chain
  rule~\eqref{eq:134} and estimates~\eqref{eq:135} and \eqref{eq:255}
  satisfied by $u_{\lambda}$ hold, in particular, for $u$, proving
  that $\phi(u)\in W^{1,2}_{loc}((0,T];L^{2}(\Sigma,\mu))$, chain
  rule~\eqref{eq:150} and the estimates~\eqref{eq:140} and
  \eqref{eq:139} hold. Moreover, by \eqref{eq:133}, \eqref{eq:136},
  and~\eqref{eq:204}, we see that $u$ has the regularity as stated
  in~\eqref{eq:124}. Next, recall that $A$ is the subgradient
  $\partial_{L^{2}}\Psi$ in $L^{2}(\Sigma,\mu)$ of a convex, proper,
  lower semicontinuous functional $\Psi$ on
  $L^{2}(\Sigma,\mu)$. Further, for every $0<\delta<T$, the function
  $v:=\phi(u)\in W^{1,2}([\delta,T];L^{2}(\Sigma,\mu))$ and
  inclusion~\eqref{eq:206} yields
  $g:=- F(u)-\frac{d u}{dt}\in
  L^{2}([\delta,T];L^{2}(\Sigma,\mu))$
  and satisfies $g(t)\in Av(t)$ for a.e. $t\in (\delta,T)$. Thus for
  every $0<\delta<T$, by~\cite[Lemme~3.3]{MR0348562}, the function
  $t\mapsto \Psi(v(t))$ is absolutely continuous on $[\delta,T]$ and
  for a.e. $t\in (\delta,T)$,
  \begin{displaymath}
    \tfrac{d}{dt}\Psi(v(t))= \langle g(t),\tfrac{d v}{dt}(t) \rangle.
  \end{displaymath}
  Combining this together with chain rule~\eqref{eq:150}, we see that
  equation~\eqref{eq:267} holds.
 \end{proof}

%%%%%%%%%%%%%%%%%%%%%%%%%%%%%%%%%%%%%%%%%%%%%%%%%%%%%%%%
%
%
%
%
%
%                      Subsection: weak energy solutions for initial values
%                      in $L^{\infty}$
%
%
%            
%  
%
%%%%%%%%%%%%%%%%%%%%%%%%%%%%%%%%%%%%%%%%%%%%%%%%%%%%%%%%%
 \subsection{Weak solutions for general $\phi$ and initial values in
   $L^{\infty}$}
\label{subsec:weak-sol-for-initial-in-Linfty}

This subsection is concerned with the second major step toward the
proof of Theorem~\ref{thm:weak-solutions-Li-Linfty}. Our next results
shows that mild solutions are, in fact, weak energy solutions for
initial values
$u_{0}\in \overline{D(A_{1\cap\infty}\phi)}^{\mbox{}_{L^{1}}}\cap
L^{\infty}(\Sigma,\mu)$.
This is known to be true for the homogeneous porous media equation
(cf.~\cite{Benthesis,MR2286292}), and generalises this result to
general quasi-$m$-accretive operators in $L^{1}$ with complete resolvent of the
form $\overline{A_{1\cap\infty}\phi}$.

\begin{theorem}
  \label{thm:reg-mild-sol-are-weak}
  Let $\Psi : V\to \R$ be a convex, lower semicontinuous, and
  G\^ateaux differentiable functional satisfying the hypotheses
  ($\mathcal{H}$\ref{hyp:a})-($\mathcal{H}$\ref{hyp:c}), $\phi$ be a
  strictly increasing continuous function on $\R$
  satisfying~($\mathcal{H}$\ref{hyp:phi}) and $F$ be an operator on
  $L^{q}(\Sigma,\mu)$ satisfying~($\mathcal{H}$\ref{hyp:F}).  Then,
  for every
  $u_{0}\in \overline{D(A_{1\cap\infty}\phi)}^{\mbox{}_{L^{1}}} \cap
  L^{\infty}(\Sigma,\mu)$
  and $T>0$, the mild solution $u(t)=T_{t}u_{0}$, $(t\ge 0)$, of
  Cauchy problem~\eqref{eq:149} in $L^{1}(\Sigma,\mu)$ is a weak
  energy solution of Cauchy problem~\eqref{eq:161} satisfying
  \begin{equation}
    \label{eq:273}
    u\in C([0,T];L^{q}(\Sigma,\mu))\cap L^{\infty}(0,T;L^{\infty}(\Sigma,\mu)),
  \end{equation}
  for every $1\le q<\infty$,
  \begin{equation}
   \label{eq:274} 
    \tfrac{\textrm{d}u}{\dt}\in
    L^{p^{\mbox{}_{\prime}}}(0,T;V'),\qquad
    \phi(u)\in L^{p}(0,T;V)
  \end{equation}
  and identity
  \begin{equation}
    \label{eq:275}
    \int_{0}^{T}\Big\{\Big\langle \tfrac{d u}{dt},v\Big\rangle_{V',V} +
    \langle \Psi'\phi(u(t)),v\rangle +\langle F(u(t)),v\rangle \Big\}\,\dt=0
  \end{equation}
  holds for every $v\in L^{p}(0,T;V)$. In particular, for the function
  $\Phi$ given by~\eqref{eq:276}, one has that
\begin{equation}
 \label{eq:277}
   \Phi(u)\in C([0,T];L^{1}(\Sigma,\mu))\quad\text{with}\quad 
   \int_{\Sigma}\Phi(u)\,\dmu\in W^{1,1}(0,T),
 \end{equation}
 ''integration by parts rule''~\eqref{eq:215} holds, and for every
$k\ge 0$ and $t>0$, inequality~\eqref{eq:140} and energy estimate
 \begin{equation}
   \label{eq:256}
   \begin{split}
    &\int_{0}^{t}\Psi(\phi(u(s)))\,\ds +
      \int_{\Sigma}\Phi(u(t))\,\dmu\\
    &\qquad\le \int_{\Sigma}\Phi(u_{0})\,\dmu
      -\int_{0}^{t}\int_{\Sigma} F(u(s))\phi(u(s))\,\dmu\,\ds,
  \end{split}
 \end{equation}
 holds.
\end{theorem}

%%%%%%%%%%%%%%%%%%%%%%%%%%%%%%%%%%%%%%%%%%%%%%%%%%%%%%%%
%
%                      Lemma to show that the hypothesis of graph
%                      convergence hold
%
%%%%%%%%%%%%%%%%%%%%%%%%%%%%%%%%%%%%%%%%%%%%%%%%%%%%%%%%%

For the proof of this result, we need the following approximation result.

\begin{lemma}
  \label{lem:2}
  % Let $A$ be an accretive operator in $L^{1}$ with complete resolvent
%   such that $(0,0)\in A$ and $A$ is $m$-accretive in
%   $L^{2}(\Sigma,\mu)$. Further, let $\phi :\R\to \R$ be a strictly
%   increasing, locally Lipschitz continuous function satisfying
%   $\phi(0)=0$ and for every $\lambda>0$, the Yosida operator
%   $\beta_{\lambda}$ of $\beta=\phi^{-1}$ satisfies~\eqref{eq:179} for
%   $q=2$. In addition, suppose that one of the hypotheses
%   ($\mathcal{H}$\ref{Hyp:1}) and ($\mathcal{H}$\ref{Hyp:2}) holds.
%   Now, set
%   \begin{equation}
%     \label{eq:209}
%     \phi_{\varepsilon}(s)=\phi(s)+\varepsilon\, s\qquad\text{ for every
%       $s\in \R$ and $\varepsilon>0$,}
% \end{equation}
  Let $A$ be an $m$-completely accretive operator in
  $L^{2}(\Sigma,\mu)$ of a finite measure space $(\Sigma,\mu)$, $F$ be
  the Nemytski operator of a Carath\'eodory function
  $f : \Sigma\times \R\to\R$ satisfying~\eqref{eq:2}, and $\phi$ be a
  strictly increasing continuous function on $\R$ such that for every
  $\lambda>0$, the Yosida operator $\beta_{\lambda}$ of
  $\beta=\phi^{-1}$ and of $\beta=\phi^{-1}_{\varepsilon}$ the
  regularisation $(\phi_{\varepsilon})$ of $\phi$ satisfy
  condition~\eqref{eq:252}. Further, for every $\lambda>0$, let
  $J_{\lambda}^{\varepsilon}$ denote the resolvent operator of
  $\overline{A_{1\cap\infty}\phi_{\varepsilon}}+F$ and $J_{\lambda}$
  the resolvent operator of $\overline{A_{1\cap\infty}\phi}+F$. Then,
  for every $\lambda>0$ such that $\omega\lambda<1$ and
  $u\in L^{1}\cap L^{\infty}(\Sigma,\mu)$, one has
\begin{equation}
  \label{eq:207}
  \lim_{\varepsilon\to0}J_{\lambda}^{\varepsilon}u=J_{\lambda}u\qquad\text{in
    $L^{q}(\Sigma,\mu)$ for every $1\le q<\infty$.}
   \end{equation}
\end{lemma}

\begin{proof}[Proof of Lemma~\ref{lem:2}] By
  Proposition~\ref{propo:Lipschitz-complete},
  $\overline{A_{1\cap\infty}\phi}+F$ is $m$-accretive in $L^{1}$ with
  complete resolvent and for every $\lambda>0$ such that
  $\omega\lambda<1$, $A_{1\cap\infty}\phi+F$ satisfies the range
  condition~\eqref{eq:201}. Moreover, since one has that
  \begin{displaymath}
  A_{1\cap\infty}\phi_{\varepsilon}\subseteq
  \overline{A_{1\cap\infty}\phi_{\varepsilon}}\quad\text{ and }\quad
  A_{1\cap\infty}\phi\subseteq \overline{A_{1\cap\infty}\phi},
\end{displaymath}
the range condition~\eqref{eq:201} yields that the resolvent
$J_{\lambda}^{\varepsilon}$ coincides with the resolvent of
$A_{1\cap\infty}\phi_{\varepsilon}$ on
$L^{1}\cap L^{\infty}(\Sigma,\mu)$ and the resolvent $J_{\lambda}$ of
$\overline{A_{1\cap\infty}\phi}$ coincides with the resolvent of
$A_{1\cap\infty}\phi$ on $L^{1}\cap L^{\infty}(\Sigma,\mu)$. Thus, for
every $\lambda>0$ such that $\omega\lambda<1$ and every
$u\in L^{1}\cap L^{\infty}(\Sigma,\mu)$, $\varepsilon>0$, there are
$(u_{\varepsilon}, v_{\varepsilon})\in A_{1\cap\infty}\phi_{\varepsilon}$
and $(u_{0},v_{0})\in A_{1\cap\infty}\phi$ satisfying
\begin{equation}
  \label{eq:208}
  u_{\varepsilon}+\lambda
  (v_{\varepsilon}+F(u_{\varepsilon}))= u
  \qquad\text{and}\qquad
  u_{0}+\lambda (v_{0}+F(u_{0}))= u,
\end{equation}
or equivalently, $u_{\varepsilon}=J_{\lambda}^{\varepsilon}u$ for
every $\varepsilon>0$ and $u_{0}=J_{\lambda}u$. Now, by
using~\eqref{eq:208} and since $A$ and $F+\omega I$ are accretive
operators in $L^{2}(\Sigma,\mu)$, we see that
\begin{align*}
  &(1-\omega\lambda)\,\int_{\Sigma}
    (u_{\varepsilon}-u_{0})(\phi(u_{\varepsilon})-\phi(u_{0}))\,\dmu\\
  &\qquad = (1-\omega\lambda)\, \int_{\Sigma}
    (u_{\varepsilon}-u_{0})(\phi_{\varepsilon}(u_{\varepsilon})-\phi(u_{0}))\,\dmu\\
  &\mbox{}\hspace{6cm}  - \,\int_{\Sigma}
    (u_{\varepsilon}-u_{0})(\phi(u_{\varepsilon})-\phi_{\varepsilon}(u_{\varepsilon}))\,\dmu\\
  &\qquad = -\lambda\,\int_{\Sigma}
    \Big[v_{\varepsilon}+F(u_{\varepsilon})+\omega
    u_{\varepsilon}-(v_{0}+F(u_{0})+\omega u_{0})\Big]\times\\
  &\hspace{8cm} \times \,(\phi_{\varepsilon}(u_{\varepsilon})-\phi(u_{0}))\,\dmu\\
  &\mbox{}\hspace{6cm}  - \,\int_{\Sigma}
    (u_{\varepsilon}-u_{0})(\phi(u_{\varepsilon})-\phi_{\varepsilon}(u_{\varepsilon}))\,\dmu\\
  &\qquad = -\lambda\,
    \Big[\phi_{\varepsilon}(u_{\varepsilon})-\phi(u_{0}),v_{\varepsilon}-v_{0}\Big]_{2}\\
  &\qquad\hspace{1cm}  - \lambda\,
    \Big[\phi_{\varepsilon}(u_{\varepsilon})-\phi(u_{0}),(F(u_{\varepsilon})+\omega
    u_{\varepsilon})-(F(u_{0})+\omega u_{0})\Big]_{2}\\
  &\mbox{}\hspace{6cm}  - \,\int_{\Sigma}
    (u_{\varepsilon}-u_{0})(\phi(u_{\varepsilon})-\phi_{\varepsilon}(u_{\varepsilon}))\,\dmu\\
  & \qquad\le  - \,\int_{\Sigma}
    (u_{\varepsilon}-u_{0})(\phi(u_{\varepsilon})-\phi_{\varepsilon}(u_{\varepsilon}))\,\dmu.
  \end{align*}
  By Proposition~\ref{propo:quasi-accretive-operators-in-L1-complete-resolvent},
  \begin{equation}
   \label{eq:211}
      \norm{u_{\varepsilon}}_{q}\le (1-\lambda\,\omega)^{-1}\norm{u}_{q}=:M
  \end{equation}
  for all $\varepsilon\ge 0$ and $1\le q\le \infty$ and so,
  \begin{displaymath}
    0\le \int_{\Sigma}
    (u_{\varepsilon}-u_{0})(\phi(u_{\varepsilon})-\phi(u_{0}))\,\dmu\le
    \frac{2\norm{u}_{1}}{(1-\lambda\,\omega)^{2}}
    \norm{\phi_{\varepsilon}-\phi}_{L^{\infty}(-M,M)}.
  \end{displaymath}
  Since $\phi_{\varepsilon}\to\phi$ uniformly on compact subsets of
  $\R$, since $\phi$ is strictly increasing on $\R$ and $\phi(s)=0$ if
  and only if $s=0$, it follows that
  $\lim_{\varepsilon}u_{\varepsilon}=u_{0}$ a.e. on $\Sigma$. Using
  again~\eqref{eq:211} and that the measure space $(\Sigma,\mu)$ is
  finite, we can conclude that \eqref{eq:207} holds for
  $u\in L^{1}\cap L^{\infty}(\Sigma,\mu)$. A standard density argument
  yields the statement of this lemma.
\end{proof}

The following \emph{integration by parts rule} is an important tool in
the proof of  Theorem~\ref{thm:reg-mild-sol-are-weak}. It also appears in
different versions in the literature (cf., for instance,~\cite[p. 366]{MR837254}).

\begin{lemma}
  \label{lem:3}
  Let $\phi: \R\to\R$ be a non-decreasing continuous
  function and $u\in L^{\infty}(0,T; L^{\infty}(\Sigma,\mu))
  \cap C([0,T];L^{1}(\Sigma,\mu))$
  such that $\tfrac{d u}{dt}\in L^{p^{\mbox{}_{\prime}}}(0,T;V')$ and
  $\phi(u)\in L^{p}(0,T;V)$. Set $\Phi(r)=\int_{0}^{r}\phi(s)\,\ds$
  for every $r\in \R$. Then, 
  \begin{equation}
    \label{eq:215}
  \int_{t_{1}}^{t_{2}}\Big\langle
    \frac{du}{dt},\phi(u)\Big\rangle_{V,V}\,\dt=
    \int_{\Sigma}\Phi(u(t_{2}))\dmu-\int_{\Sigma}\Phi(u(t_{1}))\dmu
\end{equation}
for every $0\le t_{1}<t_{2}\le T$.
\end{lemma}

\begin{proof}
   By assumption, there is a constant
  $M=\norm{u}_{L^{\infty}(0,T; L^{\infty}(\Sigma,\mu))}\ge 0$ such that $\phi$
  is bounded on $[-M,M]$ with constant $L_{\phi}>0$. From this, we
  easily obtain that
  \begin{equation}
    \label{eq:219}
    \norm{\Phi(u(t))-\Phi(u(s))}_{1}\le L_{\phi}\,M\,\norm{u(t)-u(s)}_{1}
  \end{equation}
  for every $t$, $s\in [0,T]$ and so $\Phi(u)\in
  C([0,T];L^{1}(\Sigma))$. Furthermore, H\"older's inequality yields
  $\langle \tfrac{d u}{dt}, \phi(u) \rangle_{V';V}\in
  L^{1}(0,T)$. Thus, both sides of equation~\eqref{eq:215} are
  finite. Now, let $0\le t_{1}<t_{2}\le  T$. For every $h>0$ and
  for $t_{1}<t<t_{2}$ such that
  $h<t_{2-t}$, the \emph{Steklov average} $[\frac{du}{dt}]_{h}$
  of $\frac{du}{dt}$ is given by
  \begin{displaymath}
    \Big[\frac{du}{dt}\Big]_{h}(t)
    :=\frac{1}{h}\int_{t}^{t+h}\frac{du}{ds}(s)\,\ds\qquad\text{in
    $V'$.}
  \end{displaymath}
  Since $\tfrac{d u}{dt}\in L^{p^{\mbox{}_{\prime}}}(0,T;V')$, one
  easily checks that
  \begin{displaymath}
    \lim_{h\to0+} \Big[\frac{du}{dt}\Big]_{h}=\frac{du}{dt}\qquad\text{in
    $L^{p^{\mbox{}_{\prime}}}(t_{1},t_{2};V')$}
  \end{displaymath}
  and so,
  \begin{equation}
    \label{eq:217}
    \lim_{h\to0+}\int_{t_{1}}^{t_{2}}\Big\langle
    \Big[\frac{du}{dt}\Big]_{h},\phi(u)\Big\rangle_{V',V}\,\dt=
   \int_{t_{1}}^{t_{2}}\Big\langle \frac{du}{dt},\phi(u)\Big\rangle_{V,V}\,\dt.
  \end{equation}
  Furthermore, for every $t\in (0,T-h)$, % $\frac{d u_{h}}{dt}(t)=
  % [\frac{du}{dt}]_{h}(t)$, 
  $[\frac{du}{dt}]_{h}(t)=h^{-1}(u(t+h)-u(t))$. Using this together
  with the convexity of $\Phi$ and~\eqref{eq:214}, we see that
  \begin{equation}
    \label{eq:220}
    \begin{split}
      \int_{t_{1}}^{t_{2}}\Big\langle
      \Big[\frac{du}{dt}\Big]_{h},\phi(u)\Big\rangle_{V',V}\,\dt&\le
      \int_{t_{1}}^{t_{2}}h^{-1}\int_{\Sigma} (\Phi(u(t+h))-
      \Phi(u(t)))\,\dmu\,\dt\\
      & = \int_{t_{1}}^{t_{2}}\int_{\Sigma}
      \frac{d}{dt}\Phi_{h}(u(t))\,\dmu\,\dt
    \end{split}
  \end{equation}
  By~\eqref{eq:219}, we can apply Fubini's Theorem, to conclude that
  \begin{displaymath}
    \int_{t_{1}}^{t_{2}}\int_{\Sigma}
    \frac{d}{dt}\Phi_{h}(u(t))\,\dmu\,\dt
    =\int_{\Sigma} \Phi_{h}(u(t_{2}))\,\dmu-\int_{\Sigma} \Phi_{h}(u(t_{2}))\,\dmu
  \end{displaymath}
  for every $h>0$. Since $\Phi(u)\in
  C([0,T];L^{1}(\Sigma))$, one has that
  \begin{displaymath}
   \lim_{h\to 0+}\Phi_{h}(u(t_{2}))\qquad\text{ in $C([t_{1},t_{2}];L^{1}(\Sigma))$}
  \end{displaymath}
  (cf.~\cite[Lemma 3.3.4]{zbMATH06275549} for $V=\R$ and note that the
  general Banach space-valued case $V$ is shown analogously). Thus, sending
  $h\to 0+$ in \eqref{eq:220} and using \eqref{eq:217} yields
  \begin{displaymath}
    \int_{t_{1}}^{t_{2}}\Big\langle
    \frac{du}{dt},\phi(u)\Big\rangle_{V',V}\,\dt\le 
    \int_{\Sigma} \Phi(u(t_{1}))\,\dmu-\int_{\Sigma} \Phi(u(t_{2}))\,\dmu.
  \end{displaymath}
  In order to see that the reverse inequality holds as well, we take
  $h<0$ such that $0<-h<t_{1}$. Then by the
  convexity of  $\Phi$ and since $h^{-1}<0$, we obtain that 
  \begin{align*}
      \int_{t_{1}}^{t_{2}}\Big\langle
      \Big[\frac{du}{dt}\Big]_{h},\phi(u)\Big\rangle_{V',V}\,\dt
    &= \int_{t_{1}}^{t_{2}}h^{-1}\langle
      u(t+h)-u(t),\phi(u(t))\rangle\,\dt\\
    &\ge
      \int_{t_{1}}^{t_{2}}h^{-1}\int_{\Sigma} (\Phi(u(t+h))-
      \Phi(u(t)))\,\dmu\,\dt\\
      & = \int_{t_{1}}^{t_{2}}\int_{\Sigma}
      \frac{d}{dt}\Phi_{h}(u(t))\,\dmu\,\dt.
    \end{align*}
    Now, proceeding as in the first part of this proof, we see
    that~\eqref{eq:215} holds.
\end{proof}

With the above preliminaries, we can now outline the proof
of Theorem~\ref{thm:reg-mild-sol-are-weak}.

\begin{proof}[Proof of Theorem~\ref{thm:reg-mild-sol-are-weak}]
  Let $\phi : \R\to \R$ be strictly increasing continuous and for
  every $\varepsilon>0$ and $\phi_{\varepsilon}$ be the regularisation
  of $\phi$ satisfying the assumptions of this theorem. Since
  $\overline{A_{1\cap\infty}\phi}+F$ and
  $\overline{A_{1\cap\infty}\phi_{\varepsilon}}+F$ are quasi
  $m$-accretive in $L^{1}(\Sigma,\mu)$, the
  Crandall-Liggett theorem yields the existence of strongly continuous
  semigroups $\{T_{t}\}_{t\ge0}\sim -(\overline{A_{1\cap\infty}\phi}+F)$ on
  $\overline{D(A_{1\cap\infty}\phi)}$ and
  $\{T_{t}^{\varepsilon}\}_{t\ge0}\sim
  -(\overline{A_{1\cap\infty}\phi_{\varepsilon}}+F)$
  on $\overline{D(A_{1\cap\infty}\phi_{\varepsilon})}$. Since 
  $\overline{A_{1\cap\infty}\phi}+F$ and
  $\overline{A_{1\cap\infty}\phi_{\varepsilon}}+F$ have each a complete
  resolvent,
  Proposition~\ref{propo:quasi-accretive-operators-in-L1-complete-resolvent}
  yields
  \begin{equation}
     \label{eq:153}
     \norm{u_{\varepsilon}(t)}_{q}\le e^{\omega
       t}\norm{u_{0}}_{q}\quad\text{ and }\quad \norm{u(t)}_{q}\le e^{\omega
       t}\norm{u_{0}}_{q}
  \end{equation}
  for every $t\ge 0$, $\varepsilon>0$ and $1\le q\le
  \infty$. Moreover, by Lemma~\ref{lem:2} and
  Theorem~\ref{thm:convergence-result}, one has for every $T>0$ that
  \begin{equation}
    %\label{eq:213}
    \lim_{\varepsilon\to0}u_{\varepsilon}=u\qquad\text{in $C([0,T];L^{1}(\Sigma,\mu))$.}
  \end{equation}
 Combining this with \eqref{eq:153} and H\"older's inequality, we find
 \begin{equation}
   \label{eq:151}
   \lim_{\varepsilon\to0}u_{\varepsilon}=u\qquad\text{in
     $C([0,T];L^{q}(\Sigma,\mu))$, for every $1\le q<\infty$ and $T>0$.}
 \end{equation}
 Now, we fix $T>0$ and set
 \begin{displaymath}
   M= e^{\omega T}\norm{u_{0}}_{\infty}.
 \end{displaymath}
 By~\eqref{eq:153}, the values of $u_{\varepsilon}$ and $u$ do not
 exceed the interval $[-M,M]$. Since the
 measure space $(\Sigma,\mu)$ is finite, we see that
\begin{align*}
  \norm{\phi_{\varepsilon}(u_{\varepsilon}(t))-\phi(u(t))}_{q}
  &\le \norm{\phi_{\varepsilon}(u_{\varepsilon}(t))-\phi(u_{\varepsilon}(t))}_{q}
    + \norm{\phi(u_{\varepsilon}(t))-\phi(u(t))}_{q}\\
  &\le\;\mu(\Sigma)^{1/q}\,\norm{\phi_{\varepsilon}-\phi}_{L^{\infty}(-M,M)}
    + \norm{\phi(u_{\varepsilon}(t))-\phi(u(t))}_{q}
\end{align*}
and so, the uniform convergence of $\phi_{\varepsilon}\to \phi$
on $[-M,M]$ as $\varepsilon\to0+$, the continuity of $\phi$ combined
with~\eqref{eq:153}, \eqref{eq:151}, and Lebesgue's dominated
convergence theorem imply
\begin{equation}
   \label{eq:154}
   \lim_{\varepsilon\to0}\phi_{\varepsilon}(u_{\varepsilon})=\phi(u)\qquad\text{in
     $C([0,T];L^{q}(\Sigma,\mu))$ for every $1\le q<\infty$.}
 \end{equation}
Since $F$ is Lipschitz continuous in all $L^{q}$-spaces,
\eqref{eq:151} and \eqref{eq:154} imply
\begin{equation}
   \label{eq:155}
   \lim_{\varepsilon\to0}\int_{0}^{t}s^{k+1}\int_{\Sigma} F(u_{\varepsilon})
   \phi_{\varepsilon}(u_{\varepsilon})\,\dmu\,\ds=
   \int_{0}^{t}s^{k+1}\int_{\Sigma} F(u) \phi(u)\,\dmu\,\ds
 \end{equation}
for every $t\in (0,T]$ and $k\ge -1$. We set
\begin{displaymath}
 \Phi_{\varepsilon}(s)=\int_{0}^{s}\phi_{\varepsilon}(r)\,\dr\qquad\text{for
 every $s\in \R$.}
\end{displaymath}
Since $\phi_{\varepsilon}\to \phi$ uniformly
on $[-2M,2M]$ as $\varepsilon\to0+$, there is some $\varepsilon_{0}>0$
such that the sequence $(\phi_{\varepsilon})_{\varepsilon>\varepsilon_{0}}$ is bounded
$L^{\infty}(-2M,2M)$. Thus and by the mean-value theorem, we obtain that
\begin{align*}
  & \norm{\Phi_{\varepsilon}(u_{\varepsilon}(t))-\Phi(u(t))}_{1}\\
  &\qquad
    \le \norm{\Phi_{\varepsilon}(u_{\varepsilon}(t))-\Phi_{\varepsilon}(u(t))}_{1}
     + \norm{\Phi_{\varepsilon}(u(t))-\Phi(u(t))}_{1}\\
  &\qquad
    \le\,\sup_{\varepsilon\ge
    \varepsilon_{0}}\norm{\phi_{\varepsilon}}_{L^{\infty}(-M,M)}\,
    \norm{u_{\varepsilon}(t)-u(t)}_{1}
    +\,\norm{\phi_{\varepsilon}-\phi}_{L^{\infty}(-M,M)}\,e^{\omega T}\norm{u_{0}}_{1}.
\end{align*}
Applying to this limit~\eqref{eq:151} and the uniform convergence of
$\phi_{\varepsilon}\to \phi$ on $[-M,M]$ as $\varepsilon\to0+$, we
obtain that
\begin{equation}
  \label{eq:156}
  \lim_{\varepsilon\to 0}\Phi_{\varepsilon}(u_{\varepsilon})=\Phi(u)\qquad\text{in
    $C([0,T];L^{1}(\Sigma,\mu))$.}
\end{equation}

By Theorem~\ref{thm:Linfty-implies-mild-are-strong-in-L1}, every
$u_{\varepsilon}$ is a strong solution of Cauchy
problem~\eqref{eq:150} for $\phi$ replaced by $\phi_{\varepsilon}$
and has the same regularity as stated in~\eqref{eq:124}. Moreover,
every $u_{\varepsilon}$ satisfies inequality~\eqref{eq:256}, that is,
\begin{equation}
  \label{eq:257}
  \begin{split}
    &\int_{0}^{t}\Psi(\phi_{\varepsilon}(u_{\varepsilon}(s)))\,\ds +
      \int_{\Sigma}\Phi_{\varepsilon}(u_{\varepsilon}(t))\,\dmu\\
    &\qquad\le \int_{\Sigma}\Phi_{\varepsilon}(u_{0})\,\dmu
      -\int_{0}^{t}\int_{\Sigma} F(u_{\varepsilon}(s))\phi_{\varepsilon}(u_{\varepsilon}(s))\,\dmu\,\ds
  \end{split}
\end{equation}
and inequality~\eqref{eq:140}, that is,
\begin{equation}
  \label{eq:258}
   \begin{split}
    &\int_{0}^{t}s^{k+1}\Psi(\phi_{\varepsilon}(u_{\varepsilon}(s)))\,\ds +
      t^{k+1} \int_{\Sigma}\Phi_{\varepsilon}(u_{\varepsilon}(t))\,\dmu\\
    % &\qquad\le
    %   (k+1)\int_{0}^{t}s^{k}\int_{\Sigma}\Phi_{\varepsilon}(u_{\varepsilon}(t))\,\dmu\,\ds\\
    % &\hspace{4cm}  -\int_{0}^{t}s^{k+1}\int_{\Sigma} F(u_{\varepsilon}(s))
    %   \phi_{\varepsilon}(u_{\varepsilon}(s))\,\dmu\,\ds\\
    &\qquad\le
      (k+1)\int_{0}^{t}s^{k}\int_{\Sigma}\phi_{\varepsilon}(u_{\varepsilon}(t))\,u_{\varepsilon}(t)\,\dmu\,\ds\\
     &\hspace{4cm} -\int_{0}^{t}s^{k+1}\int_{\Sigma}
     F(u_{\varepsilon}(s))\,\phi_{\varepsilon}(u_{\varepsilon}(s))\,\dmu\,\ds
  \end{split}
\end{equation}
for every $t\in (0,T]$ and $\varepsilon>0$, where $k\ge 0$ is
fixed. Since $\Phi_{\varepsilon}\ge 0$, inequality~\eqref{eq:257}
together with \eqref{eq:153} and the two limits~\eqref{eq:154} and
\eqref{eq:155} imply that the sequence
$(\phi_{\varepsilon}(u_{\varepsilon}))_{\varepsilon>0}$ is bounded in
$L^{p}(0,T;V)$ and so by~\eqref{ineq:bounded-in-V},
$(\Psi'\phi_{\varepsilon}(u_{\varepsilon}))_{\varepsilon>0}$ is
bounded in $L^{p^{\mbox{}_{\prime}}}(0,T;V')$. In particular, by
equation~\eqref{eq:152} we have that
$\big(\frac{d u_{\varepsilon}}{dt}\big)_{\varepsilon>0}$ is bounded in
$L^{p^{\mbox{}_{\prime}}}(0,T;V')$. By the continuous embedding of $V$
into $L^{2}(\Sigma,\mu)$ and by~\eqref{eq:154}, for any sequence
$(\varepsilon_{n})_{n\ge 1}$ of the open interval $(0,1)$ satisfying
$\lim_{n\to\infty}\varepsilon_{n}=0$, there are $\nu$ and
$\chi\in L^{p^{\mbox{}_{\prime}}}(0,T;V')$ and a subsequence of
$(u_{\varepsilon_{n}})_{n\ge 1}$, which we denote again by
$(u_{\varepsilon_{n}})_{n\ge 1}$ such that
\begin{align}
  \label{eq:160}
   &\lim_{n\to\infty}\phi_{\varepsilon_{n}}(u_{\varepsilon_{n}})=\phi(u)\qquad\text{weakly in
     $L^{p}(0,T;V)$,}\\
   \label{eq:158}
   &\lim_{n\to\infty}\tfrac{d u_{\varepsilon_{n}}}{dt}=\nu\qquad\text{weakly in
     $L^{p^{\mbox{}_{\prime}}}(0,T;V')$,}\\
   \label{eq:159}
   & \lim_{n\to\infty}\Psi'\phi_{\varepsilon_{n}}(u_{\varepsilon_{n}}) =\chi\qquad\text{weakly in
     $L^{p^{\mbox{}_{\prime}}}(0,T;V')$.}
\end{align}
By~\eqref{eq:151} and~\eqref{eq:158} combined with standard techniques
employed for vector-valued distributions (see, for instance, \cite[pp
36]{hauer07}), one sees that
\begin{displaymath}
\nu=\frac{d u}{dt}\qquad\text{in $L^{p^{\mbox{}_{\prime}}}(0,T;V')$.}
\end{displaymath}
By~\eqref{eq:214} and Lipschitz continuity of $F$ on
$L^{p^{\mbox{}_{\prime}}}(0,T;L^{2}(\Sigma,\mu))$, we may multiply
equation~\eqref{eq:152} by $v\in C^{1}_{c}(0,T;V)$ and subsequently
integrate over $(0,T)$. Then,
\begin{displaymath}
  \int_{0}^{T}\left\langle \frac{d u_{\varepsilon_{n}}}{dt},v\right\rangle_{V',V}\,\dt+
  \int_{0}^{T}\langle \Psi'\phi_{\varepsilon_{n}}(u_{\varepsilon_{n}}),v\rangle_{V',V}\,\dt
  + \int_{0}^{T}\langle F(u_{\varepsilon_{n}}),v\rangle\,\dt=0
\end{displaymath}
for every $n\ge 1$. Sending $n\to \infty$ in the latter equation and
employing~\eqref{eq:151}, \eqref{eq:158}, \eqref{eq:159} and the
Lipschitz continuity of $F$ on
$L^{p^{\mbox{}_{\prime}}}(0,T;L^{2}(\Sigma,\mu))$ yields
\begin{equation}
  \label{eq:157}
  \frac{d u}{dt} + \chi +F(u)=0\qquad\text{in $L^{p^{\mbox{}_{\prime}}}(0,T;V')$.}
\end{equation}
It remains to show that $\chi=\Psi'(u)$. To see this, let
$0<t_{1}<t_{2}<T$ and take $w\in L^{p}(0,T;V)$. By convexity of
$\Psi$, we have that
\begin{displaymath}
  \int_{t_{1}}^{t_{2}}\langle \Psi'(\phi_{\varepsilon_{n}}(u_{\varepsilon_{n}}))-\Psi'(w),
  \phi_{\varepsilon_{n}}(u_{\varepsilon_{n}})-w\rangle \,\dt\ge 0.
\end{displaymath}
Now, using that $u_{\varepsilon_{n}}$ is a solution of
equation~\eqref{eq:152}, the latter inequality can be
rewritten as
\begin{align*}
  &\int_{t_{1}}^{t_{2}}\int_{\Sigma} F(u_{\varepsilon_{n}}) 
    (\phi_{\varepsilon_{n}}(u_{\varepsilon_{n}})-w)\dmu \dt +
    \int_{\Sigma}\Phi(u_{\varepsilon_{n}})\dmu\Big\vert_{t_{1}}^{t_{2}}
    + \int_{t_{1}}^{t_{2}}\langle \tfrac{d u_{\varepsilon_{n}}}{dt}, w\rangle \,\dt\\
  &\qquad \le -\int_{t_{1}}^{t_{2}}\langle \Psi'(w),
    \phi_{\varepsilon_{n}}(u_{\varepsilon_{n}})-w\rangle \,\dt.
\end{align*}
Sending $n\to\infty$ in this inequality and using ~\eqref{eq:151},
\eqref{eq:155} for $k=-1$, \eqref{eq:160}, \eqref{eq:158}, and
\eqref{eq:156}, we obtain
\begin{align*}
 &\int_{t_{1}}^{t_{2}}\int_{\Sigma} F(u) (\phi(u)-w)\dmu \dt +
 \int_{\Sigma}\Phi(u)\dmu\Big\vert_{t_{1}}^{t_{2}}
 + \int_{t_{1}}^{t_{2}}\langle \tfrac{d u}{dt}, w\rangle \,\dt\\
 &\qquad \le -\int_{t_{1}}^{t_{2}}\langle \Psi'(w),
  \phi(u)-w\rangle \,\dt.
\end{align*}
On the other hand, if we first
multiply equation~\eqref{eq:157} with $\phi(u)$, and then integrate over
$(t_{1},t_{2})$ for $0<t_{1}<t_{2}<T$ and apply \emph{integration
  by parts formula} (Lemma~\ref{lem:3}) yields
\begin{displaymath}
  \int_{\Sigma}\Phi(u)\,\dmu\Big\vert_{t_{1}}^{t_{2}}+\int_{t_{1}}^{t_{2}}\langle
  \chi,\phi(u)\rangle\,\dt + \int_{t_{1}}^{t_{2}}
  \int_{\Sigma}F(u(t))\phi(u(t))\,\dmu\,\dt =0.
\end{displaymath}
Using this, we can rewrite the latter inequality as
\begin{displaymath}
  \int_{t_{1}}^{t_{2}}\langle \chi-\Psi'(w),
  \phi(u)-w\rangle \,\dt\ge 0.
\end{displaymath}
Since $w\in L^{p}(0,T;V')$ was arbitrary, taking $w=\phi(u)-\lambda \xi$ for
any $\lambda>0$ and for some general $\xi \in L^{p}(t_{1},t_{2};V)$ in this inequality
and applying the hemicontinuity of $\Psi'$
(hypothesis~($\mathcal{H}$\ref{hyp:b})) yields
\begin{displaymath}
  \int_{t_{1}}^{t_{2}}\langle \chi-\Psi'(\phi(u)),\xi\rangle \,\dt\ge 0
\end{displaymath}
for all $\xi \in L^{p}(t_{1},t_{2};V)$. Therefore and since
$0<t_{1}<t_{2}<T$ were arbitrary, $\chi=\Psi'(\phi(u))$ in $V'$ for
a.e. $t\in (0,T)$, showing that $u$ is a weak energy solution of Cauchy
problem~\eqref{eq:161}. 

It remains to show that $u$ satisfies the energy
inequalities~\eqref{eq:256} and~\eqref{eq:140}. To see this, we send
$\varepsilon\to0$ in the two inequalities~\eqref{eq:257} and
\eqref{eq:258} and apply limit \eqref{eq:160} together with
Fatou's lemma (note $\Psi\ge 0$ by assumption) and the lower
semicontinuity of $\psi$, limit~\eqref{eq:151}, \eqref{eq:154},
\eqref{eq:155} and \eqref{eq:156}. This completes the proof of this theorem.
\end{proof}

%%%%%%%%%%%%%%%%%%%%%%%%%%%%%%%%%%%%%%%%%%%%%%%%%%%%%%%%
%
%
%
%
%                      Subsection: Mild solutions in $L^{1}$ are weak.
%
%                                                 P R O O F
%
%
%
%
%
%%%%%%%%%%%%%%%%%%%%%%%%%%%%%%%%%%%%%%%%%%%%%%%%%%%%%%%%%
\subsection{Proof of Theorem~\ref{thm:weak-solutions-Li-Linfty}}

This subsection is dedicated to outlining the proof of
Theorem~\ref{thm:weak-solutions-Li-Linfty}.

\begin{proof}[Proof of Theorem~\ref{thm:weak-solutions-Li-Linfty}]
  Let $u_{0}\in \overline{D(A_{1\cap\infty}\phi)}^{\mbox{}_{L^{1}}}$
  and $\{T_{t}\}_{t\ge 0}$ be the semigroup generated by
  $-(\overline{A_{1\cap\infty}\phi}+F)$ on
  $\overline{D(A_{1\cap\infty}\phi)}^{\mbox{}_{L^{1}}}$. By
  assumption, the semigroup $\{T_{t}\}_{t\ge 0}$ satisfies the
  $L^{1}$-$L^{\infty}$-regularisation estimate~\eqref{eq:168} (for
  $u_{0}=0$ and $s=1$). Thus
  $T_{t}u_{0}\in
  \overline{D(A_{1\cap\infty}\phi)}^{\mbox{}_{L^{1}}}\cap
  L^{\infty}(\Sigma,\mu)$
  for every $t>0$ and the strong continuity of $\{T_{t}\}_{t\ge 0}$ in
  $L^{1}(\Sigma,\mu)$ yields the existence of a sequence
  $(u_{0,n})_{n\ge 1}$ with elements
  $u_{0,n}\in \overline{D(A_{1\cap\infty}\phi)}^{\mbox{}_{L^{1}}}\cap
  L^{\infty}(\Sigma,\mu)$
  such that $u_{0,n}=T_{t_{n}}u_{0}$ for some sequence
  $(t_{n})_{n\ge 1}$ satisfying $0<t_{n+1}<t_{n}$,
  $\lim_{n\to\infty}t_{n}=0$, and $\lim_{n\to\infty}u_{0,n}=u_{0}$ in
  $L^{1}(\Sigma,\mu)$. We set $u(t)=T_{t}u_{0}$, ($t\ge 0$), to be the
  unique mild solution of problem~\eqref{eq:149} in $L^{1}$ with
  initial value $u_{0}$ and $u_{n}(t)=T_{t}u_{0,n}$, ($t\ge 0$), the
  unique mild solution of problem~\eqref{eq:149} in $L^{1}$ with
  initial value $u_{0,n}$. By Theorem~\ref{thm:reg-mild-sol-are-weak},
  the mild solution $u_{n}$ of~\eqref{eq:149} is a weak energy
  solution of problem~\eqref{eq:161} with
  regularity~\eqref{eq:273}--\eqref{eq:274}, satisfying \eqref{eq:275}
  and \eqref{eq:277}. Now, the semigroup
  property~\eqref{eq:272} and the exponential growth property in
  $L^{1}$ (that is, $\tilde{q}=1$ in~\eqref{eq:62}) yield that
  \begin{equation}
    \label{eq:265}
    u(t+t_{n})=u_{n}(t)\qquad\text{for all $t\ge 0$ and $n\ge 1$.}
  \end{equation}
  Thus, if one replaces the interval $[0,T)$ by $[t_{n},T)$ for every
  $n\ge 1$, the mild solution $u$ is a weak energy
  solution of problem~\eqref{eq:161} with
  regularity~\eqref{eq:273}--\eqref{eq:274}, satisfying \eqref{eq:275}
  and \eqref{eq:277}. 

  It remains to show that energy estimates~\eqref{eq:260} holds. For
  this, recall that by Theorem~\ref{thm:reg-mild-sol-are-weak},
  $u_{n}$ satisfies~\eqref{eq:140} for every $k\ge 0$, that is,
  \allowdisplaybreaks
  \begin{equation}
    \label{eq:262}
  \begin{split}
    &\int_{0}^{t}s^{k+1}\Psi(\phi(u_{n}(s)))\,\ds +
      t^{k+1} \int_{\Sigma}\Phi(u_{n}(t))\,\dmu\\
    % &\qquad\le
    %   (k+1)\int_{0}^{t}s^{k}\int_{\Sigma}\Phi(u(t))\,\dmu\ds\\
    % &\hspace{4cm}  -\int_{0}^{t}s^{k+1}\int_{\Sigma} F(u(s))
    %   \phi(u(s))\,\dmu\,\ds\\
    &\qquad\le
      (k+1)\int_{0}^{t}s^{k}\int_{\Sigma}\phi(u_{n}(t))\,u_{n}(t)\,\dmu\,\ds\\
     &\hspace{4cm} 
     -\int_{0}^{t}s^{k+1}\int_{\Sigma} F(u_{n}(s))\phi(u_{n}(s))\,\dmu\,\ds.
  \end{split}
  \end{equation}
  By using that $F$ is Lipschitz continuous, H\"older's and Young's
  inequality, and then Poincar\'e type inequality~\eqref{eq:259}, we
  see that
  \begin{equation}
    \label{eq:269}
  \begin{split}
    &\pm\int_{0}^{t}s^{k+1}\int_{\Sigma} F(u_{n}(s))\phi(u_{n}(s))\,\dmu\,\ds\\
    &\qquad\le \omega\,
      \int_{0}^{t}s^{k+1}\norm{u_{n}(s)}_{p^{\mbox{}_{\prime}}}\,\norm{\phi(u_{n}(s))}_{p}\,\ds\\
    &\qquad\le \varepsilon\, \int_{0}^{t}s^{k+1}
      \Psi(\phi(u_{n}(s)))\,\ds +
      \tfrac{\omega^{p}\,C^{p-1}}{p^{\mbox{}_{\prime}}(\varepsilon\,p)^{p-1}}\,
      \int_{0}^{t}s^{k+1}\norm{u_{n}(s)}_{p^{\mbox{}_{\prime}}}^{p^{\mbox{}_{\prime}}}\,\ds
  \end{split}
\end{equation}
  for every $\varepsilon>0$. Similarly, 
  \begin{equation}
    \label{eq:270}
    \begin{split}
      &(k+1)\,\int_{0}^{t}s^{k}\int_{\Sigma} \phi(u_{n}(s))\,u_{n}(s)\,\dmu\,\ds\\
      &\qquad\le (k+1)\,
      \int_{0}^{t}s^{k}\norm{u_{n}(s)}_{p^{\mbox{}_{\prime}}}\,\norm{\phi(u_{n}(s))}_{p}\,\ds\\
      &\qquad\le \varepsilon\, \int_{0}^{t}s^{k+1}
      \Psi(\phi(u_{n}(s)))\,\ds +
      \tfrac{(k+1)^{p}\,C^{p-1}}{p^{\mbox{}_{\prime}}(\varepsilon\,p)^{p-1}}\,
      \int_{0}^{t}s^{k}\norm{u_{n}(s)}_{p^{\mbox{}_{\prime}}}^{p^{\mbox{}_{\prime}}}\,\ds
    \end{split}
  \end{equation}
  Choosing $\varepsilon=\tfrac{1}{4}$ in these two estimates and apply
  them to the right hand-side of inequality~\eqref{eq:262}, we
  obtain
  \begin{equation}
    \label{eq:264}
    \begin{split}
    &\tfrac{1}{2}\int_{0}^{t}s^{k+1}\Psi(\phi(u_{n}(s)))\,\ds +
      t^{k+1} \int_{\Sigma}\Phi(u_{n}(t))\,\dmu\\
    % &\qquad\le
    %   (k+1)\int_{0}^{t}s^{k}\int_{\Sigma}\Phi(u(t))\,\dmu\ds\\
    % &\hspace{4cm}  -\int_{0}^{t}s^{k+1}\int_{\Sigma} F(u(s))
    %   \phi(u(s))\,\dmu\,\ds\\
    &\qquad\le
      \tfrac{(k+1)^{p}\,C^{p-1}}{p^{\mbox{}_{\prime}}(4^{-1}\,p)^{p-1}}\,
      \int_{0}^{t}s^{k}\norm{u_{n}(s)}_{p^{\mbox{}_{\prime}}}^{p^{\mbox{}_{\prime}}}\,\ds
      +\tfrac{\omega^{p}\,C^{p-1}}{p^{\mbox{}_{\prime}}(4^{-1}\,p)^{p-1}}\,
      \int_{0}^{t}s^{k+1}\norm{u_{n}(s)}_{p^{\mbox{}_{\prime}}}^{p^{\mbox{}_{\prime}}}\,\ds.
  \end{split}
  \end{equation}
  By assumption, there are exponents $\alpha$, $\beta$, $\gamma>0$ and
  a constant $\tilde{C}>0$ such that the semigroup $\{T_{t}\}_{t\ge
    0}$ satisfies $L^{1}$-$L^{\infty}$ regularisation
  estimate~\eqref{eq:259}. Now, we choose
  $k=\alpha(p^{\mbox{}_{\prime}}-1)>0$. Then, by H\"older's
  inequality, by using that
  \begin{displaymath}
    \norm{u_{n}(t)}_{1}\le e^{\omega
      t}\,\norm{u_{0,n}}_{1}\qquad\text{for every $t\ge 0$,}
  \end{displaymath}
 and by $L^{1}$-$L^{\infty}$ regularisation
  estimate~\eqref{eq:259}, we see that
  \begin{displaymath}
    \begin{split}
      \int_{0}^{t}s^{\alpha(p^{\mbox{}_{\prime}}-1)}
      \norm{u_{n}(s)}_{p^{\mbox{}_{\prime}}}^{p^{\mbox{}_{\prime}}}\,\ds
      &\le
      \int_{0}^{t}s^{\alpha(p^{\mbox{}_{\prime}}-1)}
      \norm{u_{n}(s)}_{\infty}^{(p^{\mbox{}_{\prime}}-1)}\,\norm{u_{n}(s)}_{1}\,\ds\\
      & \le
      \int_{0}^{t}s^{\alpha(p^{\mbox{}_{\prime}}-1)}
      \norm{u_{n}(s)}_{\infty}^{(p^{\mbox{}_{\prime}}-1)}\,e^{\omega
        s}\,\ds\,\norm{u_{0,n}}_{1}\\
      & \le
      \int_{0}^{t}e^{\omega(\beta\,(p^{\mbox{}_{\prime}}-1)+1)\,s}\,\ds
      \,\norm{u_{0,n}}_{1}^{\gamma\,(p^{\mbox{}_{\prime}}-1)+1}
    \end{split}
  \end{displaymath}
  and
   \begin{displaymath}
    \begin{split}
      \int_{0}^{t}s^{\alpha(p^{\mbox{}_{\prime}}-1)+1}\,
      \norm{u_{n}(s)}_{p^{\mbox{}_{\prime}}}^{p^{\mbox{}_{\prime}}}\,\ds
      & \le
      \int_{0}^{t}s\,e^{\omega\,(\beta\,(p^{\mbox{}_{\prime}}-1)+1)\,s}\,\ds
      \,\norm{u_{0,n}}_{1}^{\gamma\,(p^{\mbox{}_{\prime}}-1)+1}.
    \end{split}
  \end{displaymath}
  Applying these to estimate~\eqref{eq:264}, we obtain
  \begin{displaymath}
    \begin{split}
    &\tfrac{1}{2}\int_{0}^{t}s^{\alpha(p^{\mbox{}_{\prime}}-1)+1}\Psi(\phi(u_{n}(s)))\,\ds +
      t^{\alpha(p^{\mbox{}_{\prime}}-1)+1} \int_{\Sigma}\Phi(u_{n}(t))\,\dmu\\
    % &\qquad\le
    %   (k+1)\int_{0}^{t}s^{k}\int_{\Sigma}\Phi(u(t))\,\dmu\ds\\
    % &\hspace{4cm}  -\int_{0}^{t}s^{k+1}\int_{\Sigma} F(u(s))
    %   \phi(u(s))\,\dmu\,\ds\\
    &\qquad\le
      \tfrac{(\alpha(p^{\mbox{}_{\prime}}-1)+1)^{p}\,C^{p-1}}{p^{\mbox{}_{\prime}}(4^{-1}\,p)^{p-1}}\,
      \int_{0}^{t}e^{\omega\,(\beta\,(p^{\mbox{}_{\prime}}-1)+1)\,s}\,\ds
      \,\norm{u_{0,n}}_{1}^{\gamma\,(p^{\mbox{}_{\prime}}-1)+1}\\
   &\hspace{3cm}   +\tfrac{\omega^{p}\,C^{p-1}}{p^{\mbox{}_{\prime}}(4^{-1}\,p)^{p-1}}\,
      \int_{0}^{t}s\,e^{\omega\,(\beta\,(p^{\mbox{}_{\prime}}-1)+1)\,s}\,\ds
      \,\norm{u_{0,n}}_{1}^{\gamma\,(p^{\mbox{}_{\prime}}-1)+1}.
  \end{split}
  \end{displaymath}
  Inserting relation~\eqref{eq:265} into this inequality yields
  \begin{displaymath}
    \begin{split}
    &\tfrac{1}{2}\int_{0}^{t}s^{\alpha(p^{\mbox{}_{\prime}}-1)+1}\Psi(\phi(u(s+t_{n})))\,\ds +
      t^{\alpha(p^{\mbox{}_{\prime}}-1)+1} \int_{\Sigma}\Phi(u(t+t_{n}))\,\dmu\\
    &\qquad\le
      \tfrac{(\alpha(p^{\mbox{}_{\prime}}-1)+1)^{p}\,C^{p-1}}{p^{\mbox{}_{\prime}}(4^{-1}\,p)^{p-1}}\,
      \int_{0}^{t}e^{\omega\,(\beta\,(p^{\mbox{}_{\prime}}-1)+1)\,s}\,\ds
      \,\norm{u(t_{n})}_{1}^{\gamma\,(p^{\mbox{}_{\prime}}-1)+1}\\
   &\hspace{3cm}   +\tfrac{\omega^{p}\,C^{p-1}}{p^{\mbox{}_{\prime}}(4^{-1}\,p)^{p-1}}\,
      \int_{0}^{t}s\,e^{\omega\,(\beta\,(p^{\mbox{}_{\prime}}-1)+1)\,s}\,\ds
      \,\norm{u(t_{n})}_{1}^{\gamma\,(p^{\mbox{}_{\prime}}-1)+1}.
  \end{split}
  \end{displaymath}
  for every $t>0$ and $n\ge 1$. By the continuity of $\phi$, since
  $u\in C([0,\infty);L^{1}(\Sigma,\mu))$, the lower semicontinuity of
  $\Psi$, and since
  \begin{displaymath}
     \lim_{n\to\infty}\int_{\Sigma}\Phi(u(t+t_{n}))\dmu\to
    \int_{\Sigma}\Phi(u(t))\dmu
 \end{displaymath}
 for every $t>0$, sending $n\to\infty$ in the last inequality yields
 inequality~\eqref{eq:260}. 
  
 Next, suppose that $\phi'\in L^{\infty}(\R)$, $(\phi^{-1})'$ is
 locally bounded, and exponent $0<\alpha\le 1$ in estimate~\eqref{eq:263}. By
 Theorem~\ref{thm:Linfty-implies-mild-are-strong-in-L1}, the function
 $u_{n}$ given by~\eqref{eq:265} is a strong solution of Cauchy
 problem~\eqref{eq:206} with initial value
 $u_{n}(0)=u_{0,n}=T_{t_{n}}u_{0}$ for some sequence
 $(t_{n})_{n\ge 1}\subseteq (0,\infty)$ satisfying
 $t_{n}\downarrow 0+$ as $n\to\infty$. By
 Theorem~\ref{thm:Linfty-implies-mild-are-strong-in-L1}, the function
 $u_{n}$ has regularity~\eqref{eq:124} satisfying the
 properties~$(1)$--$(4)$ of in this theorem. Thus, by~\eqref{eq:265},
 the mild solution $u$ is also a strong solution of Cauchy problem~\eqref{eq:206} 
 on the interval $[t_{n},T)$ admitting regularity~\eqref{eq:124} on
 $[t_{n},T)$ for every $n\ge 1$ large enough. 

 It remains to show that $u$ satisfies energy
 inequality~\eqref{eq:268}. To see this, we use that by
 Theorem~\ref{thm:Linfty-implies-mild-are-strong-in-L1}, $u_{n}$ satisfies
 inequality~\eqref{eq:139} for every $k\ge 0$ and $t>0$, that is,
 \begin{equation}
   \label{eq:271}
   \begin{split}
     &\tfrac{1}{2}\int_{0}^{t}s^{k+2} \int_{\Sigma}
     \phi'(u_{n}(s))\,\labs{\frac{d u_{n}}{ds}(s)}^{2}\,\dmu\,\ds +
     t^{k+2}\,\Psi(\phi(u_{n}(t)))\\
     &\qquad \le (k+2)(k+1) \int_{0}^{t}s^{k}\int_{\Sigma}\phi(u_{n}(t))\,u_{n}(t)\,\dmu\,\ds\\
     &\hspace{3cm} -(k+2) \int_{0}^{t}s^{k+1}\int_{\Sigma}
     F(u_{n}(s))\phi(u_{n}(s))\,\dmu\,\ds\\
     &\hspace{4cm} +
     \tfrac{1}{2}\int_{0}^{t}s\int_{\Sigma}\phi'(u_{n}(s))\,
     \labs{F(u_{n}(s))}^{2}\,\dmu\,\ds.
        \end{split}
 \end{equation}
 Applying the two estimates~\eqref{eq:269} and~\eqref{eq:270} to the
 right hand side of inequality~\eqref{eq:271} yields
 \begin{align*}
    &\tfrac{1}{2}\int_{0}^{t}s^{k+2} \int_{\Sigma}
     \phi'(u_{n}(s))\,\labs{\frac{d u_{n}}{ds}(s)}^{2}\,\dmu\,\ds +
     t^{k+2}\,\Psi(\phi(u_{n}(t)))\\
   &\qquad\le (k+2)\,2\,\left[\varepsilon\, \int_{0}^{t}s^{k+1}
      \Psi(\phi(u_{n}(s)))\,\ds +
      \tfrac{(k+1)^{p}\,C^{p-1}}{p^{\mbox{}_{\prime}}(\varepsilon\,p)^{p-1}}\,
      \int_{0}^{t}s^{k}\norm{u_{n}(s)}_{p^{\mbox{}_{\prime}}}^{p^{\mbox{}_{\prime}}}\,\ds\right]\\
   &\hspace{4cm} +
     \tfrac{1}{2}\int_{0}^{t}s\int_{\Sigma}\phi'(u_{n}(s))\,
     \labs{F(u_{n}(s))}^{2}\,\dmu\,\ds.
 \end{align*}
 for every $\varepsilon>0$ and $k\ge 0$. Now, taking
 $k=\alpha(p^{\mbox{}_{\prime}}-1)$ and $\varepsilon=\tfrac{1}{2}$,
 then by inequality~\eqref{eq:260}, the
 Lipschitz continuity of $F$, H\"older's inequality, since by
 assumption $\phi'\in L^{\infty}(\R)$, by the $L^{1}$-$L^{\infty}$
 regularisation estimate~\eqref{eq:263}, and by the exponential growth
 property
 \begin{displaymath}
   \norm{u(t)}_{1}\le e^{\omega t}\,\norm{u_{0}}_{1}\qquad\text{for
     all $t\ge 0$,}
 \end{displaymath}
 we see that
\begin{align*}
  &\tfrac{1}{2}\int_{0}^{t}s^{\alpha(p^{\mbox{}_{\prime}}-1)+2} \int_{\Sigma}
    \phi'(u_{n}(s))\,\labs{\frac{d u_{n}}{ds}(s)}^{2}\,\dmu\,\ds +
    t^{\alpha(p^{\mbox{}_{\prime}}-1)+2}\,\Psi(\phi(u_{n}(t)))\\
  &\qquad\le \tfrac{(\alpha(p^{\mbox{}_{\prime}}-1)+2)\,2}{p^{\mbox{}_{\prime}}}\,\left[
    \tfrac{(\alpha(p^{\mbox{}_{\prime}}-1)+1)^{p}\,C^{p-1}}{(4^{-1}\,p)^{p-1}}\,
    \int_{0}^{t}e^{\omega\,(\beta\,(p^{\mbox{}_{\prime}}-1)+1)\,s}\,\ds
    \,\norm{u_{0,n}}_{1}^{\gamma\,(p^{\mbox{}_{\prime}}-1)+1}\right.\\
  &\hspace{3,5cm}
    +\tfrac{\omega^{p}\,C^{p-1}}{(4^{-1}\,p)^{p-1}}\,
    \int_{0}^{t}s\,e^{\omega\,(\beta\,(p^{\mbox{}_{\prime}}-1)+1)\,s}\,\ds
    \,\norm{u_{0,n}}_{1}^{\gamma\,(p^{\mbox{}_{\prime}}-1)+1}\\
  & \hspace{4cm}\left.+ 
    \tfrac{(\alpha(p^{\mbox{}_{\prime}}-1)+1)^{p}\,C^{p-1}}{(2^{-1}\,p)^{p-1}}\,
    \int_{0}^{t}s^{\alpha(p^{\mbox{}_{\prime}}-1)}\norm{u_{n}(s)}_{\infty}^{p^{\mbox{}_{\prime}}-1}\, 
    \norm{u_{n}(s)}_{1}\,\ds\right]\\
  &\hspace{5cm} +
    \tfrac{\omega^{2} }{2}\norm{\phi'}_{\infty}
    \int_{0}^{t}s\,\norm{u_{n}(s)}_{\infty}\,
    \norm{u_{n}(s)}_{1}\,\ds.\\
%%%%%%%%%%%%%%%%%%%%%%%%%%%%
&\qquad\le \tfrac{(\alpha(p^{\mbox{}_{\prime}}-1)+2)\,2}{p^{\mbox{}_{\prime}}}\,\left[
        \tfrac{(k+1)^{p}\,C^{p-1}}{(4^{-1}\,p)^{p-1}}\,
        \int_{0}^{t}e^{\omega\,(\beta\,(p^{\mbox{}_{\prime}}-1)+1)\,s}\,\ds
        \,\norm{u_{0,n}}_{1}^{\gamma\,(p^{\mbox{}_{\prime}}-1)+1}\right.\\
    &\hspace{3,5cm}
        +\tfrac{\omega^{p}\,C^{p-1}}{(4^{-1}\,p)^{p-1}}\,
        \int_{0}^{t}s\,e^{\omega\,(\beta\,(p^{\mbox{}_{\prime}}-1)+1)\,s}\,\ds
        \,\norm{u_{0,n}}_{1}^{\gamma\,(p^{\mbox{}_{\prime}}-1)+1}\\
  & \hspace{4cm}\left.
    + \tfrac{(\alpha(p^{\mbox{}_{\prime}}-1)+1)^{p}\,C^{p}}{p^{\mbox{}_{\prime}}(2^{-1}\,p)^{p-1}}\,
      \int_{0}^{t}e^{\omega\,\beta (p^{\mbox{}_{\prime}}-1)\,s}\,\ds
    \norm{u_{0,n}}_{1}^{\gamma (p^{\mbox{}_{\prime}}-1)+1}\,\right]\\
   &\hspace{5cm} +
     \tfrac{\omega^{2}\,\tilde{C} }{2}\norm{\phi'}_{\infty}
     \int_{0}^{t}s^{1-\alpha}\,e^{\omega(\beta+1)\,s }\,\ds \norm{u_{0,n}}_{1}^{\gamma+1}.
 \end{align*}
 From this estimate, the assumptions on $\phi$ and by
 inequality~\eqref{eq:263}, it is not difficult to deduce that
 inequality~\eqref{eq:268} holds and $u$ admits the stated
 properties. This concludes the proof of this theorem.
\end{proof}

%%%%%%%%%%%%%%%%%%%%%%%%%%%%%%%%%%%%%%%%%%%%%%%%%%%%%%%%%%
%%%%%%%%%%%%%%%%%%%%%%%%%%%%%%%%%%%%%%%%%%%%%%%%%%%%%%%%%%
%%%%%%%%%%%%%%%%%%%%%%%%%%%%%%%%%%%%%%%%%%%%%%%%%%%%%%%%%%
%%%%%%%%%%%%%%%%%%%%%%%%%%%%%%%%%%%%%%%%%%%%%%%%%%%%%%%%%%
%%%%%%%%%%%%%%%%%%%%%%%%%%%%%%%%%%%%%%%%%%%%%%%%%%%%%%%%%%
%%%%%%%%%%%%%%%%%%%%%%%%%%%%%%%%%%%%%%%%%%%%%%%%%%%%%%%%%%
%%%%%%%%%%%%%%%%%%%%%%%%%%%%%%%%%%%%%%%%%%%%%%%%%%%%%%%%%%
%%%%%%%%%%%%%%%%%%%%%%%%%%%%%%%%%%%%%%%%%%%%%%%%%%%%%%%%%%
%%%%%%%%%%%%%%%%%%%%%%%%%%%%%%%%%%%%%%%%%%%%%%%%%%%%%%%%%%
%%%%%%%%%%%%%%%%%%%%%%%%%%%%%%%%%%%%%%%%%%%%%%%%%%%%%%%%%%
%%%%%%%%%%%%%%%%%%%%%%%%%%%%%%%%%%%%%%%%%%%%%%%%%%%%%%%%%%
%%%%%%%%%%%%%%%%%%%%%%%%%%%%%%%%%%%%%%%%%%%%%%%%%%%%%%%%%%
%%%%%%%%%%%%%%%%%%%%%%%%%%%%%%%%%%%%%%%%%%%%%%%%%%%%%%%%%%
%%%%%%%%%%%%%%%%%%%%%%%%%%%%%%%%%%%%%%%%%%%%%%%%%%%%%%%%%%
%%%%%%%%%%%%%%%%%%%%%%%%%%%%%%%%%%%%%%%%%%%%%%%%%%%%%%%%%%

%-----------------------------------------------------------
%
%
%                     E X A M P L E S / A P P L I C A T I O N S
%
%
%
%-----------------------------------------------------------

\section{Examples}
\label{sec:examples}

This section is devoted to illustrating the power of the theory
developed in the preceding sections. By using the abstract theory of
nonlinear semigroups, we show in this section that mild solutions of nonlinear
parabolic initial boundary-value problems satisfy an $L^{1}$-$L^{\infty}$-regularisation
effect provided the involved diffusion operator satisfies a Gagliardo-Nirenberg type
inequality. Comparing our first examples with the results from
the known literature, one sees that the methods developed in
Section~\ref{gn} and Section~\ref{extra} yield sharp exponents and
extend these results for solutions with exponential growth.

%Note that most of our examples could be introduced on manifolds  under
%the additional assumption that a Gagliardo-Nirenberg type inequality
%(or, simply a Sobolev type inequality) holds. As a
%matter of fact, the parameters in the Gagliardo-Nirenberg type
%inequality are dependent on the curvature and regularity of the manifold. As one can
%expect, this has an immediate impact on the regularisation effect of the semigroup
%associated with the generator satisfying the given Gagliardo-Nirenberg
%type inequality. %This phenomenon shall be outlined in more details in our forthcoming paper.

Note that in principle our theory could work for non-linear operators
on non-compact manifolds, such as the porous media operators
associated with the Laplace-Beltrami operator or the $p$-Laplace
operator (for the latter see \cite{MR1878317}). Here
the Sobolev inequalities for the gradient depend on the geometry of
the manifold (see~\cite{MR2039952}), the main task is to deduce
Gagliardo-Nirenberg inequalities for the operator under consideration
by adapting the methods of the present section, then one applies the
above machinery. We leave this for future work.

%Therefore, in order to keep our examples simple and to focus on
%the essential, namely, the regularisation effect of solutions
%of parabolic boundary-value problems, it suffices to assume that
%$\Sigma$ is an open subset of the $d$-dimen\-sional Euclidean space
%$\R^{d}$ for $d\ge 2$. 
In order to keep our examples simple and to focus on
the essential, namely, the regularisation effect of solutions
of parabolic boundary-value problems, we shall assume that
$\Sigma$ is an open subset of the $d$-dimen\-sional Euclidean space
$\R^{d}$ for $d\ge 2$. 
We shall specify at the beginning of each example which
further assumptions we impose on the boundary $\partial\Sigma$ of
$\Sigma$. We choose $\mu$ to be the $d$-dimensional Lebesgue measure on
$\Sigma$ and denote by $\mathcal{H}=\mathcal{H}^{d-1}_{\vert \partial\Sigma}$ the
$(d-1)$-dimensional Hausdorff measure $\mathcal{H}^{d-1}$ restricted
to the boundary $\partial\Sigma$. 

Under this assumptions, we simplify our
notation and write $L^{q}(\Sigma)$ to denote the Lebesgue space
$L^{q}(\Sigma,\mu)$, $L^{q}(\partial\Sigma)$ to denote the Lebesgue space
$L^{q}(\partial\Sigma,\mathcal{H})$, and $L^{q}_{0}(\Sigma,\mu)$ the closed
linear subspace $u\in L^{q}(\Sigma,\mu)$ with \emph{mean value}
$\overline{u}:=\tfrac{1}{\abs{\Sigma}}\int_{\Sigma}u\,\dx=0$ for $1\le q\le \infty$.

Here, we employ the following notation: for $1\le p,q\le \infty$, let
$W^{1}_{p,q}(\Sigma)$ be the linear subspace of all functions $u\in
L^{q}(\Sigma)$ having weak partial derivatives $\tfrac{\partial
  u}{\partial x_{1}}, \dots,\tfrac{\partial u}{\partial x_{d}}\in
L^{p}(\Sigma)$ equipped with the norm
\begin{displaymath} 
\norm{u}_{W^{1}_{p,q}}:=\norm{u}_{q}+\norm{|\nabla u|}_{p}.
\end{displaymath} 
Moreover, for $1\le p, q<\infty$, we denote by
$\dot{W}^{1}_{p,q}(\Sigma)$ the closure of the set of test functions
$C_{c}^{\infty}(\Sigma)$ in $W_{p,q}^{1}(\Sigma)$,
$W^{1}_{p,q,m}(\Sigma)$ the space
$L^{q}_{0}(\Sigma)\cap W_{p,q}^{1}(\Sigma)$, and for $0<s<1$,
$W^{s}_{p,q}(\Sigma)$ denotes the set of all $u\in L^{q}(\Sigma)$ with
finite semi-norm
\begin{displaymath}
  \abs{u}_{s,p}^{p}:=\int_{\Sigma}
  \int_{\Sigma}\tfrac{\abs{u(x)-u(y)}^{p}
    }{\abs{x-y}^{d+sp}}\,\dx\dy.
\end{displaymath}
We equip $W^{s}_{p,q}(\Sigma)$ with the norm
$\norm{u}_{s,p,q}=\norm{u}_{q}+\abs{u}_{s,p}$. Further, we denote by
$\dot{W}^{s}_{p,q}(\Sigma)$ the closure of $C_{c}^{\infty}(\Sigma)$ in
$W^{s}_{p,q}(\Sigma)$.
% by
% $\tilde{W}^{1}_{p,q}(\Sigma)$ the closure of $W^{1}_{p,q}(\Sigma)\cap
% C_{c}(\overline{\Sigma})$ in $W^{1}_{p,q}(\Sigma)$,
For subsets $\partial\Sigma$ in $\R^{d-1}$, we denote by
$W^{1-1/p,p}(\partial\Sigma)$ the Sobolev-Slobodeckij
space given by the set of all $u\in
L^{p}(\partial\Sigma)$ having finite semi-norm
\begin{displaymath}
  \abs{u}_{p}^{p}:=\int_{\partial\Sigma}
  \int_{\partial\Sigma}\tfrac{\abs{u(x)-u(y)}^{p}
    }{\abs{x-y}^{d+p-2}}\,\dH(x)\dH(y).
\end{displaymath}

In the following, $F : L^{q}(\Sigma,\mu)\to L^{q}(\Sigma,\mu)$ be the
  Nemytski operator of a Carath\'eodory function
  $f : \Sigma\times\R\to \R$ satisfying~\eqref{eq:2} for some constant
  $L> 0$ and $f(\cdot,0)=0$ and $\beta$ be an
$m$-accretive graph in $\R$ with domain $D(\beta)=\R$ and
$(0,0)\in \beta$.\bigskip

We begin to illustrate our theory on the following classical example.

%%%%%%%%%%%%%%%%%%%%%%%%%%%%%%%%%%%
%
%
%
%         1. Subsection: p-Laplace type diffusion equations
%
%
%
%%%%%%%%%%%%%%%%%%%%%%%%%%%%%%%%%%%

\subsection{Parabolic problems involving $p$-Laplace type operators}
\label{sec:p-laplace}

The $L^{q}$-$L^{r}$-regu\-larisation effect for $1\le q <r\le \infty$ of
solutions of parabolic equations associated with the celebrated
$p$-Laplace operators equipped with homogeneous Dirichlet boundary
conditions has been first established by V\'eron~\cite{MR554377}. The
ideas in~\cite{MR554377} were followed up and extended rapidly by
Alikakos and Rostam\-ian~\cite{MR656651} and more recently
in~\cite{MR1746793,MR1741878}. By using the logarithmic Sobolev
approach, Cipriani and Grillo~\cite{MR1867617} revisited the
$L^{q}$-$L^{r}$-regularisation effect for solutions of parabolic
equations involving $p$-Laplace operators equipped with homogeneous
boundary conditions. Then many papers followed on this topic by using
the same method (see, for
instance, \cite{Takac05,MR2379911,MR2053885,MR2529737}, and more
recently, ~\cite{MR3158845} for homogeneous Robin boundary
conditions with a nonlocal term).\bigskip

To the best of our knowledge, our results stated in this section
complement the existing literature in several ways: namely, by adding
(possibly multi-valued) monotone and Lipschitz continuous
perturbations and by providing a simplified approach to a
$L^{q}$-$L^{r}$-regularisation effect of solutions of parabolic
boundary-value problems associated with $p$-Laplace type operators.\bigskip

Further, the examples in this subsection show that the parameter
$m_{0}$ appearing in the two main theorems Theorem~\ref{thm:main-1}
and Theorem~\ref{thm:GN-implies-reg-bis} is \emph{optimal} if
$m_{0}=q\,\gamma^{-1}$ (cf. Remark~\ref{rem:11}). To be more precise,
consider the case $1<p<d$ and let $\{T_{t}\}_{t\ge 0}$ be the
semigroup generated by the negative $p$-Laplace operator
$-\Delta_{p}^{\!\R^{d}}$ on $L^{2}(\R^{d})$. Then, we show in the
proof of Theorem~\ref{thm:Dirichlet-p-laplace} below that
$\{T_{t}\}_{t\ge 0}$ satisfies $L^{q}$-$L^{r}$ regularity
estimate~\eqref{eq:168} for $u_{0}=0$ with parameters
$r=\tfrac{pd}{d-p}$, $q=2$ and exponent $\gamma=\tfrac{2}{p}$. One
easily sees that $\gamma\,r>q$ and so one can deduce an
$L^{s}$-$L^{\infty}$ regularisation estimate for
$s=\gamma\,r\,q^{-1}\,m_{0}=\frac{dm_{0}}{d-p}$ and sufficiently large
$m_{0}\ge q\,\gamma^{-1}=p$. By Theorem~\ref{thm:Dirichlet-p-laplace},
if $\tfrac{2d}{d+2}<p<d$ then $m_{0}=p$ satisfies~\eqref{eq:75},
and if $\tfrac{2d}{d+1}<p<d$ then for $m_{0}=p$, the semigroup
$\{T_{t}\}_{t\ge 0}\sim -\Delta_{p}^{\!\R^{d}}$ for satisfies 
$L^{1}$-$L^{\infty}$-regularisation estimate~\eqref{eq:168} with
exponent $\alpha_{1}=\tfrac{d}{d(p-2)+p}$ und $u_{0}=0$. The exponent $\alpha_{1}$
coincides with exponent $\frac{d}{\lambda}$ in the \emph{Barenblatt
solution}
\begin{equation}
  \label{eq:250}
  \begin{split}
    &\Gamma_{p}(x,t):=t^{-\frac{d}{\lambda}}\left[1+C_{p}\left(\frac{\abs{x}}{t^{\frac{1}{\lambda}}}\right)^{\frac{p}{p-1}}
    \right]^{\frac{p-1}{p-2}}_{+},\quad\text{for $t>0$,}\\
    &\quad\text{with $\lambda=d(p-2)+p$,
      $C_{p}=\left(\frac{1}{\lambda}\right)^{\frac{1}{p-1}}\frac{2-p}{p}$,}
  \end{split}
\end{equation}
 to the prototype parabolic $p$-Laplace equation 
\begin{displaymath}
 \partial_{t}u-\Delta_{p}^{\!\R^{d}}u=0\qquad\text{on $\R^{d}\times (0,\infty)$.}
\end{displaymath}
Note, the Barenblatt solution~\eqref{eq:250} also
holds for the \emph{singular} range $1<p<2$ provided the parameter
$\lambda>0$. Moreover, $\lambda>0$ if and only if
$\tfrac{2d}{d+1}<p<2$ (see \cite[Chapter 7.4]{MR2865434} ). It is
worth noting that for singular $1<p<2$, the existence of a Barenblatt
solution coincides with the fact that semigroup $\{T_{t}\}_{t\ge 0}$
generated by the negative $p$-Laplace operator
$-\Delta_{p}^{\!\R^{d}}$ on $L^{2}(\R^{d})$ satisfies
$L^{1}$-$L^{\infty}$-regularisation estimate~\eqref{eq:168} for
$u_{0}=0$ with exponent $\alpha_{1}=\tfrac{d}{d(p-2)+p}$, but also that for
the same range $\tfrac{2d}{d+1}<p<2$, every positive weak energy solutions of
problem~\eqref{ip:p-laplace} (below) satisfy a \emph{Harnack
  inequality} (cf.~\cite[Chapter 7.4]{MR2865434}). In the degenerated
range  $2<p<\infty$, the comparison of the optimal exponents $\alpha_{1}$ has been
considered, for instance, in~\cite{MR2268115}.\bigskip

Throughout this section, let $1<p<\infty$. Then for given initial
value $u_{0} \in L^{q}(\Sigma)$, we intend to establish the
regularisation effect of solutions $u(t)=u(x,t)$ for $t>0$ of the
\emph{parabolic initial value problem}
\begin{equation}
  \label{ip:p-laplace}
  \begin{cases}
    \partial_{t}u-\textrm{div} (a(x,\nabla u))+\beta(u)+f(x,u)\ni 0 
    &\qquad\text{on $\Sigma\times
      (0,\infty)$,}\\
    u(\cdot,0)=u_{0} & \qquad\text{on $\Sigma$,}
  \end{cases}
\end{equation}
respectively equipped with one of the following types of boundary
conditions:
\begin{align}
  \label{eq:44} u=0\quad\text{on $\partial\Sigma\times
    (0,\infty)$,} & \;
    \text{if $\Sigma\subseteq \R^{d}$,}\\ 
   \label{eq:45} a(x,\nabla u)\cdot\nu
     =0\quad\text{on $\partial\Sigma\times (0,\infty)$,} & \;
     \text{if $\mu(\Sigma)<\infty$,}\\
  \label{eq:46}  a(x,\nabla u)\cdot\nu+b(x)
  \abs{u}^{p-1}u+d\,\theta_{p}(u)= 0
  \quad\text{on $\partial\Sigma\times
     (0,\infty)$,} & \;
     \text{if $\mu(\Sigma)<\infty$.}
\end{align}

Here, we suppose that $a : \Sigma \times \R^{d}\to \R^{d}$ is
a Carath\'eodory function satisfying the following
\emph{$p$-coercivity}, \emph{growth} and \emph{monotonicity} conditions
 \begin{align}
   \label{eq:coerciveness}
   &a(x,\xi)\xi\ge \eta \abs{\xi}^{p}\\
   \label{eq:growth-cond}
   &\abs{a(x,\xi)}\le c_{1}\abs{\xi}^{p-1}+h(x)\\
   \label{eq:monotonicity-of-a}
   &(a(x,\xi_{1})-a(x,\xi_{2}))(\xi_{1}-\xi_{2})>0
 \end{align}
 for a.e. $x\in \Sigma$ and all $\xi$, $\xi_{1}$, $\xi_{2}\in \R^{d}$
 with $\xi_{1}\neq \xi_{2}$, where $h\in
 L^{p^{\mbox{}_{\prime}}}(\Sigma)$ and $c_{1}$, $\eta>0$ are constants
 independent of $x\in \Sigma$ and $\xi\in \R^{d}$. Under these assumptions, the
 second order quasi linear operator
 \begin{equation}
   \label{eq:gen-div-operator}
  \mathcal{B}u:= -\divergence (a(x,\nabla u))\qquad\text{in $D'(\Sigma)$}
 \end{equation}
 for $u\in W^{1,p}_{loc}(\Omega)$ belongs to the class of
 \emph{Leray-Lions operators} (cf.~\cite{MR0194733}), of which the \emph{$p$-Laplace
 operator} $\Delta_{p}u=\textrm{div}(\abs{\nabla u}^{p-2}\nabla u)$ is a classical prototype.

In some situations, one can replace~\eqref{eq:monotonicity-of-a} by
\begin{equation}
  \label{eq:54}
  (a(x,\xi_{1})-a(x,\xi_{2}))(\xi_{1}-\xi_{2}) \ge \tilde{\eta}\,\abs{\xi_{1}-\xi_{2}}^{p}
\end{equation}
for a.e. $x\in \Sigma$ and all $\xi_{1}$, $\xi_{2}\in \R^{d}$.
In fact, it is well-known (\cite{MR1230384}) that for $p\ge 2$, the
$p$-Laplace operator satisfies inequality~\eqref{eq:54} with constant
$\tilde{\eta}=2^{2-p}$. 

Regarding \emph{homogeneous Dirichlet boundary
  conditions}~\eqref{eq:44}, we assume that $\Sigma$ is an open subset
of $\R^{d}$ and impose no further assumptions on the boundary
$\partial\Sigma$ of $\Sigma$. In the case $\Sigma=\R^{d}$, the homogeneous Dirichlet
boundary conditions~\eqref{eq:44} become the following \emph{vanishing
at infinity condition}
\begin{equation}\label{eq:43}
      \lim_{\abs{x}\to\infty}u(x,t)=0\quad\text{for every $t>0$,}  \quad
      \text{if $\Sigma=\R^{d}$.} 
\end{equation}

Concerning \emph{homogeneous Neumann boundary
  conditions}~\eqref{eq:45}, we assume that $\Sigma$ is an open bounded
domain with a Lipschitz boundary $\partial\Sigma$ (in the sense of~\cite[Sect.~1.3]{MR0227584}).
We denote by $\nu$ the (weak) outward pointing unit normal vector on
$\partial\Sigma$. Under this assumption, it is not clear whether the
co-normal derivative $a(x,\nabla u)\cdot \nu$ on $\partial\Sigma$
exists. Thus, the Neumann boundary condition~\eqref{eq:45} needs to be
understood in a \emph{weak sense} and so, we denote
by $a(x,\nabla u)\cdot \nu$ the \emph{generalised co-normal
  derivative} of $u$ at $\partial\Sigma$ associated with the operator
$\mathcal{B}$ (as, for instance, described in~\cite{MR2981020}).

Considering \emph{homogeneous Robin boundary
  conditions}~\eqref{eq:46}, we assume that $\Sigma$ is a open bounded
domain with a Lipschitz boundary $\partial\Sigma$, $d\ge 0$ is a
constant, $b\in L^{\infty}(\partial\Sigma)$ such that
$b(x)\ge b_{0}>0$ for $\mathcal{H}$-a.e. $x\in \partial\Sigma$. The
operator $\theta_{p}$ describes the \emph{nonlocal term on
  $\partial\Sigma$} and is given by
\begin{equation}
    \label{eq:87}
  \langle \theta_{p}(u),v\rangle=\int_{\partial\Sigma}
  \int_{\partial\Sigma}\tfrac{\abs{u(x)-u(y)}^{p-2}
    (u(x)-u(y))}{\abs{x-y}^{d+p-2}}(v(x)-v(y))\,\dH(x)\dH(y)
\end{equation}
for every $u$, $v\in W^{1-1/p,p}(\partial\Sigma)$.

%%%%%%%%%%%%%%%%%%%%%%%%%%%%%%%%%%%%%%%%%%%%%%%%%%%%
%
%      homogeneous Dirichlet boundary conditions
%
%%%%%%%%%%%%%%%%%%%%%%%%%%%%%%%%%%%%%%%%%%%%%%%%%%%%

\subsubsection{Homogeneous Dirichlet boundary conditions}
\label{sec:homog-dirichl-bound}

Let $\Sigma$ be an open subset of $\R^{d}$. It is well-known
(cf.~\cite{Benilan1990}), at least in the case when
$\Sigma$ is bounded that the Leray-Lions
operator $\mathcal{B}$ given by~\eqref{eq:gen-div-operator} equipped
with homogeneous Dirichlet boundary conditions can be realised as follows:
\begin{align}\label{eq:p-laplace-homogeneous}
 \begin{split}
      &B^{D} = \Big\{ (u,v)\in
      L^{2}(\Sigma)\times L^{2}(\Sigma)\;\Big\vert\;u\in
      \dot{W}_{p,2}^{1}(\Sigma)\; \text{ such that }
      \Big.\\
      & \hspace{3cm}\Big.\int_\Sigma a(x,\nabla u)\nabla
      \xi\dx=\int_{\Sigma} v\,\xi\dx \;\text{ for all }\;\xi\in \dot{W}_{p,2}^{1}(\Sigma)\Big\}.
  \end{split}
\end{align}
We call $B^D$ the \emph{Dirichlet-Leray-Lions operator} in
$L^{2}(\Sigma)$. Note that, since the set of test functions
$C^{\infty}_{c}(\Sigma)$ is contained in $\dot{W}_{p,2}^{1}(\Sigma)$
and dense in $L^{2}(\Sigma)$, $B^{D}$ defines a single-valued operator
on $L^{2}(\Sigma)$ and by using~\eqref{eq:coerciveness}, one obtains that
the domain $D(B^{D})$ is dense in $L^{2}(\Sigma)$. Furthermore,
condition~\eqref{eq:coerciveness} yields $a(x,0)=0$ a.e. on
$\Sigma$ hence, $(0,0)\in B^{D}$.

In the case $\Sigma=\R^{d}$, the space
$\dot{W}_{p,2}^{1}(\Sigma)=W_{p,2}^{1}(\R^{d})$.  Hence the operator
$B^D$ becomes a realisation in $L^{2}(\Sigma)$ of the Leray-Lions
operator $\mathcal{B}$ equipped with vanishing
conditions~\eqref{eq:43}.

To see that $B^D$ is completely accretive in $L^{2}(\Sigma)$,
let $T\in C^{\infty}(\R)$ be such that the derivative $0\le T'\le 1$
with compact support $\textrm{supp}(T')$ and $T(0)=0$. Since for every
$u$, $\hat{u}\in \dot{W}_{p,2}^{1}(\Sigma)$, $T(u-\hat{u})\in
\dot{W}_{p,2}^{1}(\Sigma)$ with
\begin{displaymath}
\nabla T(u-\hat{u})=T'(u-\hat{u})\,\nabla (u-\hat{u})
\end{displaymath}
and by monotonicity
condition~\eqref{eq:monotonicity-of-a}, one sees that
\begin{align*}
 &\int_{\Sigma}T(u-\hat{u})(B^Du-B^D\hat{u})\dx \\
 &\qquad = \int_{\Sigma} (a(x,\nabla
    u)-a(x,\nabla\hat{u}))\nabla(u-\hat{u})T'(u-\hat{u})\,\dx\ge 0.
\end{align*}
Thus, by Proposition~\ref{prop:completely-accretive},
the operator $B^D$ is completely accretive. 

Under the assumptions~\eqref{eq:coerciveness}-\eqref{eq:monotonicity-of-a},
the restriction of the operator $I+ B^D$ on the reflexive Banach space
$V=\dot{W}_{p,2}^{1}(\Sigma)$ satisfies the hypotheses
of~\cite[Th\'eo\-r\`eme~1]{MR0194733}. Recall that an operator $I+B$ on some Banach
space $V$ is \emph{coercive} in $V$ if
\begin{equation}
  \label{eq:91}
  \lim_{\norm{u}_{V}\to\infty}\frac{\langle (I+B)u,u\rangle_{V',V}}{\norm{u}_{V}}=\infty,
\end{equation}
where we denote by $\langle v',v\rangle_{V',V}$ the value of $v'\in V'$
at $v\in V$. In practice, it is often easier to verify that the following
statement holds, which is equivalent to~\eqref{eq:91}: \emph{for every $\alpha \in R$,
the set of all $u\in V$ satisfying
\begin{displaymath}
 \frac{\langle (I+B)u,u\rangle_{V',V}}{\norm{u}_{V}}\le \alpha
\end{displaymath}
is bounded in $V$}. For the operator $B=B^{D}$, the latter statement
holds since for every $\alpha \in \R_{+}:=[0,\infty)$, the set
$\{(a,b)\in \R_{+}^{2}\;\vert a^2 +b^p\le \alpha (a+b)\,\}$ is bounded
in $\R^{2}$. Thus and since $\dot{W}_{p,2}^{1}(\Sigma)$ is
continuously and densely embedded into $L^{2}(\Sigma)$, it follows
that $B^D$ satisfies the range condition~\eqref{eq:range-condition}
for $X=L^{2}(\Sigma)$.

By hypothesis on the $m$-accretive graph $\beta$ on $\R$, one has that
the domain $D(\beta_{2})$ of the associated accretive operator
$\beta_{2}$ in $L^{2}(\Sigma)$ contains the set of test
functions $C_{c}^{\infty}(\Sigma)$. Recall, for every $\lambda>0$, the Yosida operator
$\beta_{2,\lambda}$ of $\beta_{2}$ is given by
$(\beta_{2,\lambda}u)(x)=\beta_{\lambda}(u(x))$ for a.e. $x\in \Sigma$, where
$\beta_{\lambda}$ denotes the Yosida operator of $\beta$ on
$\R$. Since the Yosida operator $\beta_{\lambda} : \R\to \R$ of $\beta$ is
monotone, Lipschitz continuous and satisfies
$\beta_{\lambda}(0)=0$, one has that for every $u\in\dot{W}_{p,2}^{1}(\Sigma)$,
$\beta_{2,\lambda}(u)\in\dot{W}_{p,2}^{1}(\Sigma)$ with
$\nabla \beta_{2,\lambda}(u)=\beta'_{\lambda}(u)\,\nabla u$ a.e. on
$\Sigma$ for all $\lambda>0$. Thus, by definition of $B^{D}$ and by
\eqref{eq:coerciveness},
\begin{displaymath}
  [v,\beta_{2,\lambda}(u)]_{2}=\int_{\Sigma}a(x,\nabla u)\nabla 
  \beta_{2,\lambda}(u)\,\dx \ge \eta\int_{\Sigma}\abs{\nabla
    u}^{p}\beta'_{\lambda}(u)\,\dx\ge 0
\end{displaymath}
 for every $(u,v)\in B^{D}$. Therefore, by
 Proposition~\ref{propo:Lipschitz-complete-accretive}, the operator
 \begin{equation}
   \label{eq:79}
   A^{D}:=B^D+\beta_{2}+F
 \end{equation}
 is quasi $m$-completely accretive in $L^{2}(\Sigma)$ with dense
 domain.

 By the Crandall-Liggett theorem~\cite{MR0287357}, $-A^D$
 generates a strongly continuous semigroup $\{T_{t}\}_{t\ge 0}$ on
 $L^{2}(\Sigma)$ of Lipschitz continuous mappings $T_{t} :
 L^{2}(\Sigma)\to L^{2}(\Sigma)$. Since $-A^{D}$ is completely
 accretive, each mapping $T_{t}$ has a unique Lip\-schitz continuous
 extension on $L^{q}(\Sigma)$ for all $1\le q< \infty$ and
 on $\overline{L^{2}\cap L^{\infty}(\Sigma)}^{\mbox{}_{L^{\infty}}}$ if $q=\infty$,
 respectively with constant $e^{\omega t}$.\bigskip

The complete description of the $L^{q}$-$L^{\infty}$-regularisation effect
of the semigroup $\{T_{t}\}_{t\ge 0}\sim-A^D$ is as follows.

\begin{theorem}
  \label{thm:Dirichlet-p-laplace}
  Suppose the Carath\'eodory function
  $a : \Sigma\times\R^{d}\to \R^{d}$ satisfies the
  conditions~\eqref{eq:growth-cond}, ~\eqref{eq:54} and $a(x,0)=0$ for
  a.e. $x\in \Sigma$. Then the semigroup $\{T_{t}\}_{t\ge 0}\sim-A^D$
  for the operator $A^D$ given by~\eqref{eq:79} satisfies the following
  $L^{q}$-$L^{r}$ regularisation estimates.
  \begin{enumerate}
    \item\label{thm:Dirichlet-p-laplace-claim1} If $1<p<d$, then~\eqref{eq:80} holds for 
      \begin{displaymath}
        \mbox{}\qquad\alpha_{s}=\tfrac{\alpha^{\ast}}{1-\gamma^{\ast}\left(1-\frac{s(d-p)}{d
              m_{0}} \right)},\quad
        \beta_{s}=\tfrac{\frac{\beta^{\ast}}{2}+\gamma^{\ast}
          \frac{s(d-p)}{d m_{0}}}{1-\gamma^{\ast}\left(1-\frac{s(d-p)}{d
              m_{0}}\right)},\quad
        \gamma_{s}=\tfrac{\gamma^{\ast}\,s(d-p)}{d
        m_{0}\left(1-\gamma^{\ast}\left(1-\frac{s (d-p)}{d m_{0}}\right)\right)}
      \end{displaymath}
      for every $m_{0}\ge p$ satisfying
    \begin{math}
      (\frac{d}{d-p}-1)m_{0}+p-2>0,
    \end{math}
    and every $1\le s\le \frac{d m_{0}}{d-p}$ satisfying $s>\frac{d(2-p)}{p}$, where 
    \begin{displaymath}
      \mbox{}\qquad\alpha^{\ast}=\tfrac{d-p}{p m_{0}+(d-p)(p-2)},\;
      \beta^{\ast}=\tfrac{(\frac{2}{p}-1)d+p}{p m_{0}+(d-p)(p-2)}+1,\;
      \gamma^{\ast}=\tfrac{p m_{0}}{p m_{0}+(d-p)(p-2)}.
    \end{displaymath}
    Moreover, if $\frac{2d}{d+2}<p<d$ then one can take $m_{0}=p$ and
    if $\frac{2d}{d+1}<p<d$, then~\eqref{eq:80} holds for every $1\le s\le \frac{d p}{d-p}$.

  \item\label{thm:Dirichlet-p-laplace-claim2} If $p=d\ge 2$, then for
    every $0<\theta <1$, inequality~\eqref{eq:80} holds with exponents
    \begin{displaymath}
     \alpha_{s}=\tfrac{\alpha_{\theta}^{\ast}}{1-\gamma_{\theta}^{\ast}(1-\frac{s(1-\theta)}{2})},\quad
      \beta_{s}=\tfrac{\frac{\beta^{\ast}}{2}+\gamma_{\theta}^{\ast} 
        \frac{s(1-\theta)}{2}}{1-\gamma_{\theta}^{\ast}(1-\frac{s(1-\theta)}{2})},\quad
     \gamma_{s}=\tfrac{\gamma_{\theta}^{\ast}\frac{s(1-\theta)}{2}}{1-\gamma_{\theta}^{\ast}
       (1-\frac{s(1-\theta)}{2})}
    \end{displaymath}
    for every $1\le s \le \frac{2}{1-\theta}$, where
    \begin{displaymath}
      \qquad\begin{array}[c]{c}
      \alpha_{\theta}^{\ast}=\tfrac{[2\theta+p(1-\theta)](1-\theta)}{\theta^{2}},\;
      \beta^{\ast}_{\theta}=\tfrac{[2\theta+p(1-\theta)]^{2}-(1-\theta)p^{2}}{p^{2}2\theta}+1,\; 
      \gamma_{\theta}^{\ast}=\tfrac{2\theta^{2}}{2\theta+p(1-\theta)}.      
      \end{array}
    \end{displaymath}

  \item\label{thm:Dirichlet-p-laplace-claim3} If $d<p<\infty$, then
    inequality~\eqref{eq:80} holds with exponents
    \begin{displaymath}
      \alpha_{s}=\tfrac{\alpha^{\ast}}{1-\gamma^{\ast}(1-\frac{s}{2})},\quad
      \beta_{s}=\tfrac{}{1-\gamma^{\ast}(1-\frac{s}{2})},\quad
      \gamma_{s}=\tfrac{\gamma^{\ast}s}{2(1-\gamma^{\ast}(1-\frac{s}{2}))}
    \end{displaymath}
     for every $1\le s \le 2$, where
    \begin{displaymath}
      \alpha^{\ast}=\tfrac{d}{pd+2(p-d)},\;
      \beta^{\ast}=\gamma^{\ast}+1,\;
      \gamma^{\ast}=\tfrac{2\theta_{0}+p(1-\theta_{0})}{p},\;
      \theta_{0}=\tfrac{p d}{pd+2(p-d)}.
    \end{displaymath}
 \end{enumerate}
 Under the assumptions that $a$
 satisfies~\eqref{eq:coerciveness}-\eqref{eq:monotonicity-of-a}, the
 statements~\eqref{thm:Dirichlet-p-laplace-claim1}-\eqref{thm:Dirichlet-p-laplace-claim3}
 remain true with~\eqref{eq:80} replaced by~\eqref{eq:168} and for $u_{0}=0$.
\end{theorem}

For the proof of this theorem, we employ the classical
Gagliardo-Nirenberg inequalities (\cite{MR0109940}, see
also~\cite{MR1245890}). The Gagliardo-Nirenberg inequalities are valid
for functions $u\in W^{1}_{p,q}(\R^{d})$ and so, in particular, for test
functions $u\in C^{\infty}_{c}(\Sigma)$. Thus we can use of the
following version of Gagliardo-Nirenberg inequalities.

\begin{lemma}[\cite{MR0109940}]
  \label{lem:Sobolev-Gagliardo-Nirenberg}
  For $1\le q, p\le \infty$, let $u\in \dot{W}_{p,q}^{1}(\Sigma)$. Then there
  is a constant $C>0$ depending only on $d$, $q$, $p$, $\theta$ such
  that
  \begin{equation}
    \label{eq:31}
    \norm{u}_{p^{\ast}}\le C\,\norm{|\nabla u|}_{p}^{\theta}\,
    \norm{u}_{q}^{1-\theta},
  \end{equation}
  where
  \begin{equation}
    \label{eq:32}
    \tfrac{1}{p^{\ast}}=\theta \left(\tfrac{1}{p}-\tfrac{1}{d}\right) +
    (1-\theta)\tfrac{1}{q}\;,
  \end{equation}
  for all $\theta \in [0,1]$ with the following exceptional cases:
  \begin{enumerate}
  \item\label{case:1} If $p< d$ and $q=\infty$, then we make the additional assumption
    that either $u$ tends to zero at infinity or $u\in
    L^{\tilde{q}}(\R^{d})$ for some finite $\tilde{q}>0$.
 
  \item\label{case:2} If $1<p<\infty$ and $1-d/p$ is a non negative integer,
    then~\eqref{eq:31} holds only for $\theta\in [0,1)$.
  
  \item If $\Sigma$ is a bounded domain with a Lipschitz boundary,
    then inequality~\eqref{eq:31} is replaced by
    \begin{equation}
      \label{eq:92}
      \norm{u}_{p^{\ast}}\le C\,\left(\norm{|\nabla u|}_{p}^{\theta}\,
    \norm{u}_{q}^{1-\theta}+\norm{u}_{\tilde{q}}\right),
    \end{equation}
    for every $u\in W_{p,q}^{1}(\Sigma)\cap L^{\tilde{q}}(\Sigma)$ and
    any $\tilde{q}>0$, where $p^\ast$ is given by~\eqref{eq:32} for every $\theta\in
    [0,1]$ with the exceptional cases~\eqref{case:1}
    and~\eqref{case:2}, and the constant $C>0$ also depends on the
    domain.
  \end{enumerate}
\end{lemma}

\begin{proof}[Proof of Theorem~\ref{thm:Dirichlet-p-laplace}]
  We begin to consider the case $1<p<d$. Then by
  Lemma~\ref{lem:Sobolev-Gagliardo-Nirenberg}, there is a constant
  $C>0$ such that
  \begin{equation}
    \label{eq:162}
    \norm{u}_{\frac{p d}{d-p}}\le C \norm{|\nabla u|}_{p}
  \end{equation}
  for every $u\in\dot{W}_{p,2}^{1}(\Sigma)$. Thus, by definition of
  the operator $B^D$ and by~\eqref{eq:54}, 
  \begin{align*}
    \norm{u-\hat{u}}_{\frac{p d}{d-p}}^{p}
    &\le C\, \norm{\abs{\nabla(u-\hat{u})}}_{p}^{p}\\
    & \le C\,\tilde{\eta}^{-1} \int_{\Sigma} (a(x,\nabla u)-a(x,\nabla
      \hat{u}))\nabla (u-\hat{u})\,\dx\\
    & =C\,\tilde{\eta}^{-1} \langle u-\hat{u},B^{D}u-B^{D}\hat{u}\rangle
  \end{align*}
 for every $u$, $\hat{u}\in D(B^{D})$. Now, Remark~\ref{rem:2} yields
 the operator $A^{D}$ given
  by~\eqref{eq:79} satisfies the Gagliardo-Nirenberg
 inequality~\eqref{eq:245} with parameters
 \begin{equation}
   \label{eq:235}
   r=\frac{p d}{d-p},\quad \sigma=p,\quad
  \varrho=0,\;\text{ and }\; \omega=L.
\end{equation}
For $\gamma=\tfrac{2}{p}$, one has $\gamma\,r> 2$ and
$m_{0}=2 \gamma^{-1}=p$ satisfies~\eqref{eq:75} if and only if
$p> 2d/(d+2)$. Thus, Theorem~\ref{thm:main-1} yields the first
statement of this theorem.

Next, consider the case $p=d\ge 2$. By
Lemma~\ref{lem:Sobolev-Gagliardo-Nirenberg}, 
  \begin{displaymath}
    \norm{u}_{\frac{2}{1-\theta}}\le C \norm{|\nabla u|}_{p}^{\theta}\,\norm{u}_{2}^{1-\theta}
  \end{displaymath}
  for every $u\in\dot{W}_{p,2}^{1}(\Sigma)$, $0\le \theta
  <1$ and some constant $C>0$. Let $0<\theta<1$. Then by definition of
  the operator $B^D$ and by~\eqref{eq:54}, 
\begin{align*}
    \norm{u-\hat{u}}_{\frac{2}{1-\theta}}^{\frac{p}{\theta}}
    &\le C^{\frac{p}{\theta}}\,
      \norm{\abs{\nabla(u-\hat{u})}}_{p}^{p}\,\norm{u-\hat{u}}_{2}^{\frac{p(1-\theta)}{\theta}}\\
    & \le C^{\frac{p}{\theta}} \,\tilde{\eta}^{-1}\int_{\Sigma} (a(x,\nabla u)-a(x,\nabla
      \hat{u}))\nabla (u-\hat{u})\,\dx\,\norm{u-\hat{u}}_{2}^{\frac{p(1-\theta)}{\theta}}\\
    & = C^{\frac{p}{\theta}}\,\tilde{\eta}^{-1}\,
      \langle u-\hat{u},B^{D}u-B^{D}\hat{u}\rangle\,\norm{u-\hat{u}}_{2}^{\frac{p(1-\theta)}{\theta}}
  \end{align*}
  for every $u$, $\hat{u}\in D(B^{D})$. Thus, by Remark~\ref{rem:2}, the
  operator $A^D$ given by~\eqref{eq:79} satisfies the Gagliardo-Nirenberg
  inequality~\eqref{eq:245} with parameters
 \begin{equation}
   \label{eq:236}
   r_{\theta}=\frac{2}{1-\theta},\; \sigma_{\theta}=\frac{p}{\theta},\;
  \varrho_{\theta}=\frac{p(1-\theta)}{\theta},\; \omega=L\quad
  \text{for every $0<\theta<1$.}
\end{equation}
For $0<\theta<1$, $\gamma_{\theta}:=\frac{2\theta+p(1-\theta)}{p}$
satisfies $\gamma_{\theta}\, r_{\theta} >2$ and by taking
$m_{0}=2 \gamma_{\theta}^{-1}=\frac{2 p }{2\theta+p (1-\theta)}$, one has
\begin{displaymath}
  \left(\frac{\gamma_{\theta}
      r_{\theta}}{2}-1\right)m_{0}+2\left(\frac{1}{\gamma_{\theta}}-1\right)
  =\frac{2\theta}{1-\theta}>0
\end{displaymath}
hence, condition~\eqref{eq:75} holds. Moreover, since
$0<\gamma_{\theta}\le 1$, one easily sees that
$\gamma_{\theta} (1-\frac{s}{r_{\theta}})<1$ for every
$1\le s\le 2^{-1}\gamma_{\theta} r_{\theta} m_{0}
=r_{\theta}$.
Therefore by Theorem~\ref{thm:main-1}, the second statement of this
theorem holds.

Finally, let $d<p<\infty$. Then there is an $0<\theta_{0}<1$ such that
$\theta_{0}
(\frac{1}{p}-\frac{1}{d})+(1-\theta_{0})\frac{1}{2}=0$ or
equivalently, $\theta_{0}=\tfrac{p d}{pd+2(p-d)}$. We
apply Lemma~\ref{lem:Sobolev-Gagliardo-Nirenberg} for this
$\theta_{0}$, to conclude that there is
a constant $C>0$ such that
\begin{displaymath}
    \norm{u}_{\infty}\le C \norm{|\nabla u|}_{p}^{\theta_{0}}\,\norm{u}_{2}^{1-\theta_{0}}
  \end{displaymath}
  for every $u\in\dot{W}_{p,2}^{1}(\Sigma)$. Proceeding as in the
  previous step, we
  see that by~\eqref{eq:54} and by Remark~\ref{rem:2}, the operator
  $A^{D}$ satisfies the Gagliardo-Nirenberg inequality~\eqref{eq:245}
  with parameters
 \begin{equation}
   \label{eq:98}
   r=\infty,\quad \sigma=\frac{p}{\theta_{0}},\quad
  \varrho=\frac{p(1-\theta_{0})}{\theta_{0}},\;\text{ and }\; \omega=L.
\end{equation}
Then, by the first statement of Theorem~\ref{thm:main-1},
$\gamma^{\ast}=\frac{2+\varrho}{\sigma}=\tfrac{2\theta_{0}+p(1-\theta_{0})}{p}$,
$\alpha^{\ast}=\frac{\theta_{0}}{p}$ and
$\beta^{\ast}=\gamma^{\ast}+1$. Moreover, since
$\frac{1}{p}<\gamma^{\ast}<\frac{2}{p}<1$, one has for all
$1\le s\le 2$ that $\gamma (1-\frac{s}{2})<1$. Thus,
Theorem~\ref{thm:extrapol-L1-differences} implies that the third
statement of this theorem holds.
\end{proof}

%%%%%%%%%%%%%%%%%%%%%%%%%%%%%%%%%%%%%%%%%%%%%%%%%%%%
%%%%%%%%%%%%%%%%%%%%%%%%%%%%%%%%%%%%%%%%%%%%%%%%%%%%
%%%%%%%%%%%%%%%%%%%%%%%%%%%%%%%%%%%%%%%%%%%%%%%%%%%%
%%%%%%%%%%%%%%%%%%%%%%%%%%%%%%%%%%%%%%%%%%%%%%%%%%%%
%%%%%%%%%%%%%%%%%%%%%%%%%%%%%%%%%%%%%%%%%%%%%%%%%%%%
%%%%%%%%%%%%%%%%%%%%%%%%%%%%%%%%%%%%%%%%%%%%%%%%%%%%
%%%%%%%%%%%%%%%%%%%%%%%%%%%%%%%%%%%%%%%%%%%%%%%%%%%%
%
%      homogeneous Neumann boundary conditions
%
%%%%%%%%%%%%%%%%%%%%%%%%%%%%%%%%%%%%%%%%%%%%%%%%%%%%
%%%%%%%%%%%%%%%%%%%%%%%%%%%%%%%%%%%%%%%%%%%%%%%%%%%%
%%%%%%%%%%%%%%%%%%%%%%%%%%%%%%%%%%%%%%%%%%%%%%%%%%%%
%%%%%%%%%%%%%%%%%%%%%%%%%%%%%%%%%%%%%%%%%%%%%%%%%%%%
%%%%%%%%%%%%%%%%%%%%%%%%%%%%%%%%%%%%%%%%%%%%%%%%%%%%
%%%%%%%%%%%%%%%%%%%%%%%%%%%%%%%%%%%%%%%%%%%%%%%%%%%%

\subsubsection{Homogeneous Neumann boundary conditions}

In this subsection, we assume that $\Sigma$ is a bounded domain with a
Lipschitz boundary.

Further, we assume that the monotone graph
 $\beta$ on $\R$ either satisfies
 \begin{equation}
   \label{eq:84}
   (v-\hat{v})(u-\hat{u})\ge \eta_{0}\abs{u-\hat{u}}^{p}
 \end{equation}
 or
 \begin{equation}
   \label{eq:85}
   v u\ge \eta_{0}\abs{u}^{p}
 \end{equation}
 for every $(u,v)$, $(\hat{u},\hat{v})\in \beta$.

 We define the realisation $B^{N}$ in $L^{2}(\Sigma)$ of the Leray-Lions
  operator $\mathcal{B}$ equipped with homogeneous Neumann boundary
  conditions~\eqref{eq:45} by
\begin{equation}
  \label{eq:83}
  \begin{split}
  B^{N}&= \Big\{ (u,v)\in
      L^{2}(\Sigma)\times L^{2}(\Sigma)\;\Big\vert\;u\in
      W_{p,2}^{1}(\Sigma)\; \text{ such that }
      \Big.\\
      & \hspace{2cm}\Big.\int_\Sigma a(x,\nabla u)\nabla
      \xi\dx=\int_{\Sigma} v\,\xi\dx \;\text{ for all }\;\xi\in
      W_{p,2}^{1}(\Sigma)\Big\}.
  \end{split}
\end{equation}

Under the assumption that $u$, $\xi$ and $a(\cdot,\nabla
u)$ are smooth functions up to the boundary
$\partial\Sigma$ and $\nu$ denotes the outward pointing unit normal vector on
$\partial\Sigma$, the application of Green's first identity yields
 \begin{displaymath}
  \int_{\Sigma}a(x,\nabla u)\nabla \xi\,\dx
  =-\int_{\Sigma}\textrm{div}\left(a(x,\nabla u)\right)\,\xi\,\dx+
  \int_{\partial\Sigma} a(x,\nabla u)\cdot\nu\; \xi\dH. 
 \end{displaymath}
 Thus, if $u\in D(B^{N})$, one has that
 $v=-\textrm{div}\left(a(x,\nabla u)\right)$ and
 $ a(x,\nabla u)\cdot\nu=0$ for
 $\mathcal{H}^{d-1}$-a.e. $x\in \partial\Sigma$, showing that our
 definition of the operator $B^{N}$ is consistent with the
 \emph{smooth} situation. We call $B^{N}$ the \emph{Neumann
   Leray-Lions operator} in $L^{2}(\Sigma)$.

 In order to see that $B^N$ is $m$-completely accretive in $L^{2}(\Sigma)$
 and that the monotone graph $\beta_{2}$ in $L^{2}(\Sigma)$ satisfies
 the hypothesis~\eqref{eq:78} in
 Proposition~\ref{propo:Lipschitz-complete-accretive} with respect to
 the operator $B^N$, one proceeds as in the previous example (for homogeneous
 Dirichlet boundary conditions), but here one needs to replace the
 space $\dot{W}_{p,2}^{1}(\Sigma)$ by $W_{p,2}^{1}(\Sigma)$. In addition,
 it is not difficult to check that the domain $D(B^{N})$ is dense in
 $L^{2}(\Sigma)$. Therefore, the operator
 \begin{equation}
   \label{eq:82}
   A^N:=B^N+\beta_{2}+F
 \end{equation}
 is quasi $m$-completely accretive in $L^{2}(\Sigma)$ with dense
 domain. By the Cran\-dall-Liggett theorem, $-A^N$ generates a
 strongly continuous semigroup $\{T_{t}\}_{t\ge 0}$ on $L^{2}(\Sigma)$
 of Lipschitz continuous mappings $T_{t}$ on $L^{2}(\Sigma)$. The
 space $L^{\infty}(\Sigma)$ is continuously embedded into
 $L^{2}(\Sigma)$ since $\Sigma$ is bounded. Thus, and since
 $T_{t} : L^{q}\cap L^{2}(\Sigma)\to L^{q}\cap L^{2}(\Sigma)$ is
 Lipschitz continuous with respect to the $L^{q}$-norm with constant
 $e^{\omega t}$ for $1\le q\le\infty$, $T_{t}$ admits a unique
 Lipschitz continuous extension on $L^{q}(\Sigma)$ with the same
 Lipschitz constant $e^{\omega t}$ for every
 $1\le q\le \infty$.

 Now, we state the
 complete description of the $L^{q}$-$L^{\infty}$-regularisation effect of
 the semigroup $\{T_{t}\}_{t\ge 0}\sim-A^N$.

\begin{theorem}
  \label{thm:Neumann-p-laplace}
  Suppose the Carath\'eodory function
  $a : \Sigma\times\R^{d}\to \R^{d}$ satisfies growth
  condition~\eqref{eq:growth-cond}, $A_{\phi}^{N}$ is the operator given
  by~\eqref{eq:79}, and
  $\overline{u}:=\tfrac{1}{\mu(\Sigma)}\int_{\Sigma}u\,\dx$ for any $u\in
  L^{1}(\Sigma)$. Then the following statements hold:
 \begin{enumerate}
 \item If $a$ satisfies the strong monotonicity
   condition~\eqref{eq:54}, $a(x,0)=0$ for a.e. $x\in \Sigma$, and the
   monotone graph $\beta$ satisfies~\eqref{eq:84}, then the semigroup
   $\{T_{t}\}_{t\ge 0}\sim-A_{\phi}^{N}$ on $L^{2}(\Sigma)$ satisfies the
   regularisation estimates~\eqref{eq:18} and~\eqref{eq:80} with the
   same exponents as the semigroup generated by $-A^D$.

 \item If $a$
   satisfies~\eqref{eq:coerciveness}-\eqref{eq:monotonicity-of-a}, and
   the monotone graph $\beta$ satisfies~\eqref{eq:85}, then the
   semigroup $\{T_{t}\}_{t\ge 0}\sim-A_{\phi}^{N}$ on $L^{2}(\Sigma)$
   satisfies the regularisation estimates~\eqref{eq:20}
   and~\eqref{eq:168} with the same exponents as the semigroup
   generated by $-A^D$. Moreover, the semigroup
   $\{T_{t}\}_{t\ge 0}\sim-B^N$ on $L^{2}(\Sigma)$ satisfies
   \begin{equation}
     \label{eq:196}
    \norm{T_{t}u-\overline{u}}_{\infty}\lesssim \; t^{-\alpha_{s}}\;
    \;e^{\omega \beta_{s} t}\; \norm{u-\overline{u}}_{s}^{\gamma_{s}}
  \end{equation}
  for every $t>0$, $u\in L^{s}(\Sigma)$ for $1\le s<\infty$ and
  exponents $\alpha_{s}$, $\beta_{s}$ and $\gamma_{s}$ as given in
  Theorem~\ref{thm:Dirichlet-p-laplace} for the semigroup generated by
  $-A^{D}$.
 \end{enumerate}
\end{theorem}

For the proof of this theorem,
Lemma~\ref{lem:Sobolev-Gagliardo-Nirenberg} provides the crucial estimates.

\begin{proof}[Proof of Theorem~\ref{thm:Neumann-p-laplace}]
  First, let $1<p<d$. Then by
  Lemma~\ref{lem:Sobolev-Gagliardo-Nirenberg}, there is a constant $C>0$ such that
  \begin{equation}
    \label{eq:93}
    \norm{u}_{\frac{p d}{d-p}}\le C \left(\norm{|\nabla u|}_{p}+\norm{u}_{p}\right)
  \end{equation}
  for every $u\in W_{p,p}^{1}(\Sigma)$. Taking $p$th power on both
  sides of this inequality, using that for $q>1$,
  \begin{equation}
    \label{eq:86}
    (a+b)^{q}\le  2^{q-1}(a^q+b^q)\qquad\text{ for every $a$, $b\ge0$,}
  \end{equation}
  by definition of
  the operator $B^{N}$ and by~\eqref{eq:54} and~\eqref{eq:84}, we see that
  \begin{align*}
    \norm{u-\hat{u}}_{\frac{p d}{d-p}}^{p}
    &\le C\, \left(\norm{\abs{\nabla(u-\hat{u})}}_{p}^{p}+\norm{u-\hat{u}}_{p}^{p}\right)\\
    & \le C\,\tilde{\eta}^{-1} \int_{\Sigma} (a(x,\nabla u)-a(x,\nabla
      \hat{u}))\nabla (u-\hat{u})\,\dx\\
    &\hspace{2cm}+  C\,\eta_{0}^{-1}\int_{\Sigma} (v-\hat{v})(u-\hat{u})\,\dx\\
    & = C\,\max\{\tilde{\eta}^{-1},\eta_{0}^{-1}\}\,\langle u-\hat{u},B^Nu+v-(B^N\hat{u}+\hat{v})\rangle
  \end{align*}
  for every $u$, $\hat{u}\in D(B^{N})$, $v\in \beta_{2}(u)$,
  $\hat{v}\in \beta_{2}(\hat{u})$. Now, Remark~\ref{rem:2} yields that the
  operator $A^N$ given by~\eqref{eq:82} satisfies the Gagliardo-Nirenberg
  inequality~\eqref{eq:245} with parameters~\eqref{eq:235}. Thus,
  Theorem~\ref{thm:main-1} yields the first statement of this theorem
  for $1<p<d$.

  Next, let $p=d\ge 2$. Then by
  Lemma~\ref{lem:Sobolev-Gagliardo-Nirenberg}, there is a constant
  $C>0$ such that
  \begin{displaymath}
    \norm{u}_{\frac{2}{1-\theta}}\le C \left(\norm{|\nabla
        u|}_{p}^{\theta}\,\norm{u}_{2}^{1-\theta}+\norm{u}_{\tilde{q}}\right)
  \end{displaymath}
  for every
  $u\in\dot{W}_{p,2}^{1}(\Sigma)\cap L^{\tilde{q}}(\Sigma)$,
  $0\le \theta <1$, and $\tilde{q}>0$. Thus, if $p=d=2$, we choose $\tilde{q}=2$ and use that
   $\norm{u}_{2}=\norm{u}_{2}^{\theta}\,\norm{u}_{2}^{1-\theta}$ for
   every $0\le \theta<1$, and if $p=d>2$ then we choose $\tilde{q}_{\theta}$ by
  $\tfrac{1}{\tilde{q}_{\theta}}=\tfrac{\theta}{p}+\tfrac{1-\theta}{2}$
  for any given $0<\theta<1$
  and apply H\"older's inequality. Then, in both cases, we obtain 
  \begin{equation}
    \label{eq:95}
    \norm{u}_{\frac{2}{1-\theta}}\le C \left(\norm{|\nabla
        u|}_{p}^{\theta}+\norm{u}_{p}^{\theta}\right)\,\norm{u}_{2}^{1-\theta}
  \end{equation}
  for every $u\in W_{p,2}^{1}(\Sigma)$. Thus, by definition of the
  operator $B^{N}$,~\eqref{eq:54} and~\eqref{eq:84}, 
\begin{align*}
    \norm{u-\hat{u}}_{\frac{2}{1-\theta}}^{\frac{p}{\theta}}
    &\le C^{\frac{p}{\theta}}\,2^{\frac{p}{\theta}-1}\,\left(
      \norm{\abs{\nabla(u-\hat{u})}}_{p}^{p}
      +\norm{u-\hat{u}}_{p}^{p}\right) 
      \,\norm{u-\hat{u}}_{2}^{\frac{p(1-\theta)}{\theta}}\\
    & \le C^{\frac{p}{\theta}}\,2^{\frac{p}{\theta}-1}
      \left(\eta^{-1}\,\int_{\Sigma} (a(x,\nabla u)-a(x,\nabla
      \hat{u}))\nabla (u-\hat{u})\,\dx\right.\\
     & \hspace{4cm}\left.+ \tilde{\eta}_{0}^{-1} \,\int_{\Sigma}
      (v-\hat{v})(u-\hat{u})\,\dx\right)\,
       \norm{u-\hat{u}}_{2}^{\frac{p(1-\theta)}{\theta}}\\
    & \le  C^{\frac{p}{\theta}} 2^{\frac{p}{\theta}-1} \max\{\tilde{\eta}^{-1},\eta_{0}^{-1}\} 
       \langle u-\hat{u},B^Nu+v-(B^N
      \hat{u}+\hat{v})\rangle \norm{u-\hat{u}}_{2}^{\frac{p(1-\theta)}{\theta}}
  \end{align*}
  By Remark~\ref{rem:2}, $A^N$ satisfies the Gagliardo-Nirenberg
  inequality~\eqref{eq:245} with parameters~\eqref{eq:236} hence,
  Theorem~\ref{thm:main-1} yields the first statement of this
  theorem for $p=d$.

  Now, let $d<p<\infty$. Then, there is an $0<\theta_{0}<1$ such that
  $\theta_{0} (\frac{1}{p}-\frac{1}{d})+(1-\theta_{0})\frac{1}{2}=0$
  or equivalently, $\theta_{0}=\tfrac{p d}{pd+2(p-d)}$. We apply
  Lemma~\ref{lem:Sobolev-Gagliardo-Nirenberg} for $\theta_{0}$,
  $\tilde{q}$ given by
  $\tfrac{1}{\tilde{q}}=\tfrac{\theta_{0}}{p}+\tfrac{1-\theta_{0}}{2}$
  and apply H\"older's inequality. Then,
  \begin{equation}
    \label{eq:96}
    \norm{u}_{\infty}\le C \left(\norm{|\nabla
        u|}_{p}^{\theta_{0}}+\norm{u}_{p}^{\theta_{0}}\right)\,\norm{u}_{2}^{1-\theta_{0}}
  \end{equation}
  for every $u\in\dot{W}_{p,2}^{1}(\Sigma)$ and some constant
  $C>0$. Proceeding as in the first step of this proof, we see that
  by~\eqref{eq:54}, \eqref{eq:84} and Remark~\ref{rem:2}, the operator
  $A^N$ satisfies the Gagliardo-Nirenberg inequality~\eqref{eq:245} with
  parameters~\eqref{eq:98}. Therefore, by
  Theorem~\ref{thm:extrapol-L1-differences}, the first statement of
  this theorem holds for $p>d$.

  Under the assumption that merely the
  hypotheses~\eqref{eq:coerciveness}-\eqref{eq:monotonicity-of-a} are
  satisfied, one proceeds as in the previous three steps of this
  proof and applies Theorem~\ref{thm:GN-implies-reg-bis} with
  $u_{0}=0$. Thus, the second statement of this theorem holds.

  To see that also the last claim of this theorem holds, one
  combines the Gagliardo-Nirenberg inequality~\eqref{eq:92} with the
  Poincar\'e inequality
  \begin{equation}
    \label{eq:187}
    \norm{u-\overline{u}}_{p}\le C_{N}\,\norm{\abs{\nabla u}}_{p},
  \end{equation}
  which holds for all $u\in W^{1}_{p,p}(\Sigma)$ and some constant
  $C_{N}>0$ independent of $u$. Then, for $1<p<d$,
  inequality~\eqref{eq:93} reduces to
  \begin{equation}
    \label{eq:188}
    \norm{u-\overline{u}}_{\frac{p d}{d-p}}\le C \norm{\abs{\nabla u}}_{p}
  \end{equation}
  for every $u\in W^{1}_{p,p}(\Sigma)$, if $p=d\ge 2$ then
  inequality~\eqref{eq:95} reduces to
  \begin{equation}
    \label{eq:189}
    \norm{u-\overline{u}}_{\frac{2}{1-\theta}}\le C_{\theta}\,
    \norm{\abs{\nabla u}}_{p}^{\theta}\,\norm{u-\overline{u}}_{2}^{1-\theta}
  \end{equation}
  for every $u\in W^{1}_{p,2}(\Sigma)$ and $0<\theta<1$, and if
  $d<p<\infty$ then inequality~\eqref{eq:96} reduces to
  \begin{displaymath}
    \norm{u-\overline{u}}_{\infty}\le C\,\norm{\abs{\nabla
        u}}_{p}^{\theta_{0}}\,\norm{u-\overline{u}}_{2}^{1-\theta_{0}}
  \end{displaymath}
  for every $u\in W^{1}_{p,2}(\Sigma)$, where
  $\theta_{0}=\tfrac{p d}{pd+2(p-d)}$ and the constant $C$ can differ
  from line to line. Now, proceeding as in the first three steps of
  this proof and using the inequalities~\eqref{eq:187}-\eqref{eq:189}
  instead of~\eqref{eq:93}, \eqref{eq:95} and \eqref{eq:96}, and
  noting that for every $u\in L^{2}(\Sigma)$, the element
  $(\overline{u},0)\in B^{N}$, then one obtains that for all
  $1<p<\infty$, the operator $B^N$ satisfies the Gagliardo-Nirenberg
  inequality~\eqref{eq:190} for the same exponents as found in the
  first three steps of this proof. Thus
  Theorem~\ref{thm:GN-implies-reg-bis} yields the third statement of
  this theorem. 
\end{proof}

%%%%%%%%%%%%%%%%%%%%%%%%%%%%%%%%%%%%%%%%%%%%%%%%%%%%
%
%      homogeneous Robin boundary conditions
%
%%%%%%%%%%%%%%%%%%%%%%%%%%%%%%%%%%%%%%%%%%%%%%%%%%%%

\subsubsection{Homogeneous Robin boundary conditions}
\label{subsection:robin}

In this subsection, we assume that $\Sigma$ is a bounded domain with a
Lipschitz boundary. Then the mapping $u\mapsto
 u_{\vert\partial\Sigma}$ from $C^{0,1}(\overline{\Sigma})$ to
 $C^{0,1}(\partial\Omega)$ has a unique continuous and surjective extension 
  \begin{displaymath}
   \mathrm{Tr} : W^{1}_{p,p}(\Sigma)\to W^{1-1/p,p} (\partial\Sigma)
 \end{displaymath}
 called \emph{trace operator} 
 (cf.~\cite[Th\'eor\`eme~4.2,~4.6, and Section~3.8]{MR0227584}).  For
 convenience, we write $u_{\vert\partial\Omega}:=\mathrm{Tr}(u)$ for $u\in
 W^{1}_{p,p}(\Sigma)$ even if $u$ does not belong to $C(\overline{\Sigma})$
 and call $u_{\vert\partial\Omega}$ the \emph{trace} of $u$. Thus, if
 $\theta$ denotes the boundary operator given by~\eqref{eq:87} then
 $\langle \theta(u),u\rangle$ is finite for every $u\in
 W^{1}_{p,p}(\Sigma)$ hence, under the assumptions of this
 section, we can define the realisation $B^{R}$ in $L^{2}(\Sigma)$ of the
Leray-Lions operator $\mathcal{B}$ equipped with homogeneous Robin
boundary conditions~\eqref{eq:46} by
\begin{equation}
  \label{eq:88}
  \begin{split}
  B^{R}&= \Big\{ (u,v)\in
      L^{2}\times L^{2}(\Sigma)\,\Big\vert\,u\in
      W_{p,p}^{1}(\Sigma)\text{ s.t. for all }\xi\in
      W_{p,p}^{1}\cap L^{2}(\Sigma)
      \Big.\\
      &\hspace{0.5cm}\Big.\int_\Sigma a(x,\nabla u)\nabla
      \xi\dx+\int_{\partial\Sigma}b \abs{u}^{p-2}u\xi\dH + d \langle\theta(u),\xi\rangle
      =\int_{\Sigma} v \xi\dx\Big\}.
  \end{split}
\end{equation}
We call $B^R$ the \emph{Robin Leray-Lions operator} in $L^2(\Sigma)$.

Since  $C^{\infty}(\overline{\Sigma})$ is
contained in $W_{p,p}^{1}\cap L^{2}(\Sigma)$ and dense in
$L^{2}(\Sigma)$, $B^{R}$ defines a single-valued and densely defined
operator on $L^{2}(\Sigma)$. To see that $B^R$ is completely
accretive, let $T\in P_{0}$. Then by definition of $B^{R}$,
by~\eqref{eq:monotonicity-of-a} and since
$T$ is monotonically increasing and Lipschitz continuous  on
$\R$, and since $s\mapsto \abs{s}^{p-2}s$ is monotonically increasing,
we have that
\begin{align*}
    & \int_{\Sigma} T(u-\hat{u})(B^{R}u-B^{R}\hat{u}) \dx\\
    & = \int_\Sigma\left( a(x,\nabla u)-a(x,\nabla
      \hat{u})\right)\nabla (u-\hat{u})
    T'(u-\hat{u})\dx\\
    &\qquad + \int_{\partial\Sigma}\,b(x)
    (\abs{u}^{p-2}u-\abs{\hat{u}}^{p-2}\hat{u})T(u-\hat{u})\dH\\
    &\qquad + d \int_{\partial\Sigma}
    \int_{\partial\Sigma}\tfrac{\abs{u(x)-u(y)}^{p-2}(u(x)-u(y))
      -\abs{\hat{u}(x)-\hat{u}(y)}^{p-2}(\hat{u}(x)-\hat{u}(y))
    }{\abs{x-y}^{d+p-2}}\times\\
    &\hspace{4cm} \times T((u(x)-u(y))-(\hat{u}(x)-\hat{u}(y)))\,\dH(x)\dH(y)\\
    &\ge 0
\end{align*}
for every $u$, $\hat{u}\in D(B^{R})$. Thus, $B^{R}$ is completely
accretive. Since for every $\lambda>0$, the Yosida operator
 $\beta_{\lambda} : \R\to \R$ of $\beta$ is monotonically increasing and Lipschitz
 continuous and since $\Sigma$ is bounded, we may replace $T$ by
 $\beta_{\lambda}$ in the previous calculation, showing that the $m$-accretive
 graph $\beta$ on $\R$ satisfies condition~\eqref{eq:78} in
 Proposition~\ref{propo:Lipschitz-complete-accretive}.

In order to see that $B^{R}$ satisfies the range condition~\eqref{eq:range-condition} for
$X=L^{2}(\Sigma)$, we employ the following
 $p$-variant of Maz'ya's inequality 
 \begin{equation}
   \label{eq:90}
 \norm{u}_{\frac{pd}{d-1}}\le C\,\left(\norm{\abs{\nabla u}}_{p}
 + \norm{u_{\vert\partial\Sigma}}_{p}\right)
 \end{equation}
 holding for all $u\in W^{1}_{p,p}(\Sigma)$ provided $1\le p<\infty$.
 Here the constant $C>0$ depends on $p$, the volume $\abs{\Sigma}$, and
 the isoperimetric constant $C(d)$ (cf.~\cite[Cor. 3.6.3]{MR817985}
 and see also~\cite[Section~2.1]{MR3369257}). Now, let
 $V=W^{1}_{p,p}(\Sigma)\cap L^{2}(\Sigma)$ be equipped with the sum
 norm. Then by~\eqref{eq:coerciveness}, since
 for every $u\in V$, $\langle \theta(u),u\rangle\ge 0$, since $b(x)\ge b_{0}>0$ a.e. on
 $\partial\Sigma$ and by \eqref{eq:90}, we obtain that 
 \begin{align*}
   % \alpha\left(\norm{u}_{2}+\norm{\abs{\nabla u}}_{p}+\norm{u}_{p}\right)
   % & \ge
     \langle (I+B^{R})u,u\rangle%\\
   & \ge \norm{u}_{2}^{2}+\tfrac{\eta}{2}\norm{\abs{\nabla
     u}}_{p}^{p}+C_{1} \norm{u}_{p}^{p}\\
   & \ge C_{2}\left( \norm{u}_{2}^{2}+\norm{\abs{\nabla
     u}}_{p}^{p}+ \norm{u}_{p}^{p}\right)
 \end{align*}
 for every $u\in V$. Thus, the restriction of the operator
 $I+B^{R}$ on $V$ satisfies condition~\eqref{eq:91} and so, the
 operator $I+B^{R} : V\to V'$ is surjective
 by~\cite[Th\'eor\`eme~1]{MR0194733}, proving that $B^{R}$ satisfies
 the range condition~\eqref{eq:range-condition} in
 $L^{2}(\Sigma)$.

 Therefore, by Proposition~\ref{propo:Lipschitz-complete-accretive},
 the operator
 \begin{equation}
   \label{eq:89}
   A^{R}:=B^{R}+\beta_{2}+F
 \end{equation}
 is quasi $m$-completely accretive in $L^{2}(\Sigma)$ with dense
 domain and so by the Cran\-dall-Liggett theorem, $-A^R$ generates a
 strongly continuous semigroup $\{T_{t}\}_{t\ge 0}$ on $L^{2}(\Sigma)$
 of Lipschitz continuous mappings $T_{t}$ on $L^{2}(\Sigma)$, and each
 mapping $T_{t}$ admits a unique Lipschitz continuous extension on
 $L^{q}(\Sigma)$ with constant $e^{\omega t}$ for every $1\le q\le
 \infty$.

 Here, we state the
 complete description of the $L^{q}$-$L^{r}$-regularisation effect of
 the semigroup $\{T_{t}\}_{t\ge 0}\sim-A^R$.

\begin{theorem}
  \label{thm:Robin-p-laplace}
  Suppose the Carath\'eodory function
  $a : \Sigma\times\R^{d}\to \R^{d}$ satisfies growth
  conditions~\eqref{eq:growth-cond}. Further, suppose
  $b\in L^{\infty}(\partial\Sigma)$ such that $b(x)\ge b_{0}>0$
  a.e. on $\partial\Sigma$, $d\ge 0$, and $A^{R}$ is the operator
  given by~\eqref{eq:89}. Then the following statements
  hold:
  \begin{enumerate}
  \item If $a$ satisfies the strong monotonicity condition~\eqref{eq:54} and
    $a(x,0)=0$ for a.e. $x\in \Sigma$, then the semigroup
    $\{T_{t}\}_{t\ge 0}\sim-A^{R}$ on $L^{2}(\Sigma)$ satisfies the
    regularisation estimates~\eqref{eq:18} and~\eqref{eq:80} with the
    same exponents as the semigroup generated by $-A^D$.

    \item If $a$
      satisfies~\eqref{eq:coerciveness}-\eqref{eq:monotonicity-of-a}, then the semigroup
    $\{T_{t}\}_{t\ge 0}\sim-A^{R}$ on $L^{2}(\Sigma)$ satisfies the
    regularisation estimates~\eqref{eq:20} and~\eqref{eq:168} with the
    same exponents as the semigroup generated by $-A^D$.
 \end{enumerate}
\end{theorem}

\begin{proof}[Proof of Theorem~\ref{thm:Robin-p-laplace}]
  Note that, for every $q>1$, there is a constant $C_{q}>0$
  such that
  \begin{equation}
    \label{eq:97}
    (\abs{s}^{q-2}s-\abs{t}^{q-2}t)(s-t)\ge C_{q}\abs{s-t}^{q}
  \end{equation}
  for all $s$, $t\in \R$ (cf.~\cite[Appendix]{MR3262196}). Due to
  inequality~\eqref{eq:97}, we can show that the semigroup $\{T_{t}\}_{t\ge
    0}\sim-A^R$ satisfies inequality~\eqref{eq:80} provided the
  Carath\'e\-odory function $a$ satisfies~\eqref{eq:54}.

  First, let $1<p<d$. By Lemma~\ref{lem:Sobolev-Gagliardo-Nirenberg}, we have that
  inequality~\eqref{eq:93} holds. Applying Maz'ya's
  inequality~\eqref{eq:90} to estimate the term $\norm{u}_{p}$ in~\eqref{eq:93} gives
  \begin{equation}
    \label{eq:144}
    \norm{u}_{\frac{p d}{d-p}}\le C \left(\norm{|\nabla u|}_{p}+
      \norm{u_{\vert\partial\Sigma}}_{p}\right)
  \end{equation}
  for every $u\in W_{p,p}^{1}(\Sigma)$, where the constant $C$ can be
  different from the one in~\eqref{eq:93}. Taking $p$th power on both
  sides of the last inequality, applying~\eqref{eq:86} and using the
  definition of the operator $B^{R}$ combined with~\eqref{eq:54}
  and~\eqref{eq:97} shows that
  \begin{align*}
    \norm{u-\hat{u}}_{\frac{p d}{d-p}}^{p}
    &\le C\, \left(\norm{\abs{\nabla(u-\hat{u})}}_{p}^{p}
      +\norm{u_{\vert\partial\Sigma}-\hat{u}_{\vert\partial\Sigma}}_{p}^{p}\right)\\
    & \le C\,\tilde{\eta}^{-1} \int_{\Sigma} (a(x,\nabla u)-a(x,\nabla
      \hat{u}))\nabla (u-\hat{u})\,\dx\\
    &\hspace{2cm}+  C\,b_{0}^{-1}C^{-1}_{p}\int_{\partial\Sigma} 
      b(x)\,(\abs{u}^{p-2}u-\abs{\hat{u}}^{p-2}\hat{u})(u-\hat{u})\,\dH\\
    & \le  C\,\max\{\tilde{\eta}^{-1},(b_{0}C_{p})^{-1}\}\,
      \langle u-\hat{u},B^{R}u-B^{R}\hat{u}\rangle
  \end{align*}
  for every $u$, $\hat{u}\in D(B^{R})$. Thus, Remark~\ref{rem:2}
  yields the operator $A^R$ given by~\eqref{eq:89} satisfies
  the Gagliardo-Nirenberg inequality~\eqref{eq:245} with
  parameters~\eqref{eq:235}. By Theorem~\ref{thm:main-1}, the first
  statement of this theorem holds for $1<p<\infty$.

  If $p=d\ge 2$, then applying Maz'ya's inequality~\eqref{eq:90}
  to~\eqref{eq:95} yields
  \begin{displaymath}
    \norm{u}_{\frac{2}{1-\theta}}\le C \left(\norm{|\nabla
        u|}_{p}^{\theta}+\norm{u_{\vert\partial\Sigma}}_{p}^{\theta}\right)\,\norm{u}_{2}^{1-\theta}
  \end{displaymath}
  for every $u\in W_{2,p}^{1}(\Sigma)$ and $0<\theta<1$. Thus
  by~\eqref{eq:86}, the definition of $B^{R}$,~\eqref{eq:54} and by
  inequality~\eqref{eq:97} for $q=p$ shows that
  \begin{align*}
    &\norm{u-\hat{u}}_{\frac{2}{1-\theta}}^{\frac{p}{\theta}}\\
    &\quad\le C^{\frac{p}{\theta}}\,2^{\frac{p}{\theta}-1}\,\left(
      \norm{\abs{\nabla(u-\hat{u})}}_{p}^{p}
      +\norm{u-\hat{u}}_{p}^{p}\right) \,\norm{u-\hat{u}}_{2}^{\frac{p(1-\theta)}{\theta}}\\
    & \quad\le C^{\frac{p}{\theta}}\,2^{\frac{p}{\theta}-1}\left(\eta^{-1}\,\int_{\Sigma} (a(x,\nabla u)-a(x,\nabla
      \hat{u}))\nabla (u-\hat{u})\,\dx\right.\\
    & \hspace{2.5cm}\left.+ (a_{0}\,C_{p})^{-1}\int_{\partial\Sigma}
      a\,(\abs{u}^{p-2}u-\abs{\hat{u}}^{p-2}\hat{u})(u-\hat{u})\,\dH\right)
      \,\norm{u-\hat{u}}_{2}^{\frac{p(1-\theta)}{\theta}}\\
    & \quad\le  C^{\frac{p}{\theta}} 2^{\frac{p}{\theta}-1} \max\{\tilde{\eta}^{-1},(a_{0}\,C_{p})^{-1}\} 
      \langle u-\hat{u},B^Ru-B^R\hat{u}\rangle \norm{u-\hat{u}}_{2}^{\frac{p(1-\theta)}{\theta}}
  \end{align*}
  By Remark~\ref{rem:2}, the operator $A^R$ satisfies the Gagliardo-Nirenberg
  inequality~\eqref{eq:245} with
 \begin{displaymath}
   r_{\theta}=\frac{2}{1-\theta},\quad \sigma_{\theta}=\frac{p}{\theta},\quad
   \varrho_{\theta}=\frac{p(1-\theta)}{\theta},\quad \omega=L\qquad
   \text{for every $0<\theta<1$.}
\end{displaymath}
Therefore, by Theorem~\ref{thm:main-1}, the first statement of this theorem
holds for $p=d$.

Next, let $d<p<\infty$. Applying Maz'ya's
inequality~\eqref{eq:90} to~\eqref{eq:96} with $\theta_{0}=\tfrac{p
  d}{pd+2(p-d)}$ and subsequently taking $p$th power and employing
inequality~\eqref{eq:86} gives
\begin{displaymath}
    \norm{u}_{\infty}\le C \left(\norm{|\nabla
        u|}_{p}^{\theta_{0}}+\norm{u_{\vert\partial\Sigma}}_{p}^{\theta_{0}}\right)\,
    \norm{u}_{2}^{1-\theta_{0}}
  \end{displaymath}
  for every $u\in\dot{W}_{p,2}^{1}(\Sigma)$. Proceeding as above, we
  see that by~\eqref{eq:54}, \eqref{eq:97} and by Remark~\ref{rem:2}, the operator
  $A^{R}$ satisfies the Gagliardo-Nirenberg inequality~\eqref{eq:245} with
  parameters $r$, $\sigma$, $\varrho$ and $\omega$ as given in~\eqref{eq:98}.
  By Theorem~\ref{thm:extrapol-L1-differences}, the first
statement of this theorem holds for $p>d$. 

Under the assumption that merely the
hypotheses~\eqref{eq:coerciveness}-\eqref{eq:monotonicity-of-a} are satisfied,
one proceeds as in the previous threes steps of this proof and applies
Theorem~\ref{thm:GN-implies-reg-bis} with $u_{0}=0$. Thus, the second
statement of this theorem holds as well, completing the proof.
\end{proof}

%%%%%%%%%%%%%%%%%%%%%%%%%%%%%%%%%%%%%%%%%%%%%%%%%%%%%
%
%
%         2. Subsection: Nonlocal diffusion equations
%
%
%%%%%%%%%%%%%%%%%%%%%%%%%%%%%%%%%%%%%%%%%%%%%%%%%%%%%
\subsection{Parabolic problems involving nonlocal operators}
\label{sec:nonlocal}

In the following two subsections, we outline two examples currently
attracting much interest. We begin in Subsection~\ref{subsec:DtN} by
establishing the $L^{q}$-$L^{r}$-regularisation estimates for the
\emph{semigroup generated by the Dirichlet-to-Neumann operator
  associated with a Leray-Lions operator} (cf, for
instance,~\cite{MR3369257} and the references
therein). Subsection~\ref{subsection:fractional-p-laplace} is
dedicated to the $L^{q}$-$L^{r}$-regularisation estimates for the
\emph{semigroup generated by the fractional $p$-Laplace operator}
equipped with either homogeneous Dirichlet or Neumann boundary
conditions (cf, for instance,~\cite{MR2927356,MaRoTo2015}). One can
easily see in both examples that the standard construction of a
\emph{one-parameter family of Sobolev type inequalities}
fails. Recall, this is an important intermediate step in the known
literature to achieve an $L^{q}$-$L^{\infty}$-regularisation estimates
for $1\le q<\infty$ of the semigroup (cf Section~\ref{subsec:story}).

For instance, consider the example of the semigroup generated by the
Dirichlet-to-Neumann operator associated with a Leray-Lions operator
\eqref{eq:gen-div-operator} satisfying the
hypotheses~\eqref{eq:coerciveness}-~\eqref{eq:monotonicity-of-a}. The
construction of this Dirichlet-to-Neumann operator proceeds in two
steps. First, one needs to know the solvability of \emph{Dirichlet
  problem}
\begin{equation}
   \label{eq:99}
   \begin{cases}
     -\divergence (a(x,\nabla u))=0 & \text{in $\Sigma$,}\\
     \hspace{2,6cm}u=\varphi & \text{on $\partial\Sigma$}
   \end{cases}
 \end{equation}
 for every boundary function $\varphi\in
 W^{1-1/p,p}(\partial\Sigma)$.
 For given boundary-value $\varphi$, let $P\varphi:=u$ be the unique
 weak energy solution $u$ of~\eqref{eq:99}.  Then, in order to construct a
 one-parameter family of Sobolev type inequalities, one needs that
\begin{displaymath}
    P(\abs{\varphi}^{q-p}\varphi)=\abs{P\varphi}^{q-p}P\varphi\qquad\text{for
    every $q\ge p>1$.}
\end{displaymath}
However, this does not hold in general. Thus, our next example
demonstrates the strength of Theorem~\ref{thm:main-1} and
Theorem~\ref{thm:GN-implies-reg-bis}.

\subsubsection{The Dirichlet-to-Neumann operators associated with
  Leray-Lions operators}
\label{subsec:DtN}
In this subsection, we suppose that $\Sigma$ is either the half space
$\R^{d}_{+}:=\R^{d-1}\times(0,\infty)$ or a bounded domain with a Lipschitz
boundary. 

We begin by outlining the construction of the Dirichlet-to-Neumann operator
in the case $\Sigma$ is a bounded domain with a Lipschitz
  continuous boundary. The construction of the operator on the
half space $\Sigma=\R^{d}_{+}$ proceeds similarly (see also
Remark~\ref{rem:DtN-half-space} below). Under this assumption on
$\Sigma$, the trace operator
$\textrm{Tr} : W^{1}_{p,p}(\Sigma)\to W^{1-1/p,p}(\partial\Sigma)$ has
a linear bounded right inverse
\begin{displaymath}
  %\label{eq:right-inverse}
     Z : W^{1-1/p,p}(\partial\Sigma)\to W^{1,p}(\Sigma)
\end{displaymath}
(cf.~\cite[Th\'eor\`eme 5.7]{MR0227584}) and the kernel of
$\textrm{Tr}$ coincides with $\dot{W}^{1}_{p,p}(\Sigma)$. If the
Carath\'eodory function $a : \Sigma\times \R^{d}\to \R^{d}$
satisfies~\eqref{eq:coerciveness}-\eqref{eq:monotonicity-of-a}, then
by the classical theory of monotone operators
(~\cite[Th\'eor\`eme~1]{MR0194733}), we have that for every given boundary
value $\varphi\in W^{1-1/p,p}(\partial\Omega)$, the Dirichlet
problem~\eqref{eq:99} admits a unique \emph{weak} solution
$u\in W^{1}_{p,p}(\Sigma)$ in the following sense: for given boundary
value $\varphi\in W^{1-1/p,p}(\partial\Sigma)$, a function
$u\in W^{1}_{p,p}(\Sigma)$ is a \emph{weak energy solution} of Dirichlet
problem~\eqref{eq:99} on $\Sigma$ if
$u-Z\varphi\in \dot{W}^{1}_{p,p}(\Sigma)$ and
   \begin{displaymath}
          \int_{\Sigma} a(x,\nabla u)\nabla \xi\,\dx=0
   \end{displaymath}
 for all $\xi\in \dot{W}^{1}_{p,p}(\Sigma)$. Let
 $P : W^{1-1/p,p}(\partial\Sigma) \to W^{1}_{p,p}(\Sigma)$ be the mapping
 which assigns to each boundary value $\varphi\in W^{1-1/p,p}(\partial\Sigma)$ the
 unique weak energy solution $u\in W^{1}_{p,p}(\Sigma)$ of \eqref{eq:99}. Then
 $P$ is injective and continuous. Furthermore,
 for every $\varphi \in W^{1-1/p,p}(\partial\Sigma)$ and
 $\Phi\in W^{1}_{p,p}(\Sigma)$ satisfying
 $\Phi_{\vert\partial\Sigma}=\varphi$, there is a unique
 $u_{\Phi}\in \dot{W}^{1}_{p,p}(\Sigma)$ such that
  \begin{equation}
    \label{eq:100}
    P\varphi=u_{\Phi}+\Phi
  \end{equation}
 (cf.~\cite[Lemma~2.5]{MR3369257}).

 The Dirichlet-to-Neumann operator associated with the
 operator~$\mathcal{B}$ defined in~\eqref{eq:gen-div-operator} assigns
 to each Dirichlet boundary data $\varphi$ the corresponding co-normal
 derivative $a(x,\nabla P\varphi)\cdot\nu=:\Lambda\varphi$ on $\partial\Sigma$. 

 If $P\varphi$ and $a(\cdot,\nabla
 P\varphi)$ are smooth enough up to the boundary $\partial\Sigma$, Green's formula yields
 \begin{displaymath}
  \int_{\partial\Sigma} \Lambda\varphi\,\xi\,\dH 
  = \int_{\Sigma}a(x,\nabla P\varphi)\nabla \xi\,\dx
 \end{displaymath}
 for every $\xi\in C^{\infty}(\overline{\Sigma})$ and if
 $\Lambda\varphi\in L^{p\prime}(\partial\Sigma)$, then an
 approximation argument shows that
  \begin{displaymath}
   \int_{\partial\Sigma}\Lambda\varphi\,\xi\,\dH = \int_{\Sigma}a(x,\nabla P\varphi)
   \nabla Z\xi\,\dx
 \end{displaymath}
 for every $\xi\in W^{1-1/p,p}(\partial\Sigma)$. Even if $\varphi$
 and $\xi$ merely belong to $W^{1-1/p,p}(\partial\Sigma)$, the
 integral on the right-hand side of this equation exists. Thus, we can
 use this integral to define the operator
 $\Lambda$ for the more general class of functions
 $W^{1-1/p,p}(\partial\Sigma)$. By linearity of $Z$ and by using
 H\"older's inequality together with growth
 condition~\eqref{eq:growth-cond}, one easily sees that the functional
  \begin{displaymath}
   \psi\mapsto \int_{\Omega}a(x,\nabla P\varphi)\nabla Z\psi\,\dx
 \end{displaymath}
 belongs to the dual space
 $W^{-(1-1/p),p\prime}(\partial\Omega)$. This justifies to define the
 \emph{Dirichlet-to-Neumann operator associated with the quasi-linear
   operator} $\mathcal{B}$ as the operator
 $\Lambda : W^{1-1/p,p}(\partial\Sigma) \to
 W^{-(1-1/p),p\prime}(\partial\Sigma)$ defined by
 \begin{displaymath}%\label{def-general:DNp}
   \langle \Lambda\varphi, \xi\rangle=\int_{\Sigma}
   a(x,\nabla P\varphi)\nabla Z\xi\,\dx
 \end{displaymath}
 for every $\varphi$, $\xi\in W^{1-1/p,p}(\partial\Sigma)$. The
 Dirichlet-to-Neumann operator $\Lambda$ realised as an operator on
 $L^{2}(\partial\Sigma)$ is given by the restriction
 $\Lambda_{2}:=\Lambda\cap (L^{2}(\partial\Sigma))\times
 L^{2}(\partial\Sigma))$. In fact, one can show (cf.~\cite[Proposition~3.9]{MR3369257}) that
 \begin{displaymath}
   \begin{split}
  \Lambda_{2}&= \Big\{ (\varphi,\psi)\in
      L^{2}(\partial\Sigma)\times L^{2}(\partial\Sigma)\;\Big\vert\;\varphi\in
      W^{1-1/p,p}(\partial\Sigma)\; \text{ such that }
      \Big.\\
      & \hspace{0.5cm}\Big.\int_\Sigma a(x,\nabla P\varphi)\nabla
      Z\xi\dx=\int_{\Sigma} \psi\,\xi\dx \;\text{ for all }\;\xi\in
     W^{1-1/p,p}(\partial\Sigma)\cap L^{2}(\partial\Sigma)\Big\}.
  \end{split}
 \end{displaymath}

It is well-known (cf.~\cite[Proposition~3.9]{MR3369257} or
\cite{MR2294196}), that $\Lambda_{2}$ is completely accretive. To see
that $\Lambda_{2}$ satisfies the range condition~\eqref{eq:range-condition} for
$X=L^{2}(\partial\Sigma)$, we take $V=W^{1-1/p,p}(\partial\Sigma)\cap
L^{2}(\partial\Sigma)$ equipped with the sum norm. Then, by~\eqref{eq:100},
\begin{displaymath}
  \langle \psi, \varphi\rangle_{V',V}=\int_{\Sigma}
   a(x,\nabla P\varphi)\nabla P\varphi\,\dx
\end{displaymath}
for every $(\varphi,\psi)\in \Lambda_{2}$. By using Maz'ya's
inequality~\eqref{eq:90} and Poincar\'e's inequality on
 $W^{1}_{p,p}(\Sigma)$, one can deduce the following useful inequality
 \begin{equation}
   \label{aux-ineq:2-proof}
 \norm{u}_{p}\le \tilde{C} \Big(\norm{\abs{\nabla u}}_{p}
 + \norm{u_{\vert\partial\Sigma}}_{L^{2}(\partial\Sigma)}\Big)
 \end{equation}
 holding for all $u\in W^{1}_{p,p}(\Sigma)$ with trace
 $u_{\vert\partial\Omega}\in L^{2}(\partial\Omega)$ (cf.~\cite[Section
 2]{MR3369257}). Now, let $\alpha\in \R$ and $\varphi\in
 V$. Then, by using
 ~\eqref{eq:coerciveness}, the
 boundedness of the trace operator $\textrm{Tr} $ and
 inequality~\eqref{aux-ineq:2-proof}, we see that
 \begin{align*}
   \norm{\varphi_{\vert\partial\Sigma}}_{2}^{2}+\eta\norm{\abs{\nabla
     P\varphi}}_{p}^{p} &\le  
   \langle  (I+\Lambda_{2})\varphi, \varphi\rangle_{V',V}\\
   & \le \alpha\,C
     (\norm{\varphi_{\vert\partial\Sigma}}+\norm{P\varphi}_{p}+\norm{\abs{\nabla
     P\varphi}}_{p})\\
   & \le \alpha\,\tilde{C}(\norm{\varphi_{\vert\partial\Sigma}}+\norm{\abs{\nabla
     P\varphi}}_{p})
 \end{align*}
 Thus, the restriction of the operator $I+\Lambda_{2}$ on
 $V$ satisfies condition~\eqref{eq:91} hence
 $I+\Lambda_{2} : V\to V'$ is surjective
 by~\cite[Th\'eor\`eme~1]{MR0194733}, proving that $\Lambda_{2}$
 satisfies the range condition~\eqref{eq:range-condition} in
 $X=L^{2}(\partial\Sigma)$.

 By hypothesis on the $m$-accretive graph $\beta$ on $\R$, the domain
 $D(\beta_{2})$ of the associated accretive operator $\beta_{2}$ in $L^{2}(\partial\Sigma)$ contains the
 set $\{v_{\vert\partial\Sigma}\,\vert\, v\in C^{\infty}(\overline{\Sigma})\}$.
 Thus, the domain $D(\beta_{2})$ is dense in $L^{2}(\partial\Sigma)$
 (cf.~\cite[Lemma~2.1 2]{MR3369257}). For every $\lambda>0$, the
 Yosida operator $\beta_{\lambda}$ of $\beta$ is Lipschitz continuous,
  $\beta_{\lambda}(0)=0$, and the Yosida operator $\beta_{2,\lambda}$
 of the operator $\beta_{2}$ is given by
 $(\beta_{2,\lambda}\varphi)(x)=\beta_{\lambda}(\varphi(x))$ for
 a.e. $x\in \partial\Sigma$ and every $\varphi\in L^{2}(\partial\Sigma)$. Therefore,
 $\beta_{2,\lambda}(\varphi)\in W^{1-1/p,p}(\partial\Sigma)\cap
 L^{2}(\partial\Sigma)$ for every $\varphi\in W^{1-1/p,p}(\partial\Sigma) \cap
 L^{2}(\partial\Sigma)$. Moreover, by~\eqref{eq:100}, there is a unique
 $u_{\Phi}\in \dot{W}^{1}_{p,p}(\Sigma)$ such that
 $P(\beta_{\lambda}(\varphi))=u_{\Phi}+\beta_{\lambda}(P\varphi)$ for
 $\Phi=\beta_{\lambda}(P\varphi)$. Combining this with the definition of
 $\Lambda_{2}$, \eqref{eq:coerciveness}, and the fact that
 $\beta'_{\lambda}\ge0$, we see that
\begin{displaymath}
  [\psi,\beta_{\lambda}(\varphi)]_{2}=\int_{\Sigma}a(x,\nabla
  P\varphi)\nabla \beta_{\lambda}(P\varphi)\,\dx\ge \eta \int_{\Sigma}\abs{\nabla
    P\varphi}^{p}\beta'_{\lambda}(P\varphi)\,\dx\ge 0
\end{displaymath}
 for every $(\varphi,\psi)\in \Lambda_{2}$. Therefore, by
 Proposition~\ref{propo:Lipschitz-complete-accretive}, the operator
 \begin{equation}
   \label{eq:101}
   A^{\Lambda}:=\Lambda_{2}+\beta_{2}+F
 \end{equation}
 is quasi $m$-completely accretive in $L^{2}(\partial\Sigma)$ with dense
 domain.

 By the Crandall-Liggett theorem~\cite{MR0287357}, $-A^\Lambda$
 generates a strongly continuous semigroup $\{T_{t}\}_{t\ge 0}$ on
 $L^{2}(\partial\Sigma)$ of Lipschitz continuous mappings $T_{t}$,
 which admits a unique Lipschitz continuous extension on
 $L^{q}(\partial\Sigma)$ with constant $e^{\omega t}$ for all $1\le
 q\le  \infty$.

 \begin{remark}
   \label{rem:DtN-half-space}
   In the case $\Sigma$ is the half space $\R^{d}_{+}$, the
   construction is of $\Lambda_{2}$ is exactly the same. But one needs
   to replace the space $W^{1}_{p,p}(\Sigma)$ by the space
   $D^{1,p}(\R^{d}_{+})$ which is the completion of the
   space of all $u\in C^{\infty}_{c}(\overline{\R^{d}_{+}})$ with
   respect to $\norm{\abs{\nabla u}}_{p}$ and the space
   $W^{1-1/p,p}(\R^{d-1})$ needs to be replaced by the completion of
   the space of all $\varphi\in C^{\infty}_{c}(\R^{d-1})$ with respect
   to $|\varphi|_{p}$. We leave the details to the interested reader.
 \end{remark}

 Here is the complete description of the
 $L^{q}$-$L^{\infty}$-regularisation effect of the semigroup
 $\{T_{t}\}_{t\ge 0}\sim-A^\Lambda$.

\begin{theorem}
  \label{thm:DtN-p-laplace}
  Suppose the Carath\'eodory function
  $a : \Sigma\times\R^{d}\to \R^{d}$ satisfies growth
  conditions~\eqref{eq:growth-cond} and $A^\Lambda$ be the operator
  given by~\eqref{eq:101}. Then the following statements are true.
  \begin{enumerate}
  \item Suppose $\Sigma$ is a bounded domain with a Lipschitz
    boundary, $a$ satisfies~\eqref{eq:54} with
    $a(x,0)=0$ for a.e. $x\in \Sigma$, and the
    monotone graph $\beta$ satisfies~\eqref{eq:84}. Then
    \begin{enumerate}[(i)]
    \item\label{thm:DtN-p-laplace-claim1} for $1<p<d$, the semigroup
      $\{T_{t}\}_{t\ge 0}\sim-A^\Lambda$ on $L^{2}(\partial\Sigma)$
      satisfies estimate~\eqref{eq:80} with exponents
      \begin{displaymath}
        \hspace{1cm}\begin{array}[c]{c}
          \qquad\alpha_{s}=\tfrac{\alpha^{\ast}}{1-\gamma^{\ast}\left(1-\frac{s(d-p)}{(d-1)
              m_{0}} \right)},\;
          \beta_{s}=\tfrac{\frac{\beta^{\ast}}{2}+\gamma^{\ast} 
          \frac{s(d-p)}{(d-1) m_{0}}}{1-\gamma^{\ast}\left(1-\frac{s(d-p)}{(d-1)
              m_{0}}\right)},\;
        \gamma_{s}=\tfrac{\gamma^{\ast}\,\frac{s(d-p)}{(d-1)m_{0}}}
          {1-\gamma^{\ast}\left(1-\frac{s(d-p)}{(d-1) m_{0}}
          \right)},\\
        \qquad \alpha^{\ast}=\tfrac{d-p}{(p-1)m_{0}+(d-p)(p-2)},\qquad
          \beta^{\ast}=\tfrac{(\frac{2}{p}-1)d+p-\frac{2}{p}}{(p-1) m_{0}+(d-p)(p-2)}+1,\\
        \gamma^{\ast}=\tfrac{(p-1) m_{0}}{(p-1) m_{0}+(d-p)(p-2)}.
        \end{array}
      \end{displaymath}
      for $m_{0}\ge p$ satisfying
      \begin{math}
        (\frac{d-1}{d-p}-1)m_{0}+p-2>0,
      \end{math}
      and for every $1\le s\le \frac{(d-1)m_{0}}{d-p}$ satisfying
      $s>\frac{(2-p)(d-1)}{p-1}$. Moreover, if $\frac{2d}{d+1}<p<d$
      then one can take $m_{0}=p$ and if $\frac{2d-1}{d}<p<d$,
      then estimate~\eqref{eq:80} holds with the same exponents for every
      $1\le s\le \frac{(d-1)p}{d-p}$.

    \item\label{thm:DtN-p-laplace-claim2} for $p=d\ge 2$, the semigroup
      $\{T_{t}\}_{t\ge 0}\sim-A^\Lambda$ on $L^{2}(\partial\Sigma)$
      satisfies estimate~\eqref{eq:80} with exponents
      \begin{displaymath}
        \hspace{2cm}\begin{array}[c]{c}
        \alpha_{s}=\tfrac{\alpha_{\theta}^{\ast}}{1-\gamma_{\theta}^{\ast}(1-s (1-\theta))},\quad
        \beta_{s}=\tfrac{\frac{\beta^{\ast}_{\theta}}{2}+\gamma_{\theta}^{\ast} 
          s (1-\theta)}{1-\gamma_{\theta}^{\ast}(1-s(1-\theta))},\quad
        \gamma_{s}=\tfrac{\gamma_{\theta}^{\ast}s(1-\theta)}{1-\gamma_{\theta}^{\ast}(1-s(1-\theta))},\\
        \alpha^{\ast}_{\theta}=\tfrac{1}{\frac{1}{1-\theta}-2},\;
        \beta^{\ast}_{\theta}=\tfrac{\frac{2}{p^2}\frac{1}{1-\theta}-1}{\frac{1}{1-\theta}-2}+1,\;
        \gamma_{\theta}^{\ast}=\tfrac{\frac{1}{1-\theta}-p}{\frac{1}{1-\theta}-2},
        \end{array}
      \end{displaymath}
      for every $1-\frac{1}{p}<\theta <1$ and
      $1\le s \le \frac{1}{1-\theta}$.

    \item\label{thm:DtN-p-laplace-claim3} for $d<p<\infty$, the semigroup
      $\{T_{t}\}_{t\ge 0}\sim-A^\Lambda$ on $L^{2}(\partial\Sigma)$
      satisfies estimate~\eqref{eq:80} with exponents
      \begin{displaymath}
        \alpha_{s}=\tfrac{1}{p-2+s},\quad
        \beta_{s}=\tfrac{\frac{2+p}{2}+s}{p-2+s},\quad
        \gamma_{s}=\tfrac{s}{p-2+s}
      \end{displaymath}
      for every $1\le s \le 2$.
    \end{enumerate}

  \item Suppose $\Sigma$ is a bounded domain with a Lipschitz
    boundary, $a$ satisfies~\eqref{eq:coerciveness}-\eqref{eq:monotonicity-of-a}
    and $\beta$ satisfies~\eqref{eq:85}. Then the following holds:
    \begin{enumerate}[(i)]
      \item\label{thm:DtN-p-laplace-claim4} The semigroup
        $\{T_{t}\}_{t\ge 0}\sim-A^\Lambda$ satisfies estimate~\eqref{eq:168} with
        $u_{0}=0$ for the same exponents as given in the
        statements~\eqref{thm:DtN-p-laplace-claim1}-\eqref{thm:DtN-p-laplace-claim3}.
  
      \item\label{thm:DtN-p-laplace-claim5} The semigroup $\{T_{t}\}_{t\ge 0}\sim-\Lambda_{2}$
        satisfies
        \begin{displaymath}
          \norm{T_{t}\varphi_{\vert\partial\Sigma}-\overline{\varphi_{\vert\partial\Sigma}}}_{\infty}
          \lesssim \;e^{\omega \beta_{q} t}\,
          t^{-\delta_{q}}\,\norm{\varphi_{\vert\partial\Sigma}-\overline{\varphi_{\vert\partial\Sigma}}}_{q}^{\gamma_{q}}
        \end{displaymath}
        for every $t>0$ and $\varphi\in L^{q}(\partial\Sigma)$, where
        $\overline{\varphi_{\vert\partial\Sigma}}:=\tfrac{1}{\mathcal{H}(\partial\Sigma)}\int_{\partial\Sigma}\varphi\,\dH$
        and the exponents $\alpha_{s}$, $\beta_{s}$ and $\gamma_{s}$ are
        the same as given in the
        statements~\eqref{thm:DtN-p-laplace-claim1}-\eqref{thm:DtN-p-laplace-claim3}.
    \end{enumerate}

   \item Suppose $\Sigma$ is the half space $\R^{d}_{+}$.
    \begin{enumerate}[(i)]
    
    \item[(vi)] If $a$ satisfies~\eqref{eq:54} with
    $a(x,0)=0$ for a.e. $x\in \Sigma$ and without any further
    assumptions on $\beta$, then for $1<p\le d$, the
    semigroup $\{T_{t}\}_{t\ge 0}\sim-A^\Lambda$ satisfies
    estimate~\eqref{eq:80} with the same exponents as given in the
    statements~\eqref{thm:DtN-p-laplace-claim1}-\eqref{thm:DtN-p-laplace-claim3}.

    \item[(vii)] If $a$
    satisfies~\eqref{eq:coerciveness}-\eqref{eq:monotonicity-of-a}
    and without any further assumptions on $\beta$, then for $1<p\le d$ 
    the semigroup $\{T_{t}\}_{t\ge 0}\sim-A^\Lambda$ satisfies estimate~\eqref{eq:168} with
        $u_{0}=0$ for the same exponents as given in~\eqref{thm:DtN-p-laplace-claim5}.
    \end{enumerate}
  \end{enumerate}
\end{theorem}

Since we are not aware about the existence of Gagliardo-Nirenberg
inequalities involving the trace operator, we need to construct in
each case $1<p<d$, $p=d$ and $p>d$ the
sufficient inequality from the known Sobolev-trace inequality
(\cite[Chapter 2, Sect.~4]{MR0227584}). 

\begin{proof}[Proof of Theorem~\ref{thm:DtN-p-laplace}]
  First, let $1<p<d$. Then, by the Sobolev-trace
  inequality~\cite[Th\'eor\`eme~4.2]{MR0227584} and by Maz'ya's
  inequality~\eqref{eq:90}, 
  \begin{equation}
    \label{eq:94}
    \norm{u_{\vert\partial\Sigma}}_{\frac{p (d-1)}{d-p}}\le C \left(\norm{|\nabla u|}_{p}+
      \norm{u_{\vert\partial\Sigma}}_{p}\right)
  \end{equation}
  for every $u\in W_{p,p}^{1}(\Sigma)$ and some constant $C>0$
  independent of $u$. Taking $p$th power on both
  sides of this inequality, applying~\eqref{eq:86} and using the
  definition of the operator $\Lambda_{2}$ combined with~\eqref{eq:54}
  and~\eqref{eq:84} gives
  \begin{align*}
    \norm{\varphi_{\vert\partial\Sigma}-\hat{\varphi}_{\vert\partial\Sigma}}_{\frac{p (d-1)}{d-p}}^{p}
    &\le \tilde{C}\, \left(\norm{\abs{\nabla(P\varphi-P\hat{\varphi})}}_{p}^{p}
      +\norm{\varphi_{\vert\partial\Sigma}-\hat{\varphi}_{\vert\partial\Sigma}}_{p}^{p}\right)\\
    & \le \tilde{C}\,\tilde{\eta}^{-1} \int_{\Sigma} (a(x,\nabla P\varphi)-a(x,\nabla
      P\hat{\varphi}))\nabla (P\varphi-P\hat{\varphi})\,\dx\\
    &\hspace{2cm}+  \tilde{C}\,\eta_{0}^{-1}\int_{\partial\Sigma} 
      (v-\hat{v})(\varphi-\hat{\varphi})\,\dH\\
    & \le  \tilde{C}\,\max\{\tilde{\eta}^{-1},\eta_{0}^{-1}\}\,
      \langle \varphi-\hat{\varphi},(\Lambda_{2}\varphi+v-(\Lambda_{2}\hat{\varphi}+\hat{v})\rangle
  \end{align*}
  for every $\varphi$, $\hat{\varphi}\in D(\Lambda_{2})$ and $v\in
  \beta_{2}(\varphi)$, $\hat{v}\in \beta_{2}(\hat{\varphi})$. Thus, by Remark~\ref{rem:2}, the
  operator $A^R$ satisfies the Gagliardo-Nirenberg
  inequality~\eqref{eq:245} with
 \begin{displaymath}
   r=\frac{p (d-1)}{d-p},\quad \sigma=p,\quad
  \varrho=0,\;\text{ and }\; \omega=L
\end{displaymath}
hence the first statement of this theorem holds by Theorem~\ref{thm:main-1}.

If $p=d\ge 2$, then by the Sobolev inequality for trace
  operators~\cite[Th\'eor\`eme~4.6]{MR0227584} and by Maz'ya's
  inequality~\eqref{eq:90}, for every $0\le \theta<1$, there is a constant $C>0$ such that
  \begin{equation}
    \label{eq:102}
    \norm{u_{\vert\partial\Sigma}}_{\frac{1}{1-\theta}}\le C \left(\norm{|\nabla u|}_{p}+
      \norm{u_{\vert\partial\Sigma}}_{p}\right)
  \end{equation}
  for every $u\in W_{p,p}^{1}(\Sigma)$. Proceeding as in the first step
  of this proof yields that the operator $A^\Lambda$ satisfies the Gagliardo-Nirenberg
  inequality~\eqref{eq:245} with
  \begin{displaymath}
   r=\frac{1}{1-\theta},\quad \sigma=p,\quad
  \varrho=0,\;\text{ and }\; \omega=L\quad\text{for every $0<\theta<1$.}
\end{displaymath}
Therefore by Theorem~\ref{thm:main-1}, the second statement of this theorem
holds.

Next, let $d<p<\infty$. Then, by the classical Sobolev-Morrey
inequality~\cite[Th\'eor\`eme~3.8]{MR0227584} and by Maz'ya's
inequality~\eqref{eq:90}, there is a constant such that
\begin{equation}
  \label{eq:103}
    \norm{u_{\vert\partial\Sigma}}_{\infty}\le C \left(\norm{|\nabla
        u|}_{p}+\norm{u_{\vert\partial\Sigma}}_{p}\right)
  \end{equation}
  for every $u\in\dot{W}_{p,p}^{1}(\Sigma)$. By proceeding as in the
  first step of this proof, we
  see that the operator $A^\Lambda$ satisfies the Gagliardo-Nirenberg
  inequality~\eqref{eq:245} with
  \begin{displaymath}
   r=\infty,\quad \sigma=p,\quad
  \varrho=0,\;\text{ and }\; \omega=L.
\end{displaymath}
Therefore, by Theorem~\ref{thm:main-1}, the third statement of this
theorem holds. 

Under the assumption that the Carath\'eodory function $a$
satisfies~\eqref{eq:coerciveness}-\eqref{eq:monotonicity-of-a} and the
accretive graph $\beta$ satisfies~\eqref{eq:85}, one
proceeds as in the first three steps and applies Theorem~\ref{thm:GN-implies-reg-bis}
with $u_{0}=0$. Thus, statement~\eqref{thm:DtN-p-laplace-claim4} of this theorem holds.

To see that the last statement holds, one applies
Poincar\'e's inequality
\begin{displaymath}
  \norm{u_{\vert\partial\Sigma}-\overline{u_{\vert\partial\Sigma}}}_{p}\le C\,\norm{\abs{\nabla u}}_{p}
\end{displaymath}
holding for all $u\in W^{1}_{p,p}(\Sigma)$ with mean value
$\overline{u_{\vert\partial\Sigma}}:=\tfrac{1}{\mathcal{H}(\partial\Sigma)}\int_{\partial\Sigma}u\,\dH$, for some
constant $C>0$ (cf.~\cite[Lemma~2.5]{MR3369257}) to the Sobolev-trace
inequalities~\eqref{eq:94}, \eqref{eq:102} and \eqref{eq:103}. Then
for $1<p<d$, inequality~\eqref{eq:94} reduces to
\begin{displaymath} 
  \norm{u_{\vert\partial\Sigma}-\overline{u_{\vert\partial\Sigma}}}_{\frac{p (d-1)}{d-p}}\le C \norm{\abs{\nabla u}}_{p}
\end{displaymath}
for every $u\in W^{1}_{p,p}(\Sigma)$, if $p=d\ge 2$ then
inequality~\eqref{eq:102} reduces to
\begin{displaymath}
  \norm{u_{\vert\partial\Sigma}-\overline{u_{\vert\partial\Sigma}}}_{\frac{1}{1-\theta}}\le C_{\theta}\,
  \norm{\abs{\nabla u}}_{p}
\end{displaymath}
for every $u\in W^{1}_{p,p}(\Sigma)$ and $0\le \theta<1$, and if $d<p<\infty$,
inequality~\eqref{eq:103} reduces to
\begin{displaymath}
  \norm{u_{\vert\partial\Sigma}-\overline{u_{\vert\partial\Sigma}}}_{\infty}\le C\,\norm{\abs{\nabla
      u}}_{p}
\end{displaymath}
for every $u\in W^{1}_{p,p}(\Sigma)$, where the constant $C$ can
differ from line to line. Now, by proceeding as in the first three
steps of this proof, where one employs these three new Sobolev-trace
inequalities involving the average value
$\overline{u_{\vert\partial\Sigma}}$ and by noting that for every
$\varphi\in L^{2}(\partial\Sigma)$, the element
$(\overline{\varphi_{\vert\partial\Sigma}},0)\in \Lambda_{2}$, one
sees that for all $1<p<\infty$, the operator $\Lambda_{2}$ satisfies
the Gagliardo-Nirenberg inequality~\eqref{eq:242} with the same
exponents as in the
statements~\eqref{thm:DtN-p-laplace-claim1}-\eqref{thm:DtN-p-laplace-claim3}. This
proves that statement~\eqref{thm:DtN-p-laplace-claim5} of 
this theorem holds.

If $\Sigma$ is the half space $\R^{d}_{+}$, then one replaces the
above used Sobolev-trace inequalities with the ones given
in~\cite[Theorem~15.17 \& Exercise 15.19]{MR2527916} and proceeds as
the first two steps of this proof. This completes the proof of this
theorem.
\end{proof}

%%%%%%%%%%%%%%%%%%%%%%%%%%%%%%%%%%%%%%%%%%%%%%%%%%%%%
%%%%%%%%%%%%%%%%%%%%%%%%%%%%%%%%%%%%%%%%%%%%%%%%%%%%%
%%%%%%%%%%%%%%%%%%%%%%%%%%%%%%%%%%%%%%%%%%%%%%%%%%%%%
%%%%%%%%%%%%%%%%%%%%%%%%%%%%%%%%%%%%%%%%%%%%%%%%%%%%%
%%%%%%%%%%%%%%%%%%%%%%%%%%%%%%%%%%%%%%%%%%%%%%%%%%%%%
%%%%%%%%%%%%%%%%%%%%%%%%%%%%%%%%%%%%%%%%%%%%%%%%%%%%%
%%%%%%%%%%%%%%%%%%%%%%%%%%%%%%%%%%%%%%%%%%%%%%%%%%%%%
%%%%%%%%%%%%%%%%%%%%%%%%%%%%%%%%%%%%%%%%%%%%%%%%%%%%%
%%%%%%%%%%%%%%%%%%%%%%%%%%%%%%%%%%%%%%%%%%%%%%%%%%%%%
%%%%%%%%%%%%%%%%%%%%%%%%%%%%%%%%%%%%%%%%%%%%%%%%%%%%%
%%%%%%%%%%%%%%%%%%%%%%%%%%%%%%%%%%%%%%%%%%%%%%%%%%%%%
%%%%%%%%%%%%%%%%%%%%%%%%%%%%%%%%%%%%%%%%%%%%%%%%%%%%%
%%%%%%%%%%%%%%%%%%%%%%%%%%%%%%%%%%%%%%%%%%%%%%%%%%%%%
%%%%%%%%%%%%%%%%%%%%%%%%%%%%%%%%%%%%%%%%%%%%%%%%%%%%%
%%%%%%%%%%%%%%%%%%%%%%%%%%%%%%%%%%%%%%%%%%%%%%%%%%%%%
%%%%%%%%%%%%%%%%%%%%%%%%%%%%%%%%%%%%%%%%%%%%%%%%%%%%%
%%%%%%%%%%%%%%%%%%%%%%%%%%%%%%%%%%%%%%%%%%%%%%%%%%%%%
%%%%%%%%%%%%%%%%%%%%%%%%%%%%%%%%%%%%%%%%%%%%%%%%%%%%%
%%%%%%%%%%%%%%%%%%%%%%%%%%%%%%%%%%%%%%%%%%%%%%%%%%%%%
%%%%%%%%%%%%%%%%%%%%%%%%%%%%%%%%%%%%%%%%%%%%%%%%%%%%%
%%%%%%%%%%%%%%%%%%%%%%%%%%%%%%%%%%%%%%%%%%%%%%%%%%%%%
%%%%%%%%%%%%%%%%%%%%%%%%%%%%%%%%%%%%%%%%%%%%%%%%%%%%%
%%%%%%%%%%%%%%%%%%%%%%%%%%%%%%%%%%%%%%%%%%%%%%%%%%%%%
%%%%%%%%%%%%%%%%%%%%%%%%%%%%%%%%%%%%%%%%%%%%%%%%%%%%%
%%%%%%%%%%%%%%%%%%%%%%%%%%%%%%%%%%%%%%%%%%%%%%%%%%%%%
%%%%%%%%%%%%%%%%%%%%%%%%%%%%%%%%%%%%%%%%%%%%%%%%%%%%%
%%%%%%%%%%%%%%%%%%%%%%%%%%%%%%%%%%%%%%%%%%%%%%%%%%%%%
%%%%%%%%%%%%%%%%%%%%%%%%%%%%%%%%%%%%%%%%%%%%%%%%%%%%%
%%%%%%%%%%%%%%%%%%%%%%%%%%%%%%%%%%%%%%%%%%%%%%%%%%%%%
%%%%%%%%%%%%%%%%%%%%%%%%%%%%%%%%%%%%%%%%%%%%%%%%%%%%%

\subsubsection{Parabolic problems involving the fractional $p$-Laplace operator}
\label{subsection:fractional-p-laplace}

Let $\Sigma$ be an open subset of $\R^{d}$, $1<p<\infty$ and
$0<s<1$. Then, for given initial value $u(0) \in L^{q}(\Sigma)$, we
intend to establish the $L^{q}$-$L^{r}$-regularisation estimates of
solutions $u(t)=u(x,t)$ for $t>0$ of the nonlocal diffusion equation
\begin{equation}
  \label{ip:Neumann-fractional-p-laplace}
  \partial_{t}u-(-\Delta_{p})^{s}u+\beta(u)+f(x,u)\ni 0 \qquad\text{on $\Sigma\times
      (0,\infty)$,}
\end{equation}
equipped with either \emph{homogeneous Dirichlet boundary conditions}
\begin{equation}
  \label{ip:Dirichlet-fractional-p-laplace}
   u=0  \qquad \text{on $\R^{d}\setminus\Sigma\;\times
      (0,\infty)$,}
\end{equation}
or with \emph{homogeneous Neumann boundary conditions}, that is, equation
\eqref{ip:Neumann-fractional-p-laplace} without any further
conditions. We refer the interested reader to~\cite{MR2437810} for a
thorough discussion on Neumann boundary conditions in nonlocal
diffusion problems.

Concerning homogeneous Dirichlet boundary
conditions~\eqref{ip:Dirichlet-fractional-p-laplace}, we impose no
further regularity conditions on the boundary $\partial\Sigma$ of $\Sigma$. Note that,
if $\Sigma=\R^{d}$, then the homogeneous Dirichlet boundary conditions
become \emph{vanishing conditions near infinity}~\eqref{eq:43}. In the
case of homogeneous Neumann boundary conditions, we assume that
$\Sigma$ is a bounded domain with a Lipschitz boundary.

The operator $(-\Delta_{p})^{s}$ in
equation~\eqref{ip:Neumann-fractional-p-laplace} denotes the
\emph{fractional $p$-Laplace operator} defined by
\begin{equation}
  \label{eq:105}
  (-\Delta_{p})^{s}u(x) = \textrm{P.V. }\int_{\Sigma}
  \frac{\abs{u(y)-u(x)}^{p-2}(u(y)-u(x))}{\abs{y-x}^{d+sp}}\,\dy 
\end{equation}
for a.e. $x\in \Sigma$ and any sufficiently regular function $u :
\Sigma\to \R$. The notation $\textrm{P.V. }$ in~\eqref{eq:105} indicates that the
integral at the right hand side is to be understood in the
\emph{Cauchy principal value sense}, that is, for given $x\in \Sigma$,
the value $(-\Delta_{p})^{s}u(x)$ denotes the limit
\begin{displaymath}
  \lim_{\varepsilon\to0+}\int_{\Sigma\setminus B_{\varepsilon}(x)}
  \frac{\abs{u(y)-u(x)}^{p-2}(u(y)-u(x))}{\abs{y-x}^{d+sp}}\,\dy 
\end{displaymath}
provided the limit exists. For every $u\in
W^{s}_{p,q}(\Sigma)$, $(q\ge 1)$, such that
\begin{equation}
  \label{eq:244}
  (x,y)\mapsto\frac{\abs{u(y)-u(x)}^{p-1}}{\abs{y-x}^{d+sp}}\quad\text{
  belongs to }\quad
  L^{1}(\Sigma\times \Sigma),
\end{equation}
Fubini's theorem yields that $(-\Delta_{p})^{s}u\in
L^{1}(\Sigma)$ and
\begin{displaymath}
  (-\Delta_{p})^{s}u(x)=\int_{\Sigma}
  \frac{\abs{u(y)-u(x)}^{p-2}(u(y)-u(x))}{\abs{y-x}^{d+sp}}\,\dy 
\end{displaymath}
for a.e. $x\in \Sigma$. In other words, the integral on the right hand
side of~\eqref{eq:105} holds without
the $\textrm{P.V.}$-symbol. Employing Fubini's theorem again,
subsequently interchanging $x$ and $y$, and using the symmetry of the kernel
$\abs{y-x}^{-(d+sp)}$, one sees that
\begin{equation}
  \label{eq:191}
  \begin{split}
    &\int_{\Sigma} - (-\Delta_{p})^{s}u\,\xi\,\dx\\
    % &=-\tfrac{1}{2}\int_{\Sigma} \int_{\Sigma}
    % \tfrac{\abs{u(y)-u(x)}^{p-2}(u(y)-u(x))}{\abs{y-x}^{N+sp}}\,\dy
    % \varphi(x)\,\dx\\
    % &\qquad -\tfrac{1}{2}\int_{\Sigma} \int_{\Sigma}
    % \tfrac{\abs{u(y)-u(x)}^{p-2}(u(y)-u(x))}{\abs{y-x}^{N+sp}}\,\dy
    % \varphi(x)\,\dx\\
    % &=-\tfrac{1}{2}\int_{\Sigma} \int_{\Sigma}
    % \tfrac{\abs{u(y)-u(x)}^{p-2}(u(y)-u(x))}{\abs{y-x}^{N+sp}}\,\dy
    % \varphi(x)\,\dx\\
    % &\qquad -\tfrac{1}{2}\int_{\Sigma} \int_{\Sigma}
    % \tfrac{\abs{u(y)-u(x)}^{p-2}(u(y)-u(x))}{\abs{y-x}^{N+sp}}\,
    % \varphi(x)\,\dx\,\dy\\
    % &=-\tfrac{1}{2}\int_{\Sigma} \int_{\Sigma}
    % \tfrac{\abs{u(y)-u(x)}^{p-2}(u(y)-u(x))}{\abs{y-x}^{N+sp}}\,\dy
    % \varphi(x)\,\dx\\
    % &\qquad +\tfrac{1}{2}\int_{\Sigma} \int_{\Sigma}
    % \tfrac{\abs{u(y)-u(x)}^{p-2}(u(y)-u(x))}{\abs{y-x}^{N+sp}}\,
    % \varphi(y)\,\dy\,\dx\\
    % &
   &\qquad =\frac{1}{2}\int_{\Sigma} \int_{\Sigma}
    \frac{\abs{u(y)-u(x)}^{p-2}(u(y)-u(x))}{\abs{y-x}^{d+sp}}\,
    (\xi(y)-\xi(x))\,\dy\,\dx
  \end{split}
\end{equation}
for every $\xi\in C^{\infty}_{c}(\overline{\Sigma})$. Since the double integral
in the right hand side of~\eqref{eq:191} exists merely for $u$,
$\xi\in W^{s}_{p,q}(\Sigma)$, it makes sense to employ this
double integral in order to define the fractional $p$-Laplace operator $(-\Delta_{p})^{s}$ in a
\emph{weak} sense. In particular, the previous calculation shows that our next
definition of the fractional $p$-Laplace operator is consistent with
the smooth case, that is, when $u\in W^{s}_{p,q}$ satisfies~\eqref{eq:244}.

For any open subset $\Sigma$ of $\R^{d}$, we define the realisation of the
\emph{Dirichlet-fractional $p$-Laplace operator}
$(-\Delta_{p}^{D})^{s}$ in $L^{2}(\Sigma)$ by
\begin{displaymath}
  \begin{split}
  &(-\Delta_{p}^{D})^{s}= \Bigg\{ (u,v)\in
      L^{2}\times L^{2}(\Sigma)\,\Big\vert\,u\in
      \dot{W}_{p,2}^{s}(\Sigma)\text{ s.t. for all }\xi\in
      \dot{W}_{p,2}^{s}(\Sigma)
      \Bigg.\\
      &\Bigg. 
-\frac{1}{2}\int_{\Sigma} \int_{\Sigma}
  \frac{\abs{u(y)-u(x)}^{p-2}(u(y)-u(x))}{\abs{y-x}^{d+sp}}\,
    (\xi(y)-\xi(x))\,\dy\,\dx=\int_{\Sigma} v\,\xi\,\dx
\Bigg\}.
  \end{split}
\end{displaymath}

If $\Sigma$ is a bounded Lipschitz domain, then we define the
realisation of the
\emph{Neumann-fractional $p$-Laplace operator}
$(-\Delta_{p}^{N})^{s}$ in $L^{2}(\Sigma)$ by
\begin{displaymath}
  \begin{split}
  &(-\Delta_{p}^{N})^{s}= \Bigg\{ (u,v)\in
      L^{2}\times L^{2}(\Sigma)\,\Big\vert\,u\in
      W_{p,2}^{s}(\Sigma)\text{ s.t. for all }\xi\in
      W_{p,2}^{s}(\Sigma)
      \Bigg.\\
      &\Bigg. 
-\frac{1}{2}\int_{\Sigma} \int_{\Sigma}
  \frac{\abs{u(y)-u(x)}^{p-2}(u(y)-u(x))}{\abs{y-x}^{d+sp}}\,
    (\xi(y)-\xi(x))\,\dy\,\dx=\int_{\Sigma} v\,\xi\,\dx
\Bigg\}.
  \end{split}
\end{displaymath}

Both operators $-(-\Delta_{p}^{D})^{s}$ and $-(-\Delta_{p}^{N})^{s}$ are completely
accretive in $L^{2}(\Sigma)$ (cf.~\cite{MaRoTo2015}). We show this
only at the operator $-(-\Delta_{p}^{D})^{s}$ since the proof for the
operator  $-(-\Delta_{p}^{N})^{s}$ proceeds similarly. Let $(u,v)$,
$(\hat{u},\hat{v})\in (-\Delta_{p}^{D})^{s}$ and $T\in P_{0}$. By
the Lipschitz property of $T$ and since $T(0)=0$, one has that $T\circ u\in
\dot{W}_{p,2}^{s}(\Sigma)$. Furthermore, since $T$ and $s\mapsto
\abs{s}^{p-2}s$ are monotonically increasing on $\R$,
\begin{align*}
  &\int_{\Sigma}T(u-\hat{u})((-v)-(-\hat{v}))\,\dx\\
  & =\frac{1}{2}\int_{\Sigma}
    \int_{\Sigma}\tfrac{\abs{u(x)-u(y)}^{p-2}(u(x)-u(y))
      -\abs{\hat{u}(x)-\hat{u}(y)}^{p-2}(\hat{u}(x)-\hat{u}(y))
    }{\abs{x-y}^{d+sp}}\times\\
    &\hspace{4cm} \times
      T((u(x)-u(y))-(\hat{u}(x)-\hat{u}(y)))\,\dx\dy\ge 0,
\end{align*}
proving that $-(-\Delta_{p}^{N})^{s}$ is completely accretive by
Proposition~\ref{prop:completely-accretive}. 

Further, for any given $m$-accretive
graph $\beta$ on $\R$ satisfying $0\in \beta(0)$, one has for every
$\lambda>0$ that the Yosida operator
$\beta_{\lambda} : \R\to \R$ of $\beta$ is
monotonically increasing, Lipschitz continuous and satisfies
$\beta_{\lambda}(0)=0$. Therefore, taking $T=\beta_{\lambda}$ in the
previous calculation shows that the monotone graph $\beta_{2}$ in
$L^{2}(\Sigma)$ satisfies condition~\eqref{eq:78} in
Proposition~\ref{propo:Lipschitz-complete-accretive}. 

To see that $-(-\Delta_{p}^{D})^{s}$ satisfies the range
condition~\eqref{eq:range-condition} for $X=L^{2}(\Sigma)$, we take
$V=\dot{W}_{p,2}^{s}(\Sigma)$. If for every $\alpha\in \R_{+}$, $E_{\alpha}$
 denotes the set of all $u\in
 V$ satisfying 
 \begin{displaymath}
   \frac{\langle (I-(-\Delta_{p}^{D})^{s})u,u\rangle_{V',V}}{\norm{u}_{V}}\le
   \alpha\, ,
 \end{displaymath}
 then by 
 \begin{displaymath}
   \langle (I-(-\Delta_{p}^{D})^{s})u,u\rangle_{V',V}
   = \norm{u}_{2}^{2}+\abs{u}_{s,p}^{p},
   \end{displaymath}
 one has
 \begin{displaymath}
   \norm{u}_{2}^{2}+\abs{u}_{s,p}^{p}\le \alpha (\norm{u}_{2}+\abs{u}_{s,p})
 \end{displaymath}
 for every $u\in E_{\alpha}$, implying that $E_{\alpha}$ is bounded in $V$. Thus, the operator
 $(I-(-\Delta_{p}^{D})^{s}) :V\to V'$ is surjective
 by~\cite[Th\'eor\`eme~1]{MR0194733}, proving that $-(-\Delta_{p}^{D})^{s}$ satisfies
 the range condition~\eqref{eq:range-condition} in
 $X=L^{2}(\Sigma)$. Analogously, one shows that $-(-\Delta_{p}^{N})^{s}$
 is $m$-completely accretive in $L^{2}(\Sigma)$. Moreover, both
 operators have a dense domain in $L^{2}(\Sigma)$ as proven in~\cite{MaRoTo2015}.

Therefore, by Proposition~\ref{propo:Lipschitz-complete-accretive}, the operators
 \begin{equation}
  \label{eq:104}
   A^{D,s}:=-(-\Delta_{p}^{D})^{s}+\beta_{2}+F
 \end{equation}
 and
\begin{equation}
  \label{eq:106}
   A^{N,s}:=-(-\Delta_{p}^{N})^{s}+\beta_{2}+F
 \end{equation}
 are quasi $m$-completely accretive in $L^{2}(\Sigma)$ with dense
 domain in $L^{2}(\Sigma)$. By the Crandall-Liggett
 theorem, $-A^{D,s}$ and $-A^{N,s}$
 generate respectively a strongly continuous semigroup $\{T_{t}\}_{t\ge 0}$ on
 $L^{2}(\Sigma)$ of Lipschitz continuous mappings $T_{t}$, which admits
 a unique Lipschitz continuous extension on $L^{q}(\Sigma)$ for all
 $1\le q< \infty$ and on $\overline{L^{2}\cap
   L^{\infty}(\Sigma)}^{\mbox{}_{L^{\infty}}}$ if $q=\infty$,
 respectively with constant $e^{\omega t}$.

 We begin by giving the complete description of the
 $L^{q}$-$L^{r}$-regularisation estimates of the semigroup
 $\{T_{t}\}_{t\ge 0}\sim -A^{D,s}$ on $L^{2}(\Sigma)$.

\begin{theorem}
  \label{thm:D-fractional-p-laplace}
  Let $1<p<\infty$, $0<s<1$, $\Sigma$ be an open subset of
  $\R^{d}$ and $A^{D,s}$ given by~\eqref{eq:104}. Then the following statements are true.
  \begin{enumerate}
  \item\label{thm:D-fractional-p-laplace-claim1} 
    For $1<s p<d$, the semigroup $\{T_{t}\}_{t\ge 0}\sim-A^{D,s}$
    on $L^{2}(\Sigma)$ satisfies~\eqref{eq:80} for
      \begin{displaymath}
        \alpha_{q}=\tfrac{\alpha^{\ast}}{1-\gamma^{\ast}(1-\frac{q(d-sp)}{d
              m_{0}})},\;
        \beta_{q}=\tfrac{\frac{\beta^{\ast}}{2}+\gamma^{\ast} 
          \frac{q(d-sp)}{d m_{0}}}{1-\gamma^{\ast}(1-\frac{q(d-sp)}{d m_{0}})},\;
        \gamma_{q}=\tfrac{\gamma^{\ast}\frac{q(d-sp)}{d m_{0}}}{1-\gamma^{\ast}(1-\frac{q (d-sp)}{d m_{0}})}
      \end{displaymath}
      for every $m_{0}\ge p$ satisfying
      \begin{math}
        sp m_{0}+(p-2)(d-sp)>0
      \end{math}
      and $1\le q\le \frac{d m_{0}}{d-sp}$ satisfying
      $q>\frac{sp}{d}+(p-2)-d\frac{2+p}{sp}$, where
      \begin{displaymath}
        \begin{array}[c]{c}
        \alpha^{\ast}=\tfrac{1}{(\frac{d}{d-sp}-1)m_{0}+p-2},\qquad
        \beta^{\ast}=\tfrac{\frac{2}{p}\frac{d}{d-sp}-1}{(\frac{d}{d-sp}-1)m_{0}+p-2}+1,\\
        \gamma^{\ast}=\tfrac{(\frac{d}{d-sp}-1)m_{0}}{(\frac{d}{d-sp}-1)m_{0}+p-2}.
        \end{array}
      \end{displaymath}
      Moreover, if $\frac{2d}{d+2s}<p<d$ then one can take $m_{0}=p$
      and if $\frac{2d}{d+s}<p<d$, then~\eqref{eq:80} holds for every
      $1\le q\le \frac{dp}{d-sp}$.

    \item\label{thm:D-fractional-p-laplace-claim2} For $sp=d$,
      suppose that either $\Sigma$ is unbounded and $\beta$
      satisfies~\eqref{eq:84} or $\Sigma$ is bounded and no further
      assumptions on $\beta$. Then the semigroup
      $\{T_{t}\}_{t\ge 0}\sim-A^{D,s}$ satisfies~\eqref{eq:80} with
      exponents
      \begin{displaymath}
        \alpha_{q}=\tfrac{\alpha^{\ast}_{\theta}}{1-\gamma_{\theta}^{\ast}(1-\frac{q(1-\theta)}{p})},\quad
        \beta_{q}=\tfrac{\frac{\beta^{\ast}_{\theta}}{2}+\gamma_{\theta}^{\ast} 
          \frac{q(1-\theta)}{p}}{1-\gamma_{\theta}^{\ast}(1-\frac{q(1-\theta)}{p})},\quad
        \gamma_{q}=\tfrac{\gamma_{\theta}^{\ast}\,
          \frac{q(1-\theta)}{p}}{1-\gamma_{\theta}^{\ast}(1-\frac{q(1-\theta)}{p})}
      \end{displaymath}
      for every $1\le q \le \frac{p}{1-\theta}$ and $\max\{1-\frac{p}{2}, 2-p,0\}<\theta <1$, where
      \begin{displaymath}
        \begin{array}[c]{c}
          \alpha^{\ast}_{\theta}=\tfrac{1}{\frac{p}{1-\theta}-2},\quad 
          \beta_{\theta}^{\ast}=\tfrac{\frac{2}{1-\theta}-1}{\frac{p}{1-\theta}-2}+1,
          \quad\gamma_{\theta}^{\ast}=\tfrac{(\frac{1}{1-\theta}-1)p}{\frac{p}{1-\theta}-2}.
        \end{array}
      \end{displaymath}

    \item\label{thm:D-fractional-p-laplace-claim3} For $d<s p<\infty$,
      suppose that either $\Sigma$ is unbounded and $\beta$
      satisfies~\eqref{eq:84} or $\Sigma$ is bounded and no further
      assumptions on $\beta$. Then the semigroup
      $\{T_{t}\}_{t\ge 0}\sim-A^{D,s}$ satisfies~\eqref{eq:80} with
      exponents
      \begin{displaymath}
        \alpha_{q}=\tfrac{1}{p-2(1-\frac{q}{2})},\quad
        \beta_{q}=\tfrac{1+\frac{p}{2}+q}{p-2(1-\frac{q}{2})},\quad
        \gamma_{q}=\tfrac{q}{p-2(1-\frac{q}{2})}
      \end{displaymath}
      for every $1\le q \le 2$.
    \end{enumerate}
\end{theorem}

Since we could not find an appropriate reference to Gagliardo-Nirenberg
inequalities available for Sobolev or Besov
spaces of fractional order in the spirit of the classical ones
(cf. Lemma~\ref{lem:Sobolev-Gagliardo-Nirenberg}), we construct in
each case $1<sp<d$, $sp=d$ and $sp>d$ our
sufficient inequalities from the known Sobolev inequality for
fractional Sobolev spaces partially combined with Poincar\'e inequalities.

\begin{proof}[Proof of Theorem~\ref{thm:D-fractional-p-laplace}]
  First, let $1<s p<d$. Then, by ~\cite[Theorem 14.29]{MR2527916}, there is a
  constant $C>0$ such that
  \begin{displaymath}
    \norm{u}_{\frac{p d}{d-sp}}\le C \abs{u}_{s,p}
  \end{displaymath}
  for every $u\in \dot{W}_{p,q}^{s}(\Sigma)$ and $q\ge 1$. Taking
  $p$th power on both sides of this inequality and
  applying~\eqref{eq:97} to $a=u(x)-u(y)$ and
  $b=\hat{u}(x)-\hat{u}(y)$ yields
  \begin{align*}
    &\norm{u-\hat{u}}_{\frac{p d}{d-sp}}^{p}\\
    &\; \le C\, \int_{\Sigma} \int_{\Sigma}
  \tfrac{\abs{(u(x)-\hat{u}(x))-(u(y)-\hat{u}(y)}^{p}}{\abs{x-y}^{d+sp}}\,\dx\,\dy\\
    & \; \le \tilde{C}\, \frac{1}{2}\int_{\Sigma}\int_{\Sigma} 
      \tfrac{\abs{u(x)-u(y)}^{p-2}(u(x)-u(y))-\abs{\hat{u}(x)-\hat{u}(y)}^{p-2}(\hat{u}(x)-\hat{u}(y))}{\abs{x-y}^{d+sp}}\\
    &\hspace{5cm} \times\,((u(x)-u(y))-(\hat{u}(x)-\hat{u}(y)))\,\dx\,\dy\\
    & \; = \tilde{C}\,
      \langle u-\hat{u},(-(-\Delta_{p}^{D})^{s}u)-(-(-\Delta_{p}^{D})^{s}\hat{u})\rangle
  \end{align*}
  for every $u$, $\hat{u}\in D((-\Delta_{p}^{D})^{s})$. Thus, by Remark~\ref{rem:2}, the
  operator $A^{D,s}$ given by~\eqref{eq:104} satisfies the Gagliardo-Nirenberg
  inequality~\eqref{eq:245} with parameters
 \begin{equation}
   \label{eq:197}
   r=\frac{p d}{d-sp},\quad \sigma=p,\quad\text{ and }\quad
  \varrho=0,
\end{equation}
hence the first statement of this theorem holds by Theorem~\ref{thm:main-1}.

Next, let $sp=d\ge 2$. By~\cite[Theorem~6.9]{MR2944369}, there is a
constant $C=C(d,p,s)>0$ such that 
\begin{equation}
  \label{eq:192}
  \norm{u}_{\frac{p}{1-\theta}}\le C \,\left(\abs{u}_{s,p}+\norm{u}_{p}\right)
\end{equation}
for every $u\in C_{c}^{\infty}(\Sigma)$ and by a standard
approximation argument we see that this inequality holds for all
$u\in \dot{W}^{s}_{p,p}(\Sigma)$ and every $0\le \theta<1$. If
$\Sigma$ is unbounded and $\beta$ satisfies~\eqref{eq:84}, then taking
$p$th power on both sides of inequality~\eqref{eq:192} and
applying~\eqref{eq:97} to $a=u(x)-u(y)$ and $b=\hat{u}(x)-\hat{u}(y)$
together with~\eqref{eq:84} yields
\begin{align*}
    &\norm{u-\hat{u}}_{\frac{p}{1-\theta}}^{p}\\
    &\; \le \,C\, \left(\int_{\Sigma} \int_{\Sigma}
  \tfrac{\abs{(u(x)-\hat{u}(x))-(u(y)-\hat{u}(y)}^{p}}{\abs{x-y}^{d+sp}}\,\dx\,\dy
      + \int_{\Sigma}\abs{u-\hat{u}}^{p}\,\dx \right)\\
   & \quad \le \,C\, \left(2\,\frac{1}{2}\int_{\Sigma}\int_{\Sigma} 
      \tfrac{\abs{u(x)-u(y)}^{p-2}(u(x)-u(y))-
     \abs{\hat{u}(x)-\hat{u}(y)}^{p-2}(\hat{u}(x)-\hat{u}(y))}{\abs{x-y}^{d+sp}}\right.\\
    &\hspace{5cm}
      \times\,((u(x)-u(y))-(\hat{u}(x)-\hat{u}(y)))\,\dx\,\dy\\
   & \hspace{8cm} \left.+ \eta_{0}^{-1}\int_{\Sigma}(v-\hat{v})(u-\hat{u})\,\dx\right)\\
    & \quad \le C\,\max\{2,\eta_{0}^{-1}\}\, 
      \langle u-\hat{u},(-(-\Delta_{p}^{D})^{s}u+v)-(-(-\Delta_{p}^{D})^{s}\hat{u}+\hat{v})\rangle
  \end{align*}
  for every $u$, $\hat{u}\in D((-\Delta_{p}^{D})^{s})$,
  $v\in \beta_{2}(u)$, $\hat{v}\in \beta_{2}(\hat{u})$ and some
  constant $C>0$ which might be different from line to line. Therefore, by
  Remark~\ref{rem:2}, the operator $A^{D,s}$ given by~\eqref{eq:104}
  satisfies the Gagliardo-Nirenberg inequality~\eqref{eq:245} with parameters
  \begin{equation}
    \label{eq:193}
    r=\frac{p}{1-\theta},\quad \sigma=p,\quad\text{ and }\quad
    \varrho=0,\qquad\text{for every $0\le \theta<1$.}
  \end{equation}
  If $\Sigma$ is bounded, then by~\cite[Theorem~5]{MR3148135},
  the first eigenvalue of $-(-\Delta_{p}^{D})^{s}$ is positive hence
  the following Poincar\'e inequality
  \begin{equation}
    \label{eq:107}
    \norm{u}_{p}\le C\,\abs{u}_{s,p}
  \end{equation}
  holds for every $u\in \dot{W}_{p,p}^{s}(\Sigma)$, $1<p<\infty$ and
  $0<s<1$. Using~\eqref{eq:107} to estimate the term
  $\norm{u}_{p}$ in~\eqref{eq:192} yields
\begin{displaymath}
  \norm{u}_{\frac{p}{1-\theta}}\le C\,\abs{u}_{s,p}
\end{displaymath}
 for every $u\in \dot{W}^{s}_{p,2}(\Sigma)$. Now, proceeding as previously,
 we see that the operator $A^{D,s}$ satisfies the Gagliardo-Nirenberg
 inequality~\eqref{eq:245} with exponents~\eqref{eq:193}.
 Therefore by Theorem~\ref{thm:main-1}, the second statement of this theorem
 holds.

Next, let $d<s p$. Then, by~\cite[Theorem~8.2]{MR2944369}, there is a
constant $C=C(d,p,s)>0$ such that 
\begin{equation}
  \label{eq:194}
  \norm{u}_{\infty}\le C \,\left(\abs{u}_{s,p}^{p}+\norm{u}_{p}^{p}\right)^{1/p}
\end{equation}
for every $u\in \dot{W}_{p,p}^{s}(\Sigma)$. If $\Sigma$ is unbounded but $\beta$
satisfies~\eqref{eq:84}, then proceeding as in the case $s p=d$, we
obtain that
\begin{align*}
    &\norm{u-\hat{u}}_{\infty}^{p}\\
    & \quad \le C\,\max\{2,\eta_{0}^{-1}\}\, 
      \langle u-\hat{u},(-(-\Delta_{p}^{D})^{s}u+v)-(-(-\Delta_{p}^{D})^{s}\hat{u}+\hat{v})\rangle
  \end{align*}
  for every $u$, $\hat{u}\in D((-\Delta_{p}^{D})^{s})$ and
  $v\in \beta_{2}(u)$, $\hat{v}\in \beta_{2}(\hat{u})$. Thus,
  Remark~\ref{rem:2} implies that the operator $A^{D,s}$ satisfies
 the  Gagliardo-Nirenberg inequality~\eqref{eq:245} with exponents
  \begin{equation}
    \label{eq:195}
    r=\infty,\quad \sigma=p,\quad\text{ and }\quad
  \varrho=0.
  \end{equation}
  If $\Sigma$ is bounded, then we apply Poincar\'e
  inequality~\eqref{eq:107} to estimate the term $\norm{u}_{p}$
  in~\eqref{eq:194} and obtain
  \begin{displaymath}
    \norm{u}_{\infty}\le C \abs{u}_{s,p}
  \end{displaymath}
  for every $u\in\dot{W}_{p,2}^{s}(\Sigma)$. Proceeding as above, we
  see that the operator $A^{D,s}$ satisfies the Gagliardo-Nirenberg
  inequality~\eqref{eq:245} with exponents~\eqref{eq:195}. Therefore
  by Theorem~\ref{thm:main-1}, the third statement of this theorem
  holds. This completes the proof.
\end{proof}

Next, we state the $L^{q}$-$L^{r}$-regularisation estimates of the
semigroup $\{T_{t}\}_{t\ge 0}\sim -A^{N,s}$ on $L^{2}(\Sigma)$.

\begin{theorem}
  \label{thm:Neumann-fractional-p-laplace}
  Let $\Sigma$ be a bounded domain of
  $\R^{d}$ with a Lipschitz continuous boundary. Then the following statements are true.
  \begin{enumerate}
  \item Suppose the monotone graph $\beta$ satisfies~\eqref{eq:84} and
    $A^{N,s}$ is given by~\eqref{eq:106}. Then, for every $1<p<\infty$
    and $0<s<1$, the semigroup $\{T_{t}\}_{t\ge 0}\sim-A^{N,s}$
    satisfies~\eqref{eq:80} with the same exponents as satisfied by
    the semigroup $\{T_{t}\}_{t\ge 0}\sim-A^{D,s}$ in the
    statements~\eqref{thm:D-fractional-p-laplace-claim1}-\eqref{thm:D-fractional-p-laplace-claim3}
    of Theorem~\ref{thm:D-fractional-p-laplace}.
 
    \item Suppose the monotone graph $\beta$ satisfies~\eqref{eq:85} and $A^{N,s}$
      is given by~\eqref{eq:106}. Then, for every $1<p<\infty$
    and $0<s<1$, the semigroup $\{T_{t}\}_{t\ge 0}\sim-A^{N,s}$
    satisfies estimate~\eqref{eq:168} for $u_{0}=0$ with the same exponents as satisfied by
    the semigroup $\{T_{t}\}_{t\ge 0}\sim-A^{D,s}$ in the
    statements~\eqref{thm:D-fractional-p-laplace-claim1}-\eqref{thm:D-fractional-p-laplace-claim3}
    of Theorem~\ref{thm:D-fractional-p-laplace}.  
    
    \item For every $1<p<\infty$
    and $0<s<1$, the semigroup $\{T_{t}\}_{t\ge 0}\sim (-\Delta_{p}^{N})^{s}$
    satisfies~\eqref{eq:196} with the same exponents
   as given in the
   statements~\eqref{thm:D-fractional-p-laplace-claim1}-\eqref{thm:D-fractional-p-laplace-claim3}
   of Theorem~\ref{thm:D-fractional-p-laplace}.
    \end{enumerate}
\end{theorem}

In particular, concerning Neumann boundary condition, we need to
construct for each case $1<sp<d$, $sp=d$ and $sp>d$ our
sufficient inequalities from the known Sobolev inequalities for
fractional Sobolev spaces partially combined with a Poincar\'e inequalities.

\begin{proof}[Proof of Theorem~\ref{thm:Neumann-fractional-p-laplace}]
  We only derive the Sobolev inequalities in each case $1<sp<d$,
  $sp=d$ and $sp>d$ needed to deduce the $L^{q}$-$L^{r}$- regularity
  estimates for the semigroups, since then one proceeds as in the proof of
  Theorem~\ref{thm:D-fractional-p-laplace}. 

  First, let $1<s p<d$. Then, by~\cite[Theorem~6.9]{MR2944369}, there is a
  constant $C>0$ such that 
  \begin{equation}
    \label{eq:198}
    \norm{u}_{\frac{p d}{d-sp}}\le C (\abs{u}_{s,p}+\norm{u}_{p})
  \end{equation}
  for every $u\in W_{p,p}^{s}(\Sigma)$. Now, one takes $p$th power on
  both sides of this inequality and applies assumption~\eqref{eq:84}
  and inequality~\eqref{eq:97} for $s=u(x)-u(y)$ and
  $t=\hat{u}(x)-\hat{u}(y)$. This, together with Remark~\ref{rem:2}
  yields the operator $A^{N,s}$ satisfies Gagliardo-Nirenberg
  inequality~\eqref{eq:245} with exponents~\eqref{eq:197}. If $\beta$
  satisfies~\eqref{eq:85}, then one does not need
  inequality~\eqref{eq:97} to show that $A^{N,s}$ satisfies
  Gagliardo-Nirenberg inequality~\eqref{eq:242} with
  exponents~\eqref{eq:197}.

  In the case $sp=d\ge 2$, one uses the  Sobolev
  inequality~\cite[Theorem~6.10]{MR2944369}
  \begin{equation}
    \label{eq:199}
    \norm{u}_{\frac{p}{1-\theta}}\le C (\abs{u}_{s,p}+\norm{u}_{p}),
  \end{equation}
  which holds for all $u\in W_{p,p}^{s}(\Sigma)$ and any $\theta\in
  [0,1)$, where the constant $C$ is independent of $\theta$ and $u$.
  In the case and $d<s p$, one employs the Sobolev
  inequality~\cite[Theorem~8.2]{MR2944369})
  \begin{equation}
    \label{eq:200}
    \norm{u}_{\infty}\le C (\abs{u}_{s,p}^{p}+\norm{u}_{p}^{p})^{1/p}
  \end{equation}
  holding for all $u\in W_{p,p}^{s}(\Sigma)$. In both cases, one
  proceeds analogously as in the proof of
  Theorem~\ref{thm:D-fractional-p-laplace}. If $\beta$
  satisfies~\eqref{eq:84}, then one sees that operator $A^{N,s}$
  satisfies the Gagliardo-Nirenberg inequality~\eqref{eq:245} with 
  exponents~\eqref{eq:193} and \eqref{eq:195}, respectively. If
  $\beta$ satisfies~\eqref{eq:85}, then by proceeding as in the
  previous case but without using inequality~\eqref{eq:97} one sees
  that $A^{N,s}$ satisfies the Gagliardo-Nirenberg
  inequality~\eqref{eq:242} with  exponents~\eqref{eq:193}
  and~\eqref{eq:195}, respectively. This shows that the first and the
  second statement of this theorem holds.

  In order to see that the last statement holds, one employs the
  following Poincar\'e inequality (see, for instance,~\cite{MR3095149})
  \begin{displaymath}
    \norm{u-\overline{u}}_{p}\le C\, \abs{u}_{s,p}
  \end{displaymath}
  holding for all $u\in W^{1}_{p,p}(\Sigma)$ and some constant $C>0$,
  to estimate the term $\norm{u}_{p}$ in~\eqref{eq:198},
  \eqref{eq:199} and \eqref{eq:200}. Then, one proceeds as in the
  previous steps of this proof and obtains that the operator
  $(-\Delta_{p}^{N})^{s}$ satisfies the Gagliardo-Nirenberg
  inequality~\eqref{eq:245} with exponents~\eqref{eq:197} if $1< s p
  < d$, \eqref{eq:193} if $ps=d\ge 2$ and \eqref{eq:195} if $sp >d$. Therefore
  by Theorem~\ref{thm:main-1}, the third statement of this theorem
  holds.
\end{proof}

%%%%%%%%%%%%%%%%%%%%%%%%%%%%%%%%%%%%%%%%%%%%%%%%%%%%%
%%%%%%%%%%%%%%%%%%%%%%%%%%%%%%%%%%%%%%%%%%%%%%%%%%%%%
%%%%%%%%%%%%%%%%%%%%%%%%%%%%%%%%%%%%%%%%%%%%%%%%%%%%%
%%%%%%%%%%%%%%%%%%%%%%%%%%%%%%%%%%%%%%%%%%%%%%%%%%%%%
%%%%%%%%%%%%%%%%%%%%%%%%%%%%%%%%%%%%%%%%%%%%%%%%%%%%%
%%%%%%%%%%%%%%%%%%%%%%%%%%%%%%%%%%%%%%%%%%%%%%%%%%%%%
%%%%%%%%%%%%%%%%%%%%%%%%%%%%%%%%%%%%%%%%%%%%%%%%%%%%%
%%%%%%%%%%%%%%%%%%%%%%%%%%%%%%%%%%%%%%%%%%%%%%%%%%%%%
%%%%%%%%%%%%%%%%%%%%%%%%%%%%%%%%%%%%%%%%%%%%%%%%%%%%%
%%%%%%%%%%%%%%%%%%%%%%%%%%%%%%%%%%%%%%%%%%%%%%%%%%%%%
%%%%%%%%%%%%%%%%%%%%%%%%%%%%%%%%%%%%%%%%%%%%%%%%%%%%%
%%%%%%%%%%%%%%%%%%%%%%%%%%%%%%%%%%%%%%%%%%%%%%%%%%%%%
%%%%%%%%%%%%%%%%%%%%%%%%%%%%%%%%%%%%%%%%%%%%%%%%%%%%%
%%%%%%%%%%%%%%%%%%%%%%%%%%%%%%%%%%%%%%%%%%%%%%%%%%%%%
%%%%%%%%%%%%%%%%%%%%%%%%%%%%%%%%%%%%%%%%%%%%%%%%%%%%%

%%%%%%%%%%%%%%%%%%%%%%%%%%%%%%%%%%%%%%%%%%%%%%%%%%%%%
%
%
%         3. Subsection: Nonlinear diffusion equations in $L^1$
%
%
%%%%%%%%%%%%%%%%%%%%%%%%%%%%%%%%%%%%%%%%%%%%%%%%%%%%%

\subsection{Nonlinear diffusion equations in $L^{1}$}
\label{sec:doubly-nonl-diff}

This subsection is concerned with the application of our theory
developed in Section~\ref{gn} and Section~\ref{extra} to semigroups
generated by quasi accretive operator in $L^{1}$. The here established
$L^{1}$-$L^{\infty}$-regularisation estimates are used in the
subsequent Section~\ref{subsec:Mild-is-strong} to show that mild
solutions are strong. 

For the sake of readability, we outline the
example here only on the $p$-Laplace operator but we emphasise that it
is clear that this example and the corresponding results hold very
well for the general Leray-Lions operator considered in
Section~\ref{sec:p-laplace}.

% \subsubsection{Parabolic problems involving doubly nonlinear
%   operators in divergence type}

Let $1<p<\infty$, $m>0$, and $\phi \in C(\R)\cap C^{1}(\R\setminus\{0\})$ be a
non-decreasing function satisfying
\begin{equation}
  \label{eq:10}
  \phi(0)=0\qquad\text{and}\qquad \phi'(s)\ge
  C\,\abs{s}^{m-1}\quad\text{ for every $s\neq 0$,}
\end{equation}
for some $C>0$ independent of $s\in \R$. 

\begin{remark}
  Typical examples of non-decreasing functions
  $\phi\in C(\R)\cap C^{1}(\R\setminus\{0\})$ satisfying the two
  conditions in~\eqref{eq:10} are $\phi_{1}(s)=\abs{s}^{m-1}s$,
  ($s\in \R$), for any $m>0$ or for $m=1$,
  $\phi_{2}(s):=a\,s^{+}-b\, s^{-}$, ($s\in \R$), for $a$, $b>0$. Note
  that for the function $\phi_{1}$, the operator $\Delta_{p}\phi_{1}$
  coincides with the celebrated \emph{doubly nonlinear operator}
  $\Delta_{p}(\cdot^{m})$ (cf.~\cite{MR2268115}).
\end{remark}

Then, for given initial value $u_{0} \in L^{1}(\Sigma)$, we
investigate in this subsection the regularisation effect of
\emph{mild} solutions $u(t)=u(x,t)$ for $t>0$ of the \emph{parabolic
  initial value problem}
\begin{equation}
  \label{ip:doubly-nonlinear}
  \begin{cases}
    \partial_{t}u-\textrm{div} (\abs{\nabla \phi(u)}^{p-2}\nabla \phi(u))+f(x,u)\ni 0 
    &\qquad\text{on $\Sigma\times
      (0,\infty)$,}\\
    \hspace{5,5cm}u(\cdot,0)=u_{0} & \qquad\text{on $\Sigma$,}
  \end{cases}
\end{equation}
respectively equipped with one of the following types of boundary
conditions:
\begin{align}
 \label{eq:127} u=0\quad \text{on $\partial\Sigma\times
    (0,\infty)$,}& \;
    \text{if $\Sigma\subseteq \R^{d}$,}\\  
 \label{eq:128} \abs{\nabla \phi(u)}^{p-2}\nabla \phi(u)\cdot\nu
     =0\quad \text{on $\partial\Sigma\times (0,\infty)$,}& \;
     \text{if $\mu(\Sigma)<\infty$,}\\
 \label{eq:129} \abs{\nabla \phi(u)}^{p-2}\nabla \phi(u)\cdot\nu+a \abs{\phi(u)}^{p-1}= 0
  \quad \text{on $\partial\Sigma\times
     (0,\infty)$,}& \;
     \text{if $\mu(\Sigma)<\infty$.}
\end{align}

Concerning homogeneous Dirichlet boundary condition~\eqref{eq:127}, we
make no further assumptions on the boundary of $\Sigma$. However,
regarding homogeneous Neumann or Robin boundary conditions
\eqref{eq:128} and \eqref{eq:129}, respectively, we need to ensure the
validity of the Gagliardo-Niren\-berg inequalities~\eqref{eq:31}
hence, we assume that $\Sigma$ is a bounded domain with a Lipschitz
boundary. In addition, concerning homogeneous Neumann boundary
conditions \eqref{eq:128}, we state the $L^{q}$-$L^{r}$ regularisation
effect merely for initial values $u_{0}\in L^{q}_{0}(\Sigma)$ for
$1\le q\le \infty$.

%  we assume that there is $\eta_{1}>0$ such that
% \begin{equation}
%   \label{eq:142}
%   \beta(u)u\ge \eta_{1}\abs{u}^{m(p-1)+1}\qquad\text{for all $u\in \R$}
% \end{equation}

% In order to keep this section readable, we state the
% $L^{q}$-$L^{r}$ regularisation effect of solutions of value
% problem~\eqref{ip:doubly-nonlinear} equipped with homogeneous Neumann
% conditions~\eqref{eq:128} merely for initial values
% $u_{0}\in L^{q}_{0}(\Sigma)$ for $1\le q\le \infty$ and mention the
% result for general $u\in L^{q}(\Sigma)$ in Remark~\ref{} below.

\begin{remark}
  \label{rem:5}
  If $\Sigma$ is unbounded, then the Dirichlet boundary
  conditions~\eqref{eq:127} become \emph{vanishing conditions at
    infinity}~\eqref{eq:43}. It is well-known 
  (cf.~\cite[Theorem~9.12]{MR2286292} for the case $p=2$ and the
  references therein) that the $L^{q}$-$L^{r}$-regular\-isa\-tion effect of
  mild (respectively, strong) solutions of
  problem~\eqref{ip:doubly-nonlinear} for $\Sigma=\R^{d}$ has been
  deduced from the uniform estimates obtained in the case
  $\Sigma_{n}=B(0,n)$ the open ball centred at $x=0$ and radius
  $r=n\ge 1$, or for $p\neq 2$ under the assumption that the solutions
  have enough regularity (cf., for
  instance,~\cite{MR2268115,MR2726546}). In this monograph, we show that
  we do not need to proceed in this way. We treat the case of Dirichlet boundary
  condition~\eqref{eq:127} for general open subsets $\Sigma$ of
  $\R^{d}$ \emph{at once}. This simplifies essentially the known approaches in the
  literature and has the great advantage that we know
  the infinitesimal generator of the nonlinear semigroup.
\end{remark}

Now, let $A$ denote either the negative Dirichlet $p$-Laplace operator
$-\Delta_{p}^{D}$ on $L^{2}(\Sigma)$, the negative Neumann
$p$-Laplace operator $-\Delta_{p}^{N}$ on $L^{2}(\Sigma)$ or on $L^{2}_{m}(\Sigma)$, or the
negative Robin $p$-Laplace operator $-\Delta_{p}^{R}$ realised on
$L^{2}(\Sigma)$. Then, $A$ is a single-valued, $m$-completely
accretive operators in $L^{2}(\Sigma)$ satisfying $A0=0$
(cf. Section~\ref{sec:homog-dirichl-bound}
and~\ref{subsection:robin}). Furthermore, let $A_{1\cap\infty}$ be the
trace of $A$ on $L^{1}\cap L^{\infty}(\Sigma,\mu)$. Then, by
Proposition~\ref{propo:extrapol-prop}, the closure
$\overline{A_{1\cap\infty}}$ in $L^{1}(\Sigma)$ of the trace
$A_{1\cap\infty}$ is $m$-completely accretive in $L^{1}(\Sigma)$ with
dense domain. In the specific case $A=-\Delta_{p}^{N}$ on $L^{2}_{m}(\Sigma)$,
Proposition~\ref{propo:charact-of-c-complete-operators} yields that
$\overline{A_{1\cap\infty}}$ is $m$-completely accretive in
$L^{1}_{m}(\Sigma)$ with dense domain and with $c$-complete resolvent.

Next, we first consider the case when $\phi : \R\to\R$ is a general
continuous, non-decreasing function satisfying $\phi(0)=0$. For every
$\lambda>0$, let $\beta_{\lambda}(s)=(1+\lambda\beta)^{-1}(s)$,
$(s\in \R)$, denote the Yosida operator of $\beta:=\phi^{-1}$. Then,
by the Lipschitz continuity of $\beta_{\lambda} : \R\to\R$, since
$\beta_{\lambda}(0)=0$ and since $\beta_{\lambda}$ monotonically
increasing,
\begin{align*}
  [\beta_{\lambda}(u),A(u)]_{2}
  &=\int_{\Sigma}\abs{\nabla
    u}^{p-2}\nabla u\nabla \beta_{\lambda}(u)\,\dx +
    a\,\int_{\partial\Sigma}\abs{u}^{p-2}u \beta_{\lambda}(u)\,\dH\\
  &= \int_{\Sigma}\abs{\nabla
    u}^{p}\beta'_{\lambda}(u)\,\dx +
    a\,\int_{\partial\Sigma}\abs{u}^{p-2}u \beta_{\lambda}(u)\,\dH\\
  &\ge 0,
\end{align*}
for every $u\in D(A)$ and $\lambda>0$ provided $A$ is one of the three
operators $-\Delta_{p}^{D}$, $-\Delta_{p}^{N}$ or $-\Delta_{p}^{R}$,
where $a=0$ if $A$ is not $-\Delta_{p}^{R}$. Thus,
condition~\eqref{eq:78} holds for $q=2$. To see that the trace
$A_{1\cap\infty}$ of $A$ in $L^{1}\cap L^{\infty}(\Sigma,\mu)$
satisfies ~\eqref{eq:78} for $q=1$, let
$(\gamma_{\varepsilon})_{\varepsilon>0}$ be the sequence given
by~\eqref{eq:64}, then by the Lipschitz continuity of
$\gamma_{\varepsilon}$ and $\beta_{\lambda}$ on $\R$, since
$\gamma_{\varepsilon}(0)=0$ and $\beta_{\lambda}(0)=0$, and by the
monotonicity of $\gamma_{\varepsilon}$ and $\beta_{\lambda}$ on $\R$,
\begin{align*}
  &\int_{\Sigma}\gamma_{\varepsilon}(\beta_{\lambda}(u))\,A_{1\cap\infty}(u)\,\dmu\\
  &\qquad=\int_{\Sigma}\abs{\nabla
    u}^{p-2}\nabla u\nabla \gamma_{\varepsilon}(\beta_{\lambda}(u))\,\dx +
    a\,\int_{\partial\Sigma}\abs{u}^{p-2}u \gamma_{\varepsilon}(\beta_{\lambda}(u))\,\dH\\
  &\qquad= \int_{\Sigma}\abs{\nabla
    u}^{p}\gamma'_{\varepsilon}(\beta_{\lambda}(u))\beta'_{\lambda}(u)\,\dx +
    a\,\int_{\partial\Sigma}\abs{u}^{p-2}u \gamma_{\varepsilon}(\beta_{\lambda}(u))\,\dH\\
  &\qquad\ge 0,
\end{align*}
for every $u\in D(A_{1\cap\infty})$, $\lambda>0$ and $\varepsilon>0$. Since 
\begin{displaymath}
  \lim_{\varepsilon\to0+}\gamma_{\varepsilon}(\beta_{\lambda}(u(x)))=
  \textrm{sign}_{0}(\beta_{\lambda}(u(x)))\qquad\text{for
a.e. $x\in \Sigma$,}
\end{displaymath}
and $\abs{\gamma_{\varepsilon}(\beta_{\lambda}(u))\,A_{1\cap\infty}(u)}\le
\abs{A_{1\cap\infty}(u)}\in L^{1}(\Sigma)$, Lebesgue's dominated convergence
theorem yields
\begin{displaymath}
  \lim_{\varepsilon\to 0+}\int_{\Sigma}\gamma_{\varepsilon}(\beta_{\lambda}(u))\,A_{1\cap\infty}(u)\,\dmu=
  \int_{\Sigma}\textrm{sign}_{0}(\beta_{\lambda}(u(x)))A_{1\cap\infty}(u)\,\dmu.
\end{displaymath}
Thus,
\begin{displaymath}
  [\beta_{\lambda}(u),\,A_{1\cap\infty}(u)]_{1}\ge 0\qquad
  \text{for all $u\in D(A_{1\cap\infty})$ and $\lambda>0$.}
\end{displaymath}
Therefore, if $A$ is one of the three operators $-\Delta_{p}^{D}$,
$-\Delta_{p}^{N}$ or $-\Delta_{p}^{R}$ then for every continuous
non-decreasing function $\phi$ on $\R$ satisfying $\phi(0)=0$, one has
that condition~\eqref{eq:252} of
Proposition~\ref{propo:Range-cond-in-Rd} holds and so, under either
hypothesis~\eqref{propo:Range-cond-in-Rd-Hyp-3},
hypothesis~\eqref{propo:Range-cond-in-Rd-Hyp-1} or
hypothesis~\eqref{propo:Range-cond-in-Rd-Hyp-2} of
Proposition~\ref{propo:Range-cond-in-Rd}, we can conclude that the
closure $\overline{A_{1\cap\infty}\phi}$ of $A_{1\cap\infty}\phi$ in
$L^{1}(\Sigma,\mu)$ is an $m$-accretive operator in $L^{1}(\Sigma)$
with complete resolvent. In particular, for $A=-\Delta_{p}^{N}$, the
operator $\overline{A_{1\cap\infty}\phi}$ is $m$-accretive in
$L^{1}_{m}(\Sigma)$ with $c$-complete resolvent. If $\phi$ satisfies
either hypothesis~\eqref{propo:Range-cond-in-Rd-Hyp-1} or
hypothesis~\eqref{propo:Range-cond-in-Rd-Hyp-2} holds, then
$A_{1\cap\infty}\phi$ satisfies the range condition~\eqref{eq:182},
which is important in order to apply
Theorem~\ref{thm:GN-implies-reg-for-oper-in-L1}. Therefore, if either
$\phi(s)$ is locally Lipschitz continuous or $A$ is defined on
$L^{2}(\Sigma)$ with $\Sigma$ an open subset of $\R^{d}$ of finite
Lebesgue measure, then $A_{1\cap\infty}\phi$ satisfies the range
condition~\eqref{eq:182}.  Moreover, we can state the following
result.
%However, neither the range
% condition~\eqref{eq:182} nor the density of the domain
% $D(A_{1\cap\infty}\phi)$ in $L^{1}(\R^{d})$ is trivial to state.

 \begin{lemma}
   \label{lem:4}
   Let $\phi$ be a continuous non-decreasing function satisfying
   $\phi(0)=0$ and $\Sigma$ an open subset of $\R^{d}$. Let $A$ be the
   negative Dirichlet-$p$-Laplace operator $-\Delta_{p}^{D}$ on
   $L^{2}(\Sigma)$ and $A_{1\cap\infty}$ the trace of $A$ on
   $L^{1}\cap L^{\infty}(\Sigma)$.
%In the case
%   $A=-\Delta_{p}^{N}$ or $A=-\Delta_{p}^{R}$ suppose, in addition, $\Sigma$
%   is a bounded domain with a Lipschitz boundary. 
 Then, $A_{1\cap\infty}\phi$ satisfies range condition~\eqref{eq:182}.
 \end{lemma}
 
 By using that $A_{1\cap\infty}\phi$ satisfies the range
 condition~\eqref{eq:182} and under the assumption that $\phi$ is a
 continuous strictly increasing function satisfying $\phi(0)=0$ and
 $\Sigma$ be either an open bounded subset of $\R^{d}$ or $\R^{d}$,
 it is not difficult to see that the domain
 $D(A_{1\cap\infty}\phi)$ is dense in $L^{1}(\Sigma)$.

We briefly outline the proof of Lemma~\eqref{lem:4}.

 \begin{proof}[Proof of Lemma~\ref{lem:4}]
   Let $(\Sigma_{n})_{n\ge 1}$ be a sequence of subsets
   $\Sigma_{n}\subseteq \Sigma$ satisfying
   $\Sigma_{n}\subseteq \Sigma_{n+1}$ and
   $\bigcup_{n\ge 1}\Sigma_{n}=\Sigma$. Let $\Delta_{p}^{D,n}$ be the
   Dirichlet-$p$-Laplace operator on $L^{2}(\Sigma_{n})$ and
   $A_{n}$ the trace of $\Delta_{p}^{D,n}$ on
   $L^{1}\cap L^{\infty}(\Sigma_{n})$. Since $\Sigma_{n}$ has finite
   Lebesgue measure, Proposition~\ref{propo:Range-cond-in-Rd} implies
   that the operator $A_{n}\phi$ satisfies range
   condition~\eqref{eq:182}. For every $\lambda>0$, let $J_{\lambda}^{n}$ be the
   resolvent operator of $A_{1\cap\infty}\phi$. 

   Now, let $f\in L^{1}\cap L^{\infty}(\Sigma)$,
   $\lambda>0$ and for every $n\ge 1$, set $f_{n}=f \mathds{1}_{\Sigma_{n}}$,
   $u_{n}=J_{\lambda}^{n}[f_{\vert \Sigma_{n}}]$ and $\tilde{u}_{n}$
   the extension of $u_{n}$ on $\R^{d}$ by zero. Then, our first aim
   is to show that there is $u\in D(A_{1\cap\infty}\phi)$ satisfying
   $u+\lambda A_{1\cap\infty}\phi(u)\ni f$ and after eventually
   passing to a subsequence, 
   \begin{equation}
     \label{eq:241}
     \lim_{n\to\infty}\tilde{u}_{n}=u\qquad\text{in $L^{1}(\Sigma)$.}
   \end{equation}
   Since $f\in L^{1}\cap L^{\infty}(\Sigma)$ and $A_{1\cap\infty}\phi$
   has a complete resolvent, it follows that
   \begin{equation}
     \label{eq:240}
     \norm{\tilde{u}_{n}}_{q}\le \norm{f}_{q}\qquad\text{for every
       $n\ge 1$}
   \end{equation}
   and all $1\le q\le \infty$. By reflexivity of $L^{q}(\Sigma)$ for
   $q>1$, there is $u\in L^{q}(\Sigma)$ such that after eventually
   passing to a subsequence $\tilde{u}_{n}$ converges to $u$ weakly in
   $L^{q}(\Sigma)$ and $\norm{u}_{q}\le \norm{f}_{q}$. Moreover, since
   for every $\varepsilon>0$ and for every $A\subseteq \Sigma$
   satisfying
   $\abs{A}<\delta(\varepsilon):=\norm{f}_{\infty}^{-1}\,\varepsilon$,
   one sees that 
   \begin{math}
     \int_{A}\abs{\tilde{u}_{n}}\,\dx\le \varepsilon
   \end{math}
   for all $n\ge 1$. Thus the Dunford-Pettis theorem implies
   $u\in L^{1}(\Sigma)$, $\norm{u}_{1}\le \norm{f}_{1}$ and
   $\tilde{u}_{n}$ converges to $u$ weakly in $L^{1}(\Sigma)$ after
   passing eventually to a subsequence of $(\tilde{u}_{n})_{n\ge 1 }$.
   If $f\ge 0$, then by the $T$-accretivity of $A_{n}$, we have that
   $0\le \tilde{u}_{n}\le \tilde{u}_{n+1}$ a.e. on $\Sigma$ for every
   $n\ge 1$. Moreover, by~\eqref{eq:240} for $q=1$, Beppo-Levi's
   monotone convergence theorem yields \eqref{eq:241} for some
   $u\in L^{1}(\Sigma)$ satisfying $u\ge 0$ provided $f\ge 0$. Similar
   arguments show that $\tilde{u}_{n+1}\le \tilde{u}_{n}\le 0$ and
   \eqref{eq:241} holds for some $u\in L^{1}(\Sigma)$ satisfying
   $u\le 0$ provided $f\le 0$. Since $-f^{-}\le f\le f^{+}$, the
   $T$-accretivity of $A_{n}$ yields
   $J_{\lambda}^{n}[-f_{\vert \Sigma_{n}}^{-}]\le u_{n}\le
   J_{\lambda}^{n}[f_{\vert \Sigma_{n}}^{+}]$
   a.e. on $\Sigma$. Let $\tilde{u}_{-,n}$ denote the extension on
   $\Sigma$ of $J_{\lambda}^{n}[-f_{\vert \Sigma_{n}}^{-}]$ by zero
   and $\tilde{u}_{+,n}$ denote the extension on $\Sigma$ of
   $J_{\lambda}^{n}[f_{\vert \Sigma_{n}}^{+}]$ by zero. Then, there
   are $u_{-}$ and $u_{+}\in L^{1}(\Sigma)$ such that
   $\lim_{n\to\infty}\tilde{u}_{-,n}=u_{-}$ and
   $\lim_{n\to\infty}\tilde{u}_{+,n}=u_{+}$ in $L^{1}(\Sigma)$. By the
   monotonicity of $(\tilde{u}_{-,n})_{n\ge 1}$ and
   $(\tilde{u}_{+,n})_{n\ge 1}$, we obtain
   $u_{-}\le \tilde{u}_{n}\le u_{+}$ a.e. on $\Sigma$ for all
   $n\ge 1$. Thus, by Lebesgue's dominated convergence theorem,
  ~\eqref{eq:241} holds provided $\tilde{u}_{n}$ converges to
   $u$ a.e. on $\Sigma$.

   We multiply equation $u_{n}+\lambda A_{n}\phi= f_{n}$ by
   $\phi(u_{n})$ with respect to the $L^{2}$-inner product. Then
   coercivity condition~\eqref{eq:coerciveness} yields
   $(\phi(\tilde{u}_{n}))_{n\ge 1}$ is bounded in
   $W^{1,p}_{0}(\Sigma)$. Hence and by using Rellich-Kondrachov's
   compactness result combined with a diagonal-sequence argument
   yields the existence of a subsequence of
   $(\tilde{u}_{n})_{n\ge 1}$, which we denote again by
   $(\tilde{u}_{n})_{n\ge 1}$ and some $v\in W^{1,p}_{0}(\Sigma)$ such
   that $\phi(\tilde{u}_{n})$ converges weakly to $v$ in
   $W^{1,p}_{0}(\Sigma)$ and strongly in $L^{p}_{loc}(\Sigma)$. By the
   continuity of $\phi^{-1}$ on $\R$, it follows that
   $\tilde{u}_{n}=\phi^{-1}(\phi(\tilde{u}_{n}))$ converges to
   $\phi^{-1}(v)$ in $L^{p}_{loc}(\Sigma)$ and a.e. on $\Sigma$ after
   passing again to a subsequence. Comparing this with the weak limit
   $u$ of $(\tilde{u}_{n})$ in $L^{q}(\Sigma)$, it follows that
   $\phi^{-1}(v)=u$ and that limit~\eqref{eq:241} holds. Now, by using
   classical monotonicity arguments due to
   Leray-Lions~\cite{MR0194733} (as employed, for instance,
   in~\cite[Lemma~2.5]{MR3369257}) yields $u\in D(A_{1\cap\infty}\phi)$ with
   $u+\lambda A_{1\cap\infty}\phi(u)\ni f$.
 \end{proof}

\begin{remark}
  Note that, for $A=-\Delta_{p}^{D}$, the
  operator $\overline{A_{1\cap\infty}\phi}$ coincides with the associated
  \emph{entropy solution} operator of the composition $-\Delta_{p}^{D}\phi$. This
  has been investigate in the celebrated paper~\cite{MR1354907}. 
\end{remark}

Let $F$ denote the Nemytski operator on $L^{1}(\Sigma)$ associated
with a Carath\'eo\-dory function $f : \Sigma\times \R\to\R$
satisfying~\eqref{eq:2} for some Lipschitz constant $L>0$. If
$\Delta_{p,1}^{D}$, $\Delta_{p,1}^{N}$ and $\Delta_{p,1}^{R}$ are
respectively the traces on $L^{1}\cap L^{\infty}(\Sigma,\mu)$ of the
operators $\Delta_{p}^{D}$, $\Delta_{p}^{N}$ or $\Delta_{p}^{R}$, and
if $\phi\in C(\R)$ is a non-decreasing function satisfying $\phi(0)=0$
then, by Proposition~\ref{propo:Lipschitz-complete} and
Lemma~\ref{lem:4}, the operators
\begin{displaymath}
  \begin{array}[c]{c}
  A^{D}_{\phi}:=\overline{(-\Delta_{p,1}^{D})\phi}+F\text{ on $L^{1}(\Sigma)$},\qquad
  A^{N}_{\phi}:=\overline{(-\Delta_{p,1}^{N})\phi}+F\text{ on $L^{1}_{m}(\Sigma)$},\\[7pt]
  A^{R}_{\phi}:=\overline{(-\Delta_{p,1}^{R})\phi}+F\text{ on $L^{1}(\Sigma)$},
  \end{array}
\end{displaymath}
are quasi-$m$-accretive in $L^{1}$ with complete resolvent. %and dense
%domain. %  In particular, $\overline{(-\Delta_{p,1}^{N})\phi}$ is
% $m$-accretive in $L^{1}(\Sigma)$ with $c$-complete resolvent and dense
% domain. 
By the Crandall-Liggett theorem, the operators $-A^{D}_{\phi}$,
$-A^{N}_{\phi}$ and $-A^{R}_{\phi}$ generate a strongly continuous
semigroup $\{T_{t}\}_{t\ge 0}$ on
$\overline{D(A_{1\cap\infty}\phi)}^{\mbox{}_{L^{1}}}$ of Lipschitz
continuous mappings $T_{t}$ on the set
$\overline{D(A_{1\cap\infty}\phi)}^{\mbox{}_{L^{1}}}$ with constant
$e^{\omega t}$ and satisfying exponential
growth~\eqref{eq:60} with respect to the $L^{\tilde{q}}$-norm for all
$1\le \tilde{q}\le \infty$. % In particular,
% $-\overline{(-\Delta_{p,1}^{N}\phi)}$
% generates a strongly continuous semigroup $\{T_{t}\}_{t\ge
%   0}$ on $L^{1}(\Sigma)$ of contractive mappings
% $T_{t}$
% on $L^{1}(\Sigma)$
% satisfying~\eqref{eq:120} in $L^{\tilde{q}}(\Sigma)$
% for all $1\le \tilde{q}\le \infty$ every $c\in \R$.
\bigskip

%%%%%%%%%%%%%%%%%%%%%%%%%%%%%%%%%%%%%%%%%%%%%%%%%%%%%%%%%
%
%                   Dirichlet case
%
%%%%%%%%%%%%%%%%%%%%%%%%%%%%%%%%%%%%%%%%%%%%%%%%%%%%%%%%%

Here, we state the complete description of the
$L^{q}$-$L^{r}$-regularisation effect of the Dirichlet-semigroup
$\{T_{t}\}_{t\ge 0}\sim-A_{\phi}^{D}$ on
$\overline{D(A_{\phi}^{D})}^{\mbox{}_{L^{1}}}$ for a non-decreasing
function $\phi\in C(\R)\cap C^{1}(\R\setminus\{0\})$
satisfying~\eqref{eq:10} for some $m>0$ and $C>0$.

\begin{theorem}
  \label{thm:Lq-Lr-reg-of-doubly-nonlinear-Dirichlet}
  Let $\phi\in C(\R)\cap C^{1}(\R\setminus\{0\})$ be non-decreasing
  function satisfying~\eqref{eq:10} for some $m>0$ and $C>0$, and
  let $\Sigma$ be an arbitrary open set of $\R^{d}$. Then, the semigroup
  $\{T_{t}\}_{t\ge 0}\sim-A_{\phi}^{D}$ on
  $\overline{D(A_{\phi}^{D})}^{\mbox{}_{L^{1}}}$ satisfies the
  following regularisation estimates.
   \begin{enumerate}
   \item\label{thm:Lq-Lr-reg-of-doubly-nonlinear-Dirichlet-claim1} If
     $1<p<d$, then there is $\beta^{\ast}\ge 0$ such that the
     semigroup $\{T_{t}\}_{t\ge 0}$ satisfies estimate~\eqref{eq:168}
     with $u_{0}=0$ for every
     $u\in \overline{D(A_{\phi}^{D})}^{\mbox{}_{L^{1}}}\cap
     L^{\infty}(\Sigma)$ with exponents
      \begin{displaymath}
         \alpha_{s}=
        \tfrac{\alpha^{\ast}}{1-\gamma^{\ast}\big(1-\frac{s(d-p)}{dmq_{0}}\big)},\quad
        \beta_{s}= \tfrac{\frac{\beta^{\ast}}{2}+\gamma^{\ast}
          \frac{s(d-p)}{dmq_{0}}}{1-\gamma^{\ast}(1-\frac{s(d-p)}{dmq_{0}})},\quad
        \gamma_{s}=\tfrac{\gamma^{\ast}\,s}{\frac{dm q_{0}}{(d-p)}
          \big(1-\gamma^{\ast}(1-\frac{s(d-p)}{dmq_{0}}) \big)}
      \end{displaymath}
      for every $q_{0}\ge p$ satisfying
    \begin{math}
     \frac{pq_{0}}{d-p}+p-1-\frac{1}{m}>0
    \end{math}
    and $1\le s\le \frac{d m\,q_{0}}{d-p}$ satisfying
    $\gamma^{\ast}(1-\frac{s(d-p)}{dmq_{0}})<1$, where
    \begin{displaymath}
      \alpha^{\ast}=\tfrac{1}{\frac{mp}{d-p}q_{0}+mp-m-1},\quad\text{ and }\quad
      \gamma^{\ast}=\tfrac{p\, q_{0}}{p\, q_{0}+(d-p)(p-1-\frac{1}{m})}.
    \end{displaymath}
    Moreover, for $\frac{d(1+\frac{1}{m})}{1+d+\frac{1}{m}}<p<d$, one can
    take $q_{0}=p$ and for $\frac{d(m+1)}{dm+1}<p<d$,
    the semigroup $\{T_{t}\}_{t\ge 0}$ satisfies~\eqref{eq:168} for every
    $u\in \overline{D(A_{\phi}^{D})}^{\mbox{}_{L^{1}}}\cap
    L^{\infty}(\Sigma)$ and $1\le s\le \frac{d m\,q_{0}}{d-p}$.

   % \item If $p=d\ge 2$, then for every $\theta \in (0,1)$, there is is
   %   a constant $C>0$ such that the semigroup $\{T_{t}\}_{t\ge 0}$
   %   satisfies estimate~\eqref{eq:239} with $u_{0}=0$,
   %   $r=\frac{m+1}{1-\theta}$ and for every $u\in \overline{D(A_{\phi}^{D})}^{\mbox{}_{L^{1}}}\cap
   %  L^{\infty}(\Sigma)$ and $1\le s \le m+1$ with exponents
   %   \begin{displaymath}

   %     \sigma=\tfrac{d m}{\theta}\quad\text{ and }\quad
   %     \beta=\tfrac{m+1+d\frac{1-\theta}{\theta}m}{\frac{d m}{\theta}}.
   %   \end{displaymath}
   
  \item If $p=d\ge 2$ and $\Sigma$ has finite Lebesgue measure, then
    for every $\theta \in (0,1)$, there is a $\beta^{\ast}_{\theta}\ge 0$ such that the semigroup
    $\{T_{t}\}_{t\ge 0}$ satisfies estimate~\eqref{eq:168} with
    $u_{0}=0$ for every
    $u\in \overline{D(A_{\phi}^{D})}^{\mbox{}_{L^{1}}}\cap
    L^{\infty}(\Sigma)$ with exponents
      \begin{displaymath}
          \alpha_{s}=
        \tfrac{\alpha^{\ast}}{1-\gamma^{\ast}_{\theta}\big(1-\frac{s(1-\theta)}{
            m q_{0}}\big)},\quad
        \beta_{s}= \tfrac{\frac{\beta^{\ast}}{2}+\gamma^{\ast}_{\theta}
          \frac{s(1-\theta)}{mq_{0}}}{1-\gamma^{\ast}_{\theta}(1-\frac{s(\theta-1)}{mq_{0}})},\quad
        \gamma_{s}=\tfrac{\gamma^{\ast}_{\theta}\,s}{\frac{m q_{0}}{(1-\theta)}
          (1-\gamma^{\ast}_{\theta}(1-\frac{s(1-\theta)}{mq_{0}}))}
      \end{displaymath} 
      for every $q_{0}\ge p$ satisfying
      $ \frac{\theta q_{0}}{1-\theta}+p-1-\frac{1}{m}>0$ and
      $1\le s\le \frac{m\,q_{0}}{1-\theta}$, where
    \begin{displaymath}
      \alpha^{\ast}_{\theta}=\tfrac{1}{m(\frac{\theta q_{0}}{1-\theta}+p-1-\frac{1}{m})},\quad\text{ and }\quad
      \gamma^{\ast}_{\theta}=\tfrac{\frac{\theta}{1-\theta}\,q_{0}}{\frac{\theta q_{0}}{1-\theta}+p-1-\frac{1}{m}}.
    \end{displaymath}
    If one takes $\max\Big\{0,\tfrac{1+m(1-p)}{m+1}\Big\}<\theta<1$,
    then one can take $q_{0}=p$ and the semigroup $\{T_{t}\}_{t\ge 0}$
    satisfies estimate~\eqref{eq:168} with $s=1$ for every
    $u\in \overline{D(A_{\phi}^{D})}^{\mbox{}_{L^{1}}}$.

   \item If $p>d$, then the semigroup
    $\{T_{t}\}_{t\ge 0}$ satisfies estimate~\eqref{eq:168} for every $u\in \overline{D(A_{\phi}^{D})}^{\mbox{}_{L^{1}}}\cap
    L^{\infty}(\Sigma)$ with exponents
     \begin{displaymath}
       \alpha_{s}=\frac{\alpha^{\ast}}{1-\gamma^{\ast}(1-\frac{s}{m+1})},\quad
       \beta_{s}=\tfrac{\frac{\beta^{\ast}}{2}+\gamma^{\ast}
         \tfrac{s}{m+1}}{1-\gamma^{\ast}(1-\frac{s}{m+1})},\quad
       \gamma_{s}=\tfrac{\gamma^{\ast}\frac{s}{m+1}}{1-\gamma^{\ast}(1-\frac{s}{m+1})},
     \end{displaymath}
    for every $1\le s\le m+1$, where
     \begin{equation}
       \label{eq:174}
       \alpha^{\ast}=\tfrac{1}{pm(1-\frac{m+1}{mp}+\frac{m+1}{md})},\quad
       \beta^{\ast}=\gamma^{\ast}+1
       \quad\text{and}\quad
       \gamma^{\ast}=\tfrac{m+1}{
       d m (1-\frac{m+1}{mp}+\frac{m+1}{md})}.
     \end{equation}
     %In particular, for $m\ge 1$, $\{T_{t}\}_{t\ge 0}$ satisfies
     %estimate~\eqref{eq:168} for every $u\in \overline{D(A_{\phi}^{D})}^{\mbox{}_{L^{1}}}$.
   \end{enumerate}
\end{theorem}

% For the proof, we shall make use of the following lemma. We leave the
% easy proof to the reader as an exercise.

% \begin{lemma}
%   \label{lem:derivative-of-composition-function}
%   Let $2\le r<\infty$ and $u\in W^{1}_{p,\infty}(\Sigma)$. Then
%   $\abs{u}^{r-2}u\in W^{1}_{p,\infty}(\Sigma)$ with
%   \begin{displaymath}
%     \nabla (\abs{u}^{r-2}u)=(r-1) \abs{u}^{r-2}\,\nabla
%     u\qquad\text{a.e. on $\Sigma$.}
%   \end{displaymath} 
% \end{lemma}

We outline the proof of
Theorem~\ref{thm:Lq-Lr-reg-of-doubly-nonlinear-Dirichlet}.

\begin{proof}[Proof of Theorem~\ref{thm:Lq-Lr-reg-of-doubly-nonlinear-Dirichlet}]
  % Note, for either $m\ge 1$ or for $0<m<1$ and bounded $\Sigma$, the
  % operator $(-\Delta_{p,1}^{D})\phi$ with $\phi(s)=\abs{s}^{m-1}s$
  % satisfies range condition~\eqref{eq:182} in
  % Proposition~\ref{propo:Range-cond-in-Rd}. In order to establish the
  % statements of this theorem for $0<m<1$ and unbounded $\Sigma$ one
  % proceeds as outlined in Remark~\ref{rem:5}. Here, we only outline the
  % proof of this results under the assumption that the operator
  % $(-\Delta_{p,1}^{D})\phi$ satisfies range condition~\eqref{eq:182} in
  % Proposition~\ref{propo:Range-cond-in-Rd} since the arguments given in
  % Remark~\ref{rem:5} is very well known (see, for
  % instance,~\cite{MR2286292}).

  By Lemma~\ref{lem:4}, the operator $(-\Delta_{p,1}^{D})\phi$ satisfies range
  condition~\eqref{eq:182} in
  Proposition~\ref{propo:Range-cond-in-Rd}. Hence, we intend to apply 
  Theorem~\ref{thm:GN-implies-reg-for-oper-in-L1}.

  We begin by considering the case $1<p<d$. Then by
  Lemma~\ref{lem:Sobolev-Gagliardo-Nirenberg}, there is a constant
  $C>0$ such that inequality~\eqref{eq:162} holds for every
  $u\in\dot{W}_{p,2}^{1}(\Sigma)$. For every
  $(u,v)\in (-\Delta_{p,1}^{D})\phi$, one has
  $\phi(u)\in \dot{W}_{p,\infty}^{1}(\Sigma)$. By classical
  interior regularity results (see~\cite{MR727034}) and since $\phi'(r)>0$
  for all $r\neq 0$, one has
  $u\in C(\Sigma)\cap C^{1,\alpha}(\{u\neq0\})$ and $\nabla u\equiv 0$
  on the level set $\{u=0\}$. Combining this with coercivity
  condition~\eqref{eq:10} and Gagliardo-Nirenberg
  inequality~\eqref{eq:31} for $1<p<d$, we see that
  \begin{align*}
    [u,v]_{(q-p+1)m+1}
    &= \int_{\Sigma}\abs{\nabla \phi(u)}^{p-2}\nabla\phi(u)\nabla (\abs{u}^{((q-p+1)m+1)-2}u)\,\dx\\
    &= (q-p+1)\,m\,\int_{\{u\neq0\}} \abs{u}^{(q-p+1)m-1}\abs{\nabla u}^{p}\,[\phi']^{p-1}(u)\,\dx\\
    &\ge C^{p-1}\, (q-p+1)\,m\,\int_{\{u\neq0\}} \abs{u}^{qm-p}\abs{\nabla u}^{p}\,\dx\\
    &=  \left[\frac{C}{m}\right]^{p-1}\,\frac{(q-p+1) p^{p}}{q^{p}}\norm{\abs{\nabla (\abs{u}^{\frac{qm+p}{p}-2}u)}}_{p}^{p}\\
    &\ge \left[\frac{C}{m}\right]^{p-1}\,\frac{(q-p+1) p^{p}}{q^{p}}\,
      \tilde{C}^{-p}\,\norm{\abs{u}^{\frac{qm+p}{p}-2}u}_{\frac{pd}{d-p}}^{p}\\
    &=  \left[\frac{C}{m}\right]^{p-1}\,\frac{(q-p+1) p^{p}}{q^{p}}\,\tilde{C}^{-p}\,\norm{u}_{\frac{qmd}{d-p}}^{qm}
  \end{align*}
  for every $q\ge p$. Here, the constant $\tilde{C}>0$ is the one
  given by Gagliardo-Nirenberg inequality~\eqref{eq:31} and is independent of
  $\Sigma$. Remark~\ref{rem:2} yields that the operator $A_{\phi}^{D}$ satisfies
  the one-parameter family of Gagliardo-Nirenberg type
  inequalities~\eqref{eq:10single} with $u_{0}=0$ and
  $\kappa = \frac{d}{d-p}>1$ and so
  Theorem~\ref{thm:GN-implies-reg-for-oper-in-L1} yields the first
  statement of this theorem.

 Next, we consider the case $p=d$ and suppose that $\Sigma$ is a
 general open subset of $\R^{d}$ with finite Lebesgue measure. By
 Lemma~\ref{lem:Sobolev-Gagliardo-Nirenberg}, for every $1\le q\le
 \infty$ and every $\theta\in [0,1)$, there is a constant $\tilde{C}=\tilde{C}(q,d,\theta)>0$ such
 that
 \begin{displaymath}
 %  \label{eq:141}
   \norm{u}_{\frac{q}{1-\theta}}^{\frac{d}{\theta}}\le \tilde{C}\,
   \norm{\abs{\nabla u}}_{d}^{d}\,\norm{u}_{q}^{d\frac{1-\theta}{\theta}}
 \end{displaymath}
for every $u\in \dot{W}^{1}_{p,q}(\Sigma)$. % Using this inequality
% for $q=\frac{m+1}{m}$, we
% see that
% \begin{align*}
%     [u,v]_{m+1}\,\norm{u}_{m+1}^{d\frac{1-\theta}{\theta}m}
%     &= \norm{\abs{\nabla u^{m}}}_{d}^{d}\,\norm{u^{m}}_{\frac{m+1}{m}}^{d\frac{1-\theta}{\theta}} \\
%     &\ge  C^{-1}\,\norm{u^{m}}_{\frac{m+1}{m(1-\theta)}}^{\frac{d}{\theta}}\\
%     &=  C^{-1}\,\norm{u}_{\frac{m+1}{1-\theta}}^{\frac{d m}{\theta}}
%  \end{align*}
%  for every $(u,v)\in ((-\Delta_{p,1}^{D})\phi)_{1}$. This show
%  $(-\Delta_{p,1}^{D})\phi$ satisfies Gagliardo-Niren\-berg
%  inequality~\eqref{eq:242} for $(0,0)\in
%  (-\Delta_{p,1}^{D})\phi$ with parameters
%  \begin{displaymath}
%    r=\frac{m+1}{1-\theta},\quad \sigma=\frac{d m}{\theta},\quad
%    \varrho=d\frac{1-\theta}{\theta}m,\quad q=m+1.
%  \end{displaymath}
% Thus by remark~\ref{rem:2} and
% Theorem~\ref{thm:GN-implies-reg-for-oper-in-L1}, one sees that
% inequality~\eqref{eq:170} holds.
% Now, suppose $p=d$ and $\Sigma$
For functions $u\in C^{\infty}_{c}(\Sigma)$, Maz'ya's
inequality~\eqref{eq:90} reduces to a Poincar\'e inequality, which we apply to estimate
$\norm{u}_{q}^{d\frac{1-\theta}{\theta}}$ for $q=d$ in the last
inequality. Then for every $\theta\in [0,1)$, there is a constant
$\tilde{C}>0$, which might be different to the one given in the
previous inequality, such that
\begin{equation}
  \label{eq:169}
  \norm{u}_{\frac{d}{(1-\theta)}}^{d}\le \tilde{C}\,\norm{\abs{\nabla u}}_{d}^{d}
\end{equation}
for every $u\in C^{\infty}_{c}(\Sigma)$. Since for $1\le q<\infty$,
$\dot{W}_{p,q}^{1}(\Sigma)$ is the closure of $C^{\infty}_{c}(\Sigma)$
in $W^{1}_{p,q}(\Sigma)$, an approximation argument shows
that~\eqref{eq:169} holds also for functions
$u\in \dot{W}_{p,q}^{1}(\Sigma)$. Now, proceeding as in the case
$1<p<d$ and using~\eqref{eq:169}, yields
\begin{align*}
  [u,v]_{(q-d+1)m+1}&\ge \left[\frac{C}{m}\right]^{p-1}\,\frac{(q-d+1) d^{d}}{q^{d}}
    \tilde{C}^{-p}\,\norm{\abs{u}^{\frac{qm+d}{d}-2}u}_{\frac{d}{(1-\theta)}}^{d}\\
  &= \left[\frac{C}{m}\right]^{p-1}\,\frac{(q-d+1) d^{d}}{q^{d}} C^{-p}\,\norm{u}_{\frac{qm}{(1-\theta)}}^{qm}
\end{align*}
for every $(u,v)\in ((-\Delta_{p,1}^{D})\phi)_{1}$ and $q\ge p=d$,
where for every $\theta\in [0,1)$, the constant $C>0$ depends on
the measure of $\Sigma$, $\theta$ and $p=d$. Remark~\ref{rem:2} yields
that the operator $A_{\phi}^{D}$ satisfies the one-parameter family of
Gagliardo-Niren\-berg type inequalities~\eqref{eq:10single} with
$u_{0}=0$ and $\kappa = \frac{1}{1-\theta}>1$ and so
Theorem~\ref{thm:GN-implies-reg-for-oper-in-L1} yields the third
statement of this theorem.

Now, let $p>d$. Then by Lemma~\ref{lem:Sobolev-Gagliardo-Nirenberg}
there is a $\theta_{0}\in (0,1)$ satisfying
 \begin{displaymath}
   \theta_{0}(\tfrac{1}{p}-\tfrac{1}{d})+(1-\theta_{0})\tfrac{m}{m+1}=0
\end{displaymath}
 and a constant $\tilde{C}>0$ such that
 \begin{equation}
   \label{eq:143}
   \norm{u}_{\infty}\le \tilde{C}\,\norm{\abs{\nabla u}}_{p}^{\theta_{0}}\,
   \norm{u}_{\frac{m+1}{m}}^{1-\theta_{0}}
 \end{equation}
for every $u\in \dot{W}^{1}_{p,\frac{m+1}{m}}(\Sigma)$. By
applying~\eqref{eq:143} and the coercivity condition~\eqref{eq:10} of $\phi$, we see that
 \begin{align*}
   [u,v]_{m+1}\,\norm{u}_{m+1}^{m p
   \frac{1-\theta_{0}}{\theta_{0}}} 
   & = m\,\int_{\{u\neq 0\}}\abs{\nabla
     u}^{p}[\phi'(u)]^{p-1}\,\abs{u}^{m-1}\dx 
     \,\norm{u_{m+1}}_{\frac{m+1}{m}}^{p\frac{1-\theta_{0}}{\theta_{0}}}\\
   & \ge m\,C^{p-1}\,\int_{\{u\neq 0\}}\abs{\nabla
     u}^{p}\,\abs{u}^{p(m-1)}\dx 
     \,\norm{u_{m+1}}_{\frac{m+1}{m}}^{p\frac{1-\theta_{0}}{\theta_{0}}} \\
   &=\left[\frac{C}{m}\right]^{p-1}\,
     \int_{\Sigma}\abs{\nabla u_{m+1}}^{p}\dx \,\norm{u_{m+1}}_{\frac{m+1}{m}}^{p
     \frac{1-\theta_{0}}{\theta_{0}}} \\
   &\ge\left[\frac{C}{m}\right]^{p-1}\,
     \tilde{C}^{-\frac{p}{\theta_{0}}}\,\norm{u_{m+1}}_{\infty}^{\frac{p}{\theta_{0}}}\\
   &=\tilde{C}^{-\frac{p}{\theta_{0}}}\,\norm{u}_{\infty}^{\frac{pm}{\theta_{0}}}
 \end{align*}
 for every $(u,v)\in ((-\Delta_{p,1}^{D})\phi)_{1}$. Since
 $\theta_{0}=(1-\frac{m+1}{mp}+\frac{m+1}{md})^{-1}$ and by
 Remark~\ref{rem:2}, $A_{\phi}^{D}$ satisfies Gagliardo-Niren\-berg type
 inequality~\eqref{eq:242} with
 \begin{displaymath}
   r=\infty, \qquad \sigma=pm (1-\tfrac{m+1}{mp}+\tfrac{m+1}{md}),
   \qquad q=m+1,\qquad \varrho=m p
     \frac{1-\theta_{0}}{\theta_{0}} 
 \end{displaymath}
 and so by Theorem~\ref{thm:GN-implies-reg-for-oper-in-L1}
 and Theorem~\ref{thm:extrapol-L1-bis},
 the semigroup $\{T_{t}\}_{t\ge 0}\sim-A_{\phi}^{D}$ satisfies
 inequality~\eqref{eq:168} with $u_{0}=0$, $r=\infty$, $q=m+1$ and $\alpha^{\ast}$,
 $\beta^{\ast}$ and $\gamma^{\ast}$ given by~\eqref{eq:174}. Since for
 $m\ge 1$, $\gamma^{\ast}(1-\frac{1}{m+1})<1$,
 Theorem~\ref{thm:extrapol-L1-bis} completes the proof of the last claim of this
 theorem.
\end{proof}

%%%%%%%%%%%%%%%%%%%%%%%%%%%%%%%%%%%%%%%%%%%%%%%%%%%%%%%%%
%
%                   Neumann case
%
%%%%%%%%%%%%%%%%%%%%%%%%%%%%%%%%%%%%%%%%%%%%%%%%%%%%%%%%%

Next, we state the complete description of the
$L^{q}$-$L^{r}$-regularisation estimates of the semigroup
$\{T_{t}\}_{t\ge 0}\sim-A_{\phi}^{N}$ on $L^{1}_{m}(\Sigma)$. Here, we denote by
$L^{1}_{m}(\Sigma)$ the space of all functions $u\in L^{1}(\Sigma)$
with mean value $\overline{u}:=\tfrac{1}{\abs{\Sigma}}\int_{\Sigma}u\,\dx=0$.
%\overline{((-\Delta_{p,1}^{N})\phi)_{1}}

\begin{theorem}
  \label{thm:Lq-Lr-reg-of-doubly-nonlinear-Neumann}
  Let $\Sigma$ be a bounded domain with Lipschitz boundary and
  $\phi\in C(\R)\cap C^{1}(\R\setminus\{0\})$ be a non-decreasing
  function satisfying~\eqref{eq:10} for some $m>0$ and $C>0$. Then,
  for $1<p<\infty$, the semigroup
  $\{T_{t}\}_{t\ge 0}\sim-A_{\phi}^{N}$ on $L^{1}_{m}(\Sigma)$
  satisfies the $L^{q}$-$L^{r}$-regularisation estimate~\eqref{eq:168}
  with $u_{0}=0$ for every
  $u\in L^{1}_{m}(\Sigma)\cap L^{\infty}(\Sigma)$ with the same
  exponents and conclusions as for the semigroup generated by
  $-A_{\phi}^{D}$ on $\overline{D(A_{\phi}^{D})}^{\mbox{}_{L^{1}}}$
  stated in Theorem~\ref{thm:Lq-Lr-reg-of-doubly-nonlinear-Dirichlet}.
\end{theorem}

For the proof, we proceed similarly as in the proof of
Theorem~\ref{thm:Lq-Lr-reg-of-doubly-nonlinear-Dirichlet}.

\begin{proof}[Proof of Theorem~\ref{thm:Lq-Lr-reg-of-doubly-nonlinear-Neumann}]
  If $1<p<d$, then inequality~\eqref{eq:188} reduces to Sobolev
  inequality~\eqref{eq:162} by using functions
  $u\in W^{1}_{p,p,m}(\Sigma)$. If $p\ge d$, then applying Poincar\'e
  inequality~\eqref{eq:187} for functions $u\in W^{1}_{p,p,m}(\Sigma)$
  to Gagliardo-Nirenberg inequality~\eqref{eq:92} yields
  inequality~\eqref{eq:169} and~\eqref{eq:143}.  Thus, proceeding as
  in the proof of
  Theorem~\ref{thm:Lq-Lr-reg-of-doubly-nonlinear-Dirichlet}, we see
  that the statement of this theorem holds.
\end{proof}

%%%%%%%%%%%%%%%%%%%%%%%%%%%%%%%%%%%%%%%%%%%%%%%%%%%%%%%%%
%
%                   Robin case
%
%%%%%%%%%%%%%%%%%%%%%%%%%%%%%%%%%%%%%%%%%%%%%%%%%%%%%%%%%

To complete this subsection, we state the complete description of the
$L^{q}$-$L^{r}$-regularisation effect of the semigroup
$\{T_{t}\}_{t\ge 0}\sim-A_{\phi}^{R}$ on $L^{1}(\Sigma)$ and
$\phi(s)=\abs{s}^{m-1}s$ for $m>0$.

\begin{theorem}
  \label{thm:Lq-Lr-reg-of-doubly-nonlinear-Robin}
  Let $\Sigma$ be a bounded domain with Lipschitz boundary and
  $\phi\in C(\R)\cap C^{1}(\R\setminus\{0\})$ be a non-decreasing
  function satisfying~\eqref{eq:10} for some $m>0$ and $C>0$. Then,
  for $1<p<\infty$, the semigroup
  $\{T_{t}\}_{t\ge 0}\sim-A_{\phi}^{R}$ on $L^{1}(\Sigma)$ satisfies
  the $L^{q}$-$L^{r}$-regularisation estimate~\eqref{eq:168} with
  $u_{0}=0$ for every $u\in L^{1}(\Sigma)\cap L^{\infty}(\Sigma)$ with
  the same exponents and conclusions and as for the semigroup
  generated by $-A_{\phi}^{D}$ on
  $\overline{D(A_{\phi}^{D})}^{\mbox{}_{L^{1}}}$ stated in
  Theorem~\ref{thm:Lq-Lr-reg-of-doubly-nonlinear-Dirichlet}.
\end{theorem}

We proceed as in the proof of
Theorem~\ref{thm:Lq-Lr-reg-of-doubly-nonlinear-Dirichlet}.

\begin{proof}[Proof of Theorem~\ref{thm:Lq-Lr-reg-of-doubly-nonlinear-Robin}]
  % By range condition~\eqref{eq:182} in Proposition~\ref{propo:Range-cond-in-Rd}, the
  % operator $(-\Delta_{p,1}^{R})\phi$ is non-empty for $\phi(s)=\abs{s}^{m-1}s$. 
  We begin by considering the case $1<p<d$. Note that by coercivity
  condition~\eqref{eq:10} of $\phi$, one has
  \begin{equation}
    \label{eq:248}
    \phi(s)\ge \tfrac{C}{m}\abs{s}^{m-1}s\qquad\text{for all $s\in \R$.}
  \end{equation}
  Combining this with \eqref{eq:10} and Sobolev
  inequality~\eqref{eq:144}, %and inequality~\eqref{eq:97} with $b=0$,
  we see that
  \begin{align*}
    [u,v]_{(q-p+1)m+1}
    &\ge \left[\frac{C}{m}\right]^{p-1}\,\frac{(q-p+1)
      p^{p}}{q^{p}}\norm{\abs{\nabla
      (\abs{u}^{\frac{qm+p}{p}-2}u)}}_{p}^{p} \\
    &\hspace{2cm} +
      a\,\left[\frac{C}{m}\right]^{p-1} \int_{\partial\Sigma}
      u_{(q-p+1)m+1}\abs{u_{m+1}}^{p-2}u_{m+1}\,\dH\\
    & = \left[\frac{C}{m}\right]^{p-1}\,\frac{(q-p+1)
      p^{p}}{q^{p}}\norm{\abs{\nabla
      (\abs{u}^{\frac{qm+p}{p}-2}u)}}_{p}^{p}\\
    &\hspace{2cm} +
      a\,\left[\frac{C}{m}\right]^{p-1} \int_{\partial\Sigma} \abs{\abs{u}^{\frac{qm+p}{p}-2}u}^{p}\,\dH\\
    &\ge \left[\frac{C}{m}\right]^{p-1}\,\min\Big\{\tfrac{(q-p+1)
      p^{p}}{q^{p}},a\Big\}\,\tilde{C}^{-p}\norm{\abs{u}^{\frac{qm}{p}}}_{\frac{pd}{d-p}}^{p}\\
    &=  \left[\frac{C}{m}\right]^{p-1}\,
      \min\Big\{\tfrac{(q-p+1) p^{p}}{q^{p}},a\Big\} \tilde{C}^{-p}\,\norm{u}_{\frac{qmd}{d-p}}^{qm}
  \end{align*}
 for every $q\ge p$ and $(u,v)\in (-\Delta_{p,1}^{R})\phi$. Remark~\ref{rem:2} yields
 that the operator $A_{\phi}^{R}$ satisfies the one-parameter family of
 Gagliardo-Nirenberg type inequalities~\eqref{eq:10single} with
 $u_{0}=0$ and $\kappa = \frac{d}{d-p}>1$ and so by
Theorem~\ref{thm:GN-implies-reg-for-oper-in-L1}, the 
statement of this theorem holds for $1<p<d$.

Next, for $p=d$, then by Lemma~\ref{lem:Sobolev-Gagliardo-Nirenberg}, for every $1\le q\le
 \infty$ and every $\theta\in [0,1)$, there is a constant $C=C(d,\theta)>0$ such that
 \begin{displaymath}
   \norm{u}_{\frac{d}{1-\theta}}\le C\,\left(
   \norm{\abs{\nabla u}}_{d}^{\theta}\,\norm{u}_{d}^{1-\theta}+\norm{u}_{d}\right)
 \end{displaymath}
 for every $u\in W^{1}_{d,d}(\Sigma)$. Applying Maz'ya's
 inequality~\eqref{eq:90} and Young's
 inequality to the latter inequality and subsequently raising to the $d$th
 power  yields 
\begin{displaymath}
   \norm{u}_{\frac{d}{1-\theta}}^{d}\le C\,\left(
   \norm{\abs{\nabla u}}_{d}^{d}+\norm{u_{\vert\partial\Sigma}}_{d}\right)^{d}
 \end{displaymath}
for every $u\in W^{1}_{d,d}(\Sigma)$, where the constant $C>0$ can
differ from the previous one. By this Sobolev type inequality, we can
proceed as above and see that also for $p=d$, the statement of this
theorem holds.

Now, let $p>d$. Then by Lemma~\ref{lem:Sobolev-Gagliardo-Nirenberg}
there is a $\theta_{0}\in (0,1)$ satisfying
 \begin{displaymath}
   \theta_{0}(\tfrac{1}{p}-\tfrac{1}{d})+(1-\theta_{0})\tfrac{m}{m+1}=0
\end{displaymath}
and for every $\tilde{q}>0$, there is a constant
$C:=C(\theta_{0},p,d,\tilde{q})>0$ such that
 \begin{displaymath}
   \norm{u}_{\infty}\le C\,\left( \norm{\abs{\nabla u}}_{p}^{\theta_{0}}\,
   \norm{u}_{\frac{m+1}{m}}^{1-\theta_{0}}+\norm{u}_{\tilde{q}}\right)
 \end{displaymath}
for every $u\in W^{1}_{p,\frac{m+1}{m}}(\Sigma)\cap
L^{\tilde{q}}(\Sigma)$. Taking $\tilde{q}$ such that
$\frac{1}{\tilde{q}}=\frac{\theta_{0}}{p}+\tfrac{1-\theta_{0}}{\frac{m+1}{m}}$
yields
\begin{displaymath}
   \norm{u}_{\infty}\le C\,\left( \norm{\abs{\nabla
         u}}_{p}^{\theta_{0}}+ \norm{u}_{p}^{\theta_{0}}\right)\,
   \norm{u}_{\frac{m+1}{m}}^{1-\theta_{0}}
 \end{displaymath}
 and so by Maz'ya's
 inequality~\eqref{eq:90}, and subsequently raising to the $\frac{p}{\theta_{0}}$th power, we obtain that
\begin{displaymath}
   \norm{u}_{\infty}^{\frac{p}{\theta_{0}}}\le C\,\left( \norm{\abs{\nabla
         u}}_{p}^{p}+ \norm{u_{\vert\partial\Sigma}}_{p}^{p}\right)\,
   \norm{u}_{\frac{m+1}{m}}^{p\frac{1-\theta_{0}}{\theta_{0}}}
 \end{displaymath}
for every $u\in W^{1}_{p,\frac{m+1}{m}}(\Sigma)$, where the constant
can differ from the previous one. By using
this Gagliardo-Nirenberg type inequality together with~\eqref{eq:248}, we see that
 \begin{align*}
   [u,v]_{m+1}\,\norm{u}_{m+1}^{m\,p
     \frac{1-\theta_{0}}{\theta_{0}}} 
& \ge \left( m\,C^{p-1}\,\int_{\{u\neq 0\}}\abs{\nabla
     u}^{p}\,\abs{u}^{p(m-1)}\dx \right.\\
 & \hspace{3cm}\left. +
  \left[\frac{C}{m}\right]^{p-1}\,a\,\int_{\partial\Sigma}\abs{u}^{p m}\,\dH \right)
     \,\norm{u_{m+1}}_{\frac{m+1}{m}}^{p\frac{1-\theta_{0}}{\theta_{0}}}
   \\
& =\left[\frac{C}{m}\right]^{p-1} \left( \, \norm{\nabla u_{m+1}}_{p}^{p} +
      a\,\norm{u_{m+1\,\vert \partial\Sigma}}_{p}^{p} \right)
     \,\norm{u_{m+1}}_{\frac{m+1}{m}}^{p\frac{1-\theta_{0}}{\theta_{0}}}
   \\
  &\ge \left[\frac{C}{m}\right]^{p-1} \, \tilde{C}^{-1}\,\min\{1,a\}\,
    \norm{u}_{\infty}^{\frac{m p}{\theta_{0}}}
 \end{align*}
$(u,v)\in (-\Delta_{p,1}^{R})\phi$. Thus, the statement of this
theorem holds in the case $p>d$, completing the proof.
\end{proof}

%%%%%%%%%%%%%%%%%%%%%%%%%%%%%%%%%%%%%%%%%%%%%%%%%%%%%%%%%%%
%%%%%%%%%%%%%%%%%%%%%%%%%%%%%%%%%%%%%%%%%%%%%%%%%%%%%%%%%%%
%%%%%%%%%%%%%%%%%%%%%%%%%%%%%%%%%%%%%%%%%%%%%%%%%%%%%%%%%%%
%%%%%%%%%%%%%%%%%%%%%%%%%%%%%%%%%%%%%%%%%%%%%%%%%%%%%%%%%%%
%%%%%%%%%%%%%%%%%%%%%%%%%%%%%%%%%%%%%%%%%%%%%%%%%%%%%%%%%%%
%%%%%%%%%%%%%%%%%%%%%%%%%%%%%%%%%%%%%%%%%%%%%%%%%%%%%%%%%%%
%%%%%%%%%%%%%%%%%%%%%%%%%%%%%%%%%%%%%%%%%%%%%%%%%%%%%%%%%%%
%%%%%%%%%%%%%%%%%%%%%%%%%%%%%%%%%%%%%%%%%%%%%%%%%%%%%%%%%%%

\section{Application II: Mild solutions in $L^{1}$ are strong}
\label{subsec:Mild-is-strong}

Let $\phi\in C(\R)\cap C^{1}(\R\setminus\{0\})$ be a strictly
increasing function satisfying~\eqref{eq:10} and $\Sigma$ be an open
bounded subset of $\R^{d}$ satisfying the same assumption as in the
previous Section~\ref{sec:doubly-nonl-diff}. Then the aim of this
section is to show that \emph{mild solutions} in $L^{1}$ of the
nonlinear parabolic initial value problem~\eqref{ip:doubly-nonlinear}
equipped with one of the boundary conditions~\eqref{eq:127},
\eqref{eq:128}, \eqref{eq:129} on a bounded open set $\Sigma$ of
$\R^{d}$ are \emph{weak energy solutions} (see
Definition~\ref{Def:weak-and-strong-solutions} below) which are globally
bounded. This property implies global H\"older continuity of is mild
solutions of the parabolic problem~\eqref{ip:doubly-nonlinear}
(see~\cite{MR2865434,MR1218742,MR1156216}). Moreover, if  $\phi$ is either
given by 
\begin{equation}
  \label{eq:251}
  \phi(s)=\abs{s}^{m-1}s\qquad\text{for every $s\in \R$, and some $m>0$,}
\end{equation}
or $\phi$ is locally bi-Lipschitz continuous, then every mild solution
in $L^{1}$ of the nonlinear parabolic initial value
problem~\eqref{ip:doubly-nonlinear} is a \emph{strong energy} solution
(see Definition~\ref{Def:weak-and-strong-solutions} below).

In this section, we denote by
\begin{center}
  \emph{$V$ either the space $\dot{W}^{1}_{p,2}(\Sigma)$,
  $W^{1}_{p,2,m}(\Sigma)$ or $W^{1}_{p,2}(\Sigma)$}
\end{center}
and $L^{2}(\Sigma,\mu)$ is either the classical $L^{2}(\Sigma)$ space
equipped with the $d$-dimensional Lebesgue measure if we consider
Dirichlet or Robin boundary conditions or $L^{2}_{m}(\Sigma)$ if we
consider Neumann boundary conditions. Note that in each case the space
$V$ is embedded into the Hilbert space $L^{2}(\Sigma,\mu)$ by a continuous injection with a
dense image.

\begin{remark}
  Note that our approach given here is quite general and can easily be
  adapted to other nonlinear parabolic boundary-value problems. For
  instance, to problems involving the fractional $p$-Laplace operator as
  \begin{displaymath}
    \partial_{t}u-(-\Delta_{p})^{s}\phi(u)+\beta(u)+f(x,u)\ni 0 \qquad\text{on $\Sigma\times
      (0,\infty)$,}
  \end{displaymath}
  or to problems associated with the $p(x)$-Laplace operator as
  \begin{displaymath}
    \partial_{t}u-\textrm{div}(\abs{\nabla \phi(u)}^{p(x)-2}\nabla
    \phi(u))+\beta(u)+f(x,u)\ni 0 \qquad\text{on $\Sigma\times
      (0,\infty)$.}
  \end{displaymath}
  each equipped with some boundary conditions. Concerning the latter
  problem, we refer the interested reader to~\cite{Ha2016pLaplaceNonStandardGrowth}.
\end{remark}

In order to conclude that the milds solution of
problem~\eqref{ip:doubly-nonlinear} with initial value
$u_{0}\in L^{1}(\Sigma)$ is, in fact, a weak energy solution, we will
take advantage of the following two properties: the negative
$p$-Laplace operator $-\Delta_{p}$ equipped with one of
the above given boundary conditions~\eqref{eq:127}-\eqref{eq:129} can be realised 
\begin{enumerate}[(i)]
\item as the first derivative $\Psi' : V\to V'$ of a
  continuously differentiable functional $\Psi : V\to \R_{+}$ given by
  \begin{equation}
    \label{eq:266}
    \Psi(u)=\frac{1}{p}\displaystyle\int_{\Sigma}\abs{\nabla u}^{p}\,\dx + \tfrac{a}{p}
    \int_{\partial\Sigma}\abs{u}^{p}\dH 
  \end{equation}
  for very $u\in V$, where $a=0$ if one considers Dirichlet or Neumann
  boundary conditions, and $a>0$ if one considers purely \emph{Robin
    boundary conditions},

\item as an operator $A$ in $L^{2}(\Sigma,\mu)$ by taking the
  \emph{part} of $\Psi'$ in $L^{2}(\Sigma,\mu)$, that is,
  \begin{displaymath}
    A=\Big\{(u,v)\in V\times L^{2}(\Sigma,\mu)\,\vert \;\langle
    \Psi'(u),v\rangle_{V';V}=\langle h,v\rangle
    \text{ for all $v\in V$}\Big\}.
  \end{displaymath}
\end{enumerate}

Note, the part $A$ of $\Psi'$ in $L^{2}(\Sigma,\mu)$ coincides with the
subgradient $\partial_{\! L^{2}}\Psi^{L^{2}}$ in $L^{2}(\Sigma,\mu)$ of the convex, proper,
densely defined, and lower semicontinuous functional
$\Psi^{L^{2}} : L^{2}(\Sigma,\mu)\to \R\cup\{+\infty\}$ given by
 \begin{displaymath}
  \Psi^{L^{2}}(u)=
  \begin{cases}
    \Psi(u)& \text{if $u\in
      V$,}\\
    +\infty & \text{if otherwise}
  \end{cases}
\end{displaymath}
for every $u\in L^{2}(\Sigma,\mu)$. This is well-known, but if
the reader is interested in a more thorough explanation, then we
refer him to~\cite{arXiv:1412.4151}.

% In order to keep the definition of weak energy solutions more general,
% we formulated it for general Banach spaces $V$ which are continuously
% embedded into the Hilbert space $L^{2}(\Sigma,\mu)$ of a
% $\sigma$-finite measure space $(\Sigma,\mu)$. 

One easily verifies that the functional $\Psi$ defined
in~\eqref{eq:266} satisfies the
hypotheses~($\mathcal{H}$\ref{hyp:a})-($\mathcal{H}$\ref{hyp:e}). Moreover,
in this framework, the notion of \emph{weak energy solutions} given in
Definition~\ref{Def:weak-sols-in-V} concerning solutions of
problem~\eqref{ip:doubly-nonlinear} equipped with one
of the boundary condition~\eqref{eq:127}-\eqref{eq:129} makes
sense, we also in this section we use the function
\begin{displaymath}
  \Phi(s):=\int_{0}^{s}\phi(r)\,\dr\qquad\text{for every $s\in \R$.}
\end{displaymath}
We still need to clarify the notion of \emph{strong solutions} of such problems.

\begin{definition}
  \label{Def:weak-and-strong-solutions}
  For given $u_{0}\in L^{1}(\Sigma)$, we a function
  $u\in C([0,\infty);L^{1}(\Sigma))$ a \emph{strong energy solution in
    $L^{1}$ of problem~\eqref{ip:doubly-nonlinear}} if $u$ is a weak
  energy solution of problem~\eqref{ip:doubly-nonlinear} in the sense of
  Definition~\ref{Def:weak-sols-in-V} and for every $T>0$, one has
    \begin{displaymath}
      u\in W^{1,1}((0,T];L^{1}(\Sigma)).
    \end{displaymath}
\end{definition}

% We emphasise that in our next theorem, the set $\Sigma$ needs not to have finite Lebesgue
% measure if one considers problem~\eqref{ip:doubly-nonlinear} equipped
% with homogeneous Dirichlet boundary conditions~\eqref{eq:127}.

The following theorem is the main result of this section, where we
take the measure $\dmu=\dx$ the $d$-dimensional Lebesgue-measure.

\begin{theorem}
  \label{thm:weak-is-strong}
  Let $1<p<\infty$ with the restriction that
  \begin{displaymath}
    \frac{d(1+\frac{1}{m})}{dm+1}<p\qquad\text{if $1<p<d$,}
  \end{displaymath}
  and $\phi\in C(\R)\cap C^{1}(\R\setminus\{0\})$ be a strictly
  increasing function satisfying~\eqref{eq:10} for some $m>0$ and $\Sigma$ be an open
  bounded subset of $\R^{d}$ satisfying the same assumption as in the
  previous Section~\ref{sec:doubly-nonl-diff}. Further, let
  $\{T_{t}\}_{t\ge 0}$ be the semigroup either generated by 
  $-\overline{((-\Delta_{p,1}^{D})\phi)_{1}}+F$ on $L^{1}(\Sigma)$,
  $-\overline{((-\Delta_{p,1}^{N})\phi)_{1}}+F$ on $L^{1}_{m}(\Sigma)$
  or by $-\overline{((-\Delta_{p,1}^{R})\phi)_{1}}+F$ on
  $L^{1}(\Sigma)$. Then, for every $u_{0}\in L^{1}(\Sigma)$ (respectively, for every
  $u_{0}\in L^{1}_{m}(\Sigma)$), the following statements hold.
  \begin{enumerate}[(1)]
    \item The mild solution $u(t):=T_{t}u_{0}$, $t\ge 0$ of
      problem~\eqref{ip:doubly-nonlinear} equipped with either homogeneous
      Dirichlet boundary conditions~\eqref{eq:127}, homogeneous Neumann
      boundary conditions~\eqref{eq:128}, or homogeneous Robin boundary
      conditions~\eqref{eq:129} is a weak energy solution of
      \eqref{ip:doubly-nonlinear} satisfying energy inequality~\eqref{eq:260}.
  
   \item If, in addition, $\phi$ satisfies one of the following
     conditions
     \begin{enumerate}[(i)]
       \item $\phi$ is homogeneous of degree $\alpha>0$, $\alpha\neq 1$,
       that is, $\phi(\lambda s)=\lambda^{\alpha}\phi(s)$ for every
       $s\in \R$ and $\lambda> 0$,

       \item $\phi$ and $\phi^{-1}$ are locally Lipschitz continuous
         on $\R$,
     \end{enumerate}
     then the mild solution $u(t):=T_{t}u_{0}$, $t\ge 0$, is a strong energy solution.
  \end{enumerate}
\end{theorem}

For the proof of Theorem~\ref{thm:weak-is-strong}, the main ingredients are the
$L^{1}$-$L^{\infty}$-regularisation estimates established in Section~\ref{sec:doubly-nonl-diff}.

\begin{proof}[Proof of Theorem~\ref{thm:weak-is-strong}]
  The first statement of this theorem follows immediately from
  Theorem~\ref{thm:weak-solutions-Li-Linfty} due to the global
  $L^{1}$-$L^{\infty}$ regularisation estimates holding uniformly for
  all $t>0$ given by Theorem~\ref{thm:Lq-Lr-reg-of-doubly-nonlinear-Dirichlet} concerning
  Dirichlet boundary conditions,
  Theorem~\ref{thm:Lq-Lr-reg-of-doubly-nonlinear-Neumann} concerning
  Neumann boundary conditions, and
  Theorem~\ref{thm:Lq-Lr-reg-of-doubly-nonlinear-Robin} regarding
  Robin boundary conditions. Here, we chose in the case $p=d$, the
  parameter $\theta$ appearing in
  Theorem~\ref{thm:Lq-Lr-reg-of-doubly-nonlinear-Dirichlet} such that
  \begin{displaymath}
    \max\Big\{0,\tfrac{1+m(1-p)}{m+1}\Big\}<\theta<1.
  \end{displaymath}

  The second statement follows from \cite[Theorem~7]{MR648452} if
  $\phi$ is homogeneous of order $\alpha>0$, $\alpha\neq 1$, and from
  Theorem~\ref{thm:Linfty-implies-mild-are-strong-in-L1} if $\phi$ and $\phi^{-1}$ are locally Lipschitz
  continuous on $\R$. Here, we note that if one wants to conclude from
  Lipschitz continuity of the mild solution $u$ with values in
  $L^{1}(\Sigma)$ that the function
  $u\in W^{1,1}((0,T];L^{1}(\Sigma))$, one needs to apply a classical
  result from measure theory (cf. \cite[Lemma A.1]{MR2286292}, wherein
  the continuity assumption of $u$ can be omitted due to the chain
  rule given by Ambrosio and Dal Maso~\cite{MR969514}).
\end{proof}

%%%%%%%%%%%%%%%%%%%%%%%%%%%%%%%%%%%%%%%%%%%%%%%%%%%
%%%%%%%%%%%%%%%%%%%%%%%%%%%%%%%%%%%%%%%%%%%%%%%%%%%
%%%%%%%%%%%%%%%%%%%%%%%%%%%%%%%%%%%%%%%%%%%%%%%%%%%
%%%%%%%%%%%%%%%%%%%%%%%%%%%%%%%%%%%%%%%%%%%%%%%%%%%
%%%%%%%%%%%%%%%%%%%%%%%%%%%%%%%%%%%%%%%%%%%%%%%%%%%
%%%%%%%%%%%%%%%%%%%%%%%%%%%%%%%%%%%%%%%%%%%%%%%%%%%
%%%%%%%%%%%%%%%%%%%%%%%%%%%%%%%%%%%%%%%%%%%%%%%%%%%
%
%
%                                   A P P E N D I X
%
%
%%%%%%%%%%%%%%%%%%%%%%%%%%%%%%%%%%%%%%%%%%%%%%%%%%%
%%%%%%%%%%%%%%%%%%%%%%%%%%%%%%%%%%%%%%%%%%%%%%%%%%%
%%%%%%%%%%%%%%%%%%%%%%%%%%%%%%%%%%%%%%%%%%%%%%%%%%%
%%%%%%%%%%%%%%%%%%%%%%%%%%%%%%%%%%%%%%%%%%%%%%%%%%%
%%%%%%%%%%%%%%%%%%%%%%%%%%%%%%%%%%%%%%%%%%%%%%%%%%%
%%%%%%%%%%%%%%%%%%%%%%%%%%%%%%%%%%%%%%%%%%%%%%%%%%%

%\section*{Appendix}
\appendix
\renewcommand{\thesection}{\Alph{section}}
\setcounter{section}{0}

\section{More on accretive operators in $L^{1}$}
\label{Asec:accretive-L1}

We begin this section by outlining the proof of
Proposition~\ref{propo:characterisation-c-complete}.

\begin{proof}[Proof of Proposition~\ref{propo:characterisation-c-complete}]
  Let $u$, $v\in L^{1}(\Sigma,\mu)$ and suppose~\eqref{eq:118} holds
  for all $j\in \mathcal{J}$ and $\lambda>0$. For every $T\in P$, one has
  either $T<0$ on $\R$ or $T\in P_{0}$ or $T>0$ on $\R$. If $T\in
  P_{0}$, then inequality~\eqref{eq:char-c-complete} follows from
  Proposition~\ref{prop:completely-accretive}. If $T>0$ on $\R$, then
  the function $j(s):=\int_{0}^{s}T(r)\,\dr$ for every $s\in \R$
  belongs to $\mathcal{J}$. Since the support of the derivative $T'$
  of $T$ is a compact subset of $\R$, the function $T$ is bounded on
  $\R$. Thus, there is a constant $M\ge0$ such that
  \begin{displaymath}
    \abs{j(u+\lambda v)}\le \abs{j(0)}+\int_{0}^{1}\abs{T((u+\lambda
      v)s)}\,\ds\abs{u+\lambda v}\le \abs{j(0)}+M\,\abs{u+\lambda v}
  \end{displaymath}
  for a.e. $x\in \Sigma$, showing that $j(u+\lambda v)\in L^{1}(\Sigma,\mu)$.
  By~\eqref{eq:118} and since $j(u)\ge0$, we have that $j(u)\in
  L^{1}(\Sigma,\mu)$ satisfies
  \begin{displaymath}
    0\le \int_{\Sigma}\left(j(u+\lambda v)-j(u)\right)\,\dmu
  \end{displaymath}
  for all $\lambda>0$. By convexity of $j$ and since $j\in C^{1}(\R)$,
  one has that
  \begin{math}
    \frac{j(u+\lambda v)-j(u)}{\lambda}
  \end{math}
  decreases to $T(u)\,v$ a.e. on $\Sigma$. Thus and since
  $T(u)\,v\in L^{1}(\Sigma,\mu)$, it follows
  that~\eqref{eq:char-c-complete} holds for $T>0$. In the case that
  $T<0$, we first truncate $u$ and $v$ at hight $n$. More precisely,
  for every $n>1$, let $u_{n}=u$ if $\abs{u}\le n$ and $u_{n}=0$ if
  otherwise and analogously, define $v_{n}$. Further, define
  $j_{n}\in \mathcal{J}$ by
  \begin{displaymath}
    j_{n}(s)=
    \begin{cases}
      \int_{-n}^{n}T(r)\,\dr-2\,n\,T(-n) & \text{if $s\ge n$,}\\
      \int_{-n}^{s}T(r)\,\dr-2\,n\,T(-n) & \text{if $\abs{s}\le n$,}\\
      -2\,n\,T(-n) & \text{if $s\le -n$}
    \end{cases}
  \end{displaymath}
  for every $s\in \R$. Now, proceeding as above yields
  \begin{displaymath}
    \int_{\Sigma}T(u_{n})\,v_{n}\,\dmu\ge 0
  \end{displaymath}
  for every $n>1$. By dominated convergence, $T(u_{n})\,v_{n}$
  converges to $T(u)\,v$ in $L^{1}(\Sigma,\mu)$ hence we can
  conclude that~\eqref{eq:char-c-complete} holds as well for $T>0$.

  It remains to show that the other inclusion holds as well. To see
  this, let $u$, $v\in L^{1}(\Sigma,\mu)$
  satisfy~\eqref{eq:char-c-complete} for every $T\in P$. For given
  $j\in \mathcal{J}$, let $j_{\nu}(s):=\inf_{r\in \R}\{j(r)+\nu\abs{s-r}\}$ for every $s\in
  \R$ and $\nu\ge 0$. Then the sequence $(j_{\nu})_{\nu\ge0}$ consists
  of Lipschitz continuous, convex functions $j_{\nu}\in \mathcal{J}$ such that for every
  $h\in L^{1}(\Sigma,\mu)$, $j_{\nu}(h)$ converges monotone increasingly to $j(h)$ a.e. on
  $\Sigma$ and $\int_{\Sigma}j_{\nu}(h)\dmu\uparrow \int_{\Sigma}j(h)\dmu$ as
  $\nu\to\infty$. Next, for every $n\ge 1$, let $j_{\nu,n}\in \mathcal{J}$ be given
  by
  \begin{displaymath}
    j_{\nu, n}(r)=
    \begin{cases}
      j'_{\nu}(n)(r-n)+j_{\nu}(n) & \text{if $r\ge n$,}\\
      j_{\nu}(r) & \text{if $\abs{s}\le n$,}\\
      j'_{\nu}(-n)(r+n)+j_{\nu}(-n) & \text{if $s\le -n$}
    \end{cases}
  \end{displaymath}
  for every $r\in \R$. By construction, the a.e. derivative
  $j'_{\nu,n}$ is positive and bounded by the same Lipschitz constant
  $L_{\nu}>0$ of $j_{\nu}$. Thus, for every $h\in
  L^{1}(\Sigma,\mu)$,
  \begin{equation}
    \label{eq:119}
    \abs{j_{\nu,n}(h)}\le \abs{j_{\nu,n}(0)}+
    \int_{0}^{1}\abs{j'_{\nu,n}(h s)}\abs{h}\,\ds\le \abs{j_{\nu}(0)}+L_{\nu}\abs{h}
  \end{equation}
   a.e. on $\Sigma$ hence $j_{\nu,n}(h)\in L^{1}(\Sigma,\mu)$. Since, the function $j_{\nu,n}$
  is convex,
  \begin{displaymath}
    \inf_{\lambda>0}\frac{j_{\nu,n}(u+\lambda v)-j_{\nu,n}(u)}{\lambda}=j'_{\nu,n}(u)v
  \end{displaymath}
  for a.e. $x\in \Sigma$ and by the Lipschitz continuity of $j_{\nu,n}$,
  \begin{displaymath}
    \labs{\frac{j_{\nu,n}(u+\lambda v)-j_{\nu,n}(u)}{\lambda}}\le L_{\nu}\abs{v}. 
  \end{displaymath}
  Therefore, by the dominated convergence theorem and since
  $L_{\nu}^{-1}j'_{\nu,n}\in P$, it follows
  by~\eqref{eq:char-c-complete} that
  \begin{align*}
    \int_{\Sigma} \frac{L_{\nu}^{-1}j_{\nu,n}(u+\lambda
    v)-L_{\nu}^{-1}j_{\nu,n}(u)}{\lambda}\,\dmu&\ge 
    \int_{\Sigma}\inf_{\lambda>0}\frac{L_{\nu}^{-1}j_{\nu,n}(u+\lambda
                                     v)-L_{\nu}^{-1}j_{\nu,n}(u)}{\lambda}\,\dmu\\
    &= \int_{\Sigma}L_{\nu}^{-1} j'_{\nu,n}(u)v\,\dmu\ge0,
  \end{align*}
  from where one can conclude that~\eqref{eq:118} holds for
  $j_{\nu,n}$. Since for every $h\in L^{1}(\Sigma,\mu)$,
  $j_{\nu,n}(h)$ converges to $j_{\nu}(h)$ a.e. on $\Sigma$ and since
  the right hand side in~\eqref{eq:119} does not depend on $n$, it
  follows that ~\eqref{eq:118} holds for $j_{\nu}$. By the properties
  of the sequence $(j_{\nu})_{\nu\ge0}$, one easily concludes that
  inequality~\eqref{eq:118} holds for $j$. This completes the proof of
  this proposition.
\end{proof}

Next, we outline the proof of Proposition~\ref{propo:composition-operators-in-L1}.

\begin{proof}[Proof of Proposition~\ref{propo:composition-operators-in-L1}]
  We begin by showing that $A\phi$ is accretive in $L^{1}(\Sigma,\mu)$. To
  do so, let $(u,v)$, $(\hat{u},\hat{v})\in A\phi$ and
   $(u,w)$, $(\hat{u},\hat{w})\in \phi$. First, we assume that
  hypothesis~\eqref{propo:composition-operators-in-L1-H1} holds. Then, 
  \begin{equation}
    \label{eq:12}
    \int_{\Sigma}\psi\,(v-\hat{v})\,\dmu\ge 0
  \end{equation}
  for every $\psi\in L^{\infty}(\Sigma,\mu)$ satisfying
  $\psi(x)\in \textrm{sign}(w(x)-\hat{w}(x))$ for a.e. $x\in \Sigma$
  and since by assumption, $A$ is single-valued, the situation
  $w=\hat{w}$ implies that~ \eqref{eq:12} holds only for
  $\psi\equiv 0$. Consider, the function
  $\psi\in L^{\infty}(\Sigma,\mu)$ defined by
  \begin{displaymath}
   \psi(x):=
   \begin{cases}
     1 & \text{if $u(x)>\hat{u}(x)$,}\\
     \textrm{sign}_{0}(w(x)-\hat{w}(x)) & \text{if
       $u(x)=\hat{u}(x)$,}\\
     -1 & \text{if $u(x)<\hat{u}(x)$,}
   \end{cases}
  \end{displaymath}
  for a.e. $x\in \Sigma$. Then by construction,
  \begin{displaymath}
     \psi\in\;\textrm{sign}(w(x)-\hat{w}(x))\cap
   \textrm{sign}(u(x)-\hat{u}(x)).
 \end{displaymath}
 In particular, $\psi$ satisfies~\eqref{eq:12} hence $A\phi$ is
 accretive in $L^{1}(\Sigma,\mu)$. If we assume that
 hypothesis~\eqref{propo:composition-operators-in-L1-H2} holds, then
 by definition of $A\phi$ and since $\phi$ is a function, one has that
 $v\in A\phi(u)$ and $\hat{v}\in A\phi(\hat{u})$. Thus and since $A$
 is accretive in $L^{1}(\Sigma,\mu)$,
  \begin{align*}
    &[\phi(u)-\phi(\hat{u}),v-\hat{v}]_{1}\\
    &\quad = \int_{\{\phi(u)\neq\phi(\hat{u})\}}
      \textrm{sign}_{0}(\phi(u)-\phi(\hat{u}))\,(v-\hat{v})\,\dmu +
    \int_{\{\phi(u)=\phi(\hat{u})\}}\abs{v-\hat{v}}\,\dmu\ge 0.
  \end{align*}
  Since $\phi$ is injective, one has that
  $\{\phi(u)=\phi(\hat{u})\}= \{u=\hat{u}\}$. Therefore,
  \begin{align*}
    &\int_{\{u\neq\hat{u}\}}\textrm{sign}_{0}(u-\hat{u})\,(v-\hat{v})\,\dmu +
    \int_{\{u=\hat{u}\}}\abs{v-\hat{v}}\,\dmu\\
    &\qquad = \int_{\{\phi(u)\neq\phi(\hat{u})\}}\textrm{sign}_{0}(\phi(u)-\phi(\hat{u}))\,(v-\hat{v})\,\dmu +
    \int_{\{\phi(u)=\phi(\hat{u})\}}\abs{v-\hat{v}}\,\dmu\ge 0,
  \end{align*}
  showing that $A\phi$ is accretive in $L^{1}(\Sigma,\mu)$. 

  Moreover, for every $\varepsilon>0$, the sum
  $\varepsilon\phi_{1}+A\phi$ is accretive in $L^{1}(\Sigma,\mu)$
  under the assumption that either~\eqref{propo:composition-operators-in-L1-H1}
  or~\eqref{propo:composition-operators-in-L1-H2}
  holds. This follows easily from the fact that the operator
  $\phi_{1}$ in $L^{1}(\Sigma,\mu)$ of the monotone function $\phi$ on
  $\R$ is $s$-accretive in $L^{1}(\Sigma,\mu)$
  (cf.~\cite{Benilanbook}).

  Similarly,
  one shows under the assumptions $\phi$ is injective and $A$ is $T$-accretive in
  $L^{1}(\Sigma,\mu)$ that for every $\varepsilon\ge 0$, one has
  $\varepsilon\phi_{1} + A\phi$ is $T$-accretive in
  $L^{1}(\Sigma,\mu)$ (cf.~\cite[Proposition~2.5]{Benthesis}).

  Next, suppose that $A$ has a complete resolvent and $\phi$ is
  continuous satisfying $\phi(0)=0$. Then, for every $T\in P_{0}$,
  $T\circ \phi^{-1}$ is continuous on $Rg (\phi)=]a,b[$ for some $a$,
  $b\in \R$, bounded and $T\circ \phi^{-1}(0)=0$. If $(\rho_{n})$ is a
  standard positive mollifier sequence on $\R$, then
  $T_{n}:=(T\circ \phi^{-1})\ast \rho_{n}\in P_{0}$ and
  $T_{n}\to T\circ \phi^{-1}$ uniformly on compact subsets of $\R$ as
  $n\to\infty$. Thus, for every $u$, $v\in L^{1}(\Sigma,\mu)$,
  $T_{n}(\phi(u))v\to T(u) v$ a.e. on $\Sigma$ as $n\to \infty$ and
  since $(T_{n}(\phi(u)))$ is uniformly bounded in
  $L^{\infty}(\Sigma,\mu)$, it follows that $\lim_{n\to\infty}
  T_{n}(\phi(u))v =T(u) v$ in $L^{1}(\Sigma,\mu)$. For every
  $(u,v)\in A\phi$, one has $(\phi(u),v)\in A$ hence by
  Proposition~\ref{prop:completely-accretive},
  \begin{displaymath}
    \int_{\Sigma}T(\phi(u))v\,\dmu\ge 0
  \end{displaymath}
  for every $T\in P_{0}$. Hence, for every $T\in P_{0}$, replacing $T$
  by $T_{n}$ in the latter inequality and sending $n\to\infty$ yields
  \begin{equation}
    \label{eq:22}
    \int_{\Sigma}T(u)v\,\dmu\ge 0,
  \end{equation}
  showing that $A\phi$ has a complete resolvent. If $(\Sigma,\mu)$ is
  finite and $A$ has a $c$-complete resolvent, then
  similar arguments and  replacing
  Proposition~\ref{prop:completely-accretive} by
  Proposition~\ref{propo:charact-of-c-complete-operators} yields that
  $A\phi$ has a $c$-complete resolvent. Now, for every
  $\varepsilon>0$, recall that $\varepsilon \phi_{1}$ is completely
  accretive in $L^{1}(\Sigma,\mu)$. Thus, if $\phi(0)=0$, then
  $\phi_{1}$ has a complete resolvent and so
  \begin{displaymath}
    \int_{\Sigma}T(u)\varepsilon\phi(u)\,\dmu\ge 0
  \end{displaymath}
  for every $T\in P_{0}$. For any $T\in P_{0}$, adding this inequality
  to~\eqref{eq:22} for $u\in D(\phi_{1})\cap D(A\phi)$ and
  $v\in A\phi(u)$ shows that for every $\varepsilon>0$,
  $\varepsilon \phi_{1}+A\phi$ has a complete resolvent by
  Proposition~\ref{prop:completely-accretive}. Again, the same
  arguments and using
  Proposition~\ref{propo:charact-of-c-complete-operators} yields that for
  every $\varepsilon>0$, $\varepsilon \phi_{1}+A\phi$ has a
  $c$-complete resolvent.
\end{proof}

The statements of Proposition~\ref{propo:composition-operators-in-L1}
are used in following proof.

\begin{proof}[Proof of Proposition~\ref{propo:Range-cond-in-Rd}]
  Here, we have been inspired by the proof of
  ~\cite[Proposition~2]{MR647071}. Let $A_{\phi}$ denote the
  operator on $L^{1}(\Sigma,\mu)$ given by
  \begin{displaymath}
    A_{\phi}=\Bigg\{(u,f)\in L^{1}\times L^{1}(\Sigma,\mu) \Bigg\vert
    \begin{array}[c]{l}
      \text{there are }\lambda>0,\;g\in L^{1}\cap L^{\infty}(\Sigma,\mu)\textrm{ such that }\\
      \displaystyle\lim_{\varepsilon\to 0+}J_{\lambda}^{\varepsilon\phi+A_{1\cap\infty}\phi}g= u\text{
      in $L^{1}(\Sigma,\mu)$ and }f=\frac{g-u}{\lambda}
    \end{array}\!\!\!
    \Bigg\},
  \end{displaymath}
  where for every $\lambda>0$ and every $\varepsilon>0$, the operator
  $J_{\lambda}^{\varepsilon\phi+A_{1\cap\infty}\phi}$ denotes the
  resolvent of $\varepsilon\phi+A_{1\cap\infty}\phi$.

  We begin by showing that under the
  hypotheses~\eqref{propo:Range-cond-in-Rd-Hyp-3}-\eqref{propo:Range-cond-in-Rd-Hyp-2},
  for every $\varepsilon>0$ sufficiently small, $\lambda>0$ and every
  $g\in L^{1}\cap L^{\infty}(\Sigma,\mu)$, there is a unique
  $u_{\varepsilon}\in D(A_{1\cap\infty}\phi)$
  satisfying
  \begin{equation}
    \label{eq:228}
    u_{\varepsilon}+\lambda (\varepsilon \phi(u_{\varepsilon})+
    A_{1\cap\infty}\phi(u_{\varepsilon}))\ni g
  \end{equation}
  or equivalently,
  $u_{\varepsilon}=J_{\lambda}^{\varepsilon\phi+A_{1\cap\infty}\phi}g$,
  and there is an $u\in L^{1}\cap L^{\infty}(\Sigma,\mu)$ such that
  \begin{equation}
    \label{eq:224}
    \lim_{\varepsilon\to0+}u_{\varepsilon}=u\qquad\text{ in
      $L^{1}(\Sigma,\mu)$}
  \end{equation}
  and 
  \begin{equation}
  \label{eq:231}
  \lim_{\varepsilon\to0+}
  \varepsilon\,\varphi(u_{\varepsilon})=0\qquad\text{in
    $L^{\tilde{q}}(\Sigma,\mu)$, for every $1\le \tilde{q}\le \infty$.}
\end{equation}
  By Proposition~\ref{propo:Lipschitz-complete-accretive},
  the operator $\phi^{-1}_{q}+\lambda A$ is $m$-completely accretive
  in $L^{q}(\Sigma,\mu)$. Thus, and since $(0,0)\in \phi^{-1}_{q}+\lambda A$,
  for every $\varepsilon>0$, there are
  $v_{\varepsilon}\in L^{1}\cap L^{\infty}(\Sigma,\mu)\cap
  D(\phi^{-1}_{q})\cap D(A)$
  and $w_{\varepsilon}\in Av_{\varepsilon}$ satisfying
  \begin{equation}
    \label{eq:218}
    v_{\varepsilon}+\tfrac{1}{\varepsilon\lambda
    }(\phi^{-1}(v_{\varepsilon})+\lambda w_{\varepsilon})=
    \tfrac{1}{\varepsilon\lambda}g.
  \end{equation}
  In fact (cf. the proof of~\cite[Proposition~3.8]{MR2582280}), the
  solution $v_{\varepsilon}$ of \eqref{eq:218} is the limit
  \begin{displaymath}
    \lim_{\nu\to0+}v_{\varepsilon,\nu}=v_{\varepsilon}\qquad\text{in $L^{q}(\Sigma,\mu)$}
  \end{displaymath}
  of the sequence $(v_{\varepsilon,\nu})_{\nu>0}$ of solutions
  $v_{\varepsilon,\nu}\in L^{1}\cap L^{\infty}(\Sigma,\mu)\cap D(A)$ of
  \begin{equation}
    \label{eq:223}
    v_{\varepsilon,\nu}+\tfrac{1}{\varepsilon\lambda
    }(\beta_{\nu}(v_{\varepsilon,\nu})+\lambda w_{\varepsilon,\nu}=
    \tfrac{1}{\varepsilon\lambda}g
  \end{equation}
  with $w_{\varepsilon,\nu}\in Av_{\varepsilon,\nu}$. Moreover, one has
  \begin{equation}
    \label{eq:222}
    \lim_{\nu\to0+}\beta_{\nu}(v_{\varepsilon,\nu})=\beta(v_{\varepsilon})
    \qquad\text{weakly in $L^{q}(\Sigma,\mu)$,}
  \end{equation}
  where $\beta_{\nu}$ denotes the Yosida operator of
  $\beta:=\phi^{-1}$. We note that for every $\nu>0$,
  $v_{\varepsilon,\nu}\in D(A_{1\cap \infty})$ owing to the Lipschitz continuity
  of $\beta_{\nu}$ and since $\beta_{\nu}(0)=0$. First, multiplying
  equation~\eqref{eq:223} with $\beta_{\nu}(v_{\varepsilon,\nu})$ with
  respect to the $1$-bracket $[\cdot,\cdot]_{1}$,  then  using that
  $\beta_{\nu}$ is accretive in $L^{1}(\Sigma,\mu)$ and that
  $\beta_{\nu}$ satisfies~\eqref{eq:179} for $q=1$, we see that
  \begin{align*}
    \norm{\beta_{\nu}(v_{\varepsilon,\nu})}_{1}
    &\le [\beta_{\nu}(v_{\varepsilon,\nu}),v_{\varepsilon,\nu}]_{1}
      +[\beta_{\nu}(v_{\varepsilon,\nu}),\beta_{\nu}(v_{\varepsilon,\nu})]_{1}
    +[\beta_{\nu}(v_{\varepsilon,\nu}),w_{\varepsilon,\nu}]_{1}\\
    &=\tfrac{1}{\varepsilon\lambda}[\beta_{\nu}(v_{\varepsilon,\nu}),g]_{1}\\
    &\le \tfrac{1}{\varepsilon\lambda}\,\norm{g}_{1}
  \end{align*}
  for all $\nu>0$. By this estimate together with~\eqref{eq:222} and
  H\"older's inequality yields that there is a constant $C>0$ such that
  \begin{displaymath}
    \norm{\beta_{\nu}(v_{\varepsilon,\nu})}_{p}\le C\qquad\text{for
      all $\nu>0$ and $1<p<q$}
  \end{displaymath}
  hence, the weak limit $\beta(v_{\varepsilon})$ satisfies
  \begin{displaymath}
   \norm{\beta(v_{\varepsilon})}_{p}\le C\qquad\text{for
     all $1<p<q$.}
  \end{displaymath}
  Sending $p\to 1+$ in the latter inequality and using Fatou's lemma,
  we obtain that
  $\beta(v_{\varepsilon})=\phi^{-1}(v_{\varepsilon})\in
  L^{1}(\Sigma,\mu)$
  and so by continuity of $\phi^{-1}$ on $\R$,
  $\phi^{-1}(v_{\varepsilon})\in L^{1}\cap
  L^{\infty}(\Sigma,\mu)$.
  Thus, equation~\eqref{eq:218} yields $v_{\varepsilon}\in D(A_{1\cap \infty})$
  with $w_{\varepsilon}\in L^{1}\cap L^{\infty}(\Sigma,\mu)$, hence,
  for every $\varepsilon>0$, there is
  $v_{\varepsilon}\in L^{1}\cap L^{\infty}(\Sigma,\mu)\cap
  D(\phi^{-1}_{q})\cap D(A_{1\cap \infty})$
  such that $\phi^{-1}(v_{\varepsilon})\in L^{1}\cap L^{\infty}(\Sigma,\mu)$ and
  \begin{displaymath}
    v_{\varepsilon}+\tfrac{1}{\varepsilon\lambda
    }(\phi^{-1}(v_{\varepsilon})+\lambda A_{1\cap\infty}v_{\varepsilon})\ni
    \tfrac{1}{\varepsilon\lambda}g.
  \end{displaymath}
  Taking $u_{\varepsilon}=\phi^{-1}(v_{\varepsilon})$, one has
  $\phi(u_{\varepsilon})=v_{\varepsilon}$.  Thus and by the last
  inclusion, we have shown that for every $\varepsilon>0$, there is an
  $u_{\varepsilon}\in L^{1}\cap L^{\infty}(\Sigma,\mu)$ such that
  $\phi(u_{\varepsilon})\in D(A_{1\cap \infty})$ and \eqref{eq:228}
  holds, or, equivalently,
  $u_{\varepsilon}=J_{\lambda}^{\varepsilon\phi+A_{1\cap\infty}\phi}g$. We
  still need to show that the ~\eqref{eq:224} and~\eqref{eq:231}
  hold. 

  We begin, by assuming that
  hypothesis~\eqref{propo:Range-cond-in-Rd-Hyp-3} holds. Then, by
  Proposition~\ref{propo:Range-cond-in-Rd}, for every
  $\varepsilon> 0$, $\varepsilon \phi+A_{1\cap\infty}\phi$ is
  $T$-accretive in $L^{1}(\Sigma,\mu)$. Thus, for every
  $\tilde{g}\in L^{1}\cap L^{\infty}(\Sigma,\mu)$ satisfying
  $\tilde{g}\ge 0$ and $\varepsilon>0$, one has that
  $\tilde{u}_{\varepsilon}:=J_{\lambda}^{\varepsilon\phi+A_{1\cap\infty}\phi}\tilde{g}$
  satisfies $\tilde{u}_{\varepsilon}\ge 0$,
  $\norm{\tilde{u}_{\varepsilon}}_{\tilde{q}}\le
  \norm{\tilde{g}}_{\tilde{q}}$ for $1\le\tilde{q}\le \infty$
  hence, by the assumptions on $A\phi$ and $\phi$,
  \begin{equation}
    \label{eq:171}
    \begin{split}
      \norm{\tilde{u}_{\varepsilon}+\lambda\varepsilon
        \phi(\tilde{u}_{\varepsilon})}_{1} & =
      \norm{\tilde{u}_{\varepsilon}}_{1}+\lambda\varepsilon
      \norm{\phi(\tilde{u}_{\varepsilon})}_{1} \\
      & \le \norm{\tilde{g}-\lambda\varepsilon
        \phi(\tilde{u}_{\varepsilon})}_{1}+\lambda\varepsilon
      \norm{\phi(\tilde{u}_{\varepsilon})}_{1} = \norm{\tilde{g}}_{1}
 \end{split}
\end{equation}
for every sufficiently small $\varepsilon>0$. Moreover, if
$\tilde{w}_{\varepsilon}\in
A_{1\cap\infty}\phi(\tilde{u}_{\varepsilon})$
satisfies
$\tilde{u}_{\varepsilon}+\lambda
(\eta\phi(\tilde{u}_{\varepsilon})+\tilde{w}_{\varepsilon}))=\tilde{g}$,
then for every $\varepsilon>\eta>0$,
\begin{displaymath}
  \tilde{u}_{\varepsilon}+\lambda
  (\eta\phi(\tilde{u}_{\varepsilon})+\tilde{w}_{\varepsilon}))=\tilde{g}-\lambda
  (\varepsilon-\eta) \tilde{g}\le \tilde{g} = \tilde{u}_{\eta}+\lambda
  (\eta\phi(\tilde{u}_{\eta})+\tilde{w}_{\eta}))
\end{displaymath}
and so, since the resolvent
$J_{\lambda}^{\varepsilon \phi+A_{1\cap\infty}\phi}$ of
$\varepsilon \phi+A_{1\cap\infty}\phi$ is order-preserving, one has
that $\tilde{u}_{\varepsilon}\le \tilde{u}_{\eta}$ for every
$\varepsilon>\eta>0$. Since $\tilde{u}_{\varepsilon}\ge 0$ and
$\sup_{\varepsilon>0}\norm{\tilde{u}_{\varepsilon}}_{1}\le \norm{\tilde{g}}_{1}$,
Beppo-Levi's monotone convergence theorem implies that there is
$u_{+}\in L^{1}\cap L^{\infty}(\Sigma,\mu)$ such that
$\tilde{u}_{\varepsilon}\uparrow u_{+}$ in $L^{1}(\Sigma,\mu)$ as
$\varepsilon\downarrow 0+$. Similarly, one shows that for every
$\tilde{g}\in L^{1}\cap L^{\infty}(\Sigma,\mu)$ satisfying
$\tilde{g}\le 0$, one has $\tilde{u}_{\varepsilon}\le 0$,
$\tilde{u}_{\varepsilon}\ge \tilde{u}_{\eta}$ for every
$\varepsilon>\eta>0$ and there is
$u_{-}\in L^{1}\cap L^{\infty}(\Sigma,\mu)$ such that
$\tilde{u}_{\varepsilon}\downarrow u_{-}$ in $L^{1}(\Sigma,\mu)$ as
$\varepsilon\downarrow 0+$. Now, we apply this to a general function
$g\in L^{1}\cap L^{\infty}(\Sigma,\mu)$. Let $g^{+}=g\vee 0$ be the
positive part of $g$ and $g^{-}=(-g)\vee 0$ be the negative part of
$g$.  Since by assumption, $J_{\lambda}^{A_{1\cap\infty}\phi}$ is
order-preserving, $u_{\varepsilon}$,
$u_{\varepsilon^{+}}:=J_{\lambda}^{\varepsilon\phi+A_{1\cap\infty}\phi}(g^{+})$
and
$u_{\varepsilon^{-}}:=J_{\lambda}^{\varepsilon\phi+A_{1\cap\infty}\phi}(-g^{-})$
satisfy
\begin{equation}
  \label{eq:175}
  u_{\varepsilon^{-}}\le u_{\varepsilon}\le u_{\varepsilon^{+}}\qquad
  \text{for every $\varepsilon>0$,}
\end{equation}
and there are $u_{+}$, $u_{-}\in L^{1}\cap L^{\infty}(\Sigma,\mu)$
satisfying $u_{+}\ge 0$, $u_{-}\le 0$,
$u_{\varepsilon}\uparrow u_{+}$ in $L^{1}(\Sigma,\mu)$ as
$\varepsilon\downarrow 0$ and $u_{\varepsilon}\downarrow u_{-}$
in $L^{1}(\Sigma,\mu)$ as $\varepsilon\uparrow 0+$. In particular,
\begin{displaymath}
%  \label{eq:232}
  u_{-}\le u_{\varepsilon}\le u_{+}\qquad\text{for every $\varepsilon>0$.}
\end{displaymath}
Thus, to see that ~\eqref{eq:224} holds for some function
$u\in L^{1}\cap L^{\infty}(\Sigma,\mu)$, it is enough to show that for
every sequence $(\varepsilon_{n})_{n\ge 1}\subseteq (0,1)$ with
$\varepsilon_{n}>\varepsilon_{n+1}$ and every $\delta>0$, one has
 \begin{equation}
   \label{eq:77}
   \lim_{n,m\to\infty} 
   \mu\left(\{\abs{u_{\varepsilon_{n}}-u_{\varepsilon_{m}}}>\delta\}\right)=0,
 \end{equation}
 that is, $(u_{\varepsilon})_{\varepsilon>0}$ is a Cauchy sequence $\mu$-measure.
 First, we note that by the boundedness of
 $(u_{\varepsilon_{n}})_{n\ge 1}$ and by the continuity and
 infectivity of $\phi$, for every given $\delta>0$, there is an
 $N>0$ such that
 \begin{displaymath}
   \{\abs{u_{\varepsilon_{n}}-u_{\varepsilon_{m}}}>\delta\}
   \subseteq
   \{\abs{\phi(u_{\varepsilon_{n}})-\phi(u_{\varepsilon_{m}})}>N\}\qquad
   \text{for every $n$, $m\ge 1$.}
 \end{displaymath}
 By \eqref{eq:171} and \eqref{eq:175}, every
 $\phi(u_{\varepsilon_{n}})\in L^{1}(\Sigma,\mu)$. Furthermore, by the
 continuity of $\phi$ and since $\norm{u_{\varepsilon_{n}}}_{\infty}\le
 \norm{g}_{\infty}$, we obtain that every $\phi(u_{\varepsilon_{n}})\in L^{\infty}(\Sigma,\mu)$.
 Thus, $\phi(u_{\varepsilon_{n}})-\phi(u_{\varepsilon_{m}})\in
 L^{1}\cap L^{\infty}(\Sigma,\mu)$ and so,
 $\psi(\phi(u_{\varepsilon_{n}})-\phi(u_{\varepsilon_{m}}))\in L^{q}(\Sigma,\mu)$
 for $\psi(r):=1$ if $r>N$, $\psi(r)=0$ if $\abs{r}\le N$ and
 $\psi(r)=-1$ if $r<-N$. Thus, multiplying inclusion
 \begin{displaymath}
   (u_{\varepsilon_{n}}-u_{\varepsilon_{m}})+\lambda
   (A_{1\cap\infty}\phi(u_{\varepsilon_{n}})-A_{1\cap\infty}\phi(u_{\varepsilon_{m}}))\ni
   \varepsilon_{m}\phi(u_{\varepsilon_{m}})-\varepsilon_{n}\phi(u_{\varepsilon_{n}}),
 \end{displaymath}
 by $(\psi(\phi(u_{\varepsilon_{n}})-\phi(u_{\varepsilon_{m}})))_{q}$,
 and using that $A_{1\cap\infty}$ is accretive in $L^{q}(\Sigma,\mu)$
 together with H\"older's inequality yields
 \begin{equation}
   \label{eq:145}
   \begin{split}
    & \int_{\{\abs{\phi(u_{\varepsilon_{n}})-\phi(u_{\varepsilon_{m}})}>N\}}
     \abs{u_{\varepsilon_{n}}-u_{\varepsilon_{m}}}\,\dmu\\
     &\qquad \le \lambda\;
     \int_{\{\abs{\phi(u_{\varepsilon_{n}})-\phi(u_{\varepsilon_{m}})}>N\}}
     \abs{\varepsilon_{m}\phi(u_{\varepsilon_{m}})-\varepsilon_{n}\phi(u_{\varepsilon_{n}})}\,\dmu\\
     &\qquad \le \lambda\;
     \norm{\varepsilon_{m}\phi(u_{\varepsilon_{m}})-\varepsilon_{n}\phi(u_{\varepsilon_{n}})}_{2}\;
     \mu\left(\{\abs{\phi(u_{\varepsilon_{n}})-\phi(u_{\varepsilon_{m}})}>N\}\right)^{1/2}.
   \end{split}
 \end{equation}
 By continuity of $\phi$ and boundedness of $(u_{\varepsilon_{n}})_{n\ge 1}$, there is an
 $M>0$ such that
 \begin{displaymath}
   \{\abs{\phi(u_{\varepsilon_{n}})-\phi(u_{\varepsilon_{m}})}>N\}
   \subseteq \{\abs{u_{\varepsilon_{n}}-u_{\varepsilon_{m}}}>M\}\qquad\text{}
 \end{displaymath}
and so,~\eqref{eq:145} gives
\begin{equation}
  \label{eq:7}
  M\,\mu\left(\{\abs{\phi(u_{\varepsilon_{n}})-\phi(u_{\varepsilon_{m}})}>N\}\right)^{1/2}\\
  \le \lambda\;
    \norm{\varepsilon_{m}\phi(u_{\varepsilon_{m}})-\varepsilon_{n}\phi(u_{\varepsilon_{n}})}_{2}
\end{equation}
for every $n$, $m\ge 1$. By \eqref{eq:171} and \eqref{eq:175}, one has
$(\varepsilon\phi(u_{\varepsilon}))_{\varepsilon>0}$ is bounded in
$L^{1}(\Sigma,\mu)$. By continuity of $\phi$ and since
$\norm{u_{\varepsilon}}_{\infty}\le \norm{g}_{\infty}$, one has
$\lim_{\varepsilon\to0+}\varepsilon \phi(u_{\varepsilon})=0$ in
$L^{\infty}(\Sigma,\mu)$. Thus, under the
hypothesis~\eqref{propo:Range-cond-in-Rd-Hyp-3}, for every
$g\in L^{1}\cap L^{\infty}(\Sigma,\mu)$, ~\eqref{eq:231}
holds. In particular, the right hand side in~\eqref{eq:7} tends to
zero as $n$, $m\to\infty$ showing that ~\eqref{eq:77} holds and
hence  \eqref{eq:224} holds for some function
$u\in L^{1}\cap L^{\infty}(\Sigma,\mu)$ provided
hypothesis~\eqref{propo:Range-cond-in-Rd-Hyp-3} holds. Next, suppose
that  hypotheses~\eqref{propo:Range-cond-in-Rd-Hyp-1}
and~\eqref{propo:Range-cond-in-Rd-Hyp-2} hold. Since
$u_{\varepsilon}=J_{\lambda}^{\varepsilon\phi+A_{1\cap\infty}\phi}g$
can be rewritten as
$u_{\varepsilon}=J_{\lambda}^{A_{1\cap\infty}\phi}[g-\lambda\varepsilon\phi(u_{\varepsilon})]$,
the accretivity of $A_{1\cap\infty}\phi$ in $L^{1}(\Sigma,\mu)$ yields
\begin{equation}
  \label{eq:234}
  \norm{u_{\varepsilon}-u_{\eta}}_{1}\le \lambda\;
  \norm{\varepsilon\phi(u_{\varepsilon})-\eta\phi(u_{\eta})}_{1}
\end{equation}
for every $\varepsilon$, $\eta>0$. If
hypothesis~\eqref{propo:Range-cond-in-Rd-Hyp-1} holds,
then for every $g\in L^{1}\cap L^{\infty}(\Sigma,\mu)$, there is
another $K_{1}>0$ such that
\begin{equation}
  \label{eq:21}
  \abs{\phi(r)}\le K_{1}\,\abs{r}\qquad\text{for every $\abs{r}\le \norm{g}_{\infty}$.}
\end{equation}
Moreover, since $\varepsilon\phi + A_{1\cap \infty}\phi$ has a complete resolvent,
$u_{\varepsilon}=J_{\lambda}^{\varepsilon\phi+A_{1}\phi}g$ satisfies
\begin{equation}
  \label{eq:226}
  \norm{u_{\varepsilon}}_{\infty}\le \norm{g}_{\infty}
\end{equation}
for every $\varepsilon>0$ and so,~\eqref{eq:21} yields
\begin{equation}
  \label{eq:227}
  \abs{\phi(u_{\varepsilon})}\le
  K_{1}\,\abs{u_{\varepsilon}}\qquad\text{for a.e. $x\in \Sigma$
    and all $\varepsilon>0$.}   
\end{equation}
Thus and since $\norm{u_{\varepsilon}}_{\tilde{q}}\le
\norm{g}_{\tilde{q}}$ for every $1\le q\le \infty$, it
follows that
\begin{displaymath}
  \norm{\varepsilon\phi(u_{\varepsilon})}_{\tilde{q}}\le 
  K_{1}\,\varepsilon\,\norm{u_{\varepsilon}}_{\tilde{q}}\le K_{1}\,\varepsilon\,\norm{g}_{\tilde{q}},
\end{displaymath}
for every for every $1\le q\le \infty$, from where we can conclude
that the sequence $(u_{\varepsilon})_{\varepsilon>0}$ has
limit~\eqref{eq:231} under
hypothesis~\eqref{propo:Range-cond-in-Rd-Hyp-1}. In particular,
by~\eqref{eq:234}, $(u_{\varepsilon})_{\varepsilon>0}$ is a Cauchy
sequence in $L^{1}(\Sigma,\mu)$. Therefore and by \eqref{eq:226},
 \eqref{eq:224} holds for some function
$u\in L^{1}\cap L^{\infty}(\Sigma,\mu)$ also under
hypothesis~\eqref{propo:Range-cond-in-Rd-Hyp-1}. Moreover, if we
assume that hypothesis~\eqref{propo:Range-cond-in-Rd-Hyp-2} holds,
then the continuous embedding of $L^{\infty}(\Sigma,\mu)$ into
$L^{1}(\Sigma,\mu)$ and the boundedness of $\phi$ on
$[-\norm{g}_{\infty},\norm{g}_{\infty}]$ imply that
~\eqref{eq:231} holds. Thus and by \eqref{eq:226}, 
\eqref{eq:224} holds for some function
$u\in L^{1}\cap L^{\infty}(\Sigma,\mu)$ also under
hypothesis~\eqref{propo:Range-cond-in-Rd-Hyp-2}.

With these preliminaries, we can begin proving the statements of this
proposition. First, we show that $A_{\phi}$ is an extension of
$A_{1\cap\infty}\phi$ in $L^{1}(\Sigma,\mu)$. To do so, let
$(\hat{u},\hat{v})\in A_{1\cap\infty}\phi$ and for $\lambda>0$, set
$g=\hat{u}+\lambda \hat{v}$. Then, $f:=\frac{g-\hat{u}}{\lambda}=\hat{v}$ and for every
$\varepsilon>0$ sufficiently small, there is a unique
$u_{\varepsilon}=J_{\lambda}^{\varepsilon\phi+A_{1\cap\infty}\phi}g\in
D(A_{1\cap\infty}\phi)$ and there is a function $u\in L^{1}\cap
L^{\infty}(\Sigma,\mu)$ satisfying ~\eqref{eq:224}.
Since $\hat{u}=J_{\lambda}^{A_{1\cap\infty}\phi}g$ can be rewritten as
$\hat{u}=J_{\lambda}^{\varepsilon\phi+A_{1\cap\infty}\phi}[g+\lambda\varepsilon\phi(\hat{u})]$
and since the operator $\varepsilon\phi+A_{1\cap\infty}\phi$ is accretive in
$L^{1}(\Sigma,\mu)$,
\begin{displaymath}
  \norm{u_{\varepsilon}-\hat{u}}_{1}=
  \norm{J_{\lambda}^{\varepsilon\phi+A_{1\cap\infty}\phi}g-
    J_{\lambda}^{\varepsilon\phi+A_{1\cap\infty}\phi}[g+\lambda\varepsilon\phi(\hat{u})]}_{1}
  \le \lambda\,\varepsilon\,\norm{\phi(\hat{u})}_{1}.
\end{displaymath}
Since by assumption, $\hat{u}\in D(A_{1\cap\infty}\phi)$, one has
$\phi(u)\in L^{1}(\Sigma,\mu)$. Thus, sending $n\to\infty$ in the last
inequality yields $u=\hat{u}$ and so, $u\in D(A_{\phi})$ with
$f=v\in A_{\phi}u$.

Next, we show that $A_{\phi}$ is contained in the closure
$\overline{A_{1\cap\infty}\phi}$ of $A_{1\cap\infty}\phi$ in
$L^{1}(\Sigma,\mu)$. Let $(u,f)\in A_{\phi}$. Then, by definition
of $A_{\phi}$, there are $\lambda>0$ and
$g\in L^{1}\cap L^{\infty}(\Sigma,\mu)$ such that
$f=\frac{g-u}{\lambda}$ and for every $\varepsilon>0$ sufficiently
small, there is $u_{\varepsilon}=J_{\lambda}^{\varepsilon\phi+A_{1\cap
    \infty}\phi}g\in L^{1}\cap L^{\infty}(\Sigma,\mu)$ and $u \in L^{1}\cap L^{\infty}(\Sigma,\mu)$
satisfying ~\eqref{eq:224}.  By definition of the resolvent
$J_{\lambda}^{\varepsilon\phi+A_{1\cap\infty}\phi}$ of
$\varepsilon\phi+A_{1\cap\infty}\phi$, one has
\begin{displaymath}
    (u_{\varepsilon},\frac{g-u_{\varepsilon}}{\lambda}-\varepsilon\phi(u_{\varepsilon}))\in
    A_{1\cap\infty}\phi
\end{displaymath}
and so, by ~\eqref{eq:231} for $\tilde{q}=1$, we can conclude
that $(u,\frac{g-u}{\lambda})=(u,f)\in \overline{A_{1\cap\infty}\phi}$. 
  
The operator $A_{\phi}$ is accretive in $L^{1}(\Sigma,\mu)$ since by
construction of $A_{\phi}$, the operator $A_{\phi}$ is contained in
the \emph{limit inferior}
$\liminf_{\varepsilon\to0+}(\varepsilon\phi+A_{1}\phi)$ (see, for
instance, \cite[Definition~(2.17) and Proposition~(2.18)]{Benilanbook}
or \cite[Proposition 4.4]{MR2582280}) of the family
$(\varepsilon\phi+A_{1\cap\infty}\phi)_{\varepsilon>0}$ of accretive
operators $\varepsilon\phi+A_{1\cap \infty}\phi$ in
$L^{1}(\Sigma,\mu)$ (see
Proposition~\ref{propo:composition-operators-in-L1}). Moreover,
$A_{\phi}$ is $m$-accretive in $L^{1}(\Sigma,\mu)$. To see that
$A_{\phi}$ satisfies the range condition~\eqref{eq:range-condition}
for $X=L^{1}(\Sigma,\mu)$, note that under the
hypotheses~\eqref{propo:Range-cond-in-Rd-Hyp-3}-\eqref{propo:Range-cond-in-Rd-Hyp-2},
$A_{\phi}$ is closed in $L^{1}(\Sigma,\mu)$. Hence, it is sufficient
to show that the set
\begin{equation}
\label{eq:230}
L^{1}\cap L^{\infty}(\Sigma,\mu) \subseteq Rg (I+\lambda A_{\phi}).
\end{equation}
To this end, let $g\in L^{1}\cap L^{\infty}(\Sigma,\mu)$ and
$\lambda>0$. Then, by following the arguments in the first part of
this proof, we see that for every $\varepsilon>0$, there is
$u_{\varepsilon}=J_{\lambda}^{\varepsilon\phi+A_{1\cap
    \infty}\phi}g\in D(A_{1\cap\infty}\phi)$.
Since
$u_{\varepsilon}=J_{\lambda}^{\varepsilon\phi+A_{1\cap\infty}\phi}g$
is equivalent to
$u_{\varepsilon}=J_{\lambda}^{A_{1\cap\infty}\phi}[g-\lambda\varepsilon
\phi(u_{\varepsilon})]$, we have that
\begin{displaymath}
  \norm{u_{\varepsilon}-u_{\eta}}_{1}\le \lambda \norm{\varepsilon
    \phi(u_{\varepsilon})-\eta \phi(u_{\eta})}_{1}
\end{displaymath}
for every $\varepsilon$, $\eta>0$. Thus and since under the
hypotheses~\eqref{propo:Range-cond-in-Rd-Hyp-3}-\eqref{propo:Range-cond-in-Rd-Hyp-2},
~\eqref{eq:231} holds for $\tilde{q}=1$, we can conclude from the
previous inequality that $(u_{\varepsilon})_{\varepsilon>0}$ is a
Cauchy sequence in $L^{1}(\Sigma,\mu)$ as
$\varepsilon\to0+$. Therefore, there is an $u\in L^{1}(\Sigma,\mu)$
such that~\eqref{eq:224} holds and so, by definition of $A_{\phi}$,
$(u,f)\in A_{\phi}$ with $f:=\frac{g-u}{\lambda}\in L^{1}(\Sigma,\mu)$
and $g=u+\lambda f\in (I+\lambda A_{\phi})u$. Thus, the range condition~\eqref{eq:230} holds.

Summarising, we have shown that $A_{\phi}$ in contained in the closure
$\overline{A_{1\cap\infty}\phi}$ of $A_{1\cap\infty}\phi$ in
$L^{1}(\Sigma,\mu)$ and $A_{\phi}$ is $m$-accretive in $L^{1}(\Sigma,\mu)$. 
Thus and since $\overline{A_{1\cap\infty}\phi}$ is accretive in $L^{1}(\Sigma,\mu)$,
statement~\eqref{eq:114} implies that $\overline{A_{1\cap\infty}\phi}=A_{\phi}$.

Next, we show that the range condition~\eqref{eq:182} holds under the
hypotheses~\eqref{propo:Range-cond-in-Rd-Hyp-1} and
\eqref{propo:Range-cond-in-Rd-Hyp-2}. For this, let
$g\in L^{1}\cap L^{\infty}(\Sigma,\mu)$. Then, for every
$\varepsilon>0$ sufficiently small,
\begin{displaymath}
 u_{\varepsilon}=J_{\lambda}^{\varepsilon\phi+A_{1\cap\infty}\phi}g\in
L^{1}\cap L^{\infty}(\Sigma,\mu)\quad\text{ and }\quad 
(\phi(u_{\varepsilon}),\frac{g-u_{\varepsilon}}{\lambda}-\varepsilon\phi(u_{\varepsilon}))\in A.
\end{displaymath}
Thus, if we can show that
\begin{equation}
  \label{eq:181}
  \lim_{\varepsilon\to0+}\phi(u_{\varepsilon})=\phi(u)\qquad\text{in
    $L^{q}(\Sigma,\mu)$}
\end{equation}
and
\begin{equation}
  \label{eq:221}
  \lim_{\varepsilon\to0+}
  \frac{g-u_{\varepsilon}}{\lambda}-\varepsilon
  \phi(u_{\varepsilon})=\frac{g-u}{\lambda}
  \qquad\text{in $L^{q}(\Sigma,\mu)$}
\end{equation}
then by the assumption, $A$ is $m$-accretive in the uniformly convex
Banach space $L^{q}(\Sigma,\mu)$
(cf.~\cite[Proposition~3.4]{MR2582280}), we have that
\begin{displaymath}
  \Big(\phi(u),\frac{g-u}{\lambda}\Big)\in A.
\end{displaymath}
To see that \eqref{eq:181}
holds, recall that by~\eqref{eq:226} and~\eqref{eq:224}, one has
\begin{equation}
  \label{eq:233}
  \lim_{\varepsilon\to0+}u_{\varepsilon}=u\qquad\text{ in $L^{q}(\Sigma,\mu)$.}
\end{equation}
If hypothesis~\eqref{propo:Range-cond-in-Rd-Hyp-1} holds, then
combining~\eqref{eq:233} with the continuity of $\varphi$ and by
\eqref{eq:227}, it follows that~\eqref{eq:181} holds and
\begin{displaymath}
  \norm{\varepsilon\phi(u_{\varepsilon})}_{q}\le
  \varepsilon\,K_{1}\,\norm{u_{\varepsilon}}_{q}\le \varepsilon\,K_{1}\norm{g}_{q},
\end{displaymath}
from where we can conclude that~\eqref{eq:221} holds.  If
hypothesis~\eqref{propo:Range-cond-in-Rd-Hyp-2} holds, then
by~\eqref{eq:233}, the continuity of $\phi$, and by eventually passing
to a subsequence, we see that
$\lim_{\varepsilon\to0+}\phi(u_{\varepsilon}(x))=\phi(u(x))$ a.e. on
$\Sigma$. Thus, by~\eqref{eq:226} and the embedding of
$L^{\infty}(\Sigma,\mu)$ into $L^{q}(\Sigma,\mu)$, we see
that~\eqref{eq:181} and \eqref{eq:221} hold. Moreover, using that
$\norm{u_{\varepsilon}}_{p}\le \norm{g}_{p}$ for all $\varepsilon>0$
and $1\le p\le \infty$, we can conclude that
$u\in L^{1}\cap L^{\infty}(\Sigma,\mu)$ and by the
hypotheses~\eqref{propo:Range-cond-in-Rd-Hyp-1} and
\eqref{propo:Range-cond-in-Rd-Hyp-2}, that
$\phi(u)\in L^{1}\cap L^{\infty}(\Sigma,\mu)$. Thus, 
\begin{displaymath}
    \Big(\phi(u),\frac{g-u}{\lambda}\Big)\in A_{1\cap\infty},
  \end{displaymath}
proving the range condition~\eqref{eq:182}. This completes the proof of
this proposition.
\end{proof}

%%%%%%%%%%%%%%%%%%%%%%%%%%%%%%%%%%%%%%%%%%%%%%%%%%%
%
%                            APPENDIX: Section - Mean spaces
%
%%%%%%%%%%%%%%%%%%%%%%%%%%%%%%%%%%%%%%%%%%%%%%%%%%%

\section{The link between mean spaces and $L^{p}$}

The first part of the following theorem has been proved
in~\cite[Th\'eor\`eme 1.1 of Chapter IV]{MR0165343} by using so-called \emph{discrete
mean spaces} (cf.~\cite[Chapter II]{MR0165343}). Here, we improve this result
by showing that both spaces are isometrically isomorphic. This result
serves us in the proof of Theorem~\ref{thm:extrapolation-to-infty} and
Theorem~\ref{thm:extrapolation-to-infty-bis} to determine the
convergence of the constants in inequality~\eqref{eq:24} as $m\to \infty$.

\begin{theorem}\label{thmApp:Lpspaces}
  Let $(\Sigma,\mu)$ be a $\sigma$-finite measure space,
  $(X_{0},X_{1})$ be an interpolation couple, $1\le p_{0}$,
  $p_{1}\le \infty$ and $0<\theta<1$. Then for $1\le p\le \infty$ given
  by
  \begin{equation}
    \label{eq:1App}
    \frac{1}{p}=\frac{1-\theta}{p_{0}}+\frac{\theta}{p_{1}},
  \end{equation}
 one has that
  \begin{equation}
    \label{eq:2App}
    (L^{p_{0}}(\Sigma,X_{0};\mu),L^{p_{1}}(\Sigma,X_{1};\mu))_{\theta,p_{0},p_{1}}
    =L^{p}(\Sigma,(X_{0},X_{1})_{\theta,p_{0},p_{1}};\mu)
  \end{equation}
  with equal norms.
\end{theorem}

\begin{proof}[Proof of Theorem~\ref{thmApp:Lpspaces}]
  We only outline the proof for $1\le p_{0}<\infty$ and $1\le
  p_{1}<\infty$ since the other cases are shown similarly.

  First, let $u$ be an element of
  $(L^{p_{0}}(\Sigma,X_{0};\mu),L^{p_{1}}(\Sigma,X_{1};\mu))_{\theta,p_{0},p_{1}}$. By
  definition, there are measurable functions
  $v_{i}: (0,\infty)\to L^{p_{i}}(\Sigma,\mu)$ for $i=0,1$ such that
  $t^{-\theta}v_{0}\in L^{p_{0}}_{\ast}(L^{p_{0}}(\Sigma,X_{0};\mu))$,
  $t^{1-\theta}v_{1}\in L^{p_{1}}(\Sigma,X_{1};\mu)$ and
  \begin{displaymath}
   u(x)=v_{0}(t,x)+v_{1}(t,x)
 \end{displaymath}
  for a.e. $(t,x)\in (0,\infty)\times \Sigma$. Since $(\Sigma,\mu)$ and
  $(\R_{+},\tfrac{dt}{t})$ are both $\sigma$-finite measure spaces, Fubini's
  theorem implies that
  \begin{displaymath}
    t^{-\theta}v_{0}(\cdot,x)\in L^{p_{0}}_{\ast}(X_{0})\quad\text{and}\quad
    t^{1-\theta}v_{1}(\cdot,x)\in L^{p_{1}}(\Sigma,\mu)
  \end{displaymath}
  for a.e. $x\in \Sigma$. Thus by definition of the mean space and
  by~\eqref{eq:51}, one has $a(x)\in (X_{0},X_{1})_{\theta,p_{0},p_{1}}$ for a.e. $x\in \Sigma$
  and
  \begin{displaymath}
    \norm{u(x)}_{(X_{0},X_{1})_{\theta,p_{0},p_{1}}}\le \norm{
      t^{-\theta}v_{0}(\cdot,x)}_{L^{p_{0}}_{\ast}(X_{0})}^{1-\theta}\,
    \norm{t^{1-\theta}v_{1}(\cdot,x)}_{L^{p_{1}}_{\ast}(X_{1})}^{\theta}.
  \end{displaymath}
  Integrating the latter inequality over $\Sigma$, taking $p$th root,
  applying H\"older's inequality (where one uses~\eqref{eq:1App}) and
  Fubini's theorem, we see that
  \begin{align*}
    &\norm{u}_{L^{p}(\Sigma,(X_{0},X_{1})_{\theta,p_{0},p_{1}};\mu)}\\
    &\qquad \le
          \left[\int_{\Sigma}\norm{
      t^{-\theta}v_{0}(\cdot,x)}_{L^{p_{0}}_{\ast}(X_{0})}^{p_{0}}\,\dmu\right]^{\frac{1-\theta}{p_{0}}}\;
      \left[\int_{\Sigma}\norm{t^{1-\theta}v_{1}(\cdot,x)}_{L^{p_{1}}_{\ast}(X_{1})}^{p_{1}}\,\dmu\right]^{\frac{\theta}{p_{1}}}\\
    &\qquad =\norm{t^{-\theta}v_{0}}_{L^{p_{0}}_{\ast}(L^{p_{0}}(\Sigma,X_{0};\mu))}^{1-\theta}\; 
      \norm{t^{1-\theta}v_{1}}_{L^{p_{1}}_{\ast}(L^{p_{1}}(\Sigma,X_{1};\mu))}^{\theta}.
  \end{align*}
  Taking in this inequality the infimum over all representation pairs
  $(v_{0},v_{1})$ of $u$ and applying~\eqref{eq:51} yields
  \begin{displaymath}
    \norm{u}_{L^{p}(\Sigma,(X_{0},X_{1})_{\theta,p_{0},p_{1}};\mu)}\le 
    \norm{u}_{ (L^{p_{0}}(\Sigma,X_{0};\mu),L^{p_{1}}(\Sigma,X_{1};\mu))_{\theta,p_{0},p_{1}}}.
  \end{displaymath}
  
  Now, let $u\in
  L^{p}(\Sigma,(X_{0},X_{1})_{\theta,p_{0},p_{1}};\mu)$ be a step
  function given by
  \begin{displaymath}
    u(x)=\sum_{\nu=1}^{m}a_{\nu}\,\mathds{1}_{B_{\nu}}(x)
  \end{displaymath}
  for finitely many different values
  $a_{\nu}\in (X_{0},X_{1})_{\theta,p_{0},p_{1}}$ attained on pairwise
  disjoint measurable subsets $B_{\nu}$ of $\Sigma$. Let
  $\varepsilon>0$. By the definition of
  $(X_{0},X_{1})_{\theta,p_{0},p_{1}}$ and the infimum, for every
  $\nu=1,\dots,m$, there are measurable functions
  $v_{i\nu} : (0,\infty)\to X_{i}$ for $i=0,1$ satisfying
  \begin{equation}\label{eqApp:rel-anu-v0-v1}
    a_{\nu}=v_{0\nu}(t)+v_{1\nu}(t)
  \end{equation}
  for a.e. $t\in (0,\infty)$ and
  \begin{equation}\label{eqApp:a-nu}
    \max\Big\{\norm{
      t^{-\theta}v_{0\nu}}_{L^{p_{0}}_{\ast}(X_{0})},
    \norm{t^{1-\theta}v_{1\nu}}_{L^{p_{1}}_{\ast}(X_{1})}\Big\}\le
    \norm{a_{\nu}}_{(X_{0},X_{1})_{\theta,p_{0},p_{1}}}\,(1+\varepsilon).
  \end{equation}
  Set $\lambda=(p_{0}-p)/\theta\,p_{0}$ and for every $\nu=1,\dots,m$
  and $i=0,1$ define
  \begin{displaymath}
    w_{i\nu}(t)=v_{i\nu}(\norm{a_{\nu}}_{(X_{0},X_{1})_{\theta,p_{0},p_{1}}}^{\lambda}\,t).
  \end{displaymath}
  Then applying the substitution
  $s=\norm{a_{\nu}}_{(X_{0},X_{1})_{\theta,p_{0},p_{1}}}^{\lambda}\,t$
  together with~\eqref{eqApp:a-nu} yields
  \begin{equation}\label{eqApp:est-w0}
  \begin{split}
    \norm{t^{-\theta}w_{0\nu}}_{L^{p_{0}}_{\ast}(X_{0})}^{p_{0}}
    &=
      \norm{a_{\nu}}_{(X_{0},X_{1})_{\theta,p_{0},p_{1}}}^{-\lambda\,\theta\,p_{0}}\,
      \int_{0}^{\infty}\norm{s^{-\theta}v_{0\nu}(s)}_{X_{0}}^{p_{0}}\,\tfrac{\ds}{s}\\
    & \le (1+\varepsilon)^{p_{0}}
      \norm{a_{\nu}}_{(X_{0},X_{1})_{\theta,p_{0},p_{1}}}^{p_{0}-\lambda\,\theta\,p_{0}}\\
    & = (1+\varepsilon)^{p_{0}}\;
      \norm{a_{\nu}}_{(X_{0},X_{1})_{\theta,p_{0},p_{1}}}^{p}
  \end{split}
\end{equation}
for every $\nu=1,\dots,m$. By the relation~\eqref{eq:1App}, one sees
that the same $\lambda$ satisfies
$p_{1}+\lambda(1-\theta)p_{1}=p$. Thus the same arguments gives
\begin{equation}\label{eqApp:est-w1}
  \norm{t^{1-\theta}w_{1\nu}}_{L^{p_{1}}_{\ast}(X_{1})}^{p_{1}}\le (1+\varepsilon)^{p_{1}}\;
  \norm{a_{\nu}}_{(X_{0},X_{1})_{\theta,p_{0},p_{1}}}^{p}.
\end{equation}
For $i=0,1$ and every $t\in (0,\infty)$, we define the step functions
\begin{displaymath}
  w_{i}(t,x)=\sum_{\nu=1}^{m}w_{i\nu}(t)\,\mathds{1}_{B_{\nu}}(x)
\end{displaymath}
for a.e. $x\in \Sigma$. Then by~\eqref{eqApp:est-w0}
and~\eqref{eqApp:est-w1} as well as by Fubini's theorem, 
\begin{equation}\label{eqApp:w0-Lp0astLp0}
\begin{split}
  \int_{0}^{\infty}\norm{t^{-\theta}w_{0}(t,\cdot)}_{L^{p_{0}}(\Sigma,X_{0};\mu)}^{p_{0}}\,\tfrac{\dt}{t}
  & = \int_{\Sigma}
    \norm{t^{-\theta}w_{0}(\cdot,x)}_{L^{p_{0}}_{\ast}(X_{0})}^{p_{0}}\,\dmu\\
  &\le
    (1+\varepsilon)^{p_{0}}\,\sum_{i=1}^{m}\norm{a_{\nu}}_{(X_{0},X_{1})_{\theta,p_{0},p_{1}}}^{p}\mu(B_{\nu})\\
  & = (1+\varepsilon)^{p_{0}}\,\norm{u}_{L^{p}(\Sigma,(X_{0},X_{1})_{\theta,p_{0},p_{1}};\mu)}^{p}
\end{split}
\end{equation}
and similarly,
\begin{equation}\label{eqApp:w1-Lp1astLp1}
  \int_{0}^{\infty}\norm{t^{1-\theta}w_{1}(t,\cdot)}_{L^{p_{1}}(\Sigma,X_{1};\mu)}^{p_{1}}\,\tfrac{\dt}{t}\le 
  (1+\varepsilon)^{p_{1}}\,\norm{u}_{L^{p}(\Sigma,(X_{0},X_{1})_{\theta,p_{0},p_{1}};\mu)}^{p}.
\end{equation}
Therefore, for $i=0, 1$, the functions $w_{i} : (0,\infty)\to
L^{p_{i}}(\Sigma,X_{i};\mu)$ are well defined step functions and so strongly
measurable. In addition, by~\eqref{eqApp:rel-anu-v0-v1},
\begin{align*}
  w_{0}(t,x)+w_{1}(t,x)
  & =
    \sum_{\nu=1}^{m}(v_{0\nu}(\norm{a_{\nu}}_{(X_{0},X_{1})_{\theta,p_{0},p_{1}}}^{\lambda}\,t)
    +v_{1\nu}(\norm{a_{\nu}}_{(X_{0},X_{1})_{\theta,p_{0},p_{1}}}^{\lambda}\,t))\,\mathds{1}_{B_{\nu}}(x)\\
  & =\sum_{\nu=1}^{m}a_{\nu}\,\mathds{1}_{B_{\nu}}(x)\\
  & =u(x)
\end{align*}
for a.e. $x\in \Sigma$. Thus $u\in
(L^{p_{0}}(\Sigma,X_{0};\mu),L^{p_{1}}(\Sigma,X_{1};\mu))_{\theta,p_{0},p_{1}}$
and by~\eqref{eq:51}, \eqref{eqApp:w0-Lp0astLp0},
\eqref{eqApp:w1-Lp1astLp1}, and~\eqref{eq:1App},
\begin{align*}
  \norm{u}_{(L^{p_{0}}(\Sigma,X_{0};\mu),L^{p_{1}}(\Sigma,X_{1};\mu))_{\theta,p_{0},p_{1}}}
  & \le
    \norm{t^{-\theta}w_{0}}_{L^{p_{0}}_{\ast}(L^{p_{0}}(\Sigma,X_{0};\mu))}^{1-\theta}\,
    \norm{t^{1-\theta}w_{1}}_{L^{p_{1}}_{\ast}(L^{p_{1}}(\Sigma,X_{1};\mu))}^{\theta}\\
  & \le
    (1+\varepsilon)\;\norm{u}_{L^{p}(\Sigma,(X_{0},X_{1})_{\theta,p_{0},p_{1}};\mu)}^{\frac{(1-\theta)
    p}{p_{0}}+\frac{\theta p}{p_{1}}}\\
  & = (1+\varepsilon)\;\norm{u}_{L^{p}(\Sigma,(X_{0},X_{1})_{\theta,p_{0},p_{1}};\mu)}.
\end{align*}
Sending $\varepsilon\to0+$ shows that inequality
\begin{displaymath}
  \norm{u}_{(L^{p_{0}}(\Sigma,X_{0};\mu),L^{p_{1}}(\Sigma,X_{1};\mu))_{\theta,p_{0},p_{1}}}
  \le \norm{u}_{L^{p}(\Sigma,(X_{0},X_{1})_{\theta,p_{0},p_{1}};\mu)}.
\end{displaymath}
holds for step functions. Since the set of step functions is dense
in the space $L^{p}(\Sigma,(X_{0},X_{1})_{\theta,p_{0},p_{1}};\mu)$ the claim of
this theorem holds.
\end{proof}

As an immediate consequence of Theorem~\ref{thmApp:Lpspaces}, we
obtain the following corollary improving the statement
in~\cite[Corollaire 1.1 of Chapter IV]{MR0165343}.

\begin{corollary}\label{corApp:1}
  Let $(\Sigma,\mu)$ be a $\sigma$-finite measure space, $1\le p_{0}$,
  $p_{1}\le \infty$ and $0<\theta<1$. Then for $1\le p\le \infty$
  satisfying the relation~\eqref{eq:1App}, one has that
  \begin{displaymath}
    (L^{p_{0}}(\Sigma,\mu),L^{p_{1}}(\Sigma,\mu))_{\theta,p_{0},p_{1}}
    =L^{p}(\Sigma,\mu)
  \end{displaymath}
  with equal norms.
\end{corollary}

%%%%%%%%%%%%%%%%%%%%%%%%%%%%%%%%%%%%%%%%%%%%%%%%%%%
%
%                                   R E F E R E N C E S
%
%%%%%%%%%%%%%%%%%%%%%%%%%%%%%%%%%%%%%%%%%%%%%%%%%%%

%\bibliographystyle{doi}
%\bibliography{citations}

\def\cprime{$'$} \def\cprime{$'$} \def\cprime{$'$} \def\cprime{$'$}
  \def\ocirc#1{\ifmmode\setbox0=\hbox{$#1$}\dimen0=\ht0 \advance\dimen0
  by1pt\rlap{\hbox to\wd0{\hss\raise\dimen0
  \hbox{\hskip.2em$\scriptscriptstyle\circ$}\hss}}#1\else {\accent"17 #1}\fi}
  \def\cprime{$'$} \def\cprime{$'$}
  \def\ocirc#1{\ifmmode\setbox0=\hbox{$#1$}\dimen0=\ht0 \advance\dimen0
  by1pt\rlap{\hbox to\wd0{\hss\raise\dimen0
  \hbox{\hskip.2em$\scriptscriptstyle\circ$}\hss}}#1\else {\accent"17 #1}\fi}
\providecommand{\bysame}{\leavevmode\hbox to3em{\hrulefill}\thinspace}

\end{document}